\pdfoutput=1
\documentclass[11pt,oneside,reqno]{amsart}

\usepackage{amsmath, amssymb, amsthm, amsfonts, mathrsfs, mathtools}
			
\setbox0=\hbox{$+$}
\newdimen\plusheight
\plusheight=\ht0
\def\+{\;\lower\plusheight\hbox{$+$}\;}

\newdimen\minusheight
\minusheight=\ht0
\def\-{\;\lower\minusheight\hbox{$-$}\;}

\setbox0=\hbox{$\cdots$}
\newdimen\cdotsheight
\cdotsheight=\plusheight%\ht0
\def\cds{\lower\cdotsheight\hbox{$\cdots$}}
\usepackage{tikz}
\usetikzlibrary{matrix,decorations.pathreplacing}

\theoremstyle{definition}

\usepackage{graphicx, epsfig}

\theoremstyle{definition}

\makeatletter
\renewcommand\contentsnamefont{\bfseries}
\def\@starttoc#1#2{\begingroup
\setTrue{#1}%
\par\removelastskip\vskip\z@skip
\@startsection{}\@M\z@{\linespacing\@plus\linespacing}%
{.5\linespacing}{%\centering
\contentsnamefont}{#2}%
\ifx\contentsname#2%
\else \addcontentsline{toc}{section}{#2}\fi
\makeatletter
\@input{\jobname.#1}%
\if@filesw
\@xp\newwrite\csname tf@#1\endcsname
\immediate\@xp\openout\csname tf@#1\endcsname \jobname.#1\relax
\fi
\global\@nobreakfalse \endgroup
\addvspace{32\p@\@plus14\p@}%
\let\tableofcontents\relax
}
\def\contentsname{Contents}
\def\l@section{\@tocline{2}{.5ex}{0mm}{5pc}{}}
\def\l@subsection{\@tocline{2}{0pt}{2em}{5pc}{}}
\def\l@subsubsection{\@tocline{2}{0pt}{3em}{5pc}{}}
\def\l@subsubsubsection{\@tocline{2}{0pt}{4em}{5pc}{}}
\makeatother
\DeclareMathOperator{\codim }{codim}
\usepackage[strict]{changepage}
\makeatother

\numberwithin{equation}{section}
\theoremstyle{plain}
\newtheorem{theorem}{Theorem}[section]

\newtheorem{question}[theorem]{Question}
\newtheorem{conjecture}[theorem]{Conjecture}

\newtheorem{corollary}[theorem]{Corollary}
\newtheorem{proposition}[theorem]{Proposition}
\newtheorem{definition}[theorem]{Definition}
\newtheorem{assumption}[theorem]{Assumption}
\newtheorem{lemma}[theorem]{Lemma}
\newtheorem{claim}[theorem]{Claim}
\newtheorem{remark}[theorem]{Remark}
\newtheorem{convention}[theorem]{Convention}

\usepackage{titlesec}

\usepackage{graphicx, epsfig}
\theoremstyle{definition}
\usepackage{graphicx} 
\usepackage{color} 
\usepackage{amsthm}
\usepackage{slashed}
\usepackage{changepage}
\usepackage{blkarray}
\usepackage{transparent} 
\usepackage{titlesec}
\usepackage{changepage}
\usepackage{lipsum}
\makeatletter

\titleformat{\section}{\normalfont\LARGE\bfseries}
\titleformat{\subsubsection}{\normalfont\large\bfseries}

\renewcommand\section{\@startsection{section}{1}{\z@}%
                                  {-3.5ex \@plus -1ex \@minus -.2ex}%
                                 {2.3ex \@plus.2ex}%
                                 {\normalfont\Large\bfseries}}
\renewcommand\subsection{\@startsection{subsection}{1}{\z@}%
                                  {-3.5ex \@plus -1ex \@minus -.2ex}%
                                 {2.3ex \@plus.2ex}%
                                 {\normalfont\normalsize\bfseries}}
\renewcommand\subsubsection{\@startsection{subsubsection}{1}{\z@}%
                                  {-3.5ex \@plus -1ex \@minus -.2ex}%
                                 {2.3ex \@plus.2ex}%
                                 {\normalfont\normalsize\bfseries}}
\titlespacing*{\paragraph}
{0pt}{3.25ex plus 1ex minus .2ex}{1.5ex plus .2ex}

\usepackage{titlesec}
\usepackage{hyperref}
\usepackage{enumitem}

\newlist{alphalist}{enumerate}{1}
\setlist[alphalist,1]{label=\textbf{\alph*.}}

\titleclass{\subsubsubsection}{straight}[\subsection]

\newcounter{subsubsubsection}[subsubsection]
\renewcommand\thesubsubsubsection{\thesubsubsection.\arabic{subsubsubsection}}

\newcommand{\newparallel}{\mathrel{\mathpalette\new@parallel\relax}}
\newcommand{\new@parallel}[2]{%
  \begingroup
  \sbox\z@{$#1T$}% get the height of an uppercase letter
  \resizebox{!}{\ht\z@}{\raisebox{\depth}{$\m@th#1/\mkern-5mu/$}}%
  \endgroup}
\makeatletter
\newcommand{\notparallel}{%
  \mathrel{\mathpalette\not@parallel\relax}%
}
\newcommand{\not@parallel}[2]{%
  \ooalign{\reflectbox{$\m@th#1\smallsetminus$}\cr\hfil$\m@th#1\newparallel$\cr}%
}
\titleformat{\subsubsubsection}
  {\normalfont\normalsize}{\thesubsubsubsection}{1em}{}
\titlespacing*{\subsubsubsection}
{0pt}{3.25ex plus 1ex minus .2ex}{1.5ex plus .2ex}
\makeatother
\usepackage{graphicx} 
\usepackage{color} 
\usepackage{transparent} 
\usepackage{geometry} 
\begin{document}
\title{Classification of Hyperbolic Dehn fillings I}
\author{BoGwang Jeon}
\thanks{{\em 2020 Mathematics Subject Classification}: 57K31, 57K32, 14G25, 14J20.}
\thanks{{\em Key words and phrases}: hyperbolic 3-manifolds, Dehn filling, holonomy variety, Zilber-Pink conjecture.}
%\date{29 February 2024}

\begin{abstract}
Let $\mathcal{M}$ be a $2$-cusped hyperbolic $3$-manifold. By the work of Thurston, the product of the derivatives of the holonomies of core geodesics of each Dehn filling of $\mathcal{M}$ is an invariant of it. In this paper, we classify Dehn fillings of $\mathcal{M}$ with sufficiently large coefficients using this invariant. Further, for any given two Dehn fillings of $\mathcal{M}$ (with sufficiently larger coefficients), if their aforementioned invariants are the same, it is shown their complex volumes are the same as well. 
\end{abstract}
\maketitle
\tableofcontents

\section{Introduction}
\subsection{Main results}
The following is a celebrated result of Thurston \cite{thu}:
\begin{theorem}[Thurston]\label{19100202}
Let $\mathcal{M}$ be an $n$-cusped hyperbolic $3$-manifold and $\mathcal{M}_{p_1/q_1, \dots, p_n/q_n}$ be its $(p_1/q_1,\dots, p_n/q_n)$-Dehn filling. Then $\mathcal{M}_{p_1/q_1, \dots, p_n/q_n}$ is hyperbolic for $|p_k|+|q_k|$ sufficiently large, 
\begin{equation*}
\text{vol}\;\mathcal{M}_{p_1/q_1, \dots, p_n/q_n}<\text{vol}\;\mathcal{M}
\end{equation*}
for any $(p_1/q_1, \dots, p_n/q_n)$ and 
\begin{equation*}
\text{vol}\;\mathcal{M}_{p_1/q_1, \dots, p_n/q_n}\rightarrow \text{vol}\;\mathcal{M}
\end{equation*}
as $|p_k|+|q_k| \rightarrow \infty$ ($1\leq k\leq n$). 
\end{theorem}
As a consequence, it follows that, for $\mathcal{M}$ as defined above, there are only finitely many Dehn fillings of $\mathcal{M}$ with a fixed volume.\footnote{Combining this with J$\slashed{o}$rgensen's theory, a stronger result  is, in fact, true: there are only finitely many hyperbolic $3$-manifolds of a fixed volume \cite{thu}.}  A natural question then to ask is 
\begin{question}
\normalfont Let $\mathcal{M}$ be an $n$-cusped hyperbolic $3$-manifold. Can one classify Dehn fillings of $\mathcal{M}$ via volume?
\end{question}

More specifically, one may ask 
\begin{question}\label{20051905}
\normalfont Let $\mathcal{M}$ be an $n$-cusped hyperbolic $3$-manifold. Does there exist $c$ depending only on $\mathcal{M}$ such that the number $N_{\mathcal{M}}(v)$ of Dehn fillings of $\mathcal{M}$ of the same volume $v$ is bounded by $c$? Further, if
\begin{equation*}
\text{vol}\;\mathcal{M}_{p_1/q_1, \dots, p_n/q_n}=\text{vol}\;\mathcal{M}_{p'_1/q'_1, \dots, p'_n/q'_n},
\end{equation*}
then what is the relationship between $(p_1/q_1, \dots, p_n/q_n)$ and $(p'_1/q'_1, \dots, p'_n/q'_n)$? 
\end{question}

In \cite{jeon3}, the author answered it very partially as below, but we believe the question is widely open in general:
\begin{theorem}[B. Jeon]
Let $\mathcal{M}$ be a $1$-cusped hyperbolic $3$-manifold whose cusp shape is quadratic. Then there exists $c$ depending only on $\mathcal{M}$ such that $N_{\mathcal{M}}(v)<c$ for any $v\in \mathbb{R}$. 
\end{theorem}

The \textit{complex volume}, a generalization of the ordinary volume, of $\mathcal{M}$ is defined as 
\begin{equation*}
\text{vol}_{\mathbb{C}}\;\mathcal{M}:=\text{vol}\;\mathcal{M}+\sqrt{-1}\; \text{CS}\;\mathcal{M} \mod \sqrt{-1}\pi^2\mathbb{Z}
\end{equation*}
where CS $\mathcal{M}$ is the Chern-Simons invariant of $\mathcal{M}$, well defined modulo $\pi^2\mathbb{Z}$. 

Having this definition, analogous to Question \ref{20051905}, we may ask the following question:
\begin{question}
\normalfont Let $\mathcal{M}$ be an $n$-cusped hyperbolic $3$-manifold. Does there exist $c$ depending only on $\mathcal{M}$ such that the number $N^{\mathbb{C}}_{\mathcal{M}}(v)$ of Dehn fillings of $\mathcal{M}$ of the same complex volume $v$ is bounded by $c$? Further, if 
\begin{equation*}
\text{vol}_{\mathbb{C}}\;\mathcal{M}_{p_1/q_1, \dots, p_n/q_n}=\text{vol}_{\mathbb{C}}\;\mathcal{M}_{p'_1/q'_1, \dots, p'_n/q'_n},
\end{equation*}
then what is the relationship between $(p_1/q_1, \dots, p_n/q_n)$ and $(p'_1/q'_1, \dots, p'_n/q'_n)$? 
\end{question}

Topologically, $\mathcal{M}_{p_1/q_1, \dots, p_n/q_n}$ is obtained by adding $n$-geodesics to $\mathcal{M}$ so called the \textit{core geodesics} of $\mathcal{M}_{p_1/q_1, \dots, p_n/q_n}$. By the work of Thurston, they are the $n$-shortest geodesics of $\mathcal{M}_{p_1/q_1,\dots, p_n/q_n}$, whose lengths are all converging to $0$ as $|p_k|+|q_k|\rightarrow\infty$. 

The following theorem is viewed as an effective version of Theorem \ref{19100202} \cite{nz, Y}.

\begin{theorem}[Neumann-Zagier, Yoshida]\label{20080701}
Let $\mathcal{M}$ be an $n$-cusped hyperbolic $3$-manifold. Let $\mathcal{M}_{p_1/q_1, \dots, p_n/q_n}$ be its $(p_1/q_1, \dots, p_n/q_n)$-Dehn filling and 
\begin{equation*}
\{\lambda^1_{(p_1/q_1,\dots, p_n/q_n)}, \dots, \lambda^n_{(p_1/q_1,\dots, p_n/q_n)}\}
\end{equation*} 
be the set of the complex lengths of the core geodesics of $\mathcal{M}_{p_1/q_1, \dots, p_n/q_n}$. Then there exists a function $\epsilon(p_1/q_1, \dots, p_n/q_n)$ such that  
\begin{equation*}
\text{vol}_{\mathbb{C}}\;\mathcal{M}_{p_1/q_1, \dots, p_n/q_n}\equiv\text{vol}_{\mathbb{C}}\;\mathcal{M}-\dfrac{\pi}{2}\Big(\sum_{k=1}^n \lambda^k_{(p_1/q_1,\dots, p_n/q_n)}\Big)+\epsilon(p_1/q_1, \dots, p_n/q_n)\quad \mod \sqrt{-1}\pi^2\mathbb{Z}
\end{equation*}
and
\begin{equation*}
\epsilon(p_1/q_1, \dots, p_n/q_n)\rightarrow0\quad \text{as}\quad |p_k|+|q_k|\rightarrow \infty \;(1\leq k\leq n). 
\end{equation*}
Moreover,
\begin{equation*}
\sum_{k=1}^n \lambda^k_{(p_1/q_1,\dots, p_n/q_n)}=O\Big(\sum_{k=1}^n \dfrac{1}{p_k^2+q_k^2}\Big), \quad \epsilon(p_1/q_1, \dots, p_n/q_n)=O\Big(\sum_{k=1}^n \dfrac{1}{p_k^4+q_k^4}\Big).
\end{equation*}
\end{theorem}

In short, 
\begin{equation*}
\text{vol}_{\mathbb{C}}\;\mathcal{M}-\dfrac{\pi}{2}\sum_{k=1}^n \lambda^k_{(p_1/q_1,\dots, p_n/q_n)}
\end{equation*}
is the majority part of the complex volume of $\mathcal{M}_{p_1/q_1, \dots, p_n/q_n}$ and two invariants are getting closer very rapidly as $|p_k|+|q_k|\rightarrow \infty$. Inspired from this, one has the following definition and subsequent question:

\begin{definition}
\normalfont Using the same notation in the above theorem, the \textit{pseudo complex volume} of $\mathcal{M}_{p_1/q_1, \dots, p_n/q_n}$ is defined as
\begin{equation*}
\text{pvol}_{\mathbb{C}}\;\mathcal{M}_{p_1/q_1, \dots, p_n/q_n}:=\text{vol}_{\mathbb{C}}\;\mathcal{M}-\dfrac{\pi}{2}\sum_{k=1}^n \lambda^k_{(p_1/q_1,\dots, p_n/q_n)}
\mod \sqrt{-1}\pi^2\mathbb{Z}.
\end{equation*} 
\end{definition}

\begin{question}\label{22022708}
\normalfont Let $\mathcal{M}$ be an $n$-cusped hyperbolic $3$-manifold. Does there exist $c$ depending only on $\mathcal{M}$ such that the number $N^{ps.\mathbb{C}}_{\mathcal{M}}(v)$ of Dehn fillings of $\mathcal{M}$ having the same pseudo complex volume $v$ is bounded by $c$? Further, if 
\begin{equation*}
\text{pvol}_{\mathbb{C}}\;\mathcal{M}_{p_1/q_1, \dots, p_n/q_n}=\text{pvol}_{\mathbb{C}}\;\mathcal{M}_{p'_1/q'_1, \dots, p'_n/q'_n},
\end{equation*}
then what is the relationship between $(p_1/q_1, \dots, p_n/q_n)$ and $(p'_1/q'_1, \dots, p'_n/q'_n)$? 
\end{question}

For the $1$-cusped case, we have the following classification: 
\begin{theorem}\label{22022707}
Let $\mathcal{M}$ be a $1$-cusped hyperbolic $3$-manifold and $\tau$ be the  cusp shape $\mathcal{M}$. Then $N^{ps.\mathbb{C}}_{\mathcal{M}}(v)\leq c$ for some constant $c$ depending only on $\mathcal{M}$. In particular, 
\begin{align*}
N^{ps.\mathbb{C}}_{\mathcal{M}}(v)\leq  
\begin{dcases*}
3, & if $\tau\in\mathbb{Q}(\sqrt{-3})$,\\
2, & if $\tau\in\mathbb{Q}(\sqrt{-1})$,\\
1, & otherwise
\end{dcases*}
\end{align*}
for $v$ sufficiently close to $\text{vol}_{\mathbb{C}}\;\mathcal{M}$. 
\end{theorem}
The above theorem is essentially due to Thurston via his hyperbolic Dehn filling theory. For instance, see the appendix given by I. Agol in \cite{jeon0}. We also provide a more general statement, implying Theorem \ref{22022707} as a corollary in Section \ref{Prem I}. (See Theorem \ref{22090901}.)
 
The goal of this paper is answering Question \ref{22022708} as well as generalizing Theorem \ref{22022707} to a $2$-cusped hyperbolic $3$-manifold. 

Our first main result is the following:
\begin{theorem}\label{20121901}
Let $\mathcal{M}$ be a $2$-cusped hyperbolic $3$-manifold. Then there exists a constant $c$ depending only on $\mathcal{M}$ such that 
\begin{equation*}
N^{ps.\mathbb{C}}_{\mathcal{M}}(v)\leq c. 
\end{equation*}
\end{theorem}

To state Theorem \ref{20121901} further in detail, we shall need the following two definitions. 
\begin{definition}\label{20122206}
\normalfont For $\sigma=\left( \begin{array}{cc}
a & b  \\
c & d  
\end{array} 
\right)\in \text{GL}_2(\mathbb{Q})$ and $\dfrac{p}{q}\in \mathbb{Q}$, we denote $\dfrac{ap+bq}{cp+dq}$ simply by $\sigma(p/q)$. 
\end{definition}

\begin{definition}
\normalfont We say two algebraic numbers $\tau_1$ and $\tau_2$ are \textit{relatively dependent} if there exist $a,b,c,d\in \mathbb{Z}$ such that 
\begin{equation*}
\tau_1=\dfrac{a\tau_2+b}{c\tau_2+d}.
\end{equation*}
Otherwise, we say $\tau_1$ and $\tau_2$ are \textit{relatively independent}. 
\end{definition}

Now the following is the first refinement of Theorem \ref{20121901}: 

\begin{theorem}\label{2201}
Let $\mathcal{M}$ be a $2$-cusped hyperbolic $3$-manifold. There exists a finite set $S\big(\subset \text{GL}_2(\mathbb{Q})\big)$ such that, for any two Dehn fillings $\mathcal{M}_{p_1/q_1, p_2/q_2}$ and $\mathcal{M}_{p'_1/q'_1, p'_2/q'_2}$ of $\mathcal{M}$ satisfying
\begin{equation*}
\text{pvol}_{\mathbb{C}}\;\mathcal{M}_{p_1/q_1, p_2/q_2}=\text{pvol}_{\mathbb{C}}\;\mathcal{M}_{p'_1/q'_1, p'_2/q'_2}
\end{equation*}
with $|p_k|+|q_k|$ and $|p'_k|+|q'_k|$ sufficiently large ($k=1,2$), 
\begin{equation*}
(p_1/q_1, p_2/q_2)=(\sigma(p'_1/q'_1), \tau(p'_2/q'_2))\quad \text{or}\quad (\sigma(p'_2/q'_2), \tau(p'_1/q'_1))
\end{equation*} 
for some $\sigma, \tau\in S$. 
\end{theorem}

Throughout the paper, we offer several different variants and improvements of Theorem \ref{2201}. For instance, see Theorem \ref{22022601} or \ref{22050203}. 

When two cusp shapes $\tau_1$ and $\tau_2$ of $\mathcal{M}$ do not belong to the same quadratic field, then we can precisely describe the size and elements of $S$ in Theorem \ref{2201} as below:

\begin{theorem}\label{20103003}
Adapting the same notation and assumptions in Theorem \ref{2201}, let $\tau_1$ and $\tau_2$ be two cusp shapes of $\mathcal{M}$ not belonging to the same quadratic field. Then the following statements hold. 
\begin{enumerate} 
\item Suppose neither $\tau_1$ nor $\tau_2$ is contained in $\mathbb{Q}(\sqrt{-1})$ or $\mathbb{Q}(\sqrt{-3})$. 
\begin{enumerate}
\item If $\tau_1$ and $\tau_2$ are relatively independent, then $(p_1/q_1, p_2/q_2)=(p'_1/q'_1, p'_2/q'_2)$.
\item If $\tau_1$ and $\tau_2$ are relatively dependent, there exists unique $\rho\in S$ such that $(p_1/q_1, p_2/q_2)$ is either $(p'_1/q'_1, p'_2/q'_2)$ or $\big(\rho^{-1}(p'_2/q'_2), \rho(p'_1/q'_1)\big)$.
\end{enumerate}

\item If $\tau_1\in \mathbb{Q}(\sqrt{-1})$ and $\tau_2\notin\mathbb{Q}(\sqrt{-3})$ (resp. $\tau_1\in \mathbb{Q}(\sqrt{-3})$ and $\tau_2\notin \mathbb{Q}(\sqrt{-1})$), then there exists unique $\sigma\in S$ of order $2$ (resp. $3$) such that $(p_1/q_1, p_2/q_2)=(\sigma^i(p'_1/q'_1), p'_2/q'_2)$ where $0\leq i\leq 1$ (resp. $0\leq i\leq 2$). 
\item If $\tau_1\in \mathbb{Q}(\sqrt{-1})$ and $\tau_2\in \mathbb{Q}(\sqrt{-3})$, then there exist unique $\sigma$ and $\phi\in S$ of order $2$ and $3$ respectively such that $(p_1/q_1, p_2/q_2)=(\sigma^i(p'_1/q'_1), \phi^j(p'_2/q'_2))$ for some $0\leq i\leq 1, 0\leq j\leq 2$.
\end{enumerate}
\end{theorem}

For $\tau_1$ and $\tau_2$ belonging to the same quadratic field, the case will be addressed in a subsequent paper \cite{jeon-1}. In particular, it will be shown that the size of $S$ is also chosen to be uniform (i.e., independent of $\mathcal{M}$) and bounded above by $18$. Furthermore, in a joint work with S. Oh \cite{jeon-2}, we will provide examples confirming Theorem \ref{20103003} and the result in \cite{jeon-1} are effective. 

The next is our last main result. 
\begin{theorem}\label{22041805}
Let $\mathcal{M}$ be a $2$-cusped hyperbolic $3$-manifold. If 
\begin{equation*}
\text{pvol}_{\mathbb{C}}\;\mathcal{M}_{p_1/q_1, p_2/q_2}=\text{pvol}_{\mathbb{C}}\;\mathcal{M}_{p'_1/q'_1, p'_2/q'_2}
\end{equation*}
for $|p_k|+|q_k|$ and $|p'_k|+|q'_k|$ sufficiently large ($k=1, 2$), then
\begin{equation*}
\text{vol}_{\mathbb{C}}\;\mathcal{M}_{p_1/q_1, p_2/q_2}=\text{vol}_{\mathbb{C}}\;\mathcal{M}_{p'_1/q'_2, p'_2/q'_2}
\end{equation*}
\end{theorem}

This theorem, along with Theorem \ref{20080701}, suggests that classifying Dehn fillings of $\mathcal{M}$ using pseudo complex volume would be the same as classifying them under complex volume. That is, the following conjecture is very plausible:
\begin{conjecture}\label{22011601}
Let $\mathcal{M}$ be an $n$-cusped hyperbolic $3$-manifold. Let $\mathcal{M}_{p_1/q_1, \dots, p_n/q_n}$ and $\mathcal{M}_{p'_1/q'_2,\dots,  p'_n/q'_n}$ be two Dehn fillings of $\mathcal{M}$. Then   
\begin{equation*}
\text{pvol}_{\mathbb{C}}\;\mathcal{M}_{p_1/q_1, \dots, p_n/q_n}=\text{pvol}_{\mathbb{C}}\;\mathcal{M}_{p'_1/q'_2,\dots,  p'_n/q'_n}
\end{equation*}
if and only if 
\begin{equation*}
\text{vol}_{\mathbb{C}}\;\mathcal{M}_{p_1/q_1, \dots, p_n/q_n}=\text{vol}_{\mathbb{C}}\;\mathcal{M}_{p'_1/q'_1,\dots, p'_n/q'_n}. 
\end{equation*}
\end{conjecture}

\subsection{Application}\label{22071001}
The pseudo complex volume seems to be a key to decipher the secrets of some other arithmetic invariants of a hyperbolic $3$-manifold. For instance, as an application of our main results, one can prove the following theorem:
\begin{theorem}\label{22032601}
Let $\mathcal{M}$ be a $2$-cusped hyperbolic $3$-manifold and $\mathcal{X}$ be its holonomy variety. Suppose $\mathcal{X}$ is rational. If 
\begin{equation*}
\text{pvol}_{\mathbb{C}}\;\mathcal{M}_{p_1/q_1, p_2/q_2}=\text{pvol}_{\mathbb{C}}\;\mathcal{M}_{p'_1/q'_1, p'_2/q'_2}
\end{equation*}
with $|p_k|+|q_k|$ and $|p'_k|+|q'_k|$ sufficiently large ($k=1,2$), then the two Dehn fillings have the same Bloch invariant. 
\end{theorem}

The above theorem is a generalization of the following, which is attained by combining a result of A. Champanerkar (Theorem 10.3 in \cite{Ahbijit}) with a refined version of Theorem \ref{22022707} given in Section \ref{Prem I}.  

\begin{theorem}\label{22080201}
Let $\mathcal{M}$ be a $1$-cusped hyperbolic $3$-manifold and $\mathcal{X}$ be its holonomy variety. Suppose $\mathcal{X}$ is rational. If 
\begin{equation*}
\text{pvol}_{\mathbb{C}}\;\mathcal{M}_{p/q}=\text{pvol}_{\mathbb{C}}\;\mathcal{M}_{p'/q'}
\end{equation*}
with $|p|+|q|$ and $|p'|+|q'|$ sufficiently large, then the two Dehn fillings have the same Bloch invariant.  
\end{theorem}

Together with Conjecture \ref{22011601}, we hope Theorems \ref{22032601}-\ref{22080201} shed light on the following conjecture of D. Ramakrishnan:
\begin{conjecture}[Ramakrishnan]\label{22080202}
Let $\mathcal{M}$ and $\mathcal{N}$ be two hyperbolic $3$-manifolds. If 
\begin{equation*}
\text{vol}_\mathbb{C}\;\mathcal{M}=\text{vol}_\mathbb{C}\;\mathcal{N}, 
\end{equation*}
then both $\mathcal{M}$ and $\mathcal{N}$ have the same Bloch invariant. 
\end{conjecture}

However we will not prove Theorem \ref{22032601} in the paper, as it requires additional background knowledge far beyond the scope needed for the proofs of the main results. 

Note that once we assume a conjecture of Suslin (Conjecture 5.4 in \cite{suslin}), in return, one can remove the assumption on $\mathcal{X}$ (i.e. $\mathcal{X}$ being rational) in either Theorem \ref{22032601} or \ref{22080201}. 

\subsection{Key Idea}\label{22080901}
The main results are obtained by solving a weak version of the so called Zilber-Pink conjecture for $\mathcal{X}\times \mathcal{X}$ where $\mathcal{X}$ is the holonomy variety of a $2$-cusped hyperbolic $3$-manifold. The Zilber-Pink conjecture, which is one of the well-known conjectures in number theory, concerns the distribution of the intersections between a given algebraic variety with infinitely many algebraic subgroups. The precise formulation of the conjecture is as follows:
\begin{conjecture}[Zilber-Pink]\label{22042203}
Let $\mathbb{G}^n:=(\overline{\mathbb{Q}}^*)^n$ or $(\mathbb{C^*})^n$. For every irreducible variety $\mathcal{X}(\subset \mathbb{G}^n)$ defined over $\overline{\mathbb{Q}}$, there exists a finite set $\mathcal{H}$ of proper algebraic subgroups such that, for every algebraic subgroup $K$ and every component $\mathcal{Y}$ of $\mathcal{X}\cap K$ satisfying 
\begin{equation}\label{22042201}
\dim \mathcal{Y}>\dim \mathcal{X}+\dim K -n, 
\end{equation} 
one has $\mathcal{Y}\subset H $ for some $H\in \mathcal{H}$. 
\end{conjecture}

Note that the condition in \eqref{22042201} means that $\mathcal{X}$ and $K$ are not in general position. Thus the conjecture says that if there are infinitely many algebraic subgroups intersecting with $\mathcal{X}$ anomalously, then they are not randomly distributed but rather contained in some finite set of algebraic subgroups of bigger dimensions. 

The conjecture has been resolved for the curve case by G. Maurin in \cite{mau}, for some special cases involving varieties of the codimension $2$ by Bombieri, Masser and Zannier in \cite{za}, but remains widely open for other cases. 

In the context of hyperbolic $3$-manifolds, since the holonomy variety of a $1$-cusped (resp. $2$-cusped) hyperbolic $3$-manifold is an algebraic curve (resp. $2$-dim  algebraic variety in $\mathbb{G}^4$), the conjecture holds in this case by the work of Maurin (resp. Bombieri, Masser and Zannier). Later in Section \ref{Prem I}, we further make these more effective and show how to deduce Theorem \ref{22022707} from them.  

The main results of this paper are continuations along these lines and seen as extensions of the above approaches. The following, which immediately implies Theorem \ref{20121901}, is viewed as a resolution of some weak version of Conjecture \ref{22042203} for $\mathcal{X}\times \mathcal{X}$ where $\mathcal{X}$ is the holonomy variety of a $2$-cusped hyperbolic $3$-manifold. 

\begin{theorem}\label{22022601}
Let $\mathcal{M}$ be a $2$-cusped hyperbolic $3$-manifold and $\mathcal{X}$ be its holonomy variety. Then there exists a finite set $\mathcal{H}$ of algebraic subgroups of dimension $4$ such that, for any two Dehn fillings $\mathcal{M}_{p_{1}/q_{1},p_{2}/q_{2}}$ and $\mathcal{M}_{p'_{1}/q'_{1},p'_{2}/q'_{2}}$ such that 
\begin{equation}\label{22041801}
\text{pvol}_{\mathbb{C}}\;\mathcal{M}_{p_{1}/q_{1},p_{2}/q_{2}}=\text{pvol}_{\mathbb{C}}\;\mathcal{M}_{p'_{1}/q'_{1},p'_{2}/q'_{2}} 
\end{equation}
with sufficiently large $|p_k|+|q_k|$ and $|p'_k|+|q'_k|$ ($k=1,2$), a Dehn filling point $P$ associated to \eqref{22041801} on $\mathcal{X}\times \mathcal{X}$ is contained in some $H\in\mathcal{H}$. In particular, there exists a component of $(\mathcal{X}\times \mathcal{X})\cap H$ that contains $P$ and the identity, and has dimension $2$.\footnote{Since $\dim (\mathcal{X}\times \mathcal{X})=\dim H=4$ (in $\mathbb{G}^8$), the existence of such a component means that $\mathcal{X}\times \mathcal{X}$ and $H$ intersect non-generically.} %We coin this term in Definition \ref{anomalous1} and this will serve as the central terminology throughout the paper. 
\end{theorem}

By Thurston, each Dehn filling of $\mathcal{M}$ is identified as an intersection point (i.e. \textit{a Dehn filling point}) between $\mathcal{X}$ and an algebraic subgroup of the complementary dimension arising from the filling coefficient. Thus if two Dehn fillings $\mathcal{M}_{p_{1}/q_{1},p_{2}/q_{2}}$ and $\mathcal{M}_{p'_{1}/q'_{1},p'_{2}/q'_{2}}$ of $\mathcal{M}$, regarded as a point $P$ over $\mathcal{X}\times \mathcal{X}$, further satisfies \eqref{22041801}, as \eqref{22041801} encodes one more additional algebraic subgroup, $P$ is an intersection point between $\mathcal{X}\times \mathcal{X}$ and an algebraic subgroup whose dimension is strictly less than the complementary dimension of $\mathcal{X}\times \mathcal{X}$. We call such a point a \textit{torsion anomalous} point. Consequently, if there are infinitely many pairs of Dehn fillings of $\mathcal{M}$ of the same pseudo complex volume, then there are infinitely many torsion anomalous points over $\mathcal{X}\times \mathcal{X}$, and hence the problem now fits into the framework of the Zilber-Pink conjecture. 

In \cite{jeon1, jeon2}, the author, building up on the work of Habegger, showed that the height of any Dehn filling point is uniformly bounded, namely, it is independent of its Dehn filling coefficient (Theorem \ref{20080404}). By the result of Bombieri, Masser and Zannier (Theorem \ref{19090804}), a torsion anomalous point of bounded height is further contained in a subvariety of positive dimension possessing an analogous property, so called an \textit{torsion anomalous subvariety} of $\mathcal{X}\times \mathcal{X}$ (see Definition \ref{anomalous1}).  

According to Theorem \ref{struc}, due to Bombieri, Masser and Zannier, the set $\mathcal{S}$ of positive dimensional anomalous subvarieties of $\mathcal{X}\times \mathcal{X}$ is Zariski closed. Considering every possible dimension of $\mathcal{S}$ and analyzing the architecture of $\mathcal{S}$ more in detail, we farther break down the original problem into a couple of Zilber-Pink type conjectures of smaller dimensions, that is, conjectures with varieties and ambient spaces of lower dimensions. To overcome the last obstacles, various partially known results toward Conjecture \ref{22042203} are applied. For instance, we use the aforementioned results of both Maurin and Bombieri-Masser-Zannier (i.e. their partial resolutions of Conjecture \ref{22042203}) as well as S. Zhang's work on the distribution of torsion points over a given variety, which are all summarized in Section \ref{Back}. 

Once Theorem \ref{22022601} is proven, we further attempt to make it more effective. Namely, we count the number of elements in $\mathcal{H}$ and search the precise forms of the defining equations of each $H\in \mathcal{H}$. In particular, the matrices connecting two Dehn filling coefficients in \eqref{22041801}, described in Theorems \ref{2201}-\ref{20103003}, are derived from the exponents of the defining equations of $H\in \mathcal{H}$. Establishing Theorem \ref{22041805} is another ultimate goal of this quantification. 

Throughout the paper, many corollaries of Thurston's Dehn filling theory play crucial roles, and Neumann and Zagier's foundational work on symmetric properties of the holonomy variety is of fundamental importance as well.  %All these necessary gadgets are provided in Section \ref{Back}. 

Finally, we hope the paper provides interesting perspectives for both topologists and number theorists. For topologists, it would be interesting to see how the ideas in diophantine geometry would help to solve problems in the low dimensional topology, particularly, as we believe this is the first paper classifying Dehn fillings of a $2$-cusped hyperbolic $3$-manifold by a single invariant. For number theorists, the contents of this paper exhibit concrete applications of the research around the Zilber-Pink conjecture to the other area, beyond the realm of number theory or algebraic geometry or even logic.  

\subsection{Outline of the paper}
The organization of this paper is as follows. 

In Section \ref{Back}, we offer some necessary backgrounds in both topology and number theory, and summarize various results will be used in the proofs of the main theorems. 

In Section \ref{Prem}, we prove some preliminary lemmas required in the proofs of the main theorems. In particular, we analyze the structure of anomalous subvarieties of $\mathcal{X}\times \mathcal{X}$ where $\mathcal{X}$ is the same as in Theorem \ref{22022601}. 

Section \ref{Prem I} is a warm-up section before setting out to prove the main theorems. In the section, we cover the $1$-cusp case and prove a more general argument, which, consequently, implies Theorem \ref{22022707}. We also make Corollary \ref{21021602} as well as a special case of it more effective. The ideas and techniques utilized in this section will be  reproduced later in Sections \ref{quantaI}-\ref{20121903} to make Theorem \ref{22022601} quantitative. 

Sections \ref{ZPCI}-\ref{ZPCII} are devoted to establishing Theorem \ref{22022601}. In Section \ref{ZPCI}, we first deal with some special cases, and then prove the full statement in Section \ref{ZPCII}. These two sections constitute the largest proportion and is the technical heart of the paper.

Quantifying Theorem \ref{22022601} is the job for Sections \ref{quantaI}-\ref{20121903}. In Section \ref{quantaI}, we find the explicit relationship between two Dehn filling coefficients satisfying the equality in \eqref{22041801} and, in Section \ref{20121903}, further explore the restrictions on defining equations of each $H\in \mathcal{H}$ where $\mathcal{H}$ given in the statement of theorem \ref{22022601}.  

Finally all the results obtained in Sections \ref{ZPCI}-\ref{20121903} are merged to prove Theorem \ref{22041805} and Theorems \ref{2201}-\ref{20103003} in Section \ref{PCV} and Section \ref{22041806} respectively. 

\subsection{Acknowledgement}
The author would like to thank I. Agol for letting him know the possible application, presented in Section \ref{22071001}, of the main results of the paper. 

The author is also grateful to the anonymous referee for carefully reading the paper and offering valuable comments and suggestions, which addressed every aspect and thus improved the overall quality of the paper. 

The author is partially supported by Basic Science Research Institute Fund, whose NRF grant number is 2021R1A6A1A10042944. 

\newpage
\section{Background}\label{Back}
\subsection{Holonomy variety}\label{A-poly}

In this section, we introduce the holonomy variety of a cusped hyperbolic 3-manifold. We will not present all the technical details about the topic in this paper, but will instead provide a brief outline of its construction and necessary properties that will be utilized later in proving the main theorem.

According to Thurston \cite{thu}, for a given $n$-cusped hyperbolic $3$-manifold $\mathcal{M}$, a geometric ideal triangulation $\mathcal{T}$ on it induces a so-called \textit{gluing variety} $G(\mathcal{T})$ of $\mathcal{M}$. The variety represents the necessary conditions for how the tetrahedra in $\mathcal{T}$ are glued together along their edges to yield a hyperbolic structure on $\mathcal{M}$. Roughly, $G(\mathcal{T})$ can also be seen as the set of all the possible hyperbolic structures on $\mathcal{M}$, or simply, the moduli space of $\mathcal{M}$. 

If $T_k$ is a torus cross-section of the $k^{\text{th}}$-cusp of $\mathcal{M}$ and $m_k, l_k$ are the chosen meridian-longitude pair of $T_k$ $(1\leq k\leq n)$,  then each point of $G(\mathcal{T})$ gives rise to a (Euclidean) similarity structure on $T_k$, thus inducing the following \textit{holonomy map} 
\begin{equation*}
\pi_1(T_k)\longrightarrow \text{Aff}(\mathbb{C})=\{az+b\;:\;a\neq 0, b\in \mathbb{C}\}.
\end{equation*}
Consequently, the dilation components (i.e. derivatives) of the holonomies of $l_k$ and $m_k$ produce rational functions $M_k$ and $L_k$ respectively on $G(\mathcal{T})$. Now the \textit{holonomy variety} $\mathcal{X}$ of $\mathcal{M}$ is defined as the Zariski closure of the image under the following map:
\begin{equation*}
\xi\;:\;G(\mathcal{T})\longrightarrow (M_1, L_1, \dots, M_n, L_n).
\end{equation*}
The holonomy variety turns out to be independent of the triangulation $\mathcal{T}$, but depends only on the fundamental group of $\mathcal{M}$ as well as the chosen meridian-longitude pair of $T_k$.\footnote{For $n=1$, it is analogous to the A-polynomial of $\mathcal{M}$ introduced in \cite{ccgls}. More precisely, the A-polynomial uses the eigenvalues of $m_1$ and $l_1$, which are the square roots of the derivatives of the holonomies of $m_1$ and $l_1$ respectively, as its variables. For $n\geq 2$, a more detailed account of the holonomy variety is given in \cite{jeon1}.}

In general, the holonomy variety $\mathcal{X}$ may have several irreducible components, but we are only interested in so called the \textit{geometric component} of it. It is known that the geometric component of $\mathcal{X}$ is an $n$-dim algebraic variety in $\mathbb{C}^{2n}\big(:=(M_1, L_1, \dots, M_n, L_n)\big)$ and contains $(1, \dots, 1)$, which gives rise to the complete hyperbolic metric structure of $\mathcal{M}$. Let us denote the component by $\mathcal{X}$ and, by abuse of notation, still call it the holonomy variety of $\mathcal{M}$.   
 
The holonomy variety possesses many interesting symmetric properties. For instance, there exists a small neighborhood $\mathcal{N}_{\mathcal{X}}$ of a branch containing $(1, \dots, 1)$ in $\mathcal{X}$ such that, once we set  
\begin{equation}\label{22063015}
u_k:=\log M_k\quad, v_k:=\log L_k\quad (1\leq k\leq n), 
\end{equation}
$\mathcal{N}_{\mathcal{X}}$ is locally biholomorphic to a complex manifold $\log \mathcal{X}$ containing $(0, \dots, 0)$ in $\mathbb{C}^{2n}(:=(u_1, v_1, \dots, u_n, v_n))$, whose features are described as follows \cite{nz}:
\begin{theorem}[Neumann-Zagier] \label{potential}
\begin{enumerate}
\item Using $u_1, \dots, u_n$ as generators, each $v_k$ is represented of the form 	
\begin{equation}\label{22071901}
u_k\cdot\tau_k(u_1,\dots,u_n)
\end{equation}
over $\log \mathcal{X}$ where $\tau _k(u_1,\dots,u_n)$ is a holomorphic function with $\tau_k(0,\dots,0)=\tau_k\in \mathbb{C}\backslash\mathbb{R}$ ($1\leq k\leq n$).
\item There is a holomorphic function $\Phi(u_1,\dots,u_n)$ such that 
\begin{equation*}
v_k=\dfrac{1}{2}\dfrac{\partial \Phi}{\partial u_k}\quad \text{and}\quad \Phi(0,\dots,0)=0, \quad (1\leq k\leq n). 
\end{equation*}
In particular, $\Phi(u_1,\dots,u_n)$ is even in each argument and so its Taylor expansion is of the following form:
\begin{equation*}
\Phi(u_1,\dots,u_n)=(\tau_1u_1^2+\cdots+\tau_n u_n^2)+(m_{4,\dots,0}u_1^4+\cdots+m_{0,\dots,4}u_n^4)+\text{(higher order terms)}.\\
\end{equation*}
\end{enumerate}
\end{theorem}

Later in Section \ref{Dehn}, it will be shown that $\log \mathcal{X}$ (or $\mathcal{N}_{\mathcal{X}}$) encodes all the useful topological information pertaining Dehn fillings of $\mathcal{M}$. Thus it is natural to mainly focus our attention on $\log \mathcal{X}$ throughout the paper. 

The following immediate corollary of Theorem \ref{potential} will be repeatedly used in the proofs of many claims. 

\begin{corollary}\label{22080401}
Adapting the same notation as in Theorem \ref{potential}, 
\begin{enumerate}
\item for each $k$ ($1\leq k\leq n$), $v_k=0$ if and only if $u_k=0$ by \eqref{22071901};\footnote{Geometrically, $u_k=v_k=0$ gives rise to the complete hyperbolic metric structure of the $k$-th cusp of $\mathcal{M}$.}

\item for any $k$ and $l$ ($1\leq k,l\leq n$), 
\begin{equation*}
\dfrac{\partial v_k}{\partial u_l}=\dfrac{\partial v_l}{\partial u_k}.
\end{equation*}
\end{enumerate}
\end{corollary}

For later convenience, we also define  

\begin{definition}
\normalfont 
\begin{enumerate}
\item We refer to $\tau_k$ (in Theorem \ref{potential}) as the \textit{cusp shape} of the $k$-th cusp of $\mathcal{M}$ with respect to $m_k, l_k$ and $\Phi(u_1,\dots,u_n)$ as the \textit{Neumann-Zagier potential function} of $\mathcal{M}$ with respect to $m_k, l_k$ $(1\leq k\leq n)$. 

\item We refer to $\log\mathcal{X}$ as the \textit{analytic holonomy variety (or set)} of $\mathcal{M}$.  
\end{enumerate}
\end{definition}

Note that each cusp shape is algebraic and we normalize $\tau_k$ by $\text{Im}\;\tau_k>0$ ($1\leq k\leq n$). Moreover, it is known that each $v_k$ is not linear (i.e. $v_k\neq \tau_ku_k$). For instance, see Lemma 2.3 in \cite{jeon3}. (In fact, one can further show each $v_k$ is a transcendental function of $u_1, \dots, u_n$.)

\subsection{Dehn Filling}\label{Dehn}   
Let $\mathcal{M}$ be the same as above and 
\begin{equation*}
\mathcal{M}_{p_1/q_1,...,p_n/q_n}
\end{equation*}
be the $(p_1/q_1,\dots,p_n/q_n)$-Dehn filling of $\mathcal{M}$. By the Seifert-Van Kampen theorem, the fundamental group of $\mathcal{M}_{p_1/q_1,...,p_n/q_n}$ is obtained by adding the relations
\begin{equation}\label{22080402}
m_1^{p_1}l_1^{q_1}=1, \quad...\quad, m_n^{p_n}l_n^{q_n}=1
\end{equation}
to the fundamental group of $\mathcal{M}$ where $m_k, l_k$ ($1\leq k\leq n$) are meridian-longitude pairs of $\mathcal{M}$ as previously. Hence if $\mathcal{X}$ is the holonomy variety of $\mathcal{M}$, then the discrete faithful representation 
\begin{equation*}
\phi\;:\;\pi_1(\mathcal{M}_{p_1/q_1, \dots, p_n/q_n})\longrightarrow SL_2\mathbb{C}
\end{equation*}
corresponds to an intersection point between $\mathcal{X}$ and an algebraic subgroup defined by  
\begin{equation} \label{Dehn eq}
M_1^{p_1}L_1^{q_1}=1,\quad ... \quad , M_n^{p_n}L_n^{q_n}=1, 
\end{equation}
arising from \eqref{22080402}.
\begin{definition}
\normalfont We say \eqref{Dehn eq} the \textit{Dehn filling equations} with coefficient $(p_1/q_1,\dots,p_n/q_n)$ and a point inducing the hyperbolic structure on $\mathcal{M}_{p_1/q_1,...,p_n/q_n}$ a \textit{Dehn filling point} corresponding to $\mathcal{M}_{p_1/q_1,...,p_n/q_n}$. 
\end{definition}
Let
\begin{equation}\label{040401}
(M_1,L_1,\dots,M_n,L_n)=(\zeta_{m_1}, \zeta_{l_1},\dots, \zeta_{m_n},\zeta_{l_n})
\end{equation}
be one of the Dehn filling points corresponding to $\mathcal{M}_{p_1/q_1,...,p_n/q_n}$. By the work of Thurston,
\begin{equation}\label{22070306}
(\zeta_{m_k})^{r_k}(\zeta_{l_k})^{s_k}\quad \text{where} \quad p_ks_k-q_kr_k=1\quad \text{ and }\quad  r_k, s_k\in \mathbb{Z}, 
\end{equation}
which is the derivative of the holonomy of the $k$-th core geodesic $m_k^{r_k}l_k^{s_k}$, is always not a root of unity (i.e. $|(\zeta_{m_k})^{r_k}(\zeta_{l_k})^{s_k}|\neq 1$) for every $1\leq k\leq n$.   

The following theorem is a part of Thurston's hyperbolic Dehn filling theory \cite{nz, thu}. 
\begin{theorem}[Thurston] \label{040605}
Using the same notation as above, 
\begin{equation*}
(\zeta_{m_1}, \zeta_{l_1},\dots, \zeta_{m_n},\zeta_{l_n})\rightarrow(1,\dots,1)\in \mathbb{C}^{2n}
\end{equation*}
and 
\begin{equation*}
\big(|(\zeta_{m_1})^{r_1}(\zeta_{l_1})^{s_1}|, \dots, |(\zeta_{m_n})^{r_n}(\zeta_{l_n})^{s_n}|\big)\rightarrow(1,\dots,1)\in\mathbb{C}^{n}
\end{equation*}
as $|p_k|+|q_k|\rightarrow\infty$ for $1\leq k\leq n$. 
\end{theorem}  
By Theorem \ref{040605}, most Dehn filling points (that is, points with sufficiently large Dehn filling coefficients) are lying near $(1, \dots, 1)$ in $\mathcal{X}$ and, in particular, they are distributed over $\mathcal{N}_{\mathcal{X}}$.  
 
In the context of $\log \mathcal{X}$, \eqref{Dehn eq} is further specified as either
\begin{equation}\label{22063001}
p_ku_k+q_kv_k= 2\pi \sqrt{-1}\quad \text{ or }-2\pi \sqrt{-1},\quad  (1\leq k\leq n).  
\end{equation}
By abuse of notation, we still call \eqref{22063001} the \textit{Dehn filling equations} with coefficient $(p_1/q_1, \dots, p_n/q_n)$, and an intersection point between \eqref{22063001} and $\log \mathcal{X}$ a \textit{Dehn filling point} associated to $\mathcal{M}_{p_1/q_1, \dots, p_n/q_n}$. Note that if
\begin{equation*}
(u_1, v_1, \dots, u_n, v_n)=(\xi_{m_1}, \xi_{l_1}, \dots, \xi_{m_n}, \xi_{l_n})
\end{equation*}
is a Dehn filling point on $\log \mathcal{X}$ associated to $\mathcal{M}_{p_1/q_1, \dots, p_n/q_n}$ and matching to \eqref{040401}, then
\begin{equation*}
(e^{\xi_{m_1}}, e^{\xi_{l_1}}, \dots, e^{\xi_{m_n}}, e^{\xi_{l_n}})=(\zeta_{m_1}, \zeta_{l_1},\dots, \zeta_{m_n}, \zeta_{l_n})
\end{equation*}
and 
\begin{equation*}
e^{r_k\xi_{m_k}+s_k\xi_{l_k}}=(\zeta_{m_k})^{r_k}(\zeta_{l_k})^{s_k}
\end{equation*}
where $r_k$ and $s_k$ are as in \eqref{22070306}. For $|(\zeta_{m_k})^{r_k}(\zeta_{l_k})^{s_k}|>1$ (resp. $|(\zeta_{m_k})^{r_k}(\zeta_{l_k})^{s_k}|<1$), the logarithm of $(\zeta_{m_k})^{r_k}(\zeta_{l_k})^{s_k}$ (resp. $\frac{1}{(\zeta_{m_k})^{r_k}(\zeta_{l_k})^{s_k}}$), $r_k\xi_{m_k}+s_k\xi_{l_k}$ (resp. $-r_k\xi_{m_k}-s_k\xi_{l_k}$), is the \textit{complex length} of the $k$-th core geodesic of $\mathcal{M}_{p_1/q_1, \dots p_n/q_n}$, which appeared and is denoted by $\lambda_k$ in the statement of Theorem \ref{20080701}. Note that the complex length of each core geodesic is well-defined modulo $2\pi \sqrt{-1}$.  

As a result, when we have two Dehn fillings 
\begin{equation}\label{22063007}
\mathcal{M}_{p_{1}/q_{1},\cdots p_{n}/q_{n}}\text{  and  }\mathcal{M}_{p'_{1}/q'_{1},\cdots, p'_{n}/q'_{n}}
\end{equation}
of $\mathcal{M}$ with
\begin{equation}\label{22063009}
\text{pvol}_{\mathbb{C}}\;\mathcal{M}_{p_{1}/q_{1},\dots, p_{n}/q_{n}}=\text{pvol}_{\mathbb{C}}\;\mathcal{M}_{p'_{1}/q'_{1},\dots, p'_{n}/q'_{n}}, 
\end{equation}
identifying \eqref{22063007} as a point $P$\footnote{We say $P$ is a \textit{Dehn filling point on $\mathcal{X}\times \mathcal{X}$ associated to either \eqref{22063007} or \eqref{22063009}}.} over $\mathcal{X}\times \mathcal{X}$ by associating $\mathcal{M}_{p_{1}/q_{1},\dots p_{n}/q_{n}}$ (resp. $\mathcal{M}_{p'_{1}/q'_{1},\dots, p'_{n}/q'_{n}}$) to a Dehn filling point of the first (resp. second) copy of $\mathcal{X}\times \mathcal{X}$, it follows that $P$ is an intersection point between $\mathcal{X}\times \mathcal{X}\big(\subset \mathbb{G}^{2n}\times \mathbb{G}^{2n}(:=(M_1, L_1, \dots, M'_n, L'_n))\big)$ and an algebraic subgroup of codimension $2n+1$. Namely, $P$ satisfies, first, the following $(2n)$-Dehn filling equations 
\begin{equation*}
M_k^{p_k}L_k^{q_k}=(M'_k)^{p'_k}(L'_k)^{q'_k}=1, \quad (1\leq k\leq n)
\end{equation*}
and, in addition,  
\begin{equation}\label{22070401}
\prod^{n}_{k=1}M_k^{\epsilon_k r_k}L_k^{\epsilon_ks_k}=\prod^{n}_{k=1}(M'_k)^{\epsilon'_k r'_k}(L'_k)^{\epsilon'_k s'_k}, \quad (\text{for some }\epsilon_k, \epsilon'_k=\pm 1),
\end{equation} 
arising from \eqref{22063009}. Therefore, the existence of infinitely many pairs of Dehn fillings of $\mathcal{M}$ of the same pseudo complex volume guarantees the existence of infinitely many torsion anomalous points over $\mathcal{X}\times \mathcal{X}$, which implicates that our problem is now converted into a weak version of the Zilber-Pink conjecture. 

Let us further simplify \eqref{22070401} by introducing the following definition. 
\begin{definition}\label{22070407}
\normalfont For $\mathcal{M}_{p_1/q_1, \dots, p_n/q_n}$, let $t_k$ ($1\leq k\leq n$) be the exponential of the complex length of the $k$-th core geodesic of it. To be precise, for any given Dehn filling point \eqref{040401} associated to $\mathcal{M}_{p_1/q_1, \dots, p_n/q_n}$, 
\begin{equation*}
t_k:=\begin{dcases*}
(\zeta_{m_k})^{r_k}(\zeta_{l_k})^{s_k}, & if $|(\zeta_{m_k})^{r_k}(\zeta_{l_k})^{s_k}|>1$,\\
\frac{1}{(\zeta_{m_k})^{r_k}(\zeta_{l_k})^{s_k}}, & if $|(\zeta_{m_k})^{r_k}(\zeta_{l_k})^{s_k}|<1$.  
\end{dcases*}
\end{equation*}
We refer to $t_k$ as the \textit{$k$-th core holonomy} of $\mathcal{M}_{p_1/q_1, \dots, p_n/q_n}$ and
\begin{equation*}
\{t_1, \dots, t_n\}
\end{equation*}
as the \textit{set of the core holonomies} of $\mathcal{M}_{p_1/q_1, \dots, p_n/q_n}$. 
\end{definition}

Adapting the notation in Definition \ref{22070407}, the point given by  
\begin{equation}\label{22070411}
\begin{gathered}
(M_1, L_1, \dots, M_n, L_n)=(t_1^{-q_1}, t_1^{p_1}, \dots, t_n^{-q_n}, t_n^{p_n}) 
\end{gathered}
\end{equation}
is also a Dehn filling point associated to $\mathcal{M}_{p_1/q_1, \dots, p_n/q_n}$, which, in particular, satisfies  
\begin{equation*}
\prod^{n}_{k=1}M_k^{ r_k}L_k^{s_k}=\prod^{n}_{k=1}t_k. 
\end{equation*}
Consequently, if there are two Dehn fillings \eqref{22063007} of $\mathcal{M}$ satisfying \eqref{22063009}, there exists a Dehn filling point $P$ on $\mathcal{X}\times \mathcal{X}$ associated to it, contained in an algebraic subgroup defined by
\begin{equation*}
\prod^{n}_{k=1}M_k^{r_k}L_k^{s_k}=\prod^{n}_{k=1}(M'_k)^{r'_k}(L'_k)^{s'_k}.
\end{equation*} 
In other words, one can simply take $\epsilon_k=\epsilon'_k=1$ ($1\leq k\leq n$) in \eqref{22070401} with such $P$. 

From now on, for each given $\{(p_k, q_k)\;:\;1\leq k\leq n\}$, whenever a Dehn filling point of $\mathcal{M}_{p_1/q_1, \dots, p_n/q_n}$ is mentioned, we always mean \eqref{22070411} throughout the paper. 
\begin{remark}\label{22081303}
\normalfont Among the two given options in \eqref{22063001}, the second one gives rise to the desired Dehn filling point over $\log \mathcal{X}$. That is, the intersection point between $\log \mathcal{X}$ and $p_ku_k+q_kv_k=-2\pi \sqrt{-1}$ ($1\leq k\leq n$) is the Dehn filling point on $\log \mathcal{X}$ mapping to \eqref{22070411}.\footnote{This is due to the assumption $\text{Im}\;\tau_k>0$ on each cusp shape $\tau_k$ \cite{nz}.}  
\end{remark}

\subsection{Anomalous subvariety}\label{22030101}

Let $\mathbb{G}^n:=(\overline{\mathbb{Q}}^*)^n$ or $(\mathbb{C^*})^n$. By an \textit{algebraic subgroup} $H$ in $\mathbb{G}^n(:=(X_1, \dots, X_n))$, we mean the set defined by monomial equations $X_1^{a_1}\cdots X_n^{a_n}=1$ where $a_k\in \mathbb{Z}$ ($1\leq k\leq n$), and an \textit{algebraic coset} $K$ is defined to be a translate $gH$ of some algebraic subgroup $H$ by some $g \in \mathbb{G}^n$. An irreducible algebraic subgroup is often said to be an \textit{algebraic torus}. 

In \cite{za}, Bombieri-Masser-Zannier defined 
\begin{definition} \label{anomalous1} 
\normalfont Let $\mathcal{X}\subset \mathbb{G}^n$ be an algebraic variety. An irreducible subvariety $\mathcal{Y}$ of $\mathcal{X}$ is called \textit{anomalous} if it has positive dimension and lies in an algebraic coset $K$ in $\mathbb{G}^n$ satisfying
\begin{equation*}
\dim  \mathcal{Y} >\dim  \mathcal{X} +\dim K-n.
\end{equation*}
In particular, if $K$ is an algebraic subgroup, then we say $\mathcal{Y}$ a \textit{torsion anomalous} subvariety of $\mathcal{X}$. 
\end{definition}

Note that it is assumed $\dim \mathcal{Y} >0$ in Definition \ref{anomalous1}, but $\dim \mathcal{Y}=0$ is allowed in Conjecture \ref{22042203}. If $K$ is an algebraic subgroup and $\dim \mathcal{Y} = 0$, it is called by a \textit{torsion anomalous point} of $\mathcal{X}$ (as briefly mentioned in Section \ref{22080901}).\footnote{In Theorem \ref{19090804}, Bombieri, Masser and Zannier provide a sufficient condition for a torsion anomalous point to be contained in a torsion anomalous subvariety.}

For instance, if $\mathcal{X}$ is the holonomy variety of an $n$-cusped hyperbolic $3$-manifold, by Corollary \ref{22080401}, $M_k=1$ is equivalent to $L_k=1$ over $\mathcal{X}$, which implies $\mathcal{X}\cap (M_k=L_k=1)$ is an $(n-1)$-dim anomalous subvariety of $\mathcal{X}$. 

The structure of anomalous subvarieties of an algebraic variety is given as follows \cite{za}:

\begin{theorem}[Bombieri-Masser-Zannier]\label{struc}
Let $\mathcal{X}$ be an irreducible variety in $\mathbb{G}^n$ of positive dimension defined over $\mathbb{\overline Q}$. 
\begin{enumerate}
\item For any algebraic torus $H$ with 
\begin{equation} \label{dim1}
1\leq n- \dim  H \leq \dim  \mathcal{X},
\end{equation}
the union $\mathscr{Z}_H$ of all subvarieties $\mathcal{Y}$ of $\mathcal{X}$ contained in any coset $K$ of $H$ with
\begin{equation} \label{dim2}
\dim  H=n-(1+\dim  \mathcal{X})+\dim  \mathcal{Y} 
\end{equation}
is a closed subset of $\mathcal{X}$.
\item There is a finite collection $\Psi=\Psi_{\mathcal{X}}$ of such tori $H$ such that every maximal anomalous subvariety $\mathcal{Y}$ of $\mathcal{X}$ is a component of $\mathcal{X} \cap gH$
for some $H$ in $\Psi$ satisfying \eqref{dim1} and \eqref{dim2} and some $g$ in $\mathscr{Z}_H$. Moreover, if $\mathcal{X}^{oa}$ is defined as the complement of $\bigcup_{H\in\Psi}\mathscr{Z}_H$ in $\mathcal{X}$, then it is Zariski open in $\mathcal{X}$.
\end{enumerate}
\end{theorem} 

Using the same notation above, if there exists an algebraic subgroup $H$ such that $\mathcal{X}=\mathscr{Z}_H$ (and so $\mathcal{X}^{oa}=\emptyset$), we say $\mathcal{X}$ \textit{is foliated by anomalous subvarieties contained in} \textit{translations (or algebraic cosets) of $H$}. 

A good example to illustrate Theorem \ref{struc} is the product of two algebraic varieties. For instance, let $\mathcal{X}\subset \mathbb{G}^2 (:=(M, L))$ the holonomy variety of a $1$-cusped hyperbolic $3$-manifold defined by $f(M, L)=0$ and $\mathcal{X}\times \mathcal{X}$ be the product of two identical copies of it in $\mathbb{G}^2\times \mathbb{G}^2(:=(M, L, M', L'))$. For any $(M, L)=(\zeta_1, \zeta_2)$ satisfying $f(\zeta_1, \zeta_2)=0$, 
\begin{equation*}
(M=\zeta_1,  L=\zeta_2)\cap (\mathcal{X}\times \mathcal{X}) 
\end{equation*}
is an anomalous subvariety of $\mathcal{X}\times \mathcal{X}$ and, as 
\begin{equation*}
\bigcup_{f(\zeta_1, \zeta_2)=0}\big((M=\zeta_1,  L=\zeta_2)\cap (\mathcal{X}\times \mathcal{X})\big)=\mathcal{X}\times\mathcal{X},
\end{equation*}
it follows that $(\mathcal{X}\times \mathcal{X})^{oa}=\emptyset$ and $\mathcal{X}\times \mathcal{X}=\mathscr{Z}_H$ where $H$ is defined by $M=L=1$. (Similarly, one can check $(\mathcal{X}\times \mathcal{X})^{oa}=\emptyset$ for any algebraic variety $\mathcal{X}$.) 

One important consequence of Theorem \ref{struc}, which will be repeatedly used in the paper, is the following:
\begin{corollary}\label{22022603}
Let $\mathcal{X}$ be an irreducible variety in $\mathbb{G}^n(:=(X_1, \dots, X_n))$ of positive dimension defined over $\mathbb{\overline Q}$ and $H$ be an algebraic torus defined by
\begin{equation*}
X_1^{a_{i1}}\cdots X_n^{a_{in}}=1, \quad (1\leq i\leq k).  
\end{equation*}
Suppose every maximal anomalous subvariety of $\mathcal{X}$ is contained in a coset of $H$ and has the same dimension. Further suppose $\mathscr{Z}_H$ is irreducible and let	
\begin{equation}\label{22022501}
\begin{gathered}
\mathcal{U}_H:=\{(g_1, \dots, g_k):\;\mathcal{X}\cap gH \text{ contains a maximal anomalous subvariety of }\mathcal{X}\\
\text{ where }gH \text{ is defined by }X_1^{a_{i1}}\cdots X_n^{a_{in}}=g_i\quad (1\leq i\leq k)\}. 
\end{gathered}
\end{equation}
Then $\mathcal{U}_H$ contains a Zariski open subset of its Zariski closure $\overline{\mathcal{U}_H} (\subset\mathbb{G}^k)$, and the dimension of $\overline{\mathcal{U}_H}$ is equal to $\dim \mathscr{Z}_H-\dim (\mathcal{X}\cap gH)$.
\end{corollary}
\begin{proof}
The first claim is a well-known result of Chevalley (see Proposition 2.31 in \cite{mumford}, for instance). The second claim follows directly by the definitions of $\mathscr{Z}_H$ and $\mathcal{U}_H$. 
\end{proof}

Following \cite{BMZ1}, a refined version of Definition \ref{anomalous1} is given as  
\begin{definition}\label{22071611}
\normalfont Let $\mathcal{X}$ be an algebraic variety in $\mathbb{G}^n$. We say an irreducible subvariety $\mathcal{Y}$ of $\mathcal{X}$ is \textit{$b$-anomalous} ($b\geq1$) if it has positive dimension and lies in some coset of dimension
\begin{equation*}
n-(b+\dim \mathcal{X})+\dim \mathcal{Y}.
\end{equation*}
\end{definition}

For $X$ the holonomy variety of an $n$-cusped hyperbolic $3$-manifold, $\mathcal{X}\cap (M_k=L_k=1)$ is a $1$-anomalous subvariety of $\mathcal{X}$ and  
\begin{equation*}
\mathcal{X}\cap (M_{k_1}=L_{k_1}=\cdots =M_{k_b}=L_{k_b}=1)
\end{equation*}
is a $b$-anomalous subvariety of $\mathcal{X}$. 

Having Definition \ref{22071611}, the following strengthening version of Theorem \ref{struc} holds:
\begin{theorem}[Bombieri-Masser-Zannier]\label{struc2}
Suppose $\mathcal{X}$ is an irreducible variety in $\mathbb{G}^n$ defined over $\overline{\mathbb{Q}}$. 
\begin{enumerate}
\item For any algebraic torus $H$, the union $\mathscr{Z}_H^{(b)}$ of all subvarieties $\mathcal{Y}$ of $\mathcal{X}$ contained in any coset $K$ of $H$ with 
\begin{equation}\label{22081301}
\dim H=n-(b+\dim \mathcal{X})+\dim \mathcal{Y}\quad \text{where }b\geq 1
\end{equation}
is a closed subset of $\mathcal{X}$. 
\item There is a finite collection $\Psi =\Psi_{\mathcal{X}}$ of algebraic tori such that every maximal $b$-anomalous subvariety $\mathcal{Y}$ of $\mathcal{X}$ lies in a coset $K$ of some $H$ in $\Psi$ satisfying \eqref{22081301}. 
\end{enumerate}
\end{theorem}
\begin{proof}
The first statement is proven by simply following the proof of Theorem \ref{struc} (1) on p.18 of \cite{za}. Namely, by applying a monoidal transformation, if we set $H=\{1\}^h\times \mathbb{G}^{n-h}$ where $h:=n-\dim H$, and define $\pi_h$ as the projection from $\mathbb{G}^{n}$ to $\mathbb{G}^{n-h}$, then the following set 
\begin{equation*}
\{x\in \mathcal{X}:\dim_{x}(\mathcal{X}\cap \pi_h^{-1}(\pi_h(x))\geq s\}\quad \text{where}\quad s:=(b+\dim \mathcal{X})-h
\end{equation*}
is found to be equal to $\mathscr{Z}_H^{(b)}$. Hence the conclusion follows from the fiber dimension theorem in \cite{za}.  

The second statement corresponds to Lemma 5 of \cite{BMZ1}.   
\end{proof}
Throughout the paper, we prefer Theorem \ref{struc2} to Theorem \ref{struc}. The benefit of the above theorem is that, using this, one can reduce a given Zilber-Pink type problem into the one that are easier to handle with or already known how to solve. To explain more in detail, let $\{\mathcal{Y}_i\}_{i\in\mathcal{I}}$ be a family of maximal torsion anomalous subvarieties of $\mathcal{X}(\subset \mathbb{G}^n)$ satisfying $\mathcal{Y}_i\subset H_i$ for some algebraic subgroup $H_i$. Further, suppose each $\mathcal{Y}_i$ is a $b_i$-anomalous but not a $(b_i+1)$-anomalous subvariety of $\mathcal{X}$. By Theorem \ref{struc2}, there exists a finite set $\Psi$ of algebraic tori such that each $\mathcal{Y}_i$ is contained in a coset of some element of $\Psi$. By passing to a subsequence if necessary, we simply assume $\dim H_i$, $\dim \mathcal{Y}_i$ and $b_i$ are independent of $i$, and $\Psi=\{H\}$. That is, there exists $b\in \mathbb{N}$ such that, for each $i\in \mathcal{I}$, $\mathcal{Y}_i$ is a component of some $\mathcal{X}\cap g_iH$ ($g_i\in \mathbb{G}^n$) that satisfies
\begin{equation*}
\dim H =n-(b + \dim \mathcal{X}) + \dim \mathcal{Y}_i. 
\end{equation*} 
Note that, in this case, we further have
\begin{equation}\label{22071220}
g_iH\subset H_i
\end{equation}
for each $i\in \mathcal{I}$.\footnote{Otherwise, since 
\begin{equation*}
\mathcal{Y}_i=\mathcal{X}\cap (g_iH\cap  H_i),
\end{equation*}
it contradicts the fact that $\mathcal{Y}_i$ is not a $(b+1)$-anomalous subvariety of $\mathcal{X}$.} Hence, if $\mathcal{U}_H$ is defined as in \eqref{22022501}, $\overline{\mathcal{U}_H}$ is an algebraic variety of dimension $\dim \mathscr{Z}_H^{(b)}-\dim \mathcal{Y}_i$, intersecting with infinitely many algebraic subgroups arising from \eqref{22071220}, whose codimensions are all equal to $\codim H_i$. Later we show, in our case, the dimension of $\overline{\mathcal{U}_H}$ is strictly less than $\codim H_i$ in $\mathbb{G}^{\codim  H}$, and thus the original problem (i.e. the existence of a family $\{\mathcal{Y}_i\}_{i\in \mathcal{I}}$ of torsion anomalous subvarieties of $\mathcal{X}$) is reduced to a Zilber-Pink type conjecture over $\overline{\mathcal{U}_H}(\subset \mathbb{G}^{\codim  H})$. Since the dimensions of $\overline{\mathcal{U}_H}$ and $\mathbb{G}^{\codim  H}$ are smaller than the dimensions of $\mathcal{X}$ and $\mathbb{G}^n$ respectively, the reduced one is in general much easier to deal with. We frequently invoke this reduction trick to verify many arguments, particularly in Section \ref{ZPCII}, and therefore summarize the discussion so far in the following proposition for future reference:
\begin{proposition}\label{24101101}
Let $\mathcal{X}, \mathcal{Y}_i, H_i, H$ and $\mathcal{U}_H$ be as defined in the preceding discussion. Then $\overline{\mathcal{U}_H}$ is an algebraic variety of dimension $\dim \mathscr{Z}_H^{(b)} -\dim \mathcal{Y}_i$ in $\mathbb{G}^{\codim  H}$, intersecting with a family of algebraic subgroups of dimension $\dim H_i-\dim H$ (in $\mathbb{G}^{\codim  H}$). 
\end{proposition}
\begin{proof}
The first claim is immediate from Corollary \ref{22022603} and the second one follows from \eqref{22071220}. 
\end{proof}

Later, in the next subsection, we examine some specific resolvable cases of the Zilber-Pink conjecture and incorporate them into Proposition \ref{24101101}. See Corollary \ref{24100901}. 

Finally, we extend Definition \ref{anomalous1} further to an analytic holonomy variety $\log \mathcal{X}$. Let $K$ be an algebraic coset defined by 
\begin{equation}\label{22063011}
M_1^{a_{k1}}L_1^{b_{k1}}\cdots M_n^{a_{kn}}L_n^{b_{kn}}=\zeta_k, \quad (1\leq k\leq m)
\end{equation}
such that the component of $\mathcal{X}\cap K$ intersecting $\mathcal{N}_\mathcal{X}$ is an anomalous subvariety of $\mathcal{X}$. Under \eqref{22063015}, \eqref{22063011} is equivalent to an affine space defined by
\begin{equation}\label{22063013}
a_{k1}u_1+b_{k1}v_1+\cdots+ a_{kn}u_n+b_{kn}v_n=\log \zeta_k, \quad (1\leq k\leq m).
\end{equation}
Denote \eqref{22063013} by $\log K$ and its intersection with $\log \mathcal{X}$ by $\log (\mathcal{X}\cap K)$. Now, copying Definition \ref{anomalous1}, we analogously have 
\begin{definition}
\normalfont Adapting the same notation above, an \textit{anomalous analytic subset} of $\log \mathcal{X}$ is declared as $\log (\mathcal{X}\cap K)$ satisfying
\begin{equation*}
\dim (\log K) \leq n-\dim ( \log \mathcal{X}) + \dim  \big(\log(\mathcal{X}\cap K)\big) -1.
\end{equation*}
\end{definition}

Clearly, for any given algebraic coset $K$, $\log (\mathcal{X}\cap K)$ is an analytic anomalous subset of $\log \mathcal{X}$ if and only if the component of $\mathcal{X}\cap K$ touching $\mathcal{N}_{\mathcal{X}}$ is an anomalous subvariety of $\mathcal{X}$. Hence, throughout the proofs, we use both terminologies compatibly without any differentiation.

\begin{convention}\label{22062605}
\normalfont For $\mathcal{X}$ and $K$ as above, $\mathcal{X}\cap K$ may have several irreducible components in general. If $P\in \mathcal{X}\cap K$ for some Dehn filling point $P$, we denote a component of $\mathcal{X}\cap K$ containing $P$ by $\mathcal{X}\cap_P K$. If no confusion arises, $P$ is often skipped and $\mathcal{X}\cap_P K$ is simply represented by $\mathcal{X}\cap K$. 
 
Accordingly, $\log (\mathcal{X}\cap_P K)$ denotes the corresponding image of $\mathcal{X}\cap _P K$ over $\log \mathcal{X}$ and it will also be occasionally replaced by $\log (\mathcal{X}\cap K)$.  
\end{convention}

\begin{convention}\label{22062605}
\normalfont For simplicity, if no confusion arises, we will denote $\mathscr{Z}^{(b)}_H$ from Theorem \ref{struc2} by $\mathscr{Z}_H$.
\end{convention}

\subsection{Zilber-Pink conjecture}
In this section, we collect partial results toward the Zilber-Pink conjecture, that will be utilized later in the proofs.

First, recall the statement of the conjecture given in Section \ref{22080901}.  \\
\\
\noindent\textbf{Conjecture \ref{22042203}}
\textit{For every irreducible variety $\mathcal{X}(\subset \mathbb{G}^n)$ defined over $\overline{\mathbb{Q}}$, there exists a finite set $\mathcal{H}$ of proper algebraic subgroups such that, for every algebraic subgroup $K$ and every component $\mathcal{Y}$ of $\mathcal{X}\cap K$ satisfying 
\begin{equation*}
\dim \mathcal{Y}>\dim \mathcal{X}+\dim K -n, 
\end{equation*} 
one has $\mathcal{Y}\subset H $ for some $H\in \mathcal{H}$. }
\\

The conjecture is widely open in general, but there are some cases which have been resolved. 

First, as an extreme case, S. Zhang proved the following \cite{Zhang}:
\begin{theorem}[S. Zhang]\label{20102701}
Let $\mathcal{X}$ be an irreducible subvariety in $\mathbb{G}^n$. The set of torsion points in $\mathcal{X}$ lies in a finite number of algebraic subgroups.\footnote{By a \textit{torsion point}, we mean a point each of whose coordinate is a root of unity. Note that a torsion point is a torsion anomalous point, but the converse is not true in general. For instance, a Dehn filling point, whose holonomies are multiplicatively dependent, is a torsion anomalous point, whereas it is not a torsion point by Theorem \ref{040605}. } 
\end{theorem}

For the curve case, the conjecture was proved by G. Maurin \cite{mau}.

\begin{theorem}[G. Maurin]\label{20080407}
The Zilber-Pink conjecture is true if $\mathcal{X}$ is an algebraic curve. 
\end{theorem}

For varieties of codimension $2$, a special case (but the most basic case) of the conjecture was resolved by Bombieri, Masser and Zannier (Theorem 1.6 in \cite{za}):

\begin{theorem}[Bombieri-Masser-Zannier]\label{22060101}
For $n\geq 2$, let $\mathcal{X}$ be an irreducible variety in $\mathbb{G}^n$ of dimension $n-2$ defined over $\overline{\mathbb{Q}}$. Then there exists a finite set $\mathcal{H}$ of proper algebraic subgroups such that, for every algebraic subgroup $K$ of dimension $1$, every component of $\mathcal{X}\cap K$ is contained in some $H\in \mathcal{H}$. 
\end{theorem}

In our context, since the holonomy variety $\mathcal{X}$ of a $2$-cusped hyperbolic $3$-manifold is an algebraic surface in $\mathbb{G}^4$, the above theorems are directly applicable.

\begin{corollary}\label{21021602}
If $\mathcal{X}$ is the holonomy variety of a $2$-cusped hyperbolic $3$-manifold, then the Zilber-Pink conjecture holds for $\mathcal{X}$. 
\end{corollary}
\begin{proof}
By Theorem \ref{20080407} (resp. Theorem \ref{20080407}), there exists a finite set $\mathcal{H}$ of algebraic subgroup such that every anomalous subvariety of dimension $0$ (resp. $1$) is contained in some $H\in \mathcal{H}$. 
\end{proof}

Before presenting an application of Corollary \ref{21021602}, we state the following corollary, which partially resolves some special cases of the Zilber-Pink conjecture proposed in Proposition \ref{24101101}, thanks to Theorems \ref{20102701}-\ref{22060101}. The corollary is particularly useful and will be applied at various stages in Section \ref{ZPCII} to verify the main theorem. 

\begin{corollary}\label{24100901}
Let $\mathcal{X}, \mathcal{Y}_i, H_i, H$ and $\mathcal{U}_H$ be the same as in Proposition \ref{24101101}. Suppose it satisfies one of the following conditions:
\begin{enumerate}
\item $\dim H_i=\dim H$; 
\item $\dim \overline{\mathcal{U}_H}=1$ and $\codim  H_i\geq 2$;
\item $\codim  H-2=\dim \overline{\mathcal{U}_H}=2$ and $\dim H_i=\dim H+1$.
\end{enumerate}
Then there exists a finite number of algebraic subgroups whose union contains the intersection of $\mathcal{U}_H$ with the family of algebraic subgroups stated in the proposition.
\end{corollary}
\begin{proof}
\begin{enumerate}
\item If $\dim H_i=\dim H$ for all $i\in \mathcal{I}$, it means $\overline{\mathcal{U}_H}$ intersects with a family of algebraic subgroups of dimension $0$, that is, torsion points. Thus result follows from Theorem \ref{20102701}. 

\item For the second case, it means $\overline{\mathcal{U}_H}$ is an algebraic curve intersecting with a family of algebraic subgroups of codimension at most $2$. Hence the conclusion follows from Theorem \ref{20080407}. 

\item For the last case, it implies $\dim \overline{\mathcal{U}_H}$ is an algebraic variety of codimension $2$ in $\mathbb{G}^{\codim H}$, intersecting with a family of algebraic subgroups of dimension $1$. So the result follows by Theorem \ref{22060101}.  
\end{enumerate}
\end{proof}

Next, the following definition is required in order to explore a concrete application of Corollary \ref{21021602}. This concept is fundamental to the study of algebraic subgroups. 

\begin{definition}
\normalfont We say
\begin{equation*}
\alpha_1, \dots, \alpha_n\in \overline{\mathbb{Q}}\backslash \{0\}
\end{equation*}
are \textit{multiplicatively dependent} if there exists $(a_1, \dots, a_n)\in \mathbb{Z}^n\backslash \{(0, \dots, 0)\}$ such that 
\begin{equation*}
\alpha_1^{a_1}\cdots\alpha_n^{a_n}=1.
\end{equation*}
\end{definition}
A consequence of Corollary \ref{21021602} is 
\begin{theorem}\label{20071503}
Let $\mathcal{M}$ be a $2$-cusped hyperbolic $3$-manifold and $\mathcal{X}$ be its holonomy variety. Let $\{\mathcal{M}_{p_{1i}/q_{1i}, p_{2i}/q_{2i}}\}_{i\in \mathcal{I}}$ be a family of Dehn fillings of $\mathcal{M}$ with $|p_{1i}|+|q_{1i}|$ and $|p_{2i}|+|q_{2i}|$ sufficiently large such that the core holonomies of each $\mathcal{M}_{p_{1i}/q_{1i}, p_{2i}/q_{2i}}$ are multiplicatively dependent. Then there exists a finite set $\mathcal{H}$ of algebraic subgroups such that a Dehn filling point $P_i$ associated to $\mathcal{M}_{p_{1i}/q_{1i}, p_{2i}/q_{2i}}$ is contained in some $H\in \mathcal{H}$. Moreover, for each $H\in \mathcal{H}$, $\dim H=2$ and $\mathcal{X}\cap_{(1, \dots, 1 )}H$ is a $1$-dim anomalous subvariety of $\mathcal{X}$. 
\end{theorem}
\begin{proof}
By definition, $P_i$ is an intersection point between $\mathcal{X}$ and the Dehn filling equations defined by
\begin{equation*}
M_1^{p_{1i}}L_1^{q_{1i}}=M_2^{p_{2i}}L_2^{q_{2i}}=1.
\end{equation*}
Since the core holonomies of each $\mathcal{M}_{p_{1i}/q_{1i}, p_{2i}/q_{2i}}$ are multiplicatively dependent, $P_i$ is contained in an additional algebraic subgroup,\footnote{For instance, if $t_{1i}, t_{2i}$ are the core holonomies of $\mathcal{M}_{p_{1i}/q_{1i}, p_{2i}/q_{2i}}$ and $t_{1i}^{a_i}t_{2i}^{b_i}=1$ for some $a_i,b_i\in \mathbb{Z}$, $P_i$ is contained in an algebraic subgroup defined by
\begin{equation*}
(M_1^{r_{1i}}L_1^{s_{1i}})^{a_i}(M_2^{r_{2i}}L_2^{s_{2i}})^{b_i}=1 
\end{equation*}
where $p_{1i}s_{1i}-q_{1i}r_{1i}=p_{2i}s_{2i}-q_{2i}r_{2i}=1$.} that is, $P_i$ is contained in an algebraic subgroup $H_i$ of dim $1$. Thus, by Corollary \ref{21021602}, there exists a finite set $\mathcal{H}$ of algebraic subgroups such that each $P_i$ is contained in some $H\in \mathcal{H}$.

Without loss of generality, we assume every $H\in \mathcal{H}$ contains infinitely many $P_i$ and is an algebraic subgroup of the smallest dimension of having the property. Note that, as $P_i$ is sufficiently close to $(1, \dots, 1)$ by Theorem \ref{040605}, $\mathcal{X}\cap_{(1, \dots, 1)} H$ is the component of $\mathcal{X}\cap H$ containing those infinitely many $P_i$ for each $H\in \mathcal{H}$. By abuse of notation, let us still denote $\mathcal{X}\cap_{(1, \dots, 1)} H$ by $\mathcal{X}\cap H$.  

We claim $\dim H=2$ and $\dim (\mathcal{X}\cap H)=1$ for each $H\in \mathcal{H}$. On the contrary, if $\dim H=3$, since $\mathcal{X}$ itself is not anomalous (for instance, see Theorem 3.10 in \cite{jeon1}), $\dim ( \mathcal{X}\cap H)=1$. Projecting on $H$, if $\mathcal{X}\cap H$ is viewed as an algebraic curve in $\mathbb{G}^3$, then the projected image of each $H_i$ in $\mathbb{G}^3$ is still of dimension $1$. Hence, by Theorem \ref{20080407} (or Theorem \ref{22060101}), $P_i$ is further  contained in a finite list of algebraic subgroups in $\mathbb{G}^3$. But this contradicts the assumption $H$ is an algebraic subgroup of the smallest dimension containing infinitely many $P_i$. 

In conclusion, $\dim H=2$. Since $\mathcal{X}\cap H$ contains infinitely many $P_i$, $\dim (\mathcal{X}\cap H)=1$ is clear, that is, $\mathcal{X}\cap H$ is an anomalous subvariety of $\mathcal{X}$.  
\end{proof}

In Section \ref{22050609}, we pursue a quantified version of some special case of Theorem \ref{20071503} and use it to generalize Theorem \ref{22022707}. 

In the following, a simple criterion is proposed for a torsion anomalous point of a given algebraic variety $\mathcal{X}$ to be contained in a torsion anomalous subvariety of $\mathcal{X}$.

\begin{theorem}[Bombieri-Masser-Zannier,]\label{19090804}
Let $\mathcal{X}$ be an algebraic variety in $\mathbb{G}^n$ of dimension $k\leq n-1$ defined over $\overline{\mathbb{Q}}$. Then for any $B\geq 0$ there are at most finitely many points $P$ in $\mathcal{X}^{ta} \cap \mathcal{H}_{n-k-1}$ with $h(P)\leq B$ where $h(P)$ is the height of $P$, $\mathcal{H}_{n-k-1}$ is the set of algebraic subgroups of dimension $n-k-1$ and $\mathcal{X}^{ta}$ is the complement of torsion anomalous subvarieties of $\mathcal{X}$ in $\mathcal{X}$. 
\end{theorem}
\begin{proof}
See Lemma 8.1 in \cite{BMZ0}
\end{proof}

In \cite{jeon1, jeon2}, the author proved the following theorem based on the work of P. Habegger \cite{hab}. 

\begin{theorem}[B. Jeon]\label{20080404}
Let $\mathcal{X}$ is the holonomy variety of an $n$-cusped hyperbolic $3$-manifold. Then the height of any Dehn filling point on $\mathcal{X}$ is uniformly bounded. That is, there exists a constant $B$ depending only on $\mathcal{X}$ such that $h(P)<B$ for any Dehn filling point $P$ on $\mathcal{X}$.  
\end{theorem}

We neither define the height of an algebraic number nor provide basic properties of it, as none of them are used in the proofs. See \cite{bg} for details about them. 

Theorem \ref{19090804}, combining with Theorem \ref{20080404}, induces 
\begin{corollary}\label{20080403}
Let $\mathcal{M}$ be an $n$-cusped hyperbolic $3$-manifold, $\mathcal{X}$ be its holonomy variety of $\mathcal{M}$ and $\mathcal{M}_{p_1/q_1, \dots, p_n/q_n}$ be a Dehn filling of $\mathcal{M}$ with $|p_k|+|q_k|$ ($1\leq k\leq n$) sufficiently large. If the core holonomies of $\mathcal{M}_{p_1/q_1, \dots, p_n/q_n}$ are multiplicatively dependent, then a Dehn filling point $P$ associated to it is contained in a torsion anomalous subvariety of $\mathcal{X}$. 
\end{corollary}

Theorem \ref{19090804} (as well as Corollary \ref{20080403}) will play an utmost important role in the paper. Thanks to the theorem, the study of the distribution of torsion anomalous Dehn filling points over a holonomy variety $\mathcal{X}$ is reduced to analyzing the distribution of torsion anomalous subvarieties of $\mathcal{X}$. This allows us to incorporate the discussion from Section \ref{22030101} and combine it with various results presented in this subsection, including Theorems \ref{20102701}-\ref{22060101} as well as Corollaries \ref{21021602}-\ref{24100901}, to prove the main result.

\begin{remark}\label{22081302}
\normalfont For $P$ given as in the statement of Corollary \ref{20080403}, one can find an algebraic subgroup $H$ of dimension $n$ such that $P\in \mathcal{X}\cap H$ and $\mathcal{X}\cap_P H$ is a torsion anomalous subvariety of $\mathcal{X}$. In particular, if the \textit{norm} of $
X_1^{a_1}\cdots X_n^{a_n}=1$ is defined to be $\sum_{k=1}^na_k^2$, then $H$ is obtained by taking norm minimizing equations that vanish at $P$. See Lemma 8.1 in \cite{BMZ0} for detailed accounts. 
\end{remark}

\subsection{$\mathcal{X}^{oa}=\emptyset$ and Strong Geometric Isolation}\label{SGI}

Let $\mathcal{M}$ be a $2$-cusped hyperbolic $3$-manifold, $\mathcal{X}$ its holonomy variety and $\log \mathcal{X}$ be analytic holonomy variety. 
\begin{definition}\label{24091801}
\normalfont Let $v_k$ and $u_k$ ($k=1,2$) be coordinate functions over $\log \mathcal{X}$ given as in Section \ref{A-poly}. If $v_k$ depends only on $u_k$ for each $k$, then we say two cusps of $\mathcal{M}$ are \textit{strongly geometrically isolated (SGI)} from each other. 
\end{definition}

The concept was first introduced by W. Neumann and A. Reid in \cite{rigidity}. (See also \cite{calegari}). Simply put, it means one cusp of $\mathcal{M}$ deforms independently without affecting the other cusp. If two cusps of $\mathcal{M}$ are SGI each other, then $\mathcal{X}$ is the product of two algebraic curves defined by 
\begin{equation*}
f_1(M_1, L_1)=f_2(M_2, L_2)=0
\end{equation*}
in $\mathbb{G}^2\times \mathbb{G}^2(:=(M_1, L_1, M_2, L_2))$ where $f_1(M_1, L_1)=0$ (resp. $ f_2(M_2, L_2)=0$) is the holonomy variety obtained by keeping the $2^{\text{nd}}$ (resp. $1^{\text{st}}$) cusp of $\mathcal{M}$ complete. Hence, in this case, we have $\mathcal{X}^{oa}=\emptyset$, as explained in Section \ref{22030101}. 

However, the converse is not true in general. That is, the condition $\mathcal{X}^{oa}=\emptyset$ does not necessarily mean that two cusps of $\mathcal{M}$ are SGI each other. Instead, we have the following dichotomy proven in \cite{jeon5}: 
\begin{proposition}\label{21073101}
Let $\mathcal{M}$ be a $2$-cusped hyperbolic $3$-manifold and $\mathcal{X}$ be the holonomy variety of $\mathcal{M}$ in $\mathbb{G}^4(:=(M_1, L_1, M_2, L_2))$. If $\mathcal{X}^{oa}=\emptyset$, then $\mathcal{X}$ is the product of two curves. More precisely, if $\mathcal{X}^{oa}=\emptyset$, either one of the following two holds:
\begin{enumerate}
\item two cusps of $\mathcal{M}$ are SGI;
\item (if two cusps of $\mathcal{M}$ are not SGI) there exist $(2\times 2)$-matrices 
\begin{equation}\label{22061401}
A:=\left(\begin{array}{cc}
a_1 & b_1\\
a_2 & b_2
\end{array}\right), \quad 
B:=\left(\begin{array}{cc}
c_1 & d_1\\
c_2 & d_2
\end{array}\right)
\end{equation}
with $\det A=\det B \neq 0$ and a polynomial $f$ of two variables such that $\mathcal{X}$ is defined by 
\begin{equation}\label{21102301}
f(M_1^{a_1}L_1^{b_1}M_2^{c_1}L_2^{d_1}, M_1^{a_2}L_1^{b_2}M_2^{c_2}L_2^{d_2})=f(M_1^{a_1}L_1^{b_1}M_2^{-c_1}L_2^{-d_1}, M_1^{a_2}L_1^{b_2}M_2^{-c_2}L_2^{-d_2})=0.
\end{equation}
\end{enumerate}
\end{proposition}

In particular, $A$ and $B$ in \eqref{22061401} are uniquely determined by an algebraic subgroup $H$ satisfying $\mathcal{X}=\mathscr{Z}_H$. More precisely, if $H$ is an algebraic subgroup defined by either 
\begin{equation*}
M_1^{a_j}L_1^{b_j}M_2^{c_j}L_2^{d_j}=1 \quad (j=1,2)
\end{equation*}
or
\begin{equation*}
M_1^{a_j}L_1^{b_j}M_2^{-c_j}L_2^{-d_j}=1  \quad (j=1,2), 
\end{equation*} 
then $\mathcal{X}$ is foliated by anomalous subvarieties contained in algebraic cosets of $H$. 

\begin{definition}\label{22071301}
\normalfont For $\mathcal{X}$ defined by \eqref{21102301}, if we set
\begin{equation}\label{21102302}
\begin{gathered}
\tilde{M_1}:=M_1^{a_1}L_1^{b_1}M_2^{c_1}L_2^{d_1},\quad \tilde{L_1}:=M_1^{a_2}L_1^{b_2}M_2^{c_2}L_2^{d_2},\quad 
\tilde{M_2}:=M_1^{a_1}L_1^{b_1}M_2^{-c_1}L_2^{-d_1}, \quad \tilde{L_2}:=M_1^{a_2}L_1^{b_2}M_2^{-c_2}L_2^{-d_2}, 
\end{gathered}
\end{equation}
\eqref{21102301} is transformed into  
\begin{equation}\label{21102304}
f(\tilde{M_1}, \tilde{L_1})=f(\tilde{M_2}, \tilde{L_2})=0
\end{equation}
in $\mathbb{G}^4(:=(\tilde{M_1}, \tilde{L_1}, \tilde{M_2}, \tilde{L_2}))$. We call \eqref{21102302} the \textit{preferred coordinates} and \eqref{21102304} the \textit{preferred defining equations} of $\mathcal{X}$. 
\end{definition}

Technically, $f(\tilde{M_i}, \tilde{L_i}) = 0$ in $\mathbb{G}^2 (:= (\tilde{M_i}, \tilde{L_i}))$ ($i=1,2$) is not the holonomy variety of a 1-cusped hyperbolic 3-manifold, but it behaves in a similar way. For example, it satisfies all the properties outlined in Theorem \ref{potential}, so one can treat it as the holonomy variety of a 1-cusped hyperbolic 3-manifold without any distinction.\footnote{For instance, if we set 
\begin{equation*}
\tilde{u_1}:=\log \tilde{M_1},\quad \tilde{v_1}:=\log \tilde{L_1} 
\end{equation*}
and, locally near $(0,0)$, represent $\tilde{v_1}$ as a function of $\tilde{u_1}$ as $\tilde{\tau_1}\tilde{u_1}+\text{(higher order terms)}$, then, by the definition of $\tilde{M_1}$ and $\tilde{L_1}$, it easily follows that 
\begin{equation*}
\tilde{\tau_1}=\frac{a_1+b_1\tau_1}{a_2+b_2\tau_1}=\frac{c_1+d_1\tau_2}{c_2+d_2\tau_2}. 
\end{equation*}
If $\tilde{\tau_1}\in \mathbb{R}$, then $(a_1, b_1, c_1, d_1)=\tilde{\tau_1}(a_2, b_2, c_2, d_2)$ and so $\tilde{\tau_1}\in \mathbb{Q}$, contradicting the fact that $H$ is a $2$-dim algebraic subgroup. Thus $\tilde{\tau_1}\in \mathbb{C}\backslash \mathbb{R}$. See \cite{jeon5} for further details.}

Through the paper, whenever we deal with $\mathcal{X}$ falling into the second category of Proposition \ref{21073101}, we usually work those preferred coordinates and defining equations, instead of the standard ones.

\newpage
\section{Preliminary}\label{Prem}
\subsection{Anomalous subvarieties of $\mathcal{X}\times \mathcal{X}$}\label{22022301}
%In Section \ref{22022301}, we shall investigate every possible $H$ satisfying $\mathscr{Z}_H=\mathcal{X}\times \mathcal{X}$ when $\mathcal{X}$ is the holonomy variety of a $2$-cusped hyperbolic $3$-manifold. 

In this subsection, we analyze the structure of anomalous subvarieties of $\mathcal{X}\times \mathcal{X}$. In particular, we characterize all the possible cases that $\mathcal{X}\times \mathcal{X}$ is foliated by infinitely many anomalous subvarieties of dimension either $2$ or $3$. In some sense, the lemmas here can be seen as generalizations of Proposition \ref{21073101}. The proof of each lemma is based on a combination of Theorems \ref{potential} and \ref{struc}. 

The lemmas shall be needed in the proof of one of the main theorems, Theorem \ref{22022601}, later in Section \ref{ZPCII}. However, as it only requires the statement of each lemma, not the proof itself, we recommend a reader to skip  this subsection (along with Section \ref{24081301}) at first reading and proceed directly to Section \ref{Prem I}. 

\begin{lemma}\label{20080601}
Let $\mathcal{M}$ be a $2$-cusped hyperbolic $3$-manifold and $\mathcal{X}$ be its holonomy variety. If $\mathcal{X}\times \mathcal{X}$ is foliated by infinitely many anomalous subvarieties $\{\mathcal{Y}_i\}_{i\in \mathcal{I}}$ of dim $3$ such that each $\mathcal{Y}_i$ is contained in a translation of some algebraic subgroup $H$, then $\mathcal{X}^{oa}=\emptyset$. Moreover, 
\begin{enumerate}
\item if two cusps of $\mathcal{M}$ are SGI each other, then $H$ is defined by either $M_k=L_k=1$ or $M'_k=L'_k=1$ for some $k=1,2$;
\item otherwise, if two cusps of $\mathcal{M}$ are not SGI each other, then $H$ is defined by either $\tilde{M_k}=\tilde{L_k}=1$ or $\tilde{M'_k}=\tilde{L'_k}=1$ for some $k=1,2$ where $\tilde{M_k},\tilde{L_k},\tilde{M'_k},\tilde{L'_k}$ are the preferred coordinates of $\mathcal{X}\times \mathcal{X}$ given in \eqref{21102302}. 
\end{enumerate}
\end{lemma}
\begin{proof}
Without loss of generality, let us assume $H$ is defined by 
\begin{equation*}
\begin{gathered}
M_1^{a_{j1}}L_1^{b_{j1}}M_2^{a_{j2}}L_2^{b_{j2}}(M'_1)^{a_{j3}}(L'_1)^{b_{j3}}(M'_2)^{a_{j4}}(L'_2)^{b_{j4}}=1\quad (j=1,2).
\end{gathered}
\end{equation*}
Moving to the analytic holonomy variety, $\log\big((\mathcal{X}\times\mathcal{X})\cap H)\big)$ is defined by
\begin{equation}\label{21100906}
a_{j1}u_1+b_{j1}v_1+a_{j2}u_1+b_{j2}v_2+a_{j3}u'_1+b_{j3}v'_1+a_{j4}u'_2+b_{j4}v'_2=0, \quad (j=1,2).
\end{equation}
As $\log(\mathcal{X}\times \mathcal{X})$ is foliated by translations of \eqref{21100906}, equivalently, there exists an analytic function $\mathfrak{h}$ such that 
\begin{equation}\label{20080602}
\begin{aligned}
a_{21}u_1+b_{21}v_1+\cdots+a_{24}u'_2+b_{24}v'_2
=\mathfrak{h}(a_{11}u_1+b_{11}v_1+\cdots+a_{14}u'_2+b_{14}v'_2).
\end{aligned}
\end{equation}
Letting $u'_1=v'_1=u'_2=v'_2=0$ in \eqref{20080602}, \eqref{20080602} is reduced to 
\begin{equation*}
\begin{gathered}
a_{21}u_1+b_{21}v_1+a_{22}u_2+b_{22}v_2=\mathfrak{h}(a_{11}u_1+b_{11}v_1+a_{12}u_2+b_{12}v_2),
\end{gathered}
\end{equation*}
which implies $\mathcal{X}^{oa}=\emptyset$.

If two cusps of $\mathcal{M}$ are SGI from each other, by Definition \ref{24091801}, it means each $v_k$ (resp. $v'_k$) depends only on $u_k$ (resp. $u'_k$) for each $k=1,2$. Hence it follows that either $\mathfrak{h}$ in \eqref{20080602} is a linear function, or $a_{1l}=b_{1l}=a_{2l}=b_{2l}=0$ for all except for one $l$ ($1\leq l\leq 4$). However, if $\mathfrak{h}$ is linear, it contradicts the fact that $H$ is an algebraic subgroup of codimension $2$ and so the conclusion follows. 

Similarly, we get the desired result for the rest case (i.e. two cusps of $\mathcal{M}$ are not SGI each other).  
\end{proof}

More generally, using the same methods given in the above lemma, one is further able to prove the following:
\begin{lemma}\label{22060401}
Let $\mathcal{X}:=\mathcal{X}_1\times \cdots \times \mathcal{X}_n\big(\subset \mathbb{G}^{2}\times \cdots \times \mathbb{G}^2(:=(M_1, L_1, \dots, M_n, L_n))\big)$ be the product of $n$-algebraic curves where each $\mathcal{X}_k$ ($1\leq k\leq n$) is the holonomy variety of some $1$-cusped hyperbolic $3$-manifold. If $\mathcal{X}$ contains infinitely many $(n-1)$-dim anomalous subvarieties $\{\mathcal{Y}_i\}_{i\in \mathcal{I}}$ of $\mathcal{X}$ such that $\mathcal{Y}_i$ is contained in a translation of some algebraic subgroup $H$, then $H$ is defined by $M_k=L_k=1$ for some $1\leq k\leq n$. 
\end{lemma}

As observed earlier in Section \ref{22030101},\footnote{See the discussion after Theorem \ref{struc}.} $(\mathcal{X}\times \mathcal{X})^{oa}=\emptyset$ for any algebraic variety $\mathcal{X}$. In particular, if $\mathcal{X}$ is the holonomy variety of a $2$-cusped hyperbolic $3$-manifold, then $\mathcal{X}\times \mathcal{X}=\mathscr{Z}_H$ where $H$ is defined by either $M_1=L_1=M_2=L_2=1$ or $M'_1=L'_1=M'_2=L'_2=1$. In the lemma below, if $\mathcal{X}^{oa}\neq \emptyset$ and $\mathcal{X}\times \mathcal{X}$ is foliated by $2$-dimensional anomalous subvarieties of it, we show this is always the case. For $\mathcal{X}^{oa}=\emptyset $, if $\mathcal{X}\times \mathcal{X}$ is foliated by its $2$-dimensional anomalous subvarieties, then those are contained in $3$-dimensional ones described in Lemma \ref{20080601}.

\begin{lemma}\label{20050801}
Let $\mathcal{M}$ and $\mathcal{X}$ be as above. Suppose $\mathcal{X}\times \mathcal{X}$ is foliated by infinitely many $2$-dim anomalous subvarieties $\{\mathcal{Y}_i\}_{i\in \mathcal{I}}$ such that each $\mathcal{Y}_i$ is contained in a translation of an algebraic subgroup $H$. 
\begin{enumerate}
\item If $\mathcal{X}^{oa}\neq \emptyset$, then $\mathcal{Y}_i$ is contained in a translation of either $M_1=L_1=M_2=L_2=1$ or $M'_1=L'_1=M'_2=L'_2=1$.
\item If $\mathcal{X}^{oa}=\emptyset$ and two cusps of $\mathcal{M}$ are SGI each other, then $\mathcal{Y}_i$ is contained in a translation of either $M_k=L_k=1$ or $M'_k=L'_k=1$ for some $1\leq k\leq 2$. 
\item If $\mathcal{X}^{oa}=\emptyset$ and two cusps of $\mathcal{M}$ are not SGI each other, under the preferred coordinates introduced in Definition \ref{22071301}, $\mathcal{Y}_i$ is contained in a translation of either $\tilde{M_k}=\tilde{L_k}=1$ or $\tilde{M'_k}=\tilde{L'_k}=1$ for some $1\leq k\leq 2$. 
\end{enumerate}
\end{lemma}

\begin{proof}
Assume $\mathcal{Y}_i$ is contained in a translation of $H$, (applying Gauss elimination if necessary) defined by\footnote{Here, we switch the order of the meridian-longitude pairs, rearranging variables as
\begin{equation*}
L_1, M_1, L_2, M_2, L'_1, M'_1, L'_2, M'_2
\end{equation*}
rather than \begin{equation*}
M_1, L_1, M_2, L_2, M'_1, L'_1, M'_2, L'_2.
\end{equation*}}
\begin{equation*}
\begin{aligned}
&L_1^{a_1}M_1^{b_1}L_2^{c_1}M_2^{d_1}(L'_1)^{a'_1}(M'_1)^{b'_1}=1,\\
&M_1^{b_2}L_2^{c_2}M_2^{d_2}(L'_1)^{a'_2}(M'_1)^{b'_2}(L'_2)^{c'_2}=1,\\
&L_2^{c_3}M_2^{d_3}(L'_1)^{a'_3}(M'_1)^{b'_3}(L'_2)^{c'_3}(M'_2)^{d'_3}=1.
\end{aligned}
\end{equation*}
Moving to an analytic setting, $\log\big((\mathcal{X}\times \mathcal{X})\cap H\big)$ is given as 
\begin{equation*}
\begin{gathered}
a_1v_1+b_1u_1+c_1v_2+d_1u_2+a'_1v'_1+b'_1u'_1=0,\\
b_2u_1+c_2v_2+d_2u_2+a'_2v'_1+b'_2u'_1+c'_2v'_2=0,\\
c_3v_2+d_3u_2+a'_3v'_1+b'_3u'_1+c'_3v'_2+d'_3u'_2=0.
\end{gathered}
\end{equation*}
If $\Xi \subset \mathbb{C}^3$ is defined to be the set of all $(\xi_1, \xi_2, \xi_3) \in \mathbb{C}^3$ such that 
\begin{equation*}
\begin{gathered}
\xi_1=a_1v_1+b_1u_1+c_1v_2+d_1u_2+a'_1v'_1+b'_1u'_1,\\
\xi_2=b_2u_1+c_2v_2+d_2u_2+a'_2v'_1+b'_2u'_1+c'_2v'_2,\\
\xi_3=c_3v_2+d_3u_2+a'_3v'_1+b'_3u'_1+c'_3v'_2+d'_3u'_2
\end{gathered}
\end{equation*}
is an analytic set of dim $2$, then $\Xi$ is a complex manifold of dim $2$ by Corollary \ref{22022603}, and thus, for
\begin{equation}\label{21020401}
\begin{gathered}
\mathfrak{s}_1:=a_1v_1+b_1u_1+c_1v_2+d_1u_2+a'_1v'_1+b'_1u'_1,\\
\mathfrak{s}_2:=b_2u_1+c_2v_2+d_2u_2+a'_2v'_1+b'_2u'_1+c'_2v'_2,\\
\mathfrak{s}_3:=c_3v_2+d_3u_2+a'_3v'_1+b'_3u'_1+c'_3v'_2+d'_3u'_2,
\end{gathered}
\end{equation}
there exists a holomorphic function $\mathfrak{h}$ such that 
\begin{equation}\label{22060801}
\mathfrak{s}_2=\mathfrak{h}(\mathfrak{s}_1, \mathfrak{s}_3)=\sum n_{\alpha,\beta}\mathfrak{s}_1^{\alpha}\mathfrak{s}_3^{\beta}, \quad n_{\alpha,\beta}\in \mathbb{C}.
\end{equation}
Note that we have the following three choices for the coefficients appeared in $\mathfrak{s}_1$:
\begin{itemize}
\item $(a_1, b_1, c_1, d_1)\neq (0,0,0,0)$ and $(a'_1, b'_1)=(0,0)$;
\item $(a_1, b_1, c_1, d_1)=(0,0,0,0)$ and $(a'_1, b'_1)\neq (0,0)$;
\item $(a_1, b_1, c_1, d_1)\neq (0,0,0,0)$ and $(a'_1, b'_1)\neq (0,0)$. 
\end{itemize}
Similarly, either one of the following holds for $\mathfrak{s}_3$:
\begin{itemize}
\item $(a'_3, b'_3, c'_3, d'_3)=(0,0,0,0)$ and $(c_3, d_3)\neq (0,0)$;
\item $(a'_3, b'_3, c'_3, d'_3)\neq (0,0,0,0)$ and $(c_3, d_3)=(0,0)$;
\item $(a'_3, b'_3, c'_3, d'_3)\neq (0,0,0,0)$ and $(c_3, d_3)\neq (0,0)$. 
\end{itemize}
We consider all possible combinations of the above cases and show that each satisfies one of the statements of the lemma. 
\begin{enumerate}
\item If $(a'_1, b'_1)=(0,0)$ and $(a'_3, b'_3, c'_3, d'_3)=(0,0,0,0)$, then \eqref{21020401} is reduced to 
\begin{equation*}
\begin{gathered}
\mathfrak{s}_1=a_1v_1+b_1u_1+c_1v_2+d_1u_2,\quad \mathfrak{s}_2=b_2u_1+c_2v_2+d_2u_2,\quad \mathfrak{s}_3=c_3v_2+d_3u_2,
\end{gathered}
\end{equation*}
and, applying Gauss elimination, it is further simplified to
\begin{equation*}
\begin{gathered}
\mathfrak{s}_1=a_1v_1+b_1u_1,\quad \mathfrak{s}_2=b_2u_1+c_2v_2+d_2u_2,\quad \mathfrak{s}_3=c_3v_2+d_3u_2.
\end{gathered}
\end{equation*}
Note that if $\mathfrak{s}_1, \mathfrak{s}_2, \mathfrak{s}_3$ are constants, then $u_k, v_k$ ($k=1,2$) are all constants. In other words, each $\mathcal{Y}_i$ is contained in a translation of $M_1=L_1=M_2=L_2=1$. 

\item Similar to the above case, if $(a_1, b_1, c_1, d_1)=(0,0,0,0)$ and $(c_3, d_3)=(0,0)$, it follows that $\mathcal{Y}_i$ is contained in a translation of $M'_1=L'_1=M'_2=L'_2=1$. 

\item If $(a'_1, b'_1)=(c_3, d_3)=(0, 0)$, then \eqref{21020401} is reduced to 
\begin{equation}\label{22060713}
\begin{aligned}
&\mathfrak{s}_1=a_1v_1+b_1u_1+c_1v_2+d_1u_2,\\
&\mathfrak{s}_2=b_2u_1+c_2v_2+d_2u_2+a'_2v'_1+b'_2u'_1+c'_2v'_2,\\
&\mathfrak{s}_3=a'_3v'_1+b'_3u'_1+c'_3v'_2+d'_3u'_2. 
\end{aligned}
\end{equation}
Since $\mathfrak{s}_2$ does not contain terms of the following forms 
\begin{equation}\label{22060705}
u_1^{\alpha}(u'_1)^{\beta}, u_1^{\alpha}(u'_2)^{\beta}, u_2^{\alpha}(u'_1)^{\beta}\text{  and  }u_2^{\alpha}(u'_2)^{\beta} \text{  with  } \alpha,\beta>0, 
\end{equation}
we have $n_{\alpha,\beta}=0$ for all $\alpha,\beta>0$. That is, $\mathfrak{h}$ is represented as 
\begin{equation}\label{22060701}
\mathfrak{s}_2=\mathfrak{h}(\mathfrak{s}_1, \mathfrak{s}_3)=\sum^{\infty}_{\alpha=1}n_{\alpha,0}\mathfrak{s}_1^{\alpha}+\sum^{\infty}_{\beta=1}n_{0,\beta}\mathfrak{s}_3^{\beta},
\end{equation}
Now \eqref{22060701} splits into
\begin{equation}\label{22060703}
b_2u_1+c_2v_2+d_2u_2=\sum^{\infty}_{\alpha=1}n_{\alpha,0}(a_1v_1+b_1u_1+c_1v_2+d_1u_2)^{\alpha}
\end{equation}
and 
\begin{equation}\label{22060704}
a'_2v'_1+b'_2u'_1+c'_2v'_2=\sum^{\infty}_{\beta=1}n_{0,\beta}(a'_3v'_1+b'_3u'_1+c'_3v'_2+d'_3u'_2)^{\beta},
\end{equation}
thus it is concluded that $\mathcal{X}^{oa}=\emptyset$. 

If two cusps of $\mathcal{M}$ are SGI each other, by Lemma \ref{22060401},  \eqref{22060703} (resp. \eqref{22060704}) is equivalent to either $v_1=u_1\cdot\tau_1(u_1)$ or $v_2=u_2\cdot \tau_2(u_2)$ (resp. $v'_1=u'_1\cdot \tau_1(u'_1)$ or $v'_2=u'_2\cdot\tau_2(u'_2)$), which further implies $\mathcal{Y}_i$ is either contained in a translation of either $M_1=L_1=1$ or $M_2=L_2=1$ or $M'_1=L'_1=1$ or $M'_2=L'_2=1$. 

Similarly, if two cusps of $\mathcal{M}$ are not SGI each other, then it falls into the description of the last statement of the lemma, thanks to Proposition \ref{21073101} and Lemma \ref{22060401}.   

\item For $(a_1, b_1, c_1, d_1)=(a'_3, b'_3, c'_3, d'_3)=(0,0,0,0)$, the result is attained by an analogous manner above.

\item Next suppose 
\begin{equation*}
(a_1,b_1,c_1,d_1)\neq (0,0,0,0), (a'_1, b'_1)=(0, 0)\text{  and  }(c_3, d_3)\neq (0,0), (a'_3, b'_3, c'_3, d'_3)\neq (0,0,0,0). 
\end{equation*}
Then \eqref{21020401} is refined to 
\begin{equation}\label{22060601}
\begin{aligned}
&\mathfrak{s}_1=a_1v_1+b_1u_1+c_1v_2+d_1u_2,\\
&\mathfrak{s}_2=b_2u_1+c_2v_2+d_2u_2+a'_2v'_1+b'_2u'_1+c'_2v'_2,\\
&\mathfrak{s}_3=c_3v_2+d_3u_2+a'_3v'_1+b'_3u'_1+c'_3v'_2+d'_3u'_2.
\end{aligned}
\end{equation}

\begin{enumerate}
\item If $(a_1, b_1)\neq (0,0)$, as $\mathfrak{s}_2$ does not have any terms of the forms listed in \eqref{22060705}, we get $n_{\alpha,\beta}=0$ for all $\alpha,\beta>0$ and $n_{0,\beta}=0$ for $\beta\geq 2$. As a result, $\mathfrak{h}$ is given as 
\begin{equation}\label{22060707}
\mathfrak{s}_2=\mathfrak{h}(\mathfrak{s}_1, \mathfrak{s}_3)=\sum_{\alpha=1}^{\infty}n_{\alpha,0}\mathfrak{s}_1^{\alpha}+n_{0,1}\mathfrak{s}_3.
\end{equation}
Since $\mathfrak{s}_1$ is a function of $u_1$ and $u_2$,
\begin{equation}\label{22060709}
a'_2v'_1+b'_2u'_1+c'_2v'_2=n_{0,1}(a'_3v'_1+b'_3u'_1+c'_3v'_2+d'_3u'_2)
\end{equation}
is obtained from \eqref{22060707} and, as $v'_1$ (resp. $v'_2$) contains a term of the form $(u'_1)^{\alpha}$ (resp. $(u'_2)^{\beta}$) with $\alpha\geq 3$ (resp. $\beta\geq 3$), it follows that $a'_2=n_{0,1}a'_3$ and $c'_2=n_{0,1}c'_3$ by \eqref{22060709}. Thus $n_{0,1}\in \mathbb{Q}$ and 
\begin{equation*}
\mathfrak{s}_2-n_{0,1}\mathfrak{s}_3=(b_2u_1+c_2v_2+d_2u_2)-n_{0,1}(c_3v_2+d_3u_2)=\sum_{\alpha=1}^{\infty}n_{\alpha,0}\mathfrak{s}_1^{\alpha}.
\end{equation*}
Note that, as $\mathfrak{s}_2$ and $\mathfrak{s}_3$ are linearly independent (over $\mathbb{Q}$), $\mathfrak{s}_2-n_{0,1}\mathfrak{s}_3$ is non-trivial and so $\mathcal{X}^{oa}=\emptyset$ is concluded. 

As observed previously, if two cusps of $\mathcal{M}$ are SGI each other, then $\mathfrak{s}_2-n_{0,1}\mathfrak{s}_3=\sum_{\alpha=1}^{\infty}n_{\alpha,0}\mathfrak{s}_1^{\alpha}$ is equivalent to either $v_1=u_1\cdot \tau_1(u_1)$ or $v_2=u_2\cdot \tau_2(u_2)$, implying each $\mathcal{Y}_i$ is contained in a translation of either $M_1=L_1=1$ or $M_2=L_2=1$. For the other case (i.e. two cusps of $\mathcal{M}$ are not SGI), one gets the desired result in a similar fashion.
 
\item If $(a_1, b_1)=(0,0)$, then \eqref{22060601} is simplified to 
\begin{equation}\label{22060710}
\begin{aligned}
&\mathfrak{s}_1=c_1v_2+d_1u_2,\\
&\mathfrak{s}_2=b_2u_1+c_2v_2+d_2u_2+a'_2v'_1+b'_2u'_1+c'_2v'_2,\\
&\mathfrak{s}_3=c_3v_2+d_3u_2+a'_3v'_1+b'_3u'_1+c'_3v'_2+d'_3u'_2, 
\end{aligned}
\end{equation}
and, as $u_1$ neither appears in $\mathfrak{s}_1$ nor $\mathfrak{s}_3$, $b_2=0$ in \eqref{22060710}. Applying Gauss elimination if necessary, \eqref{22060710} is further reduced to 
\begin{equation*}
\begin{gathered}
\mathfrak{s}_1=c_1v_2+d_1u_2,\quad \mathfrak{s}_2=c_2v_2+a'_2v'_1+b'_2u'_1+c'_2v'_2,\quad \mathfrak{s}_3=a'_3v'_1+b'_3u'_1+c'_3v'_2+d'_3u'_2.
\end{gathered}
\end{equation*}
Since $\mathfrak{s}_1$ (resp. $\mathfrak{s}_3$) is a holomorphic function of $u_1$ and $u_2$ (resp. $u'_1$ and $u'_2$), the case falls into a similar case to the one discussed in \eqref{22060713}, and thus the result follows. 
\end{enumerate}

\item Using the same method shown previously right above, one attains the desired result for each of the following cases:
\begin{enumerate}
\item $(a_1, b_1, c_1, d_1)=(0, 0, 0, 0), (a'_1, b'_1)\neq (0,0), (c_3, d_3)\neq (0,0)$, $(a'_3, b'_3, c'_3, d'_3)\neq (0,0,0,0)$;
\item $(a'_1, b'_1)\neq (0,0), (a_1, b_1, c_1, d_1)\neq (0,0,0,0), (a'_3, b'_3, c'_3, d'_3)=(0,0,0,0), (c_3, d_3)\neq (0,0)$;
\item $(a'_1, b'_1)\neq (0,0),(a_1, b_1, c_1, d_1)\neq (0,0,0,0),(a'_3, b'_3, c'_3, d'_3)\neq(0,0,0,0),(c_3, d_3)=(0,0)$.
\end{enumerate}

\item Finally, let us consider 
\begin{equation*}
(a'_1, b'_1)\neq (0,0), (a_1, b_1, c_1, d_1)\neq (0,0,0,0) \text{  and  } (c_3, d_3)\neq (0,0), (a'_3, b'_3, c'_3, d'_3)\neq (0,0,0,0).
\end{equation*} 
\begin{enumerate}
\item If $(a_1, b_1)\neq (0,0)$, as $\mathfrak{s}_2$ does not contain terms of the forms given in \eqref{22060705}, $n_{\beta,0}=0$ for $\beta\geq 2$, $n_{\alpha,\beta}=0$ for $\alpha,\beta>0$, and hence $n_{0,\beta}=0$ for $\beta\geq 2$. Consequently, \eqref{22060801} is  
\begin{equation*}
\begin{aligned}
\mathfrak{s}_2&=b_2u_1+c_2v_2+d_2u_2+a'_2v'_1+b'_2u'_1+c'_2v'_2=\mathfrak{h}(\mathfrak{s}_1, \mathfrak{s}_3)=n_{1,0}\mathfrak{s}_1+n_{0,1}\mathfrak{s}_3\\
&=n_{1,0}(a_1v_1+b_1u_1+c_1v_2+d_1u_2+a'_1v'_1+b'_1u'_1)+n_{0,1}(c_3v_2+d_3u_2+a'_3v'_1+b'_3u'_1+c'_3v'_2+d'_3u'_2).
\end{aligned}
\end{equation*}
Analogous to an earlier case, it follows that $n_{1,0}, n_{0,1}\in \mathbb{Q}$, but this contradicts the fact that $H$ is an algebraic subgroup of codimension $3$. 

\item For $(c'_3, d'_3)\neq (0,0)$, it can be treated similarly to the case right above. 
\item If $(a_1, b_1)=(c'_3,d'_3)=(0,0)$, then \eqref{21020401} is reduced to 
\begin{equation}\label{22060603}
\begin{aligned}
&\mathfrak{s}_1=c_1v_2+d_1u_2+a'_1v'_1+b'_1u'_1,\\
&\mathfrak{s}_2=b_2u_1+c_2v_2+d_2u_2+a'_2v'_1+b'_2u'_1+c'_2v'_2,\\
&\mathfrak{s}_3=c_3v_2+d_3u_2+a'_3v'_1+b'_3u'_1.
\end{aligned}
\end{equation}
Since $\mathfrak{s}_1$ and $\mathfrak{s}_3$ do not have $u_1$ and $v'_2$, we get $b_2=c'_2=0$,  and thus, applying Gauss elimination again, \eqref{22060603} is further refined as
\begin{equation*}
\begin{gathered}
\mathfrak{s}_1=c_1v_2+d_1u_2,\quad \mathfrak{s}_2=c_2v_2+d_2u_2+a'_2v'_1+b'_2u'_1,\quad 
\mathfrak{s}_3=a'_3v'_1+b'_3u'_1.  
\end{gathered}
\end{equation*}
\end{enumerate}
In conclusion, $\mathfrak{s}_2=\sum^{\infty}_{\alpha=1}n_{\alpha,0}\mathfrak{s}_1^{\alpha}+\sum^{\infty}_{\beta=1}n_{0,\beta}\mathfrak{s}_3^{\beta}$, and so 
\begin{equation*}
c_2v_2+d_2u_2=\sum^{\infty}_{\alpha=1}n_{\alpha,0}(c_1v_2+d_1u_2)^{\alpha},\quad  \quad  a'_2v'_1+b'_2u'_1=\sum^{\infty}_{\beta=1}n_{0,\beta}(a'_3v'_1+b'_3u'_1)^{\beta}.
\end{equation*}
That is, two cusps of $\mathcal{M}$ are SGI and each $\mathcal{Y}_i$ is contained in a translation of either $M_2=L_2=1$ or $M'_1=L'_1=1$. 
\end{enumerate}
\end{proof}

The following lemma is proven by analogous methods shown in the proof of the above lemma. We state it here for later reference but skip the proof. 

\begin{lemma}\label{21100201}
Let $\mathcal{X}:=\mathcal{X}_1\times \mathcal{X}_2\times \mathcal{X}_3(\subset \mathbb{G}^2\times \mathbb{G}^2 \times \mathbb{G}^2:=(M_1, L_1, M_2, L_2, M_3, L_3))$ be the product of three algebraic curves where each $\mathcal{X}_k$ ($1\leq k\leq 3$) is the holonomy variety of a $1$-cusped hyperbolic $3$-manifold. If $\mathcal{X}$ is foliated by $1$-dim anomalous subvarieties $\{\mathcal{Y}_i\}_{i\in \mathcal{I}}$ of $\mathcal{X}$, then
\begin{equation*}
\mathcal{Y}_i\subset (M_k=L_k=1)
\end{equation*}
for some $1\leq k\leq 3$. 
\end{lemma}

\subsection{Anomalous subvarieties of $\overline{\text{Pr}\big((\mathcal{X}\times \mathcal{X})\cap H\big)}$}\label{24081301}

To prove Theorem \ref{22022601}, it is first shown in Theorem \ref{21082901} that there exists a finite set $\mathcal{H}$ of algebraic subgroups of codimension $1$ such that a Dehn filling point associated to \eqref{22041801} is contained in some $H\in \mathcal{H}$. Then, in Theorem \ref{20071401}, we will further promote each $H\in \mathcal{H}$ to an algebraic subgroup of codimension $2$ by solving a weak version of the Zilber-Pink conjecture over $(\mathcal{X}\times \mathcal{X})\cap_{(1, \dots, 1)} H$. 

To explain more in detail, let $H\in \mathcal{H}$ be an algebraic subgroup defined by
\begin{equation}\label{22012002}
M_1^{a}L_1^{b}M_2^cL_2^d(M'_1)^{a'}(L'_1)^{b'}(M'_2)^{c'}(L'_2)^{d'}=1. 
\end{equation}
By moving to an analytic setting, $\log \big((\mathcal{X}\times \mathcal{X}) \cap _{(1, \dots, 1)}H\big)$ is a $3$-dim analytic set defined by
\begin{equation}\label{24081403}
au_1+bv_1+cu_2+dv_2+a'u'_1+b'v'_1+c'u'_2+d'v'_2=0. 
\end{equation}
Without loss of generality, if we assume $(c',d')\neq (0,0)$ in \eqref{24081403} and solve the equation for $u'_2$, then it is expressed in the following form:
\begin{equation}\label{220122031}
u'_2=-\frac{a+b\tau_1}{c'+d'\tau_2}u_1-\frac{c+d\tau_2}{c'+d'\tau_2}u_2-\frac{a'+b'\tau_1}{c'+d'\tau_2}u'_1+\text{(higher order terms)}.
\end{equation}
Now we consider $\log \big((\mathcal{X}\times \mathcal{X}) \cap_{(1, \dots, 1)} H\big)$ (resp. $(\mathcal{X}\times \mathcal{X})_{(1, \dots, 1)} \cap H)$ as an analytic set (resp. algebraic variety) parametrized by $u_1, u_2$ and $u'_1$ (resp. $M_1, M_2$ and $M'_1$). 

Projecting $(\mathcal{X}\times \mathcal{X}) \cap_{(1, \dots, 1)} H$ under  
\begin{equation*}
\text{Pr}:\;\big(M_1, L_1, M_2, L_2, M'_1, L'_1,M'_2,L'_2\big)\longrightarrow \big(M_1, L_1, M_2, L_2, M'_1, L'_1\big), 
\end{equation*}
denote the algebraic closure of its image simply by 
\begin{equation}\label{24101901}
\overline{\text{Pr}\big((\mathcal{X}\times \mathcal{X})\cap H\big)}. 
\end{equation}
Later in Section \ref{22042905}, to prove Theorem \ref{20071401} we work with the \eqref{24101901} instead of $(\mathcal{X}\times \mathcal{X})\cap _{(1, \dots, 1)}H$. That is, we solve a weak version of the Zilber-Pink conjecture over \eqref{24101901} rather than $(\mathcal{X}\times \mathcal{X})\cap_{(1, \dots, 1)} H$ itself. Thus it is required to understand the structure of anomalous subvarieties of \eqref{24101901} and that is the job for this subsection. Similar to Section \ref{22022301}, we would be interested in the case that \eqref{24101901} is foliated by its anomalous subvarieties, in particular, of codimension $1$.  

\begin{lemma}\label{20072401}
Let $\mathcal{X}$ the holonomy variety of a $2$-cusped hyperbolic $3$-manifold satisfying $\mathcal{X}^{oa}\neq \emptyset$. Let $H$ and $\overline{\text{Pr}\big((\mathcal{X}\times \mathcal{X})\cap H\big)}$ be the same as given above. If $\overline{\text{Pr}\big((\mathcal{X}\times \mathcal{X})\cap H\big)}$ is foliated by infinitely many anomalous subvarieties $\{\mathcal{Y}_i\}_{i\in \mathcal{I}}$ of dim $2$ where each $\mathcal{Y}_i$ contained in a translation of $H'$, then $(a, b, c, d)=(0,0,0,0)$ in \eqref{22012002} and $H'$ is defined by $M'_1=L'_1=1$. 
\end{lemma}
\begin{proof}
Without loss of generality, we suppose $H'$ is defined by
\begin{equation}\label{22012001}
\begin{gathered}
M_1^{a_j}L_1^{b_j}M_2^{c_j}L_2^{d_j}(M'_1)^{a'_j}(L'_1)^{b'_j}=1\quad (j=1,2), 
\end{gathered}
\end{equation}
and hence $\log\big(\overline{\text{Pr}((\mathcal{X}\times \mathcal{X})\cap H)}\cap H'\big)$ is given as
\begin{equation*}
\begin{gathered}
a_ju_1+b_jv_1+c_ju_2+d_jv_2+a'_ju'_1+b'_jv'_1=0\quad (j=1,2). 
\end{gathered}
\end{equation*}
Since $\overline{\text{Pr}\big((\mathcal{X}\times \mathcal{X})\cap H\big)}$ is foliated by anomalous subvarieties contained in translations of $H'$, equivalently, there exists a holomorphic function $\mathfrak{h}$ such that  
\begin{equation}\label{20072402}
\begin{aligned}
a_2u_1+b_2v_1+c_2u_1+d_2v_2+a'_2u'_1+b'_2v'_1
=\mathfrak{h}(a_1u_1+b_1v_1+c_1u_1+d_1v_2+a'_1u'_1+b'_1v'_1).
\end{aligned}
\end{equation}
\begin{enumerate}
\item If 
\begin{equation*}
\left(\begin{array}{cccc}
a_1 & b_1 & c_1 & d_1\\ 
a_2 & b_2 & c_2 & d_2 
\end{array}\right)\neq 
\left(\begin{array}{cccc}
0 & 0 & 0 & 0\\ 
0 & 0 & 0 & 0 
\end{array}\right),  
\end{equation*}
plugging $u'_1=v'_1=0$, \eqref{20072402} becomes 
\begin{equation}
\begin{gathered}
a_2u_1+b_2v_1+c_2u_1+d_2v_2=\mathfrak{h}(a_1u_1+b_1v_1+c_1u_1+d_1v_2).
\end{gathered}
\end{equation}
However, this implies $\mathcal{X}$ has infinitely many anomalous subvarieties of dim $1$, which contradicts our assumption that $\mathcal{X}^{oa}\neq \emptyset$. 

As a result, 
\begin{equation*}
\left(\begin{array}{cccc}
a_1 & b_1 & c_1 & d_1\\ 
a_2 & b_2 & c_2 & d_2 
\end{array}\right)= 
\left(\begin{array}{cccc}
0 & 0 & 0 & 0\\ 
0 & 0 & 0 & 0 
\end{array}\right)
\end{equation*}
and $H'$ in \eqref{22012001} is reduced to $M'_1=L'_1=1$, which further simplifies \eqref{20072402} as 
\begin{equation}\label{22012005}
v'_1=\mathfrak{h}(u'_1).
\end{equation}
That is, $v'_1$ depends only on $u'_1$ over $\log \big((\mathcal{X}\times \mathcal{X})\cap H\big)$. 

\item Next we claim $(a,b,c,d)=(0,0,0,0)$ in \eqref{22012002}. Let $v'_1$ be represented as 
\begin{equation}\label{22072301}
\sum_{\substack{\alpha:odd\\ \beta:even
}} m_{\alpha, \beta}(u'_1)^{\alpha}(u'_2)^{\beta}
\end{equation}
over $\log (\mathcal{X}\times \mathcal{X})$. Since $\mathcal{X}^{oa}\neq \emptyset$, we get $m_{\alpha,\beta}\neq 0$ for some $\alpha,\beta>0$ in \eqref{22072301}. Combining \eqref{22072301} with \eqref{220122031}, it follows that 
\begin{equation}\label{22012006}
v'_1=\sum_{\substack{\alpha:odd\\ \beta:even
}}  m_{\alpha, \beta}(u'_1)^{\alpha}\Big(-\frac{a+b\tau_1}{c'+d'\tau_2}u_1-\frac{c+d\tau_2}{c'+d'\tau_2}u_2-\frac{a'+b'\tau_1}{c'+d'\tau_2}u'_1+\text{(higher order terms)}\Big)^{\beta}. 
\end{equation}
However, as $v'_1$ depends only on $u'_1$ over $\log \big((\mathcal{X}\times \mathcal{X})\cap H\big)$ by \eqref{22012005}, $(a,b,c,d)=(0,0,0,0)$ in \eqref{22012006}. 
\end{enumerate}
\end{proof}

For $\mathcal{X}^{oa}=\emptyset$, if two cusps of $\mathcal{M}$ are SGI each other, then $\overline{\text{Pr}\big((\mathcal{X}\times \mathcal{X})\cap H\big)}$ is the product of three algebraic curves and thus the structure of its infinitely many anomalous subvarieties of codimension $1$ is known by Lemma \ref{22060401}. 

On the other hand, if $\mathcal{X}^{oa}=\emptyset$ and two cusps of $\mathcal{M}$ are not SGI each other, we have the following lemma.

\begin{lemma}\label{21101501}
Let $\mathcal{X}$ the holonomy variety of a $2$-cusped hyperbolic $3$-manifold such that $\mathcal{X}^{oa}= \emptyset$ and two cusps of $\mathcal{M}$ are not SGI. Let $H$ be an algebraic subgroup defined by
\begin{equation}\label{21101701}
M_1^{a}L_1^{b}M_2^{c}L_2^{d}=(M'_2)^{c'}(L'_2)^{d'}\quad \big((a,b,c,d)\neq (0,0,0,0), (c', d')\neq (0,0)\big)
\end{equation}
and $\overline{\text{Pr}((\mathcal{X}\times \mathcal{X})\cap H)}$ be the same as above. If $\{\mathcal{Y}_i\}_{i\in \mathcal{I}}$ is a family of infinitely many anomalous subvarieties of $\overline{\text{Pr}((\mathcal{X}\times \mathcal{X})\cap H)}$ of dimension $2$ where each $\mathcal{Y}_i$ contained in a translation of $\tilde{H}$, then $\tilde{H}$ is defined by either $\tilde{M_1}=\tilde{L_1}=1$ or $\tilde{M_2}=\tilde{L_2}=1$ where $\tilde{M_k}, \tilde{L_k}$ ($k=1,2$) are the preferred coordinates introduced in Definition \ref{22071301} as usual.
\end{lemma}
\begin{proof}
Without loss of generality, let us assume $\tilde{H}$ is given as
\begin{equation*}
M_1^{a_j}L_1^{b_j}M_2^{c_j}L_2^{d_j}(M'_1)^{a'_j}(L'_1)^{b'_j}=1\quad (j=1,2).
\end{equation*}
Moving to an analytic setting, $\log\big(\overline{\text{Pr}((\mathcal{X}\times \mathcal{X})\cap H)}\cap \tilde{H}\big)$ is defined by
\begin{equation}\label{21101502}
a_ju_1+b_jv_1+c_ju_2+d_jv_2+a'_ju'_1+b'_jv'_1=0\quad (j=1,2)
\end{equation}
and, since $\log\big(\overline{\text{Pr}((\mathcal{X}\times \mathcal{X})\cap H)}\big)$ is foliated by anomalous subvarieties contained in translations of \eqref{21101502}, there exists an analytic function $\mathfrak{h}$ such that 
\begin{equation}\label{21101602}
a_2u_1+b_2v_1+c_2u_2+d_2v_2+a'_2u'_1+b'_2v'_1=\mathfrak{h}(a_1u_1+b_1v_1+c_1u_2+d_1v_2+a'_1u'_1+b'_1v'_1).
\end{equation} 
\begin{enumerate}
\item If 
\begin{equation*}
\left(\begin{array}{cccc}
a_1 & b_1 & c_1 & d_1\\
a_2 & b_2 & c_2 & d_2
\end{array}
\right)
=\left(\begin{array}{cccc}
0 & 0 & 0 & 0\\
0 & 0 & 0 & 0
\end{array}\right),
\end{equation*}
then \eqref{21101602} is reduced to $a'_2u'_1+b'_2v'_1=\mathfrak{h}(a'_1u'_1+b'_1v'_1)$, which is further simplified as   
\begin{equation}\label{21101706}
a'_2u'_1+b'_2v'_1=\mathfrak{h}(a'_1u'_1).
\end{equation}

If $b'_2=0$, then $\mathfrak{h}$ is linear and it contradicts the fact that $\tilde{H}$ is an algebraic subgroup of codimension $2$. So we assume $b'_2\neq 0$ and $v'_1$ is presented as 
\begin{equation}\label{22061101}
\sum_{\substack{\alpha:odd\\ \beta:even
}} m_{\alpha, \beta}(u'_1)^{\alpha}(u'_2)^{\beta}   
\end{equation}
over $\log (\mathcal{X}\times \mathcal{X})$. Since $u'_2=\dfrac{a+b\tau_1}{c'+d'\tau_2}u_1+\dfrac{c+d\tau_2}{c'+d'\tau_2}u_2+(\text{higher order terms})$ by \eqref{21101701}, combining it with \eqref{22061101}, it follows that  
\begin{equation*}
v'_1=\sum_{\substack{\alpha:odd\\ \beta:even
}}  m_{\alpha, \beta}(u'_1)^{\alpha}\Big(\dfrac{a+b\tau_1}{c'+d'\tau_2}u_1+\dfrac{c+d\tau_2}{c'+d'\tau_2}u_2+(\text{higher order terms})\Big)^{\beta}. 
\end{equation*}
However, as $(a,b,c,d)\neq (0,0,0,0)$ by the assumption, this contradicts the fact that $v'_1$ depends only on $u'_1$ over $\log ((\mathcal{X}\times \mathcal{X})\cap H)$ by \eqref{21101706}. 

\item If $\left(\begin{array}{cc}
a'_1 & b'_1 \\
a'_2 & b'_2 
\end{array}
\right)=\left(\begin{array}{cc}
0 & 0 \\
0 & 0 
\end{array}\right)$, then \eqref{21101602} is reduced to 
\begin{equation*}
a_2u_1+b_2v_1+c_2u_2+d_2v_2=\mathfrak{h}(a_1u_1+b_1v_1+c_1u_2+d_1v_2), 
\end{equation*}
implying $\mathcal{X}$ is foliated by its anomalous subvarieties contained in translations of  
\begin{equation}\label{22061103}
M_1^{a_j}L_1^{b_j}M_2^{c_j}L_2^{d_j}=1, \quad (j=1,2).
\end{equation}
Since $\mathcal{X}^{oa}=\emptyset$ and two cusps of $\mathcal{M}$ are not SGI each other, \eqref{22061103} is equivalent to either $\tilde{M_1}=\tilde{L_1}=1$ or $\tilde{M_2}=\tilde{L_2}=1$ by Lemma \ref{22060401}. 

\item Lastly suppose  
\begin{equation*}
\left(\begin{array}{cccc}
a_1 & b_1 & c_1 & d_1\\
a_2 & b_2 & c_2 & d_2
\end{array}
\right)\neq \left(\begin{array}{cccc}
0 & 0 & 0 & 0\\
0 & 0 & 0 & 0
\end{array}\right)\quad \text{ and }\quad 
\left(\begin{array}{cc}
a'_1 & b'_1 \\
a'_2 & b'_2 
\end{array}
\right)\neq \left(\begin{array}{cc}
0 & 0 \\
0 & 0 
\end{array}\right).
\end{equation*}
Since $(a, b,c,d)\neq (0,0,0,0)$, without loss of generality, we assume $(a,b)\neq (0,0)$. Letting $u_2=0$ in \eqref{21101602}, it is reduced to
\begin{equation}\label{21101605}
a_2u_1+b_2v_1+a'_2u'_1+b'_2v'_1=\mathfrak{h}(a_1u_1+b_1v_1+a'_1u'_1+b'_1v'_1).
\end{equation} 
If $v'_1$ is given as in \eqref{22061101} over $\log (\mathcal{X}\times \mathcal{X})$, since $u'_2=\dfrac{a+b\tau_1}{c'+d'\tau_2}u_1+(\text{higher order terms})$ by \eqref{21101701}, we get 
\begin{equation}\label{22081901}
v'_1=\sum_{\substack{\alpha:odd\\ \beta:even
}} m_{\alpha, \beta}(u'_1)^{\alpha}(u'_2)^{\beta}=\sum_{\substack{\alpha:odd\\ \beta:even
}}  m_{\alpha, \beta}(u'_1)^{\alpha}\Big(\dfrac{a+b\tau_1}{c'+d'\tau_2}u_1+(\text{higher order terms})\Big)^{\beta}
\end{equation}
on $\log \big((\mathcal{X}\times \mathcal{X})\cap H\big)$. Hence every term of $v'_1$ in \eqref{22081901} is of the form $(u'_1)^{\alpha}u_1^{\beta}$ with $\alpha:odd$ and $\beta: even$. However, as $\mathfrak{h}$ is an odd function, the right side of \eqref{21101605} contains a term of the form $(u'_1)^{\alpha}u_1^{\beta}$ with $\alpha:even$ and $\beta:odd$, which is a contradiction.   
\end{enumerate}
\end{proof}

Lastly, the following miscellaneous lemma will be used in the proof of Theorem \ref{21082901}.  
\begin{lemma}\label{20080305}
Let $K$ be an algebraic coset in $\mathbb{G}^8$ such that $(\mathcal{X}\times \mathcal{X})\cap K$ is an anomalous subvariety of $\mathcal{X}\times \mathcal{X}$. If dim $(\mathcal{X}\times \mathcal{X})\cap K=2$ (resp. dim $(\mathcal{X}\times \mathcal{X})\cap K=1$), then dim $K\geq 4$ (resp. dim $K\geq 2$).  
\end{lemma}
\begin{proof}
First suppose dim $(\mathcal{X}\times \mathcal{X})\cap K=2$ and $K\leq 3$. By applying Gauss elimination if necessary, we may assume $K$ is contained in an algebraic coset defined by equations of the following forms:
\begin{equation*}
\begin{aligned}
&M_1^{a_1}L_1^{b_1}M_2^{c_1}L_2^{d_1}(M'_1)^{a'_1}(L'_1)^{b'_1}(M'_2)^{c'_1}(L'_2)^{d'_1}=\zeta_1,\\
&M_1^{a_2}L_1^{b_2}M_2^{c_2}L_2^{d_2}(M'_1)^{a'_2}(L'_1)^{b'_2}(M'_2)^{c'_2}=\zeta_2,\\
&M_1^{a_3}L_1^{b_3}M_2^{c_3}L_2^{d_3}(M'_1)^{a'_3}(L'_1)^{b'_3}=\zeta_3,\\
&M_1^{a_4}L_1^{b_4}M_2^{c_4}L_2^{d_4}(M'_1)^{a'_4}=\zeta_4,\\
&M_1^{a_5}L_1^{b_5}M_2^{c_5}L_2^{d_5}=\zeta_5.
\end{aligned}
\end{equation*}
Since $M'_1, M'_2$ depend on $M_1, M_2$ and $M_2$ depends on $M_1$ from the last equation, we get $\dim \;(\mathcal{X}\times \mathcal{X})\cap K\leq 1$. 

Similarly, one can show if dim $(\mathcal{X}\times \mathcal{X})\cap K=1$, then $\dim \;K\geq 2$.
\end{proof}

\newpage
\section{$1$-cusped case}\label{Prem I}
The goal of this section is generalizing Theorem \ref{22022707}, via quantifying Theorem \ref{20071503}. We first introduce a couple of lemmas in Section \ref{22050609}, which can be seen as effective versions of Theorem \ref{20071503} and, using them, establish a more general statement implying Theorem \ref{22022707} as a corollary in Section \ref{22050610}. The lemmas offered in Section \ref{22050609} are repeatedly used later in demonstrating other arguments, including the main ones. 

This section will play the role of a prototype for the proofs of the main theorems, and hence a reader may view this as a warm-up section before actually embarking on the major tasks.

\subsection{Quantification of Theorem \ref{20071503}}\label{22050609}
Let $\mathcal{M}$ be $2$-cusped hyperbolic $3$-manifold and $\mathcal{X}$ be its holonomy variety. By Theorem \ref{20071503}, there exists a finite set $\mathcal{H}$ of algebraic subgroups such that, for any Dehn filling $\mathcal{M}_{p_1/q_1,p_2/q_2}$ of $\mathcal{M}$ with $|p_k|+|q_k|$ ($k=1,2$) sufficiently large, if the core holonomies $t_1, t_2$ of $\mathcal{M}_{p_1/q_1,p_2/q_2}$ satisfy $t_1^at_2^b=1$ for some $a, b\in \mathbb{Z}$, a Dehn filling point $P$ associated to $\mathcal{M}_{p_1/q_1, p_2/q_2}$ is contained in some $H\in \mathcal{H}$. In the following lemma, we explicitly compute $a$ and $b$ in terms of the exponents of the defining equations of $H$. The ideas and techniques utilized in the proof of the lemma will be repeatedly revisited and far generalized later in verifying other claims, in particular, Lemmas \ref{21082802} and \ref{22041320}.   

\begin{lemma}\label{22050501}
Let $\mathcal{M}, \mathcal{X},  \mathcal{M}_{p_1/q_1, p_2/q_2}, H, P$ and $t_1, t_2$ be the same as above. Since $\dim H=2$ by Theorem \ref{20071503}, without loss of generality, we assume it is defined by
\begin{equation}\label{24081302}
M_1^{a_j}L_1^{b_j}=M_2^{c_j}L_2^{d_j},\quad (j=1,2),
\end{equation}
and set $A_1:=\left(\begin{array}{cc}
a_1 & b_1\\
a_2 & b_2
\end{array}\right)$ and $A_2:= \left(\begin{array}{cc}
c_1 & d_1\\
c_2 & d_2
\end{array}\right)$. Then each $A_k$ ($k=1,2$) is invertible, and the following equations hold:
\begin{equation}\label{22051801}
\left(\begin{array}{cc}
p_{1} & q_{1}\\
\end{array}
\right)
A_1^{-1}=
\left(\begin{array}{cc}
p_{2} & q_{2}\\
\end{array}
\right)A_2^{-1}
\end{equation}
and 
\begin{equation}\label{22050419}
t_1^{\det A_1}=\epsilon t_2^{\det A_2}
\end{equation}
for some root of unity $\epsilon$. 
\end{lemma}

\begin{proof}
If $A_2$ is not invertible, without loss of generality, \eqref{24081302} is reduced to 
\begin{equation}
M_1^{a_1}L_1^{b_1}=M_2^{c_1}L_2^{d_1}, \quad  M_1^{a_2}L_1^{b_2}=1,  
\end{equation}
which implies $(t_1^{-q_1})^{a_2}(t_1^{p_1})^{b_2}=t_1^{-q_1a_2+p_1b_2}=1$. Since $t_1$ is not a root of unity, it follows that $-q_1a_2+p_1b_2=0$, but it contradicts the fact that $|p_1|+|q_1|$ is sufficiently large. Similarly one can prove $A_1$ is invertible. 

Now we turn to \eqref{22051801} and \eqref{22050419}. To simplify the proof, by changing basis if necessary, we assume $(p_{1}, q_{1})=(p_{2}, q_{2})=(1,0)$. Moving to the analytic holonomy variety, we work with $\log (\mathcal{X}\cap H)$, defined by $A_1\left(\begin{array}{c}
u_1\\
v_1
\end{array}
\right)=
A_2\left(\begin{array}{c}
u_2\\
v_2
\end{array}
\right)$, and let $(\xi_{m_{1}}, \xi_{l_{1}}, \xi_{m_{2}}, \xi_{l_{2}})$ be a Dehn filling point on $\log (\mathcal{X}\cap H)$ associated to $\mathcal{M}_{p_1/q_1, p_2/q_2}$. That is, 
\begin{equation*}
p_k\xi_{m_{k}}+q_k\xi_{l_k}=\xi_{m_{k}}=-2\pi \sqrt{-1},\quad (k=1,2)
\end{equation*}
and so
\begin{equation}\label{24081601}
\begin{gathered}
 A_1\left(\begin{array}{c}
-2\pi \sqrt{-1}\\
\xi_{l_{1}}
\end{array}
\right)
=A_2
\left(\begin{array}{c}
-2\pi \sqrt{-1}\\
\xi_{l_{2}}
\end{array}
\right)\Longrightarrow  
\left(\begin{array}{c}
-2\pi \sqrt{-1}\\
\xi_{l_{1}}
\end{array}
\right)
=A_1^{-1}A_2
\left(\begin{array}{c}
-2\pi \sqrt{-1}\\
\xi_{l_{2}}
\end{array}
\right). 
\end{gathered}
\end{equation}
By defining $\left(\begin{array}{cc}
a & b\\
c & d
\end{array}\right):=A_1^{-1}A_2$, we get 
\begin{equation*}
\left(\begin{array}{c}
-2\pi \sqrt{-1}\\
\xi_{l_{1}}
\end{array}\right)=\left(\begin{array}{c}
-2\pi \sqrt{-1}a+\xi_{l_{2}}b\\
-2\pi \sqrt{-1}c+\xi_{l_{2}}d
\end{array}
\right) 
\end{equation*}
from \eqref{24081601}, and, by exponentiating both sides of it, further obtain
\begin{equation*}
\left(\begin{array}{c}
1\\
t_1
\end{array}\right)=\left(\begin{array}{c}
e^{-2\pi \sqrt{-1}a}t_2^b\\
e^{-2\pi \sqrt{-1}c}t_2^d
\end{array}\right). 
\end{equation*}
Recall that $t_2$ is not a root of unity; thus, $b=0, a=1$ and $t_1=e^{-2\pi \sqrt{-1}c}t_2^d=e^{-2\pi \sqrt{-1}c}t_2^{\frac{\det A_2}{\det A_1}}$. As a result, the following 
\begin{equation*}
\left(\begin{array}{cc}
1 & 0
\end{array}\right)A_1^{-1}A_2=\left(\begin{array}{cc}
1 & 0
\end{array}\right)\quad  \text{and}\quad  t_1^{\det A_1}=e^{-2\pi \sqrt{-1}c\det A_1}t_2^{\det A_2}
\end{equation*}
are established, implying \eqref{22051801} and \eqref{22050419} respectively. 
\end{proof}

From now on, let $\mathcal{M}$ be a $1$-cusped hyperbolic $3$-manifold and $\mathcal{X}$ be its holonomy variety. Let $\mathcal{X}\times \mathcal{X}$ be the product of two identical copies of $\mathcal{X}$ in $\mathbb{G}^2\times \mathbb{G}^2(:=(M, L, M', L'))$. In the lemma below, we study the structure of anomalous subvarieties of $\mathcal{X}\times \mathcal{X}$. First, note that, thanks to Theorem \ref{potential}, one immediately gets that the following are anomalous subvarieties of $\mathcal{X}\times \mathcal{X}$ : 
\begin{itemize}
\item $(\mathcal{X}\times \mathcal{X})\cap (M=L=1)$;
\item $(\mathcal{X}\times \mathcal{X})\cap (M'=L'=1)$;
\item $(\mathcal{X}\times \mathcal{X})\cap (M=M',  L=L')$;
\item $(\mathcal{X}\times \mathcal{X})\cap (M=(M')^{-1},  L=(L')^{-1})$.
\end{itemize}
We call each in the above list \textit{a trivial anomalous subvariety} of $\mathcal{X}\times \mathcal{X}$ and, otherwise, an anomalous subvariety of $\mathcal{X}\times \mathcal{X}$ is \textit{non-trivial}. In the next lemma, we count the exact number of non-trivial anomalous subvarieties of $\mathcal{X}\times \mathcal{X}$. This lemma will play an important role not only in the proof of Theorem \ref{22022707} but also later in the proof of Theorem \ref{20103003}. 

\begin{lemma}\label{20102501}
Let $\mathcal{M}$, $\mathcal{X}$ and $\mathcal{X}\times \mathcal{X}$ be the same as above. Let $H$ be an algebraic subgroup defined by 
\begin{equation*}
M^{a_1}L^{b_1}=(M')^{c_1}(L')^{d_1}, \quad M^{a_2}L^{b_2}=(M')^{c_2}(L')^{d_2}
\end{equation*}
where $\det\left(\begin{array}{cc}
a_1 & b_1\\
a_2 & b_2
\end{array}\right)=\det \left(\begin{array}{cc}
c_1 & d_1\\
c_2 & d_2
\end{array}\right)(\neq 0)$. If $(\mathcal{X}\times \mathcal{X})\cap_{(1, \dots, 1)} H$ is a non-trivial anomalous subvariety of $\mathcal{X}\times \mathcal{X}$, then the cusp shape $\tau$ of $\mathcal{M}$ is either contained in $\mathbb{Q}(\sqrt{-3})$ or $\mathbb{Q}(\sqrt{-1})$. Moreover, if $\tau\in\mathbb{Q}(\sqrt{-3})$ (resp. $\mathbb{Q}(\sqrt{-1})$), then $\sigma^6=Id$ (resp. $\sigma^4=Id$) where $\sigma:=\left(\begin{array}{cc}
c_1 & d_1\\
c_2 & d_2
\end{array}\right)^{-1}
\left(\begin{array}{cc}
a_1 & b_1\\
a_2 & b_2
\end{array}\right)$. 
\end{lemma}
\begin{proof}
Moving to the analytic holonomy variety, if $\log(\mathcal{X}\times \mathcal{X})$ is defined by
\begin{equation*}
v=\tau u+(\text{higher order terms}),\quad v'=\tau u'+(\text{higher order terms}) 
\end{equation*}
in $\mathbb{C}^4(:=(u, v, u', v'))$, then $\log\big((\mathcal{X}\times \mathcal{X})\cap H\big)$ is given as 
\begin{equation}\label{20083001}
\left(\begin{array}{cc}
a_1 & b_1\\
a_2 & b_2
\end{array}\right)
\left(\begin{array}{c}
u \\
v 
\end{array}\right)
= \left(\begin{array}{cc}
c_1 & d_1\\
c_2 & d_2
\end{array}\right)
\left(\begin{array}{c}
u' \\
v' 
\end{array}\right)
\;\;
\text{or, equivalently,}\;\;
\left(\begin{array}{c}
u' \\
v' 
\end{array}\right)
=\sigma\left(\begin{array}{c}
u \\
v 
\end{array}\right).
\end{equation}
Since \eqref{20083001} is an analytic set of dim $1$, by letting $\sigma:=\left(\begin{array}{cc}
a & b\\
c & d
\end{array}\right)$, it follows that $\tau=\frac{c+d\tau}{a+b\tau}$, which implies
\begin{equation}\label{24081201}
\tau=\frac{d-a\pm \sqrt{(a-d)^2+4bc}}{2b}=\frac{d-a\pm \sqrt{(\text{tr }\sigma)^2-4\det \sigma}}{2b}. 
\end{equation} 

On the other hand, if we interpret $\sigma$ as a map from $\log\mathcal{X}$ to itself in \eqref{20083001}, then $\sigma$ is an element of finite order in the automorphism group of $\log\mathcal{X}$, since the number of anomalous subvarieties of $\mathcal{X}\times\mathcal{X}$ is finite, thanks to Theorem \ref{struc}. By the assumption, $(\mathcal{X}\times \mathcal{X})\cap_{(1, \dots, 1)} H$ is a non-trivial anomalous subvariety of $\mathcal{X}\times \mathcal{X}$, hence the minimal polynomial $m_{\sigma}(x)$ of $\sigma$ is one of the following: 
\begin{equation*}
x^2+x+1,\quad x^2-x+1,\quad x^2-1,\quad x^2+1.
\end{equation*}
If $m_{\sigma}(x)$ is either $x^2+x+1$ or $x^2-x+1$ (resp.  $x^2-1$ or $x^2+1$), then clearly the order of $\sigma$ is either $3$ or $6$ (resp. $2$ or $4$) respectively, and $\tau\in \mathbb{Q}(\sqrt{-3})$ (resp. $ \mathbb{Q}(\sqrt{-1})$)  by \eqref{24081201}. 
\end{proof}

\subsection{Proof of Theorem \ref{22022707}}\label{22050610}

Amalgamating all the previous lemmas, we now establish the following, which implies Theorem \ref{22022707} as a corollary.  

\begin{theorem}\label{22090901}
Let $\mathcal{M}$ be a $1$-cusped hyperbolic $3$-manifold, $\tau$ be its cusp shape and $\mathcal{X}$ be its holonomy variety. Let $\mathcal{M}_{p/q}$ and $\mathcal{M}_{p'/q'}$ be two Dehn fillings of $\mathcal{M}$ with $|p|+|q|$ and $|p'|+|q'|$ sufficiently large. Let $t$ (resp. $t'$) be the derivative of the holonomy of the core geodesic of $\mathcal{M}_{p/q}$ (resp. $\mathcal{M}_{p'/q'}$). If $t$ and $t'$ are multiplicatively dependent and so $t^a=\epsilon (t')^{a'}$ for some $a,a'\in \mathbb{Z}$ and root of unity $\epsilon$, then 
\begin{equation}\label{22061701}
a=a'.
\end{equation}
Further, 
\begin{enumerate}
\item if $\tau\notin \mathbb{Q}(\sqrt{-3})$ and $\tau\notin\mathbb{Q}(\sqrt{-1})$, then $p/q=p'/q'$;
\item if $\tau\in \mathbb{Q}(\sqrt{-3})$, then $p'/q'=\sigma^i(p/q)$ for some $\sigma \in GL_2(\mathbb{Q})$ of order $3$;
\item if $\tau\in \mathbb{Q}(\sqrt{-1})$, then $p'/q'=\sigma^i(p/q)$ for some $\sigma \in GL_2(\mathbb{Q})$ of order $2$.
\end{enumerate}
\end{theorem}
\begin{proof}
Following the discussion in Section \ref{Dehn}, let $P$ be a Dehn filling point on $\mathcal{X}\times \mathcal{X}$ associated to the pair $(\mathcal{M}_{p/q}, \mathcal{M}_{p'/q'})$. By Corollary \ref{21021602}, there exists an algebraic subgroup $H$ such that $P\in H$ and $(\mathcal{X}\times \mathcal{X})\cap_{(1, \dots, 1)}H$ is an anomalous subvariety of $\mathcal{X}\times \mathcal{X}$. Let $H$ be defined by
\begin{equation*}
M^{a_1}L^{b_1}=(M')^{c_1}(L')^{d_1},\quad M^{a_2}L^{b_2}=(M')^{c_2}(L')^{d_2}.
\end{equation*}
Then $\det A_1=\det A_2\neq 0$ where $A_1=\left(\begin{array}{cc}
a_1 & b_1\\
a_2 & b_2
\end{array}\right)$ and $A_2=\left(\begin{array}{cc}
c_1 & d_1\\
c_2 & d_2
\end{array}\right)$ by Lemma \ref{20102501}, thus \eqref{22061701} follows by Lemma \ref{22050501}. Also note that 
\begin{equation}\label{22061703}
\left(\begin{array}{cc}
p & q\\
\end{array}
\right)
A_1^{-1}=
\left(\begin{array}{cc}
p' & q'\\
\end{array}
\right)A_2^{-1}
\end{equation}
by \eqref{22051801} in Lemma \ref{22050501}. 

\begin{enumerate}
\item If $\tau$ is contained in neither $\mathbb{Q}(\sqrt{-3})$ nor $\mathbb{Q}(\sqrt{-1})$, then $(\mathcal{X}\times \mathcal{X})\cap_{(1, \dots, 1)}H$ is a trivial anomalous subvariety of $\mathcal{X}\times \mathcal{X}$ by Lemma \ref{20102501} and so $A_1=\pm A_2$, implying $p/q=p'/q'$ by \eqref{22061703}. 

\item For $\tau\in \mathbb{Q}(\sqrt{-3})$, if $(\mathcal{X}\times \mathcal{X})\cap_{(1, \dots, 1)}H$ is a non-trivial anomalous subvariety of $\mathcal{X}\times \mathcal{X}$, then the order of $A_2^{-1}A_1$ is $6$ by Lemma \ref{20102501} and, as $(A_2^{-1}A_1)^3=-I$ by the same lemma, the result follows. 

\item For $\tau\in \mathbb{Q}(\sqrt{-1})$, we get the desired result by Lemma \ref{20102501}, similar to the previous case.  
\end{enumerate}
\end{proof}

\newpage
\section{A weak version of the Zilber-Pink conjecture I (special case)}\label{ZPCI}
In this section, we first prove Theorem \ref{22022601} for the following two special cases:
\begin{enumerate}
\item the core holonomies of Dehn fillings are multiplicatively dependent each other;
\item two cusps of $\mathcal{M}$ are SGI from each other. 
\end{enumerate} 

\subsection{The core holonomies are multiplicatively dependent}
For the first case, we get the desired result by applying Corollary \ref{21021602} repeatedly. 
\begin{theorem}\label{20071505}
Let $\mathcal{M}$ be a $2$-cusped hyperbolic $3$-manifold and $\mathcal{X}$ be its holonomy variety. Then there exists a finite list $\mathcal{H}$ of algebraic subgroups possessing the following property. For any two Dehn fillings $\mathcal{M}_{p_{1}/q_{1},p_{2}/q_{2}}$ and $\mathcal{M}_{p'_{1}/q'_{1},p'_{2}/q'_{2}}$ satisfying
\begin{itemize}
\item $\text{pvol}_{\mathbb{C}}\;\mathcal{M}_{p_{1}/q_{1},p_{2}/q_{2}}=\text{pvol}_{\mathbb{C}}\;\mathcal{M}_{p'_{1}/q'_{1},p'_{2}/q'_{2}}$;
\item the core holonomies $t_{1}, t_{2}$ (resp. $t'_{1}, t'_{2}$) of $\mathcal{M}_{p_{1}/q_{1},p_{2}/q_{2}}$ (resp. $\mathcal{M}_{p'_{1}/q'_{1},p'_{2}/q'_{2}}$) are multiplicatively dependent;
\item $|p_k|+|q_k|$ and $|p'_k|+|q'_k|$ ($k=1,2$) sufficiently large,  
\end{itemize}
a Dehn filling point $P$ associated to $\big(\mathcal{M}_{p_{1}/q_{1},p_{2}/q_{2}}, \mathcal{M}_{p'_{1}/q'_{1},p'_{2}/q'_{2}}\big)$ is contained some $H\in \mathcal{H}$ defined by equations of the following types
\begin{equation}\label{22051507}
\begin{gathered}
M_1^{a_j}L_1^{b_j}M_2^{c_j}L_2^{d_j}=(M'_1)^{a'_j}(L'_1)^{b'_j}(M'_2)^{c'_j}(L'_2)^{d'_j}=M_1^{e_j}L_1^{f_j}(M'_1)^{e'_j}(L'_1)^{f'_j}=1, \quad (j=1,2).
\end{gathered}
\end{equation}
In particular, for each $H\in\mathcal{H}$, $(\mathcal{X}\times \mathcal{X})\cap_{(1, \dots, 1)} H$ is a $1$-dim anomalous subvariety of $\mathcal{X}\times \mathcal{X}$. 
\end{theorem}

\begin{proof}[Proof of Theorem \ref{20071505}]
Since $t_1, t_2$ (resp. $t'_1, t'_2$) are multiplicatively dependent (resp. dependent), by Theorem \ref{20071503}, there exists a finite list $\mathcal{H}_1$ of algebraic subgroups such that 
\begin{itemize}
\item $P\in H_1$ for some $H_1\in \mathcal{H}_1$ defined by the following types of equations
\begin{equation}\label{20122302}
\begin{gathered}
M_1^{a_j}L_1^{b_j}M_2^{c_j}L_2^{d_j}=(M'_1)^{a'_j}(L'_1)^{b'_j}(M'_2)^{c'_j}(L'_2)^{d'_j}=1\quad (j=1,2);
\end{gathered}
\end{equation}
\item $(\mathcal{X}\times \mathcal{X})\cap_P H_1$ is a $2$-dim anomalous subvariety of $\mathcal{X}\times\mathcal{X}$.   
\end{itemize}
Note that $(\mathcal{X}\times \mathcal{X})\cap_P H_1$ is the product of two algebraic curves $\mathcal{C}\times \mathcal{C}'$ and, if we project $\mathcal{C}\times \mathcal{C}'$ under 
\begin{equation*}
\text{Pr}:\;(M_1, L_1, M_2, L_2, M'_1, L'_1, M'_2, L'_2)\longrightarrow (M_1, L_1,M'_1, L'_1), 
\end{equation*}
the image $\text{Pr}(\mathcal{C}\times \mathcal{C}')$ is still the product of two algebraic curves. Considering $\text{Pr}(P)$ as an intersection point between
\begin{equation*}
\text{Pr}(\mathcal{C}\times \mathcal{C}')\quad\text{ and }\quad M_1^{p_{1}}L_1^{q_{1}}=(M'_1)^{p'_{1}}(L'_1)^{q'_{1}}=1, 
\end{equation*}
as $t_{1}, t'_{1}$ are multiplicatively dependent,\footnote{This follows from the assumptions, that is, $t_1t_2=t'_1t'_2$ and $t_1, t_2$ (resp. $t'_1, t'_2$) are multiplicatively dependent (resp. dependent).} we obtain a finite set $\mathcal{H}_2$ of algebraic subgroups such that
\begin{itemize}
\item $\text{Pr}(P)\in H_2 $ for some $H_2\in \mathcal{H}_2$ defined by equations of the following forms
\begin{equation}\label{20122301}
M_1^{e_j}L_1^{f_j}(M'_1)^{e'_j}(L'_1)^{f'_j}=1\quad (j=1,2);  
\end{equation}
\item $\text{Pr}(\mathcal{C}\times \mathcal{C}')\cap_{\text{Pr}(P)} H_2$ is an anomalous subvariety of $\text{Pr}(\mathcal{C}\times \mathcal{C}')$
\end{itemize}
by Theorem \ref{20071503}. 

In conclusion, $P$ is contained in an algebraic subgroup $H$ defined by equations of the forms given in \eqref{20122302}-\eqref{20122301}. Since the Dehn filling coefficients are sufficiently large by the assumption, $P$ is close enough to the identity by Theorem \ref{040605}, and hence $(\mathcal{X}\times \mathcal{X})\cap_P H$ contains the identity. 
\end{proof}

\subsection{SGI}\label{SGI2}
If two cusps of $\mathcal{M}$ are SGI each other, Theorem \ref{22022601} is further promoted as follows:
\begin{theorem}\label{22012104}
Let $\mathcal{M}$ and $\mathcal{X}$ be as in Theorem \ref{20071505}. Suppose two cusps of $\mathcal{M}$ are SGI each other. Then there exists a finite set $\mathcal{H}$ of algebraic subgroups such that, for any two Dehn fillings $\mathcal{M}_{p_1/q_1, p_2/q_2}$ and $\mathcal{M}_{p'_1/q'_1, p'_2/q'_2}$ of $\mathcal{M}$ satisfying
\begin{equation}\label{22062901}
\text{pvol}_\mathbb{C}\;\mathcal{M}_{p_1/q_1, p_2/q_2}=\text{pvol}_\mathbb{C}\;\mathcal{M}_{p'_1/q'_1, p'_2/q'_2}
\end{equation}
with $|p_k|+|q_k|$ and $|p'_k|+|q'_k|$ ($k=1,2$) sufficiently large, a Dehn filling point associated to \eqref{22062901} is contained in some $H\in \mathcal{H}$ defined by equations of the following forms either 
\begin{equation}\label{22012105}
M_1^{a_j}L_1^{b_j}(M'_1)^{a'_j}(L'_1)^{b'_j}=M_2^{c_j}L_2^{d_j}(M'_2)^{c'_j}(L'_2)^{d'_j}=1\quad (j=1,2)
\end{equation}
or 
\begin{equation}\label{22012106}
M_1^{a_j}L_1^{b_j}(M'_2)^{c'_j}(L'_2)^{d'_j}=M_2^{c_j}L_2^{d_j}(M'_1)^{a'_j}(L'_1)^{b'_j}=1\quad (j=1,2).
\end{equation}
In particular, for each $H\in \mathcal{H}$, $(\mathcal{X}\times \mathcal{X})\cap _{(1, \dots, 1)}H$ is a $2$-dim anomalous subvariety of $\mathcal{X}\times \mathcal{X}$.
\end{theorem}

Recall from Proposition \ref{21073101} that, for $\mathcal{M}$ as above, $\mathcal{X}$ is the product of two algebraic curves, thus proving Theorem \ref{22012104} is equivalent to solving (a weak version of) the Zilber-Pink conjecture over the product of four algebraic curves (i.e. $\mathcal{X}\times \mathcal{X}$). We shall prove it by induction and need two criteria beforehand. In the first criterion, Lemma \ref{22012101}, the problem is reduced to the Zilber-Pink conjecture over the product of three algebraic curves and then, via the second criterion (Lemma \ref{22012102}), it is further reduced to the same conjecture over the product of two algebraic curves. Since the Zilber-Pink conjecture known for the last case by Corollary \ref{21021602}, the result will follow. 

\begin{lemma}\label{22012101}
Let $\mathcal{X}:=\mathcal{X}_1\times \cdots\times \mathcal{X}_4 \subset \mathbb{G}^2\times \cdots \times \mathbb{G}^2(:=(M_1, L_1, \dots, M_4, L_4))$ be the product of four algebraic curves where each $\mathcal{X}_k$ ($1\leq k\leq 4$) is the holonomy variety of a $1$-cusped hyperbolic $3$-manifold. For each $k$ ($1\leq k\leq 4$), let $ M_k^{p_{k}}L_k^{q_{k}}=1$ be a Dehn filling equation with $|p_k|+|q_k|$ sufficiently large and $t_{k}$ be the core holonomy of $\mathcal{X}_k\cap (M_k^{p_{k}}L_k^{q_{k}}=1)$. If $t_{1},\dots, t_{4}$ are multiplicatively dependent, there exists a proper subset of $\{t_{1},\dots, t_{4}\}$ whose elements are multiplicatively dependent. 
\end{lemma}

\begin{proof}[Proof of Lemma \ref{22012101}]
On the contrary, suppose there is no proper subset of $\{t_{1},\dots, t_{4}\}$ whose elements are multiplicatively dependent, and let $P$ be a Dehn filling point on $\mathcal{X}$ associated to the given Dehn fillings above. Since $t_1, \dots, t_4$ are multiplicatively dependent, $P$ is contained in an algebraic subgroup of dimension $3$ and, as the height of $P$ is uniformly bounded, there exists an algebraic subgroup $H$ of dimension $4$ such that $P\in H$ and $\mathcal{X}\cap_P H$ is an anomalous subvariety of $\mathcal{X}$ by Corollary \ref{20080403}. Without loss of generality, applying Gauss elimination if necessary, let us assume $H$ is defined by
\begin{equation}\label{22012109}
\begin{aligned}
&M_1^{a_{1}}L_1^{b_{1}}M_2^{c_{1}}L_2^{d_{1}}M_3^{e_{1}}L_3^{f_{1}}M_4^{g_{1}}L_4^{h_{1}}=1,\\
&M_1^{a_{2}}L_1^{b_{2}}M_2^{c_{2}}L_2^{d_{2}}M_3^{e_{2}}L_3^{f_{2}}M_4^{g_{2}}=1,\\
&M_1^{a_{3}}L_1^{b_{3}}M_2^{c_{3}}L_2^{d_{3}}M_3^{e_{3}}L_3^{f_{3}}=1,\\
&M_1^{a_{4}}L_1^{b_{4}}M_2^{c_{4}}L_2^{d_{4}}M_3^{e_{4}}=1.
\end{aligned}
\end{equation}
Since $t_{1}, t_{2}, t_{3}$ are multiplicatively independent by the assumption, the last two equations in \eqref{22012109} are either\footnote{This is due to the norm minimizing property of the defining equations of $H$ as mentioned in Remark \ref{22081302}.} 
\begin{equation*}
M_1^{p_{1}}L_1^{q_{1}}=M_2^{p_{2}}L_2^{q_{2}}=1\;\text{   or   }\;M_2^{p_{2}}L_2^{q_{2}}=M_3^{p_{3}}L_3^{q_{3}}=1\;\text{   or   }\;M_1^{p_{1}}L_1^{q_{1}}=M_3^{p_{3}}L_3^{q_{3}}=1
\end{equation*}
and, without loss of generality, we suppose the first one. If 
\begin{equation}\label{22030201}
(M_1, L_1, M_2, L_2)=(t_1^{-q_1}, t_1^{p_1}, t_2^{-q_2}, t_2^{p_2}) 
\end{equation}
is an intersection point between $\mathcal{X}_1\times \mathcal{X}_2$ and $M_1^{p_{1}}L_1^{q_{1}}=M_2^{p_{2}}L_2^{q_{2}}=1$, corresponding to the projected image of $P$ onto $\mathbb{G}^4(:=(M_1, L_1, M_2, L_2))$, then $\eqref{22030201} \cap H$ is reduced to 
\begin{equation}\label{21100101}
\begin{gathered}
M_3^{e_{1}}L_3^{f_{1}}M_4^{g_{1}}L_4^{h_{1}}=\zeta_1,\quad
M_3^{e_{2}}L_3^{f_{2}}M_4^{g_{2}}=\zeta_2 
\end{gathered}
\end{equation}
for some $\zeta_1, \zeta_2\in \mathbb{G}$. Denoting \eqref{21100101} by $K$ and considering it as an algebraic coset in $\mathbb{G}^4(:=(M_3, L_3, M_4, L_4))$, it follows that the component of $(\mathcal{X}_{3}\times \mathcal{X}_4)\cap K$, containing the image of $P$, is an anomalous subvariety of $\mathcal{X}_{3}\times \mathcal{X}_{4}$. Thus, by Lemma \ref{22060401}, $K$ is a translation of either $M_{3}=L_{3}=1$ or $M_4=L_4=1$, that is, $e_{1}=f_{1}=e_{2}=f_{2}=0$ or $g_{1}=h_{1}=g_{2}=0$ in \eqref{22012109} and \eqref{21100101}. Without loss of generality, we choose the first case and reduce \eqref{22012109} to 
\begin{equation}\label{22062802}
M_1^{a_{1}}L_1^{b_{1}}M_2^{c_{1}}L_2^{d_{1}}M_4^{g_{1}}L_4^{h_{1}}=
M_1^{a_{2}}L_1^{b_{2}}M_2^{c_{2}}L_2^{d_{2}}M_4^{g_{2}}=M_1^{p_{1}}L_1^{q_{1}}=M_2^{p_{2}}L_2^{q_{2}}=1. 
\end{equation}
Then, as $t_1, t_2, t_4$ are multiplicatively independent, \eqref{22062802} is simplified to 
\begin{equation*}
M_4^{g_{1}}L_4^{h_{1}}=M_4^{g_{2}}=M_1^{p_{1}}L_1^{q_{1}}=M_2^{p_{2}}L_2^{q_{2}}=1.
\end{equation*}
However the last conclusion contradicts the fact that $t_4$ is not a root of unity. 
\end{proof}

The following lemma can be established by an analogous way as above, and so the proof of it is skipped here.   
\begin{lemma}\label{22012102}
Let $\mathcal{X}:=\mathcal{X}_1\times \mathcal{X}_2\times \mathcal{X}_3 \subset \mathbb{G}^2\times \mathbb{G}^2\times \mathbb{G}^2(:=(M_1, L_1,  M_2, L_2, M_3, L_3))$ be the product of three algebraic curves where each $\mathcal{X}_k$ ($1\leq k\leq 3$) is the holonomy variety of a $1$-cusped hyperbolic $3$-manifold. For $1\leq k\leq 3$, let $ M_k^{p_{k}}L_k^{q_{k}}=1$ be a Dehn filling equation with $|p_k|+|q_k|$ sufficiently large and $t_{k}$ be the core holonomy of $\mathcal{X}_k\cap (M_k^{p_{k}}L_k^{q_{k}}=1)$. If $t_{1}, t_{2}, t_{3}$ are multiplicatively dependent, there exists a proper subset of $\{t_{1}, t_{2}, t_{3}\}$ whose elements are multiplicatively dependent. 
\end{lemma}

Putting together the two lemmas, we finally confirm Theorem \ref{22012104}.

\begin{proof}[Proof of Theorem \ref{22012104}]
Since two cusps of $\mathcal{M}$ are SGI by the assumption, $\mathcal{X}$ is the product of two algebraic curves $\mathcal{X}_1\times \mathcal{X}_2$. 

As $t_{1}t_{2}=t'_{1}t'_{2}$, there exists a proper subset of $\{t_{1}, t_{2}, t'_{1}, t'_{2}\}$ whose elements are multiplicatively dependent by Lemma \ref{22012101} and, without loss of generality, we suppose $t_{1}, t_{2}, t'_{1}$ are multiplicatively dependent. By Lemma \ref{22012102}, there further exists a proper subset of $\{t_{1}, t_{2}, t'_{1}\}$ whose elements are multiplicatively dependent and thus either one of the following holds: 
\begin{enumerate}
\item If $t_{1}, t'_{1}$ are multiplicatively dependent, $t_{1}, t_{2}, t'_{2}$ are multiplicatively dependent (by $t_{1}t_{2}=t'_{1}t'_{2}$), so it falls into one of the three cases below by Lemma \ref{22012102}.  
\begin{enumerate}
\item If $t_{1}, t_{2}$ are multiplicatively dependent, (since $t_{1}, t'_{1}$ are multiplicatively dependent) $t_{1}, t'_{2}$ are also multiplicatively dependent by $t_{1}t_{2}=t'_1t'_2$. As a result, any two elements in $\{t_{1}, t_{2}, t'_{1}, t'_{2}\}$ are multiplicatively dependent and so the result holds by Theorem \ref{20071505}. 

\item If $t_{1}, t'_{2}$ are multiplicatively dependent, then $t_{1}, t_{2}$ are multiplicatively dependent by $t_{1}t_{2}=t'_1t'_2$ and hence the desired result follows by the same reason above. 

\item If $t_{2}, t'_{2}$ are multiplicatively dependent, as $t_1, t'_1$ are multiplicatively dependent, a Dehn filling point associated to \eqref{22062901} is clearly contained in an algebraic subgroup defined by equations of forms given in \eqref{22012105}.  
\end{enumerate}

\item If $t_{1}, t_{2}$ are multiplicatively dependent, then $t_{1}, t'_{1}, t'_{2}$ are multiplicatively dependent (by $t_{1}t_{2}=t'_{1}t'_{2}$). By Lemma \ref{22012102}, there exists a proper subset of $\{t_{1}, t'_{1}, t'_{2}\}$ whose elements are multiplicatively dependent. If $t'_{1}, t'_{2}$ are multiplicatively dependent, the result is attained by Theorem \ref{20071505} and, if $t_{1}, t'_{1}$ are multiplicatively dependent, it falls into the previous case. If $t_{1}, t'_{2}$ are multiplicatively dependent, as $t_1, t_2$ are multiplicatively, we get $t_1, t'_1$ are multiplicatively dependent from $t_1t_2=t'_1t'_2$, so again covered by the above case. 

\item Analogously, we obtain the desired result for $t_{2}, t'_{1}$ being multiplicatively dependent. 
\end{enumerate}

In conclusion, we have either 
\begin{equation}\label{22072201}
t_1^a(t'_1)^b=t_2^c(t'_2)^d=1
\end{equation}
or 
\begin{equation}\label{22072202}
t_1^a(t'_2)^b=t_2^c(t'_1)^d=1
\end{equation} 
for some $a,b,c,d\in \mathbb{Z}$. If \eqref{22072201} (resp. \eqref{22072202}) holds, a Dehn filling point associated to \eqref{22062901} is contained in an algebraic subgroup defined by \eqref{22012105} (resp. \eqref{22012106}). 

Since $\mathcal{X}$ (resp. $\mathcal{X}\times \mathcal{X}$) is the product of two (resp. four) algebraic curves, the last statement is easily established by
 Theorem \ref{20071503}.  
\end{proof}

\newpage
\section{A weak version of the Zilber-Pink conjecture II (general case)}\label{ZPCII}
In this section, we establish Theorem \ref{22022601} in full generality. We prove the theorem step by step. First, in Subsection \ref{22042903}, it is established that there exists a finite set $\mathcal{H}$ of algebraic subgroups of codimension $1$ such that a Dehn filling point associated to \eqref{22041801} is contained in some $H\in \mathcal{H}$. Then, in Subsection \ref{22042905}, we further show each $H\in \mathcal{H}$ is chosen to be an algebraic subgroup of codimension $2$ or $3$. Finally, every $H\in\mathcal{H}$ is promoted farther to an algebraic subgroup of codimension $4$ in Subsection \ref{22042907}. The strategy of the proof for each step will be similar. 

Throughout the section, $\mathcal{M}$ and $\mathcal{X}$ represent a $2$-cusped hyperbolic $3$-manifold and its holonomy variety respectively unless otherwise stated. We also assume two cusps of $\mathcal{M}$ are not SGI each other, since, otherwise, the theorem was already covered in Section \ref{SGI2}. 
\subsection{Codimension $1$}\label{22042903}
As an initial step, we show 
\begin{theorem}\label{21082901}
Suppose there exists a family 
\begin{equation*}
\{\mathcal{M}_{p_{1i}/q_{1i},p_{2i}/q_{2i}},\mathcal{M}_{p'_{1i}/q'_{1i},p'_{2i}/q'_{2i}} \}_{i\in \mathcal{I}}
\end{equation*}
of infinitely many pairs of Dehn fillings of $\mathcal{M}$ satisfying 
\begin{equation}\label{22021901}
\text{pvol}_{\mathbb{C}}\;\mathcal{M}_{p_{1i}/q_{1i},p_{2i}/q_{2i}}=\text{pvol}_{\mathbb{C}}\;\mathcal{M}_{p'_{1i}/q'_{1i},p'_{2i}/q'_{2i}} \quad (i\in \mathcal{I}). 
\end{equation}
Then there exists a finite list $\mathcal{H}^{(1)}$ algebraic subgroups of codimension $1$ such that a Dehn filling point $P_i$ associated to \eqref{22021901} is contained in some element of $\mathcal{H}^{(1)}$. 
\end{theorem}

Before outlining the proof, we first give the following assumption and conventions for brevity, which will be valid through the section. 

\begin{assumption}
\normalfont For a given variety $\mathcal{X}$, if $\mathcal{X}$ contains infinitely many anomalous subvarieties $\{\mathcal{Y}_i\}_{i\in \mathcal{I}}$, by Theorem \ref{struc} (or \ref{struc2}), there exists a finite set $\Psi$ of algebraic tori such that $\mathcal{Y}_i$ is contained in a translation of some element of $\Psi$. Throughout the proofs below, in situations like this, we always assume $\Psi$ has a single element, that is, $\Psi=\{H\}$ for some algebraic torus $H$. 
\end{assumption}

\begin{convention}
\normalfont The core holonomies of $\mathcal{M}_{p_{1i}/q_{1i}, p_{2i}/q_{2i}}$ (resp. $\mathcal{M}_{p'_{1i}/q'_{1i}, p'_{2i}/q'_{2i}}$) are denoted by $t_{1i}, t_{2i}$ (resp. $t'_{1i}, t'_{2i}$). 
\end{convention}

The following is a generalization of Convention \ref{22062605}. 
\begin{convention}\label{22070503}
\normalfont Let $\mathcal{X}, K$ be the same as in Convention \ref{22062605} and $\mathcal{Y}$ be an irreducible variety satisfying $\mathcal{Y}\subset \mathcal{X}\cap K$. An irreducible component of $\mathcal{X}\cap K$ containing $\mathcal{Y}$ is denoted by $\mathcal{X}\cap_\mathcal{Y} K$.  
\end{convention}

The basic strategy of the proof of Theorem \ref{21082901} goes as follows. First, since a Dehn filling point $P_i$ associated to \eqref{22021901} is contained in an algebraic subgroup of dimension $3$ (as discussed in Section \ref{Dehn}) and the height of $P_i$ is uniformly bounded by Theorem \ref{20080404}, there exists an algebraic subgroup $H_i$ such that $P_i\in H_i$ and $\mathcal{Y}_i:=(\mathcal{X}\times \mathcal{X})\cap_{P_i} H_i$ is a torsion anomalous subvariety of $\mathcal{X}\times \mathcal{X}$ by Theorem \ref{19090804}. If $\bigcup_{i\in \mathcal{I}}\mathcal{Y}_i$ is contained in a finite set of algebraic subgroups, then we are done. Otherwise, applying Theorem \ref{struc}, there exists an algebraic torus $H$ such that 
\begin{equation*}
\mathcal{Y}_i\subset (\mathcal{X}\times \mathcal{X})\cap h_iH
\end{equation*}
for some $h_i\in \mathbb{G}^8$, and $\mathcal{Z}_i:=(\mathcal{X}\times \mathcal{X})\cap_{\mathcal{Y}_i} h_iH$ is a maximal anomalous subvariety of $\mathcal{X}\times \mathcal{X}$. 

We split the problem into three cases depending on the dimension of $\mathcal{Z}_i$ and, for each case, reduce the original problem to a simpler Zilber-Pink type conjecture in which the dimensions of a given variety and an ambient space are smaller than the original ones.

For instance, if $\dim \mathcal{Z}_i=3$, then $\mathcal{X}^{oa}=\emptyset$ by Lemma \ref{20080601} and $\mathcal{X}\times \mathcal{X}$ is the product of four algebraic curves by Proposition \ref{21073101}. Showing each $\mathcal{Y}_i$ is also an anomalous subvariety of $\mathcal{Z}_i$, the original problem over $\mathcal{X}\times \mathcal{X}$ is refined to a Zilber-Pink type question over $\mathcal{Z}_i$, which is (isomorphic to) the product of three algebraic curves. 

Second, if $\dim \mathcal{Z}_i=1$, the two varieties $\mathcal{Z}_i$ and $\mathcal{Y}_i$ are indeed equal, hence we invoke the reduction trick explained in Section \ref{22030101}, as summarized in Proposition \ref{24101101}, to reduce the problem to the Zilber-Pink conjecture for the solvable cases. That is, if we further assume 
\begin{equation}\label{22022805}
h_iH\subset H_i 
\end{equation}
and define $\mathcal{U}_H$ as in \eqref{22022501}, then $\overline{\mathcal{U}_H}$ is an algebraic variety of dimension $\dim\mathscr{Z}_H-1$ in $\mathbb{G}^{\codim  H}$ by Corollary \ref{22022603} and the assumption in \eqref{22022805} implies $\overline{\mathcal{U}_H}$ intersects with a family of algebraic subgroups. For each case considered, it will be shown that the dimension conditions on $H_i, H$ and $\overline{\mathcal{U}_H}$ meet one of the criteria in Corollary \ref{24100901}, thus allowing us to apply the corollary. 

The remaining case $\dim \mathcal{Z}_i=2$ is treated similarly by a hybrid of various methods described above.

The proof of Theorem \ref{21082901} is the longest (and perhaps the hardest) part of the paper, but it contains many key ideas and techniques repeatedly used in the proofs of other subsequent theorems below. 
 
\begin{proof}[Proof of Theorem \ref{21082901}]
We use the same notation given above.

If either $t_{1i}, t_{2i}$ or $t'_{1i}, t'_{2i}$ are multiplicatively dependent, the result follows by Theorem \ref{20071503}. Hence, throughout the proof, let us assume
\begin{itemize}
\item $t_{1i}, t_{2i}$ are multiplicatively independent each other;
\item $t'_{1i}, t'_{2i}$ are multiplicatively independent each other.
\end{itemize}
To simplify the proof, we assume, in addition, that $\mathcal{Y}_i$ is a torsion anomalous subvariety of the largest dimension containing $P_i$, and $\dim \mathcal{Y}_i, \dim \mathcal{Z}_i$ are independent of $i\in \mathcal{I}$. Note that $2\leq \dim(H_i\cap h_iH)$ by Lemma \ref{20080305} for each $i\in \mathcal{I}$. We prove the theorem by splitting it into several cases depending on the dimension of $\mathcal{Z}_i$. 
\begin{enumerate}
\item If $\dim \mathcal{Z}_i=3$ for each $i\in \mathcal{I}$, then $\mathcal{X}^{oa}=\emptyset$ by Lemma \ref{20080601}. Since two cusps of $\mathcal{M}$ are not SGI each other, using the preferred coordinates and equations given Definition \ref{22071301}, we assume $\mathcal{X}$ is defined by \eqref{21102304} in $\mathbb{G}^{4}(:=(\tilde{M_1}, \tilde{L_1}, \tilde{M_2}, \tilde{L_2}))$ and, similarly, $\mathcal{X}\times \mathcal{X}$ is embedded in $\mathbb{G}^8(:=(\tilde{M_1}, \tilde{L_1}, \dots, \tilde{M'_2}, \tilde{L'_2}))$. By Lemma \ref{20080601}, $H$ is either $\tilde{M_k}=\tilde{L_k}=1$ or $\tilde{M_k}=\tilde{L'_k}=1$ for some $k=1,2$. Without loss of generality, we assume 
\begin{equation*}
\mathcal{Y}_i\subset (\tilde{M'_2}=\zeta'_{m_{2i}},\tilde{L'_2}=\zeta'_{l_{2i}})
\end{equation*}
for some $\zeta'_{m_{2i}},\zeta'_{l_{2i}}\in \mathbb{G}$. If 
\begin{equation*}
\dim \big(H_i\cap (\tilde{M'_2}=\zeta'_{m_{2i}},\tilde{L'_2}=\zeta'_{l_{2i}})\big)=\dim H_i, 
\end{equation*}
it implies $\zeta'_{m_{2i}}, \zeta'_{l_{2i}}$ are roots of unity, contradicting the fact that $t'_{1i}, t'_{2i}$ are multiplicatively independent. Thus  
\begin{equation*}
\dim \big(H_i\cap (\tilde{M'_2}=\zeta'_{m_{2i}},\tilde{L'_2}=\zeta'_{l_{2i}})\big)<\dim H_i. 
\end{equation*}
If we project $H_i\cap  (\tilde{M'_2}=\zeta'_{m_{2i}},\tilde{L'_2}=\zeta'_{l_{2i}})$ under  
\begin{equation*}
\text{Pr}:\;\mathbb{G}^8(:=(\tilde{M_1},\tilde{L_1},\tilde{M_2},\tilde{L_2},\tilde{M'_1},\tilde{L'_1}, \tilde{M'_2}, \tilde{L'_2})\longrightarrow  \mathbb{G}^6(:=(\tilde{M_1},\tilde{L_1},\tilde{M_2},\tilde{L_2}, \tilde{M'_1}, \tilde{L'_1}), 
\end{equation*}
then
\begin{equation*}
\dim \text{Pr}\big(H_i\cap  (\tilde{M'_2}=\zeta'_{m_{2i}},\tilde{L'_2}=\zeta'_{l_{2i}})\big)\;(\text{in}\; \mathbb{G}^6) \;<\;\dim H_i\; (\text{in}\; \mathbb{G}^8).
\end{equation*}
Hence, if $\text{Pr}\big((\mathcal{X}\times \mathcal{X})\cap (\tilde{M'_2}=\zeta'_{m_{2i}},\tilde{L'_2}=\zeta'_{l_{2i}})\big)$, which is the product of three algebraic curves, is denoted by $\mathcal{X}_1\times \mathcal{X}_2\times \mathcal{X}_3$, then  
\begin{equation*}
\text{Pr}(\mathcal{Y}_i)=(\mathcal{X}_1\times \mathcal{X}_2\times \mathcal{X}_3)\cap_{\text{Pr}(P_i)} \text{Pr}\big(H_i\cap  (\tilde{M'_2}=\zeta'_{m_{2i}},\tilde{L'_2}=\zeta'_{l_{2i}})\big)
\end{equation*}
is an anomalous subvariety of $\mathcal{X}_1\times \mathcal{X}_2\times \mathcal{X}_3$.  
\begin{enumerate}
\item If $\text{Pr}(\mathcal{Y}_i)$ is a $2$-dim anomalous subvariety of $\mathcal{X}_1\times \mathcal{X}_2\times \mathcal{X}_3$, each $\text{Pr}(\mathcal{Y}_i)$ is either contained in a translation of one of the following by Lemma \ref{22060401}:$\tilde{M_1}=\tilde{L_1}=1, \tilde{M_2}=\tilde{L_2}=1$ or $\tilde{M'_1}=\tilde{L'_1}=1$. Without loss of generality, if we assume the last case, then $H_i$ is defined by equations of the following types\footnote{Otherwise, $H_i\cap \big(\tilde{M'_1}=\zeta'_{m_{1i}},\tilde{L'_1}=\zeta'_{l_{2i}}, \tilde{M'_2}=\zeta'_{m_{2i}},\tilde{L'_2}=\zeta'_{l_{2i}} \big)$ is an algebraic coset $K_i$ of dimension at most $3$ in $\mathbb{G}^8$ and so  
\begin{equation}\label{24092001}
(\mathcal{X}\times\mathcal{X})\cap_{P_i}H_i=\big((\mathcal{X}\times\mathcal{X})\cap_{P_i}H_i\big)\cap \big(\tilde{M'_1}=\zeta'_{m_{1i}},\tilde{L'_1}=\zeta'_{l_{2i}}, \tilde{M'_2}=\zeta'_{m_{2i}},\tilde{L'_2}=\zeta'_{l_{2i}} \big), 
\end{equation}
which is reduced to 
\begin{equation}\label{24092002}
\big(\mathcal{X}\times(\tilde{M'_1}=\zeta'_{m_{1i}},\tilde{L'_1}=\zeta'_{l_{2i}}, \tilde{M'_2}=\zeta'_{m_{2i}},\tilde{L'_2}=\zeta'_{l_{2i}})\big)\cap_{P_i }K_i,  
\end{equation}
is an algebraic variety of dim at most $1$. But this contradicts the fact that \eqref{24092001} is a $2$-dim anomalous subvariety of $\mathcal{X}\times \mathcal{X}$.}
\begin{equation}\label{22021410}
\tilde{M'_1}^{\tilde{a_{\alpha i}}'}\tilde{L'_1}^{\tilde{b_{\alpha i}}'}\tilde{M'_2}^{\tilde{c_{\alpha i}}'}\tilde{L'_2}^{\tilde{d_{\alpha i}}'}=1\quad (1\leq \alpha \leq 3).
\end{equation}
Regarding $H_i$ as a $1$-dim algebraic subgroup in $\mathbb{G}^4\big(:=(\tilde{M'_1}, \tilde{L'_1}, \tilde{M'_2}, \tilde{L'_2})\big)$, each $H_i$ intersects with $\mathcal{X}$ (here $\mathcal{X}$ is the second copy of $\mathcal{X}\times \mathcal{X}$) nontrivially in $\mathbb{G}^4$. Since the Zilber-Pink conjecture holds for $\mathcal{X}$ by Corollary \ref{21021602}, $\bigcup_{i\in I} \mathcal{X}\cap H_i$ (and $\bigcup_{i\in I} \mathcal{Y}_i$ as well) is contained in the union of finitely many algebraic subgroups, contradicting the assumption. 

\item If $\text{Pr}(\mathcal{Y}_i)$ is a $1$-dim anomalous subvariety of $\mathcal{X}_1\times \mathcal{X}_2\times \mathcal{X}_3$, by Theorem \ref{struc}, there exists an algebraic subgroup $H'$ such that, for each $i\in \mathcal{I}$,  
\begin{equation*}
\text{Pr}(\mathcal{Y}_i)\subset (\mathcal{X}_1\times \mathcal{X}_2\times \mathcal{X}_3)\cap h'_iH' 
\end{equation*}
for some $h'_i\in \mathbb{G}^6$ and $(\mathcal{X}_1\times \mathcal{X}_2\times \mathcal{X}_3)\cap_{\text{Pr}(\mathcal{Y}_i)}h'_iH'$ is a maximal anomalous subvariety of $\mathcal{X}_1\times \mathcal{X}_2\times \mathcal{X}_3$. 
\begin{enumerate}
\item If $\dim (\mathcal{X}_1\times \mathcal{X}_2\times \mathcal{X}_3)\cap_{\text{Pr}(\mathcal{Y}_i)} h'_iH'=2$, $H'$ is defined by one of the following by Lemma \ref{22060401}: $\tilde{M_1}=\tilde{L_1}=1, \tilde{M_2}=\tilde{L_2}=1$ or $\tilde{M'_1}=\tilde{L'_1}=1$. Without loss of generality, let us assume 
\begin{equation*}
\text{Pr}(\mathcal{Y}_i)\subset (\mathcal{X}_1\times \mathcal{X}_2\times \mathcal{X}_3)\cap_{\text{Pr}(\mathcal{Y}_i)}h'_iH'\subset (\tilde{M'_1}=\zeta'_{m_{1i}}, \tilde{L'_1}=\zeta'_{l_{1i}})
\end{equation*} 
where $\zeta'_{m_{1i}},\zeta'_{l_{1i}}\in \mathbb{G}$. 
\begin{enumerate}
\item If the dimension of 
\begin{equation}\label{22021501}
H_i\cap (\tilde{M'_1}=\zeta'_{m_{1i}}, \tilde{L'_1}=\zeta'_{l_{1i}}, \tilde{M'_2}=\zeta'_{m_{2i}}, \tilde{L'_2}=\zeta'_{l_{2i}})
\end{equation}
is $3$ (in $\mathbb{G}^8$), it implies $H_i$ is contained in an algebraic subgroup defined by equations of the forms given in \eqref{22021410}, thus included in the previous case. 
\item Suppose the dimension of \eqref{22021501} is $2$. Projecting \eqref{22021501} under 
\begin{equation*}
\text{Pr}_1:\;(\tilde{M_1}, \tilde{L_1}, \dots, \tilde{M'_2}, \tilde{L'_2})\longrightarrow (\tilde{M_1}, \tilde{L_1}, \tilde{M_2}, \tilde{L_2}), 
\end{equation*}
the image of \eqref{22021501} is a $2$-dim algebraic coset in $\mathbb{G}^4$ whose intersection with $\mathcal{X}_1\times \mathcal{X}_2(=\text{Pr}_1(\mathcal{X}\times \mathcal{X}))$ is a $1$-dim anomalous subvariety of $\mathcal{X}_1\times \mathcal{X}_2$. By Lemma \ref{22060401}, each $(\mathcal{X}_1\times \mathcal{X}_2)\cap_{\text{Pr}_1(P_i)} \text{Pr}_1(\eqref{22021501})$ is contained in a translation of either $\tilde{M_1}=\tilde{L_1}=1$ or $\tilde{M_2}=\tilde{L_2}=1$. Without loss of generality, if we take the second one, then $H_i$ is defined by equations of the following types:
\begin{equation*}
\tilde{M_2}^{\tilde{c_{\alpha i}}}\tilde{L_2}^{\tilde{d_{\alpha i}}}\tilde{M'_1}^{\tilde{a_{\alpha i}}'}\tilde{L'_1}^{\tilde{b_{\alpha i}}'}\tilde{M'_2}^{\tilde{c_{\alpha i}}'}\tilde{L'_2}^{\tilde{d_{\alpha i}}'}=1\quad (1\leq \alpha \leq 4).
\end{equation*}
Now projecting $H_i$ and $\mathcal{X}\times\mathcal{X}$ under
\begin{equation*}
\text{Pr}_2:\;(\tilde{M_1},\tilde{L_1},\tilde{M_2},\tilde{L_2},\tilde{M'_1},\tilde{L'_1}, \tilde{M'_2}, \tilde{L'_2}) \longrightarrow  (\tilde{M_2},\tilde{L_2},\tilde{M'_1},\tilde{L'_1}, \tilde{M'_2}, \tilde{L'_2}), 
\end{equation*}
$\text{Pr}_2(H_i)$ is a $2$-dim algebraic subgroup intersecting with a $3$-dim algebraic variety $\text{Pr}_2(\mathcal{X}\times \mathcal{X})$ at $\text{Pr}_2(P_i)$ in $\mathbb{G}^6$. By Theorem \ref{19090804}, as the height of $\text{Pr}_2(P_i)$ is uniformly bounded, there exists a family of algebraic subgroups $\{\tilde{H_i}\}_{i\in \mathcal{I}}$ such that $\text{Pr}_2(P_i)\in \tilde{H_i}$ and $\text{Pr}_2(\mathcal{X}\times \mathcal{X})\cap_{\text{Pr}_2(P_i)} \tilde{H_i}$ is a torsion anomalous subvariety of $\text{Pr}_2(\mathcal{X}\times \mathcal{X})$ for each $i\in \mathcal{I}$. Considering an algebraic subgroup $\text{Pr}^{-1}_2(\tilde{H_i})$ in $\mathbb{G}^8$, one can check $(\mathcal{X}\times \mathcal{X})\cap_{P_i} \text{Pr}^{-1}_2(\tilde{H_i})$ is still a torsion anomalous subvariety of $\mathcal{X}\times \mathcal{X}$ satisfying $\dim  \big((\mathcal{X}\times \mathcal{X})\cap_{P_i} \text{Pr}^{-1}_2(\tilde{H_i})\big)\geq 2$. But this contradicts the initial assumption that $\mathcal{Y}_i$ is a torsion anomalous subvariety of $\mathcal{X}\times \mathcal{X}$ of the largest dimension containing $P_i$. 
\end{enumerate}

\item Now suppose $\dim (\mathcal{X}_1\times \mathcal{X}_2\times \mathcal{X}_3)\cap_{\text{Pr}(\mathcal{Y}_i)} h'_iH'=1$.\footnote{This implies $\text{Pr}(\mathcal{Y}_i)=(\mathcal{X}_1\times \mathcal{X}_2\times \mathcal{X}_3)\cap_{\text{Pr}(\mathcal{Y}_i)} h'_iH'$.} 
\begin{enumerate}
\item If $\mathscr{Z}_{H'}=\mathcal{X}_1\times \mathcal{X}_2\times \mathcal{X}_3$, by Lemma \ref{21100201}, each $(\mathcal{X}_1\times \mathcal{X}_2\times \mathcal{X}_3)\cap_{\text{Pr}(\mathcal{Y}_i)} h'_iH'$ is contained in a translation of one of the following: $\tilde{M_1}=\tilde{L_1}=1, \tilde{M_2}=\tilde{L_2}=1$ or $\tilde{M'_1}=\tilde{L'_1}=1$. However this contradicts the fact that $(\mathcal{X}_1\times \mathcal{X}_2\times \mathcal{X}_3)\cap_{\text{Pr}(\mathcal{Y}_i)} h'_iH'$ is a maximal anomalous subvariety of $\mathcal{X}_1\times \mathcal{X}_2\times \mathcal{X}_3$. 

\item Otherwise, if $\mathscr{Z}_{H'}\neq \mathcal{X}_1\times \mathcal{X}_2\times \mathcal{X}_3$, applying Theorem \ref{struc2}, we further assume $h'_iH'\subset \text{Pr}(H_i)$ for each $i\in \mathcal{I}$. Note that $\dim H'=2$ or $3$ in $\mathbb{G}^6$. If $\mathcal{U}_{H'}$ is defined as in Corollary \ref{22022603}, then $\overline{\mathcal{U}_{H'}}$ is an algebraic curve in $\mathbb{G}^{\codim H'}$ by the same corollary.\footnote{Since $\mathcal{X}_1\times \mathcal{X}_2\times \mathcal{X}_3$ is not foliated by translations of $H'$, $\dim \mathscr{Z}_{H'}=2$ and, as $\dim h'_iH'=1$ for each $i$, we have $\dim \overline{\mathcal{U}_{H'}}=1$ by Corollary \ref{22022603}. } As $\text{Pr} (H_i)$ is an algebraic subgroup of codimension at least $2$, it follows that $\mathcal{U}_{H'}$ is contained in some algebraic subgroup by Corollary \ref{24100901}, which further implies every $\mathcal{Y}_i$ is contained in the same algebraic subgroup. But this contradicts our initial assumption. 
\end{enumerate} 
\end{enumerate}
\end{enumerate}

\item Suppose $\dim \mathcal{Z}_i=2$ for each $i\in \mathcal{I}$.   
\begin{enumerate}
\item If $\mathscr{Z}_H=\mathcal{X}\times \mathcal{X}$ and $\mathcal{X}^{oa}=\emptyset$, by Lemma \ref{20050801}, $\mathcal{Z}_i$ is contained in a translation of either $\tilde{M_k}=\tilde{L_k}=1$ or $\tilde{M'_k}=\tilde{L'_k}=1$ for some $1\leq k\leq 2$. But this contradicts the fact that $\mathcal{Z}_i$ is a maximal anomalous subvariety of $\mathcal{X}\times \mathcal{X}$.

\item If $\mathscr{Z}_H=\mathcal{X}\times \mathcal{X}$ and $\mathcal{X}^{oa}\neq \emptyset$, by Lemma \ref{20050801}, $\mathcal{Z}_i$ is contained in a translation of either $M_1=L_1=M_2=L_2=1$ or $M'_1=L'_1=M'_2=L'_2=1$. Without loss of generality, we suppose the first case. 
\begin{enumerate}
\item If $\dim \mathcal{Y}_i=2$ (i.e. $\mathcal{Z}_i=\mathcal{Y}_i$), then each $H_i$ is defined by equations of the following types:
\begin{equation}\label{22021405}
M_1^{a_{\alpha i}}L_1^{b_{\alpha i}}M_2^{c_{\alpha i}}L_2^{d_{\alpha i}}=1\quad (1\leq \alpha \leq 3).
\end{equation}
Considering $H_i$ as a $1$-dim algebraic subgroup in $\mathbb{G}^{4}(:=(M_1, L_1, M_2, L_2))$, we have $\mathcal{X}\cap H_i\neq \emptyset$ for each $i\in \mathcal{I}$ (here $\mathcal{X}$ represents the first copy of $\mathcal{X}\times \mathcal{X}$). By Corollary \ref{21021602}, $\bigcup_{i\in \mathcal{I}}\mathcal{X}\cap H_i$ is contained in the union of a finite number of algebraic subgroups, and this contradicts our assumption. 
 
\item For $\dim \mathcal{Y}_i=1$ (and so $\dim H_i=4$), let us assume $M_1=t_{1i}^{-q_{1i}}, L_1=t_{1i}^{p_{1i}}, M_2=t_{2i}^{-q_{2i}}, L_2= t_{2i}^{p_{2i}}$ on $\mathcal{Y}_i$. 
\begin{enumerate}
\item If the dimension of 
\begin{equation}\label{22021605}
H_i\cap (M_1=t_{1i}^{-q_{1i}}, L_1=t_{1i}^{p_{1i}}, M_2=t_{2i}^{-q_{2i}}, L_2= t_{2i}^{p_{2i}})
\end{equation}
is $3$, it means $H_i$ is contained in an $5$-dim algebraic subgroup defined by equations of the forms given in \eqref{22021405}, thus falling into the above case. 

\item If the dimension of \eqref{22021605} is $2$, it is contained in an algebraic coset defined by equations of the following forms: 
\begin{equation}\label{22021609}
(M'_1)^{a'_{\alpha i}}(L'_1)^{b'_{\alpha i}}(M'_2)^{c'_{\alpha i}}(L'_2)^{d'_{\alpha i}}=\zeta'_{\alpha i} \quad (\alpha =1,2)
\end{equation}
where $\zeta'_{\alpha i}\in \mathbb{G}$. Regarding \eqref{22021609} simply as a $2$-dim algebraic coset $K_i$ in $\mathbb{G}^4(:=(M'_1,L'_1,M'_2,L'_2))$, $\mathcal{X}\cap K_i$ is an anomalous subvariety of $\mathcal{X}$ (here $\mathcal{X}$ is the second copy of $\mathcal{X}\times \mathcal{X}$). As $\mathcal{X}^{oa}\neq \emptyset$, $\mathcal{X}$ contains only finitely many anomalous subvarieties, and so $\bigcup_{i\in \mathcal{I}}\mathcal{X}\cap K_i$ is contained in the union of finitely many algebraic subgroups, contradicting the assumption.   
\end{enumerate}
\end{enumerate}

\item Now suppose $\mathscr{Z}_H\neq \mathcal{X}\times \mathcal{X}$ (i.e. $\dim \mathscr{Z}_H=3$). If $\mathcal{U}_H$ is defined as in Corollary \ref{22022603}, as $\dim \mathcal{Z}_i=2$, $\overline{\mathcal{U}_H}$ is an algebraic curve in $\mathbb{G}^{3(=\codim  H)}$. 
\begin{enumerate}
\item If $\dim \mathcal{Y}_i=2$ (and so $\dim H_i=5$), since $4\leq \dim (H_i\cap h_iH)\leq 5$ by Lemma \ref{20080305}, $\overline{\mathcal{U}_H}$ intersects with infinitely many algebraic subgroups of codimension at least $2$. Therefore $\overline{\mathcal{U}_H}$ is contained in some algebraic subgroup by Theorem \ref{20080407}, which further implies there exists an algebraic subgroup containing all $\mathcal{Z}_i$. But this contradicts the initial assumption. 

\item For $\dim \mathcal{Y}_i=1$ (and so $\dim H_i=4$), we claim $\dim (H_i\cap h_iH)>2$. Once it is proven, then $3\leq \dim (H_i\cap h_iH)\leq 4$ and so, by a similar reasoning as above, $\overline{\mathcal{U}_H}$ is contained in an algebraic subgroup, leading the same contradiction. 

Otherwise, if $\dim (H_i\cap h_iH)=2$ for each $i\in \mathcal{I}$, applying Gauss elimination if necessary, one may assume $H_i\cap h_iH$ is contained in an algebraic coset defined by equations of the following types:  
\begin{equation}\label{21021303}
\begin{gathered}
M_1^{a_{\alpha i}}L_1^{b_{\alpha i}}M_2^{c_{\alpha i}}L_2^{d_{\alpha i}}=\zeta_{\alpha i},\quad (\alpha =1,2)
\end{gathered}
\end{equation} 
\begin{equation}\label{21021301}
\begin{gathered}
(M'_1)^{a'_{\alpha i}}(L'_1)^{b'_{\alpha i}}(M'_2)^{c'_{\alpha i}}(L'_2)^{d'_{\alpha i}}=\zeta'_{\alpha i},\quad (\alpha =1,2) 
\end{gathered}
\end{equation} 
where $\zeta_{\alpha i},\zeta'_{\alpha i}\in \mathbb{G}$ ($1\leq \alpha \leq 2$). Let $K_i$ and $K'_i$ be algebraic cosets defined by \eqref{21021303} and \eqref{21021301} respectively. If both $(\mathcal{X}\times \mathcal{X})\cap_{\mathcal{Y}_i} K_i$ and $(\mathcal{X}\times \mathcal{X})\cap_{\mathcal{Y}_i} K'_i$ are not anomalous subvarieties of $\mathcal{X}\times \mathcal{X}$, then $\big((\mathcal{X}\times \mathcal{X})\cap_{\mathcal{Y}_i} K_i\big)\cap \big((\mathcal{X}\times \mathcal{X})\cap_{\mathcal{Y}_i}K'_i\big)$ is $0$-dim (i.e. point), contradicting the fact that $(\mathcal{X}\times \mathcal{X})\cap_{P_i} H_i(\subset K_i\cap K'_i)$ is an anomalous subvariety of $\mathcal{X}\times \mathcal{X}$. Thus either $(\mathcal{X}\times \mathcal{X})\cap K_i$ or $(\mathcal{X}\times \mathcal{X})\cap K'_i$ must be an anomalous subvariety of $\mathcal{X}\times \mathcal{X}$. If $\mathcal{X}^{oa}\neq \emptyset$, $\bigcup_{i\in \mathcal{I}}(\mathcal{X}\times \mathcal{X})\cap K'_i$ is contained in a finite set of algebraic subgroups, inducing a contradiction. On the other hand, for $\mathcal{X}^{oa}=\emptyset$, each $(\mathcal{X}\times \mathcal{X})\cap K'_i$ is contained in a translation of either $\tilde{M'_1}=\tilde{L'_1}=1$ or $\tilde{M'_2}=\tilde{L'_2}=1$ by Lemma \ref{22060401}. But this contradicts the assumption that $\mathcal{Z}_i$ is a maximal anomalous subvariety of $\mathcal{X}\times \mathcal{X}$ containing $\mathcal{Y}_i$.  
\end{enumerate}
\end{enumerate}

\item Lastly suppose $\dim \mathcal{Z}_i=1$ (so $\dim \mathcal{Y}_i=1$)  for each $i\in \mathcal{I}$. For $\dim H_i\cap h_iH=2$, we get a contradiction as seen above, thereby assume $3\leq \dim H_i\cap h_iH\leq 4$. Let $\mathcal{U}_H$ be as given in Corollary \ref{22022603}, and suppose in addition that $h_iH\subset H_i$ for every $i\in \mathcal{I}$ (by Theorem \ref{struc2}). If $\dim H_i\cap h_iH= 4$ (i.e. $\dim H=\dim H_i=4$), by Corollary \ref{24100901}, there exist finitely many algebraic subgroups whose union contains $\bigcup_{i\in \mathcal{I}} \mathcal{Z}_i$, contradicting the assumption. Otherwise, if $\dim H=3$ (i.e. $\codim  H=5$) and $\dim H_i=4$, and in addition, if $\dim \mathscr{Z}_H=2$ or $4$, then $\dim \overline{\mathcal{U}_H}=1$ or $3$ (respectively), hence the same contradiction again follows from Corollary \ref{24100901}. Lastly, if $\dim \mathscr{Z}_H=3$ (with  $\dim H=3$ and $\dim H_i=4$), then $\overline{\mathcal{U}_H}$ is an algebraic surface in $\mathbb{G}^5$ by Corollary \ref{22022603}, intersecting with infinitely many algebraic subgroups of dim $1$. Let $\mathbb{G}^4$ be a subspace of $\mathbb{G}^5$ equipped with a projection $\text{Pr}:\mathbb{G}^5\longrightarrow \mathbb{G}^4$ satisfying $\dim \text{Pr}(\overline{\mathcal{U}_H})=2$. Then $\text{Pr}(\overline{\mathcal{U}_H})$ is an algebraic surface in $\mathbb{G}^4$ for each intersecting with infinitely many algebraic subgroups of dimension at most $1$. Now we attain a contradiction either by Theorem \ref{20102701} or \ref{22060101}.  
\end{enumerate}
\end{proof}

\subsection{Codimension $2$}
Let $\mathcal{M}_{p_1/q_1, p_2/q_2}$ and $\mathcal{M}_{p'_1/q'_1, p'_2/q'_2}$ be Dehn fillings of $\mathcal{M}$ satisfying 
\begin{equation}\label{22043003}
\text{pvol}_{\mathbb{C}}\;\mathcal{M}_{p_{1}/q_{1},p_{2}/q_{2}}=\text{pvol}_{\mathbb{C}}\;\mathcal{M}_{p'_{1}/q'_{1},p'_{2}/q'_{2}}
\end{equation}
with $|p_h|+|q_h|$ and $|p'_h|+|q'_h|$ sufficiently large $(h=1,2)$. Let $t_1, t_2$ (resp. $t'_1, t'_2$) be the core holonomies of $\mathcal{M}_{p_1/q_1, p_2/q_2}$ (resp. $\mathcal{M}_{p'_1/q'_1, p'_2/q'_2}$).\footnote{We fix this notation throughout the section.}

Remark that if $t_1t_2=t'_1t'_2$ is the only relation between $t_1, t_2, t'_1$ and $t'_2$ or, equivalently, 
\\
\\
($\spadesuit$) \;\;\emph{there is no proper subset of $\{t_1, t_2, t'_1, t'_2\}$ whose elements are multiplicatively dependent}, 
\\
\\
then the conclusion of Theorem \ref{22022601} fails. That is, there is no finite set $\mathcal{H}^{(4)}$ of algebraic subgroups of dimension $4$ containing a Dehn filling point $P$ associated to \eqref{22043003} under $\spadesuit$. To explain this more in detail, let $H$ be an algebraic subgroup of codimension $4$ containing $P$ and, without loss of generality, assume it is contained in an algebraic subgroup defined by
\begin{equation}\label{22043001}
\begin{gathered}
M_1^{a_{1}}L_1^{b_{1}}M_2^{c_{1}}L_2^{d_{1}}(M'_1)^{a'_{1}}(L'_1)^{b'_{1}}(M'_2)^{c'_{1}}(L'_2)^{d'_{1}}=1,\\
M_1^{a_{2}}L_1^{b_{2}}M_2^{c_{2}}L_2^{d_{2}}(M'_1)^{a'_{2}}(L'_1)^{b'_{2}}(M'_2)^{c'_{2}}=1,\\
M_1^{a_{3}}L_1^{b_{3}}M_2^{c_{3}}L_2^{d_{3}}(M'_1)^{a'_{3}}(L'_1)^{b'_{3}}=1.
\end{gathered}
\end{equation}
As $t_1, t_2, t'_1$ are multiplicatively independent, the last equation in \eqref{22043001} implies 
\begin{equation*}
(t_1^{-q_1})^{a_{3}}(t_1^{p_1})^{b_{3}},\quad (t_2^{-q_2})^{c_{3}}(t_2^{p_2})^{d_{3}}\quad  \text{and}\quad  \big((t'_1)^{-q'_1}\big)^{a'_{3}}\big((t'_1)^{p'_1}\big)^{b'_{3}} 
\end{equation*}
are torsions and, as $t_1, t_2$ and $t'_1$ are not roots of unity, this further implies  
\begin{equation*}
-q_1a_{3}+p_1b_{3}=-q_2c_{3}+p_2d_{3}=-q'_1a'_{3}+p'_1b'_{3}=0. 
\end{equation*}
Since $(p_1, q_1), (p_2, q_2)$ and $(p'_1, q'_1)$ are co-prime pairs, they are uniquely determined by $(a_{3},b_{3}), (c_{3}, d_{3})$ and $(a'_{3}, b'_{3})$. However this contradicts the fact that $|p_h|+|q_h|$ and $|p'_h|+|q'_h|$ ($h=1,2$) are sufficiently large. 

In this subsection, we show the assumption $\spadesuit$ actually never occurs. In other words, the following theorem is proven.
\begin{theorem}\label{21091101}
Let $\mathcal{M}_{p_1/q_1, p_2/q_2}$ and $\mathcal{M}_{p'_1/q'_1, p'_2/q'_2}$ be two Dehn fillings of $\mathcal{M}$ satisfying \eqref{22043003} with $|p_h|+|q_h|$ and $|p'_h|+|q'_h|$ ($h=1,2$) sufficiently large. Then there exists a proper subset of $\{t_{1}, t_{2}, t'_{1}, t'_2\}$ such that the elements in the subset are multiplicatively dependent.
\end{theorem}

To prove Theorem \ref{21091101}, we claim two theorems beforehand, which are viewed as extensions of Theorem \ref{21082901} under $\spadesuit$. Once Theorems \ref{21082301}-\ref{22071211} are established, we then get a more quantitative version of them in Lemma \ref{21082802}. Theorem \ref{21091101} is then verified with the aid of Lemma \ref{21082802}.  
\begin{theorem}\label{21082301}
Suppose there exists a family 
\begin{equation*}
\{\mathcal{M}_{p_{1i}/q_{1i},p_{2i}/q_{2i}},\mathcal{M}_{p'_{1i}/q'_{1i},p'_{2i}/q'_{2i}} \}_{i\in \mathcal{I}}
\end{equation*}
of infinitely many pairs of Dehn fillings of $\mathcal{M}$ satisfying 
\begin{equation}\label{22021902}
\text{pvol}_{\mathbb{C}}\;\mathcal{M}_{p_{1i}/q_{1i},p_{2i}/q_{2i}}=\text{pvol}_{\mathbb{C}}\;\mathcal{M}_{p'_{1i}/q'_{1i},p'_{2i}/q'_{2i}} 
\end{equation}
with sufficiently large $|p_{hi}|+|q_{hi}|$ and $|p'_{hi}|+|q'_{hi}|$ ($h=1,2$) for each $i\in \mathcal{I}$. Suppose
\\
\\
($\clubsuit$) there is no proper subset of $\{t_{1i}, t_{2i}, t'_{1i}, t'_{2i}\}$ whose elements are multiplicatively dependent.
\\
\\
Then there exists a finite set $\mathcal{H}^{(2)}$ of algebraic subgroups of codimension $2$ such that a Dehn filling point associated to \eqref{22021902} is contained in some $H^{(2)}\in \mathcal{H}^{(2)}$. 
\end{theorem}
The general strategy of the proof is very similar to that of the proof of Theorem \ref{21082901}. However, as we work with the projected image of $\mathcal{X}\times \mathcal{X}$ onto $H$ with $H$ from Theorem \ref{21082901}, the proof itself is somewhat simpler than that of Theorem \ref{21082901}.  

Later in Theorem \ref{22071211}, we further show $(\mathcal{X}\times \mathcal{X})\cap_{(1, \dots, 1)} H^{(2)}$ is an anomalous subvariety of $\mathcal{X}\times \mathcal{X}$ for $H^{(2)}$ obtained in Theorem \ref{21082301}. 

\begin{proof}[Proof of Theorem \ref{21082301}]

By Theorem \ref{21082901}, there exists a finite set $\mathcal{H}^{(1)}$ of algebraic subgroups of codimension $1$ such that each $P_i$ is contained in some element of $\mathcal{H}^{(1)}$. Without loss of generality, we simply assume $\mathcal{H}^{(1)}=\{H\}$ and, by changing basis if necessary, $H$ is defined by
\begin{equation}\label{21082904}
M_1^{a}L_1^bM_2^cL_2^d(M'_1)^{a'}(M'_2)^{c'}=1.
\end{equation}
Let 
\begin{equation}\label{22022701}
\overline{\text{Pr}\big((\mathcal{X}\times \mathcal{X})\cap H\big)}
\end{equation}
be the algebraic closure of the image of $(\mathcal{X}\times \mathcal{X})\cap_{(1,\dots, 1)} H$ under the following projection:
\begin{equation*}
\text{Pr}:\;(M_1, L_1, M_2, L_2, M'_1, L'_1, M'_2, L'_2)\longrightarrow (M_1, L_1, M_2, L_2, M'_1, L'_1, L'_2).
\end{equation*}
Note that \eqref{22022701} is a variety of dimension $3$ and  
\begin{equation*}
\text{Pr}(P_i)\in \overline{\text{Pr}\big((\mathcal{X}\times \mathcal{X})\cap H\big)}\cap (M_1^{p_{1i}}L_1^{q_{1i}}=M_2^{p_{2i}}L_2^{q_{2i}}=(M'_1)^{p'_{1i}}(L'_1)^{q'_{1i}}=1), \quad (i\in \mathcal{I}). 
\end{equation*}
Since $t_{1i}t_{2i}=t'_{1i} t'_{2i}$, $\text{Pr}(P_i)$ is further contained in an additional algebraic subgroup. So, by Theorem \ref{19090804}, there exist an algebraic subgroup $H_i$ of dim $4$ where 
\begin{itemize}
\item $\text{Pr}(P_i)\in \overline{\text{Pr}\big((\mathcal{X}\times \mathcal{X})\cap H\big)}\cap H_i$;
\item $\overline{\text{Pr}\big((\mathcal{X}\times \mathcal{X})\cap H\big)}\cap_{\text{Pr}(P_i)} H_i$ is an anomalous subvariety of $\overline{\text{Pr}\big((\mathcal{X}\times \mathcal{X})\cap H\big)}$  
\end{itemize}
for each $i\in \mathcal{I}$. If there is a finite family of algebraic subgroups such that each   
\begin{equation}\label{22012901}
\overline{\text{Pr}\big((\mathcal{X}\times \mathcal{X})\cap H\big)}\cap_{\text{Pr}(P_i)} H_i
\end{equation}
is contained in one of them, we are done. Otherwise, let us assume there is no such a finite set and each \eqref{22012901} is contained in a different maximal anomalous subvariety of \eqref{22022701}. By $\clubsuit$, each $H_i$ is first included in some algebraic subgroup defined by
\begin{equation}\label{22020105}
M_1^{a_{1i}}L_1^{b_{1i}}M_2^{c_{2i}}L_2^{d_{2i}}(M'_1)^{a'_{1i}}(L'_1)^{b'_{1i}}(L'_2)^{d'_{1i}}=1\quad \text{where  }\;(a_{1i}, b_{1i}),(c_{2i},d_{2i}),(a'_{1i}, b'_{1i})\neq (0,0), \;\;d'_{1i}\neq 0
\end{equation}
and, second, the one defined by one of the following:\footnote{Recall the norm minimizing property of the defining equations of $H_i$ mentioned in Remark \ref{22081302}.} 
\begin{equation}\label{21080301}
M_1^{p_{1i}}L_1^{q_{1i}}=M_2^{p_{2i}}L_2^{q_{2i}}=1,\quad M_1^{p_{1i}}L_1^{q_{1i}}=(M'_1)^{p'_{1i}}(L'_1)^{q'_{1i}}=1\quad \text{or}\quad
M_2^{p_{2i}}L_2^{q_{2i}}=(M'_1)^{p'_{1i}}(L'_1)^{q'_{1i}}=1.
\end{equation}
\begin{enumerate}
\item If $H_i$ is contained in the first type of equations in \eqref{21080301}, then $M_1, L_1, M_2, L_2$ are constants on $\overline{\text{Pr}\big((\mathcal{X}\times \mathcal{X})\cap H\big)}\cap_{\text{Pr}(P_i)} H_i$ and so 
\begin{equation}\label{22021511}
M_1=t_{1i}^{-q_{1i}},\quad  L_1=t_{1i}^{p_{1i}}, \quad M_2=t_{2i}^{-q_{2i}},\quad L_2=t_{2i}^{p_{2i}}. 
\end{equation}
Then the intersection between \eqref{21082904}, \eqref{22020105} and \eqref{22021511} is contained in an algebraic coset $K_i$ defined by 
\begin{equation}\label{22061301}
\begin{gathered}
(M'_1)^{a'}(M'_2)^{c'}=t_{1i}^{aq_{1i}}t_{1i}^{-bp_{1i}}t_{2i}^{cq_{2i}}t_{2i}^{-dp_{2i}},\quad (M'_1)^{a'_{1i}}(L'_1)^{b'_{1i}}(L'_2)^{d'_{1i}}=t_{1i}^{a_{1i}q_{1i}}t_{1i}^{-b_{1i}p_{1i}}t_{2i}^{c_{1i}q_{2i}}t_{2i}^{-d_{1i}p_{2i}}.
\end{gathered}
\end{equation} 
Since   
\begin{equation}\label{22022703}
\big((\mathcal{X}\times\mathcal{X})\cap_{P_i} K_i\big)\cap  \eqref{22021511} \big(=((\mathcal{X}\times \mathcal{X})\cap_{P_i} H)\cap H_i\big)
\end{equation}
is of dimension $1$, regarding $K_i$ as a subset of $\mathbb{G}^4(:=(M'_1, L'_1, M'_2, L'_2))$, it follows that $\mathcal{X}\cap K_i$ (here $\mathcal{X}$ is the second copy of $\mathcal{X}\times \mathcal{X}$) is an anomalous subvariety of $\mathcal{X}$. 

If $\mathcal{X}^{oa}\neq \emptyset$, as $\mathcal{X}$ has only finitely many anomalous subvarieties, each $\mathcal{X}\cap K_i$ (as well as \eqref{22012901}) is contained in a finite set of algebraic subgroups, contradicting the assumption.

Otherwise, if $\mathcal{X}^{oa}=\emptyset$, using the preferred coordinates in Definition \ref{22071301}, it follows that $K_i$ is contained in a translation of either $\tilde{M'_1}=\tilde{L'_1}=1$ or $\tilde{M'_2}=\tilde{L'_2}=1$ by Lemma \ref{22060401}. Without loss of generality, we consider the first one. This implies that all $(a'_{1i}, b'_{1i}, d'_{1i})$ in \eqref{22061301} are indeed  parallel and the preferred coordinates of $\mathcal{X}$ take the following forms
\begin{equation}\label{22021601}
\tilde{M'_1}=(M'_1)^{a'}(M'_2)^{c'},  \tilde{L'_1}=(M'_1)^{a'_{1}}(L'_1)^{b'_{1}}(L'_2)^{d'_{1}},  \tilde{M'_2}=(M'_1)^{a'}(M'_2)^{-c'},  \tilde{L'_2}=(M'_1)^{a'_{1}}(L'_1)^{b'_{1}}(L'_2)^{-d'_{1}} 
\end{equation}
where $(a'_1, b'_1, d'_1)$ is a triple such that $(a'_1, b'_1, d'_1)=c_i(a'_{1i}, b'_{1i}, d'_{1i})$ for some $c_i\in \mathbb{Q}$ ($i\in \mathcal{I}$). Also, under \eqref{22021601}, \eqref{21082904} and \eqref{22020105} are transformed into equations of the forms
\begin{equation}\label{22061405}
M_1^aL_1^bM_2^cL_2^d\tilde{M'_1}=1
\end{equation}
and 
\begin{equation*}
M_1^{a_{1i}}L_1^{b_{1i}}M_2^{c_{1i}}L_2^{d_{1i}}\tilde{L'_1}^{m_i}=1
\end{equation*}
for some $m_i\in \mathbb{Z}$ respectively. As a result, $M_1, L_1, M_2, L_2, \tilde{M'_1}, \tilde{L'_1}$ are all constants over \eqref{22022703} and each \eqref{22022703} is isomorphic to an algebraic curve defined by $f(\tilde{M'_2}, \tilde{L'_2})=0$ where $f$ is the polynomial defining $\mathcal{X}$ as given in \eqref{21102304}. 

With this observation, we now reformulate the original question as follows. First, let $\overline{\tilde{\text{Pr}}\big((\mathcal{X}\times \mathcal{X})\cap H\big)}$ and $\tilde{\text{Pr}}(H_i)$ be the algebraic closures of the images of $(\mathcal{X}\times \mathcal{X})\cap_{(1, \dots, 1)} H$ and $H_i$ ($i\in \mathcal{I}$) under
\begin{equation*}
\tilde{\text{Pr}}:\;\mathbb{G}^8(:=(M_1, L_1,M_2, L_2, \tilde{M'_1}, \tilde{L'_1} , \tilde{M'_2}, \tilde{L'_2}))\longrightarrow \mathbb{G}^5(:=(M_1, L_1, M_2, L_2, \tilde{L'_1}))
\end{equation*}
respectively.\footnote{Note that $\dim \overline{\tilde{\text{Pr}}\big((\mathcal{X}\times \mathcal{X})\cap H\big)}=\dim \tilde{\text{Pr}}(H_i)=2$ in $\mathbb{G}^5$ and, as $H$ is defined by equations of the forms in \eqref{22061405}, $\overline{\tilde{\text{Pr}}\big((\mathcal{X}\times \mathcal{X})\cap H\big)}$ is isomorphic to $\mathcal{X}$.} By abuse of notation, if the image of $P_i$ under \eqref{22021601} is still denoted by $P_i$, as 
\begin{equation*}
\tilde{\text{Pr}}(P_i)\in \overline{\tilde{\text{Pr}}\big((\mathcal{X}\times \mathcal{X})\cap H\big)}\cap  \tilde{\text{Pr}}(H_i)\quad(i\in \mathcal{I})
\end{equation*}
and the height of $\tilde{\text{Pr}}(P_i)$ is uniformly bounded,\footnote{$M_1, L_1, M_2, L_2$ coordinates of $\tilde{\text{Pr}}(P_i)$ are given in \eqref{22021511}, and thus their heights are uniformly bounded by Theorem \ref{20080404}. Since $\tilde{M'_1}$ coordinate of $\tilde{\text{Pr}}(P_i)$ is determined by the four aforementioned coordinates from \eqref{22061405}, the height of it is uniformly bounded as well. Finally, the height of $\tilde{L'_1}$ coordinate of $\tilde{\text{Pr}}(P_i)$ is uniformly bounded too, as it satisfies $f(\tilde{M'_1}, \tilde{L'_1})=0$. } by Theorem \ref{19090804}, there exists a family of algebraic subgroups $\{\tilde{H_i}\}_{i\in \mathcal{I}}$ of codimension $2$ satisfying 
\begin{itemize}
\item $\tilde{\text{Pr}}(P_i)\in \tilde{H_i}$;
\item $\overline{\tilde{\text{Pr}}\big((\mathcal{X}\times \mathcal{X})\cap H\big)}\cap_{\tilde{\text{Pr}}(P_i)} \tilde{H_i}$ is a torsion anomalous subvariety of $\overline{\tilde{\text{Pr}}\big((\mathcal{X}\times \mathcal{X})\cap H\big)}$ of dimension $1$. 
\end{itemize}
Assuming there is no finite set of algebraic subgroups whose union contains 
\begin{equation}\label{22080905}
\bigcup_{i\in \mathcal{I}}\Big(\overline{\tilde{\text{Pr}}\big((\mathcal{X}\times \mathcal{X})\cap H\big)}\cap_{\tilde{\text{Pr}}(P_i)} \tilde{H_i}\Big)
\end{equation}
and applying Theorem \ref{struc2}, we find $K(\subset \mathbb{G}^5)$ such that  
\begin{itemize}
\item $k_iK\subset \tilde{H_i}$ for some $k_i(\in \mathbb{G}^5)$; 
\item $\overline{\tilde{\text{Pr}}\big((\mathcal{X}\times \mathcal{X})\cap H\big)}\cap_{\tilde{\text{Pr}}(P_i)} \tilde{H_i}$ is a component of $\overline{\tilde{\text{Pr}}\big((\mathcal{X}\times \mathcal{X})\cap H\big)}\cap k_iK$ for each $i\in \mathcal{I}$. 
\end{itemize}
If $\mathcal{U}_K$ is defined as in Corollary \ref{22022603}, then, since $\dim \overline{\mathcal{U}_K}=1$ and $\codim  \tilde{H_i}=2$ for each $i$, we get that $\mathcal{U}_K$ is contained in some algebraic subgroup by Corollary \ref{24100901}. Consequently, \eqref{22080905} is contained in the same algebraic subgroup as well. But this contradicts our initial assumption. 

\item Second we assume $H_i$ is contained in the second type of equations 
\begin{equation}\label{22071307}
M_1^{p_{1i}}L_1^{q_{1i}}=(M'_1)^{p'_{1i}}(L'_1)^{q'_{1i}}=1
\end{equation}
in \eqref{21080301}. By Theorem \ref{struc}, there exists an algebraic subgroup $K$ where 
\begin{equation}\label{210800202}
\overline{\text{Pr}\big((\mathcal{X}\times \mathcal{X})\cap H\big)}\cap_{\text{Pr}(P_i)} H_i\subset \overline{\text{Pr}\big((\mathcal{X}\times \mathcal{X})\cap H\big)}\cap k_iK
\end{equation}
for some $k_i\in \mathbb{G}^7$. To simplify notation, denote $\overline{\text{Pr}\big((\mathcal{X}\times \mathcal{X})\cap H\big)}\cap_{\text{Pr}(P_i)} H_i$ by $\mathcal{Y}_i$ and a component of $\overline{\text{Pr}\big((\mathcal{X}\times \mathcal{X})\cap H\big)}\cap k_iK$ containing $\mathcal{Y}_i$ by $\mathcal{Z}_i$.
\begin{enumerate}
\item First, suppose $\dim K=5$ and $\dim \mathcal{Z}_i=2$. Applying Gauss elimination if necessary, let $K$ be defined by 
\begin{equation}\label{21080306}
\begin{gathered}
M_1^{a_{1}}L_1^{b_{1}}M_2^{c_{1}}L_2^{d_{1}} (M_1')^{a'_{1}}(L'_1)^{b'_{1}}(L'_2)^{d'_{1}}=1,\\
M_1^{a_{2}}L_1^{b_{2}}M_2^{c_{2}}L_2^{d_{2}} (M_1')^{a'_{2}}(L'_1)^{b'_{2}}=1.
\end{gathered}
\end{equation}
By \eqref{210800202} and the fact that $H_i$ is defined by \eqref{22071307}, it follows that $c_2=d_2=0$ in \eqref{21080306}.\footnote{Otherwise, $\mathcal{Y}_i\cap k_iK$ becomes $0$-dim, contradicting the fact that it is a $1$-dim anomalous subvariety of $\overline{\text{Pr}\big((\mathcal{X}\times \mathcal{X})\cap H\big)}$.} Since $\overline{\text{Pr}\big((\mathcal{X}\times \mathcal{X})\cap H\big)}$ is foliated by anomalous subvarieties contained in translations of $K$, there exists an analytic function $\mathfrak{h}(\mathfrak{t})$ such that 
\begin{equation}\label{21080309}
a_{1}u_1+b_{1}v_1+c_{1}u_2+d_{1}v_2+a'_{1}u'_1+b'_{1}v'_1+d'_{1}v'_2=\mathfrak{h}(a_{2}u_1+b_{2}v_1+a'_{2}u'_1+b'_{2}v'_1).
\end{equation}
\begin{enumerate}
\item If $d'_1=0$, then $c_1=d_1=0$ and \eqref{21080309} is refined to 
\begin{equation}\label{21082910}
a_{1}u_1+b_{1}v_1+a'_{1}u'_1+b'_{1}v'_1=\mathfrak{h}(a_{2}u_1+b_{2}v_1+a'_{2}u'_1+b'_{2}v'_1).
\end{equation} 
\begin{enumerate}
\item If $(a_1, b_1, a_2, b_2)\neq (0, 0, 0, 0)$, we put $u'_1=0$ in \eqref{21082910} and reduce \eqref{21082910} further to 
\begin{equation*}
a_{1}u_1+b_{1}v_1=\mathfrak{h}(a_{2}u_1+b_{2}v_1).
\end{equation*} 
Applying Gauss elimination, if we assume $b_2=0$, then either $\mathfrak{h}$ is linear (for $b_1=0$) or $v_1$ depends only on $u_1$ (for $b_1\neq 0$). But both again contradict our initial assumptions.   
\item If $(a_1, b_1, a_2, b_2)=(0, 0, 0, 0)$, then \eqref{21082910} becomes
\begin{equation*}
a'_{1}u'_1+b'_{1}v'_1=\mathfrak{h}(a'_{2}u'_1+b'_{2}v'_1).
\end{equation*} 
Since $\mathfrak{h}$ is not linear, both $u'_1$ and $v'_1$ are constants over $\mathcal{Z}_i$. Since two cusps of $\mathcal{M}$ are not SGI each other, this implies $M'_1, L'_1, M'_2, L'_2$ are all constants over 
\begin{equation}\label{22073001}
\big((\mathcal{X}\times \mathcal{X})\cap_{(1,\dots,1)} H\big)\cap  k_iK. 
\end{equation} 
Since $H$ is defined by an equation of the form in \eqref{21082904}, if $M'_1, L'_1, M'_2, L'_2$ are constants over \eqref{22073001}, it contradicts the fact that \eqref{22073001} is a $2$-dim anomalous subvariety of $\mathcal{X}\times \mathcal{X}$. 
\end{enumerate}

\item Now we suppose $d'_1\neq 0$. By \eqref{21082904}, $\log H$ is defined by 
\begin{equation}\label{21080307}
u'_2=-\frac{a}{c'}u_1-\frac{b}{c'}v_1-\frac{c}{c'}u_2-\frac{d}{c'}v_2-\frac{a'}{c'}u'_1,
\end{equation}
and hence $\log \big((\mathcal{X}\times \mathcal{X})\cap H\big)$ is considered as an analytic set parametrized by $u_1, u_2$ and $u'_1$. If $v'_2$, as a function of $u'_1$ and $u'_2$ over $\log (\mathcal{X}\times \mathcal{X})$, is given by  
\begin{equation}\label{21080308}
\tau_2u'_2+\sum_{\substack{\alpha:odd,\\ \beta:even}}m_{\alpha,\beta}(u'_2)^{\alpha}(u'_1)^{\beta}+\cdots, 
\end{equation}
plugging \eqref{21080307}-\eqref{21080308} into \eqref{21080309}, we obtain the following identity:
\begin{equation}\label{21080310}
\begin{gathered}
a_{1}u_1+b_{1}v_1+c_{1}u_2+d_{1}v_2+a'_{1}u'_1+b'_{1}v'_1+d'_{1}\bigg(\tau_2\Big(-\frac{a}{c'}u_1-\frac{b}{c'}v_1-\frac{c}{c'}u_2-\frac{d}{c'}v_2-\frac{a'}{c'}u'_1\Big)\\
+\sum_{\substack{\alpha:odd,\\ \beta:even}} m_{\alpha,\beta}\Big(-\frac{a}{c'}u_1-\frac{b}{c'}v_1-\frac{c}{c'}u_2-\frac{d}{c'}v_2-\frac{a'}{c'}u'_1\Big)^{\alpha}(u'_1)^{\beta}+\cdots\bigg)\\
=\mathfrak{h}(a_{2}u_1+b_{2}v_1+a'_{2}u'_1+b'_{2}v'_1).
\end{gathered}
\end{equation}
Since $\mathfrak{h}(a_{2}u_1+b_{2}v_1+a'_{2}u'_1+b'_{2}v'_1)$ in \eqref{21080310} does not contain a term of the form $u_2^{\alpha}u_1^{\beta}$ with $\alpha$ odd and $\beta$ even, the coefficients of $u_2$ and $v_2$ in the left side of \eqref{21080310} are zero, that is, 
\begin{equation}\label{22072801}
d_{1}-d'_1\tau_2\frac{d}{c'}=0 \quad \text{and}\quad c_1-d'_1\tau_2\frac{c}{c'}=0.
\end{equation}
As $d'_1\neq 0$ and $\tau_2\in \mathbb{C}\backslash\mathbb{R}$, \eqref{22072801} further implies $c_1=d_1=c=d=0$ and so \eqref{21082904} is reduced to 
\begin{equation*}
M_1^{a}L_1^b(M'_1)^{a'}(M'_2)^{c'}=1. 
\end{equation*}
However, as $P_i\in H$ for every $i\in \mathcal{I}$, it contradicts the fact that $t_{1i}, t'_{1i}, t'_{2i}$ are multiplicatively independent. 
\end{enumerate}

\item Second, we assume 
\begin{equation*}
\dim \mathcal{Y}_i=\dim \mathcal{Z}_i=1
\end{equation*}
and, by Theorem \ref{struc2}, further suppose $k_iK\subset H_i$ for every $i\in \mathcal{I}$. Recall that $H_i$ is defined by \eqref{22071307} and thus 
\begin{equation*}
K\subset (M_1=L_1=M'_1=L'_1=1).
\end{equation*}
In other words, $\mathcal{Y}_i$ is contained in a translation of $M_1=L_1=M'_1=L'_1=1$. Since, by the initial assumption, two cusps of $\mathcal{M}$ are not SGI each other, if $M_1, L_1, M'_1, L'_1$ are all constants, then $M_2, L_2, M'_2, L'_2$ must also constants. But this contradicts the fact that $\mathcal{Y}_i$ is a $1$-dim anomalous subvariety of $\overline{\text{Pr}\big((\mathcal{X}\times \mathcal{X})\cap H\big)}$.  
\end{enumerate}
\item Lastly, if $H_i$ is contained in the third type of equations in \eqref{21080301}, the result follows by an analogous method given in the previous case.  
\end{enumerate}

In conclusion, there exists a finite set of algebraic subgroups such that each  \eqref{22012901} is contained in an element of the set. 
\end{proof}

In the theorem below, we further show, for each algebraic subgroup obtained in Theorem \ref{21082301}, its intersection with $\mathcal{X}\times \mathcal{X}$ contains a $3$-dim  anomalous subvariety of $\mathcal{X}\times \mathcal{X}$. 

\begin{theorem}\label{22071211}
We adapt the same assumptions and notation as given in Theorem \ref{21082301} and let $H^{(2)}$ be an algebraic subgroup obtained in Theorem \ref{21082301} (without loss of generality) containing all $P_i$. Then $(\mathcal{X}\times \mathcal{X})\cap_{(1, \dots, 1)}H^{(2)}$ is an anomalous subvariety of $\mathcal{X}\times\mathcal{X}$. 
\end{theorem}
\begin{proof}
Suppose $H^{(2)}$ is defined by\footnote{Note that the determinants of   
\begin{equation*}
\left(\begin{array}{cc}
a_1 & b_1\\
a_2 & b_2
\end{array}\right), \quad 
\left(\begin{array}{cc}
c_1 & d_1\\
c_2 & d_2
\end{array}\right), \quad 
\left(\begin{array}{cc}
a'_1 & b'_1\\
a'_2 & b'_2
\end{array}\right), \quad 
\left(\begin{array}{cc}
c'_1 & d'_1\\
c'_2 & d'_2
\end{array}\right)
\end{equation*}
are all nonzero by the assumption $\clubsuit$.}
\begin{equation*}
M_1^{a_j}L_1^{b_j}M_2^{c_j}L_2^{d_j}(M'_1)^{a'_j}(L'_1)^{b'_j}(M'_2)^{c'_j}(L'_2)^{d'_j}=1\quad (j=1,2)
\end{equation*}
and, for brevity, denote $(\mathcal{X}\times \mathcal{X})\cap_{(1, \dots, 1)} H^{(2)}$ still by $(\mathcal{X}\times \mathcal{X})\cap H^{(2)}$.  

On the contrary, assume $(\mathcal{X}\times \mathcal{X})\cap H^{(2)}$ is not an anomalous subvariety of $\mathcal{X}\times \mathcal{X}$ (i.e. $\dim (\mathcal{X}\times \mathcal{X})\cap H^{(2)}=2$) and let $\overline{\text{Pr}\big((\mathcal{X}\times \mathcal{X})\cap H^{(2)}\big)}$ be the algebraic closure of the image of $(\mathcal{X}\times \mathcal{X})\cap H^{(2)}$ under 
\begin{equation*}
\text{Pr}:\;(M_1, L_1, M_2, L_2, M'_1, L'_1, M'_2, L'_2)\longrightarrow (M_1, L_1, M_2, L_2, M'_1, L'_1).
\end{equation*}
Since  
\begin{equation*}
\text{Pr}(P_i)\in \overline{\text{Pr}\big((\mathcal{X}\times \mathcal{X})\cap H^{(2)}\big)}\cap (M_1^{p_{1i}}L_1^{q_{1i}}=M_2^{p_{2i}}L_2^{q_{2i}}=(M'_1)^{p'_{1i}}(L'_1)^{q'_{1i}}=1)
\end{equation*}
and the height of $\text{Pr}(P_i)$ is uniformly bounded, by Theorem \ref{19090804}, there is an algebraic subgroup $H_i$ of dimension $2$ such that
\begin{itemize}
\item $\text{Pr}(P_i)\in \overline{\text{Pr}\big((\mathcal{X}\times \mathcal{X})\cap H^{(2)}\big)}\cap H_i$;
\item $\mathcal{Y}_i:=\overline{\text{Pr}\big((\mathcal{X}\times \mathcal{X})\cap H^{(2)}\big)}\cap_{\text{Pr}(P_i)} H_i$ is a torsion anomalous subvariety of $\overline{\text{Pr}\big((\mathcal{X}\times \mathcal{X})\cap H^{(2)}\big)}$. 
\end{itemize}
By Theorem \ref{struc2}, there exists, in addition, an algebraic torus $K\subset \mathbb{G}^6$ satisfying 
\begin{itemize}
\item $k_iK\subset H_i$ for some $k_i\in \mathbb{G}^6$; 
\item $\mathcal{Y}_i$ is a component of $\overline{\text{Pr}\big((\mathcal{X}\times \mathcal{X})\cap H^{(2)}\big)}\cap k_iK$. 
\end{itemize}
Note that $H_i$ is defined by one of the following by the assumption $\clubsuit$ (and Remark \ref{22081302}):
\begin{equation}\label{22061505}
M_1^{p_{1i}}L_1^{q_{1i}}=(M'_1)^{p'_{1i}}(L'_1)^{q'_{1i}}=1,\;\;
M_2^{p_{2i}}L_2^{q_{2i}}=(M'_1)^{p'_{1i}}(L'_1)^{q'_{1i}}=1\;\;\text{  or  }\;\;
M_1^{p_{1i}}L_1^{q_{1i}}=M_2^{p_{2i}}L_2^{q_{2i}}=1. 
\end{equation}
Hence there are three cases to consider.  

\begin{enumerate}
\item If $H_i$ is defined by the first type of equations in \eqref{22061505}, then we moreover have
\begin{equation*}
K\subset (M_1=L_1=M'_1=L'_1=1)
\end{equation*}
by the assumption that $k_iK\subset H_i$. Since two cusps of $\mathcal{M}$ are not SGI, if both $M_1,L_1$ (resp. $M'_1,L'_1$) are constants, then $M_2, L_2$ (resp. $M'_2, L'_2$) are constants as well. However this again contradicts the fact that $\mathcal{Y}_i$ is a $1$-dim anomalous subvariety of $\overline{\text{Pr}\big((\mathcal{X}\times \mathcal{X})\cap H^{(2)}\big)}$.   
\item The second case can be treated analogously to the previous one. 

\item Lastly, if $H_i$ is defined by the third type of equations in \eqref{22061505}, by a similar reasoning above to the first case, it further follows 
\begin{equation*}
K\subset (M_1=L_1=M_2=L_2=1). 
\end{equation*}
Therefore $\mathcal{Y}_i$ is contained in 
\begin{equation}\label{22071214}
M_1=t_{1i}^{-q_{1i}}, \quad L_1=t_{1i}^{p_{1i}}, \quad M_2=t_{2i}^{-q_{2i}}, \quad L_2=t_{2i}^{p_{2i}},
\end{equation}
and $\eqref{22071214}\cap H^{(2)}$ is 
\begin{equation}\label{22071215}
(M'_1)^{a'_j}(L'_1)^{b'_j}(M'_2)^{c'_j}(L'_2)^{d'_j}=(t_{1i}^{q_{1i}})^{a_j}(t_{1i}^{-p_{1i}})^{b_j}(t_{2i}^{q_{2i}})^{c_j}(t_{2i}^{-p_{2i}})^{d_j}\quad (j=1,2).
\end{equation}
Since $\mathcal{Y}_i$ is a $1$-dim anomalous subvariety of $
\overline{\text{Pr}\big((\mathcal{X}\times \mathcal{X})\cap H^{(2)}\big)}$, equivalently, a component of 
\begin{equation*}
(\mathcal{X}\times \mathcal{X})\cap H^{(2)}\cap \eqref{22071214}  =(\mathcal{X}\times \mathcal{X})\cap \eqref{22071214}\cap \eqref{22071215} 
\end{equation*}
containing $P_i$ is also a $1$-dim anomalous subvariety of $\mathcal{X}\times \mathcal{X}$. Considering \eqref{22071215} as an algebraic coset in $\mathbb{G}^4(:=(M'_1, L'_1, M'_2, L'_2))$, if $\mathcal{X}\cap \eqref{22071215}$\footnote{Here $\mathcal{X}$ is the second copy of $\mathcal{X}\times \mathcal{X}$} is regarded as an anomalous subvariety of $\mathcal{X}$ for each $i\in \mathcal{I}$, it implies $\mathcal{X}^{oa}=\emptyset$. Hence, using the preferred coordinates and equations introduced in Definition \ref{22071301},\footnote{Since two cusps of $\mathcal{M}$ are not SGI each other, $\mathcal{X}$ falls into the second case of Proposition \ref{21073101}.} \eqref{22071215} is equivalent to a translation of either $\tilde{M'}_1=\tilde{L'}_1=1$ or $\tilde{M'}_2=\tilde{L'}_2=1$. Without loss of generality, if we apply the first one, by symmetry, $H^{(2)}$ is defined by equations of either one of the following forms: 
\begin{equation}\label{22071303}
\tilde{M_1}^{\tilde{a}_j}\tilde{L_1}^{\tilde{b}_j}(\tilde{M'_1})^{\tilde{a}'_j}(\tilde{L'_1})^{\tilde{b}'_j}=1,\quad (j=1,2)
\end{equation}
or 
\begin{equation*}
\tilde{M_2}^{\tilde{c}_j}\tilde{L_2}^{\tilde{d}_j}(\tilde{M'_1})^{\tilde{a}'_j}(\tilde{L'_1})^{\tilde{b}'_j}=1,\quad (j=1,2).
\end{equation*}
Without loss of generality, let us once again assume the first case. Since $(\mathcal{X}\times \mathcal{X})\cap_{(1, \dots, 1)} H^{(2)}$ is a $2$-dim subvariety of $\mathcal{X}\times \mathcal{X}$, taking \eqref{22071303} as an algebraic subgroup in $\mathbb{G}^4(:=(\tilde{M_1},\tilde{L_1}, \tilde{M'}_1, \tilde{L'}_1))$, it means the component of the intersection between \eqref{22071303} and 
\begin{equation*}
f(\tilde{M_1}, \tilde{L_1})=f(\tilde{M'_1}, \tilde{L'_1})=0\footnote{Here, $f$ is the polynomial defining $\mathcal{X}$ as given in Definition \ref{22071301}.} 
\end{equation*}
containing the identity is simply a point. That is, 
\begin{equation*}
(\tilde{M_1}, \tilde{L_1}, \tilde{M'_1}, \tilde{L'_1})=(1, 1, 1, 1)
\end{equation*}
and so
\begin{equation*}
P_i\in (\mathcal{X}\times \mathcal{X})\cap_{(1, \dots, 1)} H^{(2)}=(\mathcal{X}\times \mathcal{X})\cap_{(1, \dots, 1)} (\tilde{M_1}=\tilde{L_1}=\tilde{M'_1}=\tilde{L'_1}=1). 
\end{equation*}
As a result, $t_{1i}, t_{2i}$ ($t'_{1i}, t'_{2i}$ as well) are multiplicatively dependent. However this contradicts the hypothesis $\clubsuit$. 
\end{enumerate}
In conclusion, $(\mathcal{X}\times \mathcal{X})\cap_{(1, \dots, 1)} H^{(2)}$ is a $3$-dim anomalous subvariety of $\mathcal{X}\times \mathcal{X}$. 
\end{proof}

\subsubsection{Proof of Theorem \ref{21091101}}
In this subsection, we establish Theorem \ref{21091101}. We first provide an important lemma required for the proof of the theorem. The lemma can be seen as a quantitative version of Theorem \ref{21082301} and, in which we present the explicit relationship between two Dehn filling coefficients satisfying \eqref{22021902}. Note that the key ingredient of the lemma's proof was already presented in the proof of Lemma \ref{22050501}. 

\begin{lemma}\label{21082802}
Let $\mathcal{M}_{p_1/q_1, p_2/q_2}$ and $\mathcal{M}_{p'_1/q'_1, p'_2/q'_2}$ be two Dehn fillings of $\mathcal{M}$ having the same pseudo complex volume with $|p_h|+|q_h|, |p'_h|+|q'_h|$ ($h=1,2$) sufficiently large and $P$ be a Dehn filling point on $\mathcal{X}\times \mathcal{X}$ associated to the pair. We further suppose 
\begin{enumerate}
\item ($\spadesuit$) the elements in any proper subset of $\{t_1, t_2, t'_1, t'_2\}$ are multiplicatively independent; 
\item $P$ is contained in an algebraic subgroup $H$ defined by\footnote{Note that $\det A_h, \det A'_h\neq 0$ by the assumption $\spadesuit$ where $A_h,A'_h$ $(h=1,2)$ are as given in \eqref{22080907}.}  
\begin{equation}\label{22050811}
M_1^{a_j}L_1^{b_j}M_2^{c_j}L_2^{d_j}(M'_1)^{a'_j}(L'_1)^{b'_j}(M'_2)^{c'_j}(L'_2)^{d'_j}=1\quad (j=1,2).
\end{equation}
 \end{enumerate}
Let $r_{h}, s_{h}, r'_{h}, s'_{h}$ ($h=1,2$) be integers such that 
\begin{equation}\label{22032801}
\det \left(\begin{array}{cc}
p_{1} & q_{1}\\
r_{1} & s_{1}
\end{array}\right)
=\det \left(\begin{array}{cc}
p_{2} & q_{2}\\
r_{2} & s_{2}
\end{array}\right)
=\det \left(\begin{array}{cc}
p'_{1} & q'_{1}\\
r'_{1} & s'_{1}
\end{array}\right)
=\det \left(\begin{array}{cc}
p'_{2} & q'_{2}\\
r'_{2} & s'_{2}
\end{array}\right)
=1.
\end{equation}
Then there exist $n_{2}, n'_1, n'_2\in \mathbb{Q}$  satisfying
\begin{equation*}
\begin{gathered}
\left(\begin{array}{cc}
p_{1} & q_{1}\\
r_{1} & s_{1}
\end{array}\right)
A_1^{-1}A_2=\left(\begin{array}{cc}
\frac{\det A_2}{\det A_1} & 0\\
n_{2} & 1
\end{array}\right)
\left(\begin{array}{cc}
p_{2} & q_{2}\\
r_{2} & s_{2}
\end{array}\right),\quad \left(\begin{array}{cc}
p_{1} & q_{1}\\
r_{1} & s_{1}
\end{array}\right)
A_1^{-1}A'_1=\left(\begin{array}{cc}
\frac{\det A'_1}{\det A_1} & 0\\
n'_{1} & 1
\end{array}\right)
\left(\begin{array}{cc}
p'_{1} & q'_{1}\\
r'_{1} & s'_{1}
\end{array}\right),\\
\text{and}\quad \left(\begin{array}{cc}
p_{1} & q_{1}\\
r_{1} & s_{1}
\end{array}\right)
A_1^{-1}A'_2=\left(\begin{array}{cc}
\frac{\det A'_2}{\det A_1} & 0\\
n'_{2} & 1
\end{array}\right)
\left(\begin{array}{cc}
p'_{2} & q'_{2}\\
r'_{2} & s'_{2}
\end{array}\right)
\end{gathered}
\end{equation*}
where
\begin{equation}\label{22080907}
A_1= \left(\begin{array}{cc}
a_1 & b_1\\
a_2 & b_2
\end{array}\right),\quad
A_2= \left(\begin{array}{cc}
c_1 & d_1\\
c_2 & d_2
\end{array}\right),\quad 
A'_1= \left(\begin{array}{cc}
a'_1 & b'_1\\
a'_2 & b'_2
\end{array}\right)\quad\text{and}\quad  
A'_2= \left(\begin{array}{cc}
c'_1 & d'_1\\
c'_2 & d'_2
\end{array}\right). 
\end{equation}
\end{lemma}
\begin{proof}
Moving to an analytic setting, $\log\big((\mathcal{X}\times \mathcal{X})\cap H\big)$ is defined by
\begin{equation*}
A_1\left(\begin{array}{c}
u_1\\
v_1
\end{array}
\right)+
A_2
\left(\begin{array}{c}
u_2\\
v_2
\end{array}
\right)=A'_1
\left(\begin{array}{c}
u'_1\\
v'_1
\end{array}
\right)
+A'_2
\left(\begin{array}{c}
u'_2\\
v'_2
\end{array}
\right).
\end{equation*}
Let 
\begin{equation}\label{22071111}
(u_1, v_1, \dots , u'_2, v'_2)=(\xi_{m_{1}}, \xi_{l_{1}}, \dots, \xi'_{m_{2}}, \xi'_{l_{2}}).
\end{equation} 
be the Dehn filling point on $\log \big((\mathcal{X}\times \mathcal{X})\cap H\big)$ corresponding to $P$. That is, the coordinates in \eqref{22071111} satisfy
\begin{equation*}
p_h\xi_{m_{h}}+q_h\xi_{l_{h}}=p'_h\xi'_{m_{h}}+q'_h \xi'_{l_{h}}=-2\pi \sqrt{-1}\quad (h=1,2)
\end{equation*}
and
\begin{equation}\label{21082604}
\begin{gathered}
A_1\left(\begin{array}{c}
\xi_{m_{1}}\\
\xi_{l_{1}}
\end{array}
\right)+
A_2
\left(\begin{array}{c}
\xi_{m_{2}}\\
\xi_{l_{2}}
\end{array}
\right)=
A'_1
\left(\begin{array}{c}
\xi'_{m_{1}}\\
\xi'_{l_{1}}
\end{array}
\right)
+A'_2
\left(\begin{array}{c}
\xi'_{m_{2}}\\
\xi'_{l_{2}}
\end{array}
\right)\\
\Longrightarrow 
\left(\begin{array}{c}
\xi_{m_{1}}\\
\xi_{l_{1}}
\end{array}
\right)=-
A_1^{-1}A_2
\left(\begin{array}{c}
\xi_{m_{2}}\\
\xi_{l_{2}}
\end{array}
\right)
+A_1^{-1}A'_1
\left(\begin{array}{c}
\xi'_{m_{1}}\\
\xi'_{l_{1}}
\end{array}
\right)
+A_1^{-1}A'_2
\left(\begin{array}{c}
\xi'_{m_{2}}\\
\xi'_{l_{2}}
\end{array}
\right). 
\end{gathered}
\end{equation}
To simplify notation, similar to the proof of Lemma \ref{22050501}, we assume, by changing basis of the first cusp if necessary, that
$\left(\begin{array}{cc}
p_1 & q_1\\
r_1 & s_1
\end{array}
\right)=\left(\begin{array}{cc}
1 & 0\\
0 & 1
\end{array}
\right)$. Then $p_{1}\xi_{m_{1}}+q_{1}\xi_{l_{1}}=\xi_{m_{1}}=-2\pi \sqrt{-1}$, and so
\begin{equation}\label{22071113}
\begin{gathered}
-2\pi \sqrt{-1}=\left(\begin{array}{cc}
1 & 0
\end{array}
\right)\bigg(
-A_1^{-1}A_2
\left(\begin{array}{c}
\xi_{m_{2}}\\
\xi_{l_{2}}
\end{array}
\right)
+A_1^{-1}A'_1
\left(\begin{array}{c}
\xi'_{m_{1}}\\
\xi'_{l_{1}}
\end{array}
\right)
+A_1^{-1}A'_2
\left(\begin{array}{c}
\xi'_{m_{2}}\\
\xi'_{l_{2}}
\end{array}
\right)\bigg)
\end{gathered}
\end{equation}
by \eqref{21082604}. Since the exponentials of 
\begin{equation}\label{21082601}
\begin{gathered}
\left(\begin{array}{cc}
1 & 0
\end{array}
\right)A_1^{-1}A_2
\left(\begin{array}{c}
\xi_{m_{2}}\\
\xi_{l_{2}}
\end{array}
\right), \;\; 
\left(\begin{array}{cc}
1 & 0
\end{array}
\right)A_1^{-1}A'_1
\left(\begin{array}{c}
\xi'_{m_{1}}\\
\xi'_{l_{1}}
\end{array}
\right)\;\;\text{and}\;\;
\left(\begin{array}{cc}
1 & 0
\end{array}
\right)A_1^{-1}A'_2
\left(\begin{array}{c}
\xi'_{m_{2i}}\\
\xi'_{l_{2i}}
\end{array}
\right)
\end{gathered}
\end{equation}
are powers of $t_2, t'_1$ and $t'_2$ respectively, and because $t_{2}, t'_{1}, t'_{2}$ are multiplicatively independent by $\spadesuit$, we deduce that the exponential of every element in \eqref{21082601} is a root of unity by \eqref{22071113}. That is, there are some $k_2, k'_1, k'_2\in \mathbb{Q}$ satisfying 
\begin{equation}\label{21082711}
\begin{gathered}
\left(\begin{array}{cc}
1 & 0
\end{array}
\right)
A_1^{-1}A_2
=k_2 
\left(\begin{array}{cc}
p_2 & q_2\\
\end{array}\right),\;\; 
\left(\begin{array}{cc}
1 & 0
\end{array}
\right)A_1^{-1}A'_1
=k'_1 \left(\begin{array}{cc}
p'_1 & q'_1\\
\end{array}\right),\;\;\left(\begin{array}{cc}
1 & 0
\end{array}
\right)A_1^{-1}A'_2
=k'_2 \left(\begin{array}{cc}
p'_2 & q'_2\\
\end{array}\right).
\end{gathered}
\end{equation}
On the other hand, by multiplying $\left(\begin{array}{cc}
0 & 1
\end{array}
\right)$ to \eqref{21082604}, it follows that 
\begin{equation}\label{21082701}
\begin{gathered}
\xi_{l_1}=
\left(\begin{array}{cc}
0 & 1
\end{array}
\right)\bigg(-A_1^{-1}A_2
\left(\begin{array}{c}
\xi_{m_2}\\
\xi_{l_2}
\end{array}
\right)+A_1^{-1}A'_1
\left(\begin{array}{c}
\xi'_{m_1}\\
\xi'_{l_1}
\end{array}
\right)
+A_1^{-1}A'_2
\left(\begin{array}{c}
\xi'_{m_2}\\
\xi'_{l_2}
\end{array}
\right)\bigg).
\end{gathered}
\end{equation}
Note that $t_1=e^{r_{1}\xi_{m_1}+s_{1}\xi_{l_1}}=e^{\xi_{l_1}}$ and the exponentials of 
\begin{equation*}
\begin{gathered}
\left(\begin{array}{cc}
0 & 1
\end{array}
\right)
A_1^{-1}A_2
\left(\begin{array}{c}
\xi_{m_2}\\
\xi_{l_2}
\end{array}
\right), \;\;
\left(\begin{array}{cc}
0 & 1
\end{array}
\right)
A_1^{-1}A'_1
\left(\begin{array}{c}
\xi'_{m_1}\\
\xi'_{l_1}
\end{array}
\right)\;\;\text{and}\;\;
\left(\begin{array}{cc}
0 & 1
\end{array}
\right)
A_1^{-1}A'_2
\left(\begin{array}{c}
\xi'_{m_2}\\
\xi'_{l_2}
\end{array}
\right)
\end{gathered}
\end{equation*}
are some powers of $t_2, t'_1$ and $t'_2$ respectively. Given that $t_1t_2=t'_1t'_2$, and combining this with the assumption $\spadesuit$, it is concluded that the equation obtained by exponentiating both sides of \eqref{21082701} is indeed equivalent to $t_1t_2=t'_1t'_2$. That is,   
\begin{equation}\label{22012201}
\begin{gathered}
\left(\begin{array}{cc}
0 & 1
\end{array}
\right)
A_1^{-1}A_2
=\left(\begin{array}{cc}
r_{2} & s_{2}
\end{array}
\right)+n_2\left(\begin{array}{cc}
p_{2} & q_{2}
\end{array}
\right),  \quad \left(\begin{array}{cc}
0 & 1
\end{array}
\right)
A_1^{-1}A'_1
=\left(\begin{array}{cc}
r'_{1} & s'_{1}
\end{array}
\right)+n'_1\left(\begin{array}{cc}
p'_{1} & q'_{1}
\end{array}
\right),\\
\text{and}\quad \left(\begin{array}{cc}
0 & 1
\end{array}
\right)
A_1^{-1}A'_2
=\left(\begin{array}{cc}
r'_{2} & s'_{2}
\end{array}
\right)+n'_2\left(\begin{array}{cc}
p'_{2} & q'_{2}
\end{array}
\right)
\end{gathered}
\end{equation}
for some $n_2, n'_1, n'_2\in \mathbb{Q}$. Merging \eqref{21082711} with  \eqref{22012201}, we finally get
\begin{equation*}
\begin{gathered}
A_1^{-1}A_2
=\left(\begin{array}{cc}
k_2 & 0\\
n_{2} & 1
\end{array}\right)
\left(\begin{array}{cc}
p_{2} & q_{2}\\
r_{2} & s_{2}
\end{array}\right),\; A_1^{-1}A'_1
=\left(\begin{array}{cc}
k'_1 & 0\\
n'_{1} & 1
\end{array}\right)
\left(\begin{array}{cc}
p'_{1} & q'_{1}\\
r'_{1} & s'_{1}
\end{array}\right),\; 
A_1^{-1}A'_2
=\left(\begin{array}{cc}
k'_2 & 0\\
n'_{2} & 1
\end{array}\right)
\left(\begin{array}{cc}
p'_{2} & q'_{2}\\
r'_{2} & s'_{2}
\end{array}\right), 
\end{gathered}
\end{equation*}
with $k_2=\dfrac{\det A_2}{\det A_1},  
k'_1=\dfrac{\det A'_1}{\det A_1},  
k'_2=\dfrac{\det A'_2}{\det A_1}$ obtained from \eqref{22032801}. 
\end{proof}

\begin{remark}\label{22061607}
\normalfont Note that, for $H$ defined as in \eqref{22050811}, if $(\mathcal{X}\times \mathcal{X})\cap_{(1, \dots, 1)} H$ is an anomalous subvariety of $\mathcal{X}\times \mathcal{X}$ (of dimension $3$), we have $\mathcal{X}^{oa}=\emptyset$. In fact, for any $\zeta'_1, \zeta'_2\in \mathbb{G}$, $\big((\mathcal{X}\times \mathcal{X})\cap_{(1, \dots, 1)} H\big)\cap (M'_1=\zeta'_1, M'_2=\zeta'_2)$ is a $1$-dim anomalous subvariety of $\mathcal{X}\times \mathcal{X}$ and, further projecting under
\begin{equation*}
\text{Pr}\;:\;\mathbb{G}^8(:=(M_1, L_1, \dots, M'_2, L'_2))\longrightarrow \mathbb{G}^4(:=(M_1, L_1, M_2, L_2)), 
\end{equation*} 
$\text{Pr}\big(((\mathcal{X}\times \mathcal{X})\cap_{(1, \dots, 1)} H)\cap (M'_1=\zeta'_1, M'_2=\zeta'_2)\big)$ is regarded as an anomalous subvariety of the first copy of $\mathcal{X}\times \mathcal{X}$. Consequently, we have
\begin{equation}\label{22032803}
\det A_1=\det A_2, \quad \det A'_1=\det A'_2
\end{equation}
in Lemma \ref{21082802} by Proposition \ref{21073101}. 
\end{remark}

Now we are ready to prove Theorem \ref{21091101}. We verify it by deriving a contradiction. If there are two Dehn fillings $\mathcal{M}_{p_1/q_1, p_2/q_2}$ and $\mathcal{M}_ {p'_1/q'_1, p'_2/q'_2}$ satisfying the assumptions in the statements of Lemma \ref{21082802}, it will be shown that the core holonomies $t_1, t_2$ (resp. $t'_1, t'_2$) of $\mathcal{M}_{p_1/q_1, p_2/q_2}$ (resp. $\mathcal{M}_{p'_1/q'_1, p'_2/q'_2}$) must be in fact multiplicatively dependent (thus contradicting the given assumption).  

\begin{proof}[Proof of Theorem \ref{21091101}]
On the contrary, suppose there is no such a proper subset. By Theorem \ref{21082301}, there exists an algebraic subgroup $H$, defined by
\begin{equation*}
\begin{gathered}
M_1^{a_j}L_1^{b_j}M_2^{c_j}L_2^{d_j}(M'_1)^{a'_j}(L'_1)^{b'_j}(M'_2)^{a'_j}(L'_2)^{b'_j}=1\quad (j=1,2), 
\end{gathered}
\end{equation*} 
containing a Dehn filling point $P$ associated to \eqref{22043003}. Further, since $(\mathcal{X}\times \mathcal{X})\cap_{(1, \dots, 1)} H$ is an anomalous subvariety of $\mathcal{X}\times \mathcal{X}$ by Theorem \ref{22071211}, letting $M'_1=L'_1=M'_2=L'_2=1$, $((\mathcal{X}\times \mathcal{X})\cap_{(1, \dots, 1)} H)\cap (M'_1=L'_1=M'_2=L'_2=1)$ is an anomalous subvariety of $(\mathcal{X}\times \mathcal{X})\cap (M'_1=L'_1=M'_2=L'_2=1)$. Moving to the analytic holonomy variety and restricting over $\mathbb{C}^4(:=(u_1, v_1, u_2, v_2))$, 
it equivalently follows that  
\begin{equation}\label{21082801}
\begin{gathered}
A_1\left(\begin{array}{c}
u_1\\
v_1
\end{array}\right)
+A_2\left(\begin{array}{c}
u_2\\
v_2
\end{array}\right)=
\left(\begin{array}{c}
0\\
0
\end{array}\right)
\end{gathered}
\end{equation}
where $A_1:=\left(\begin{array}{cc}
a_1 & b_1\\
a_2 & b_2
\end{array}\right)$ and $A_2:=\left(\begin{array}{cc}
c_1 & d_1\\
c_2 & d_2
\end{array}\right)
$ is an anomalous analytic subset of $\log\mathcal{X}$ in $\mathbb{C}^4$. By Theorem \ref{potential}, if \eqref{21082801} is an anomalous analytic subset of $\log\mathcal{X}$, then 
\begin{equation}\label{22031203}
\begin{gathered}
A_1
\left(\begin{array}{c}
u_1\\
v_1
\end{array}\right)
=A_2
\left(\begin{array}{c}
u_2\\
v_2
\end{array}\right)
\end{gathered}
\end{equation}
also defines an anomalous analytic subset of $\log \mathcal{X}$. 

We claim a Dehn filling point associated to $\mathcal{M}_{p_1/q_1, p_2/q_2}$ is contained in a complex manifold defined by \eqref{22031203}. Let 
\begin{equation}\label{22050401}
(u_1, v_1, u_2, v_2)=(\xi_{m_1}, \xi_{l_1}, \xi_{m_2}, \xi_{l_2})
\end{equation}
be an intersection point between \eqref{22031203} and $p_1u_1+q_1v_1=-2\pi \sqrt{-1}$. That is, 
\begin{equation}\label{22012210}
A_1\left(\begin{array}{c}
\xi_{m_1}\\
\xi_{l_1}
\end{array}\right)
=A_2\left(\begin{array}{c}
\xi_{m_2}\\
\xi_{l_2}
\end{array}\right)\quad \text{and}\quad p_1\xi_{m_1}+q_1\xi_{l_1}=-2\pi \sqrt{-1}.
\end{equation}
By Lemma \ref{21082802} and \eqref{22032803}, there exists $n_2\in \mathbb{Q}$ such that 
\begin{equation*}
\left(\begin{array}{cc}
p_{1} & q_{1}\\
r_{1} & s_{1}
\end{array}\right)
A_1^{-1}A_2
=\left(\begin{array}{cc}
1 & 0\\
n_{2} & 1
\end{array}\right)
\left(\begin{array}{cc}
p_{2} & q_{2}\\
r_{2} & s_{2}
\end{array}\right)
\end{equation*}
and so
\begin{equation}\label{22012211}
\begin{gathered}
\left(\begin{array}{cc}
p_1 & q_1\\
r_1 & s_1
\end{array}\right)
\left(\begin{array}{c}
\xi_{m_1}\\
\xi_{l_1}
\end{array}\right)
=\left(\begin{array}{cc}
1 & 0\\
n_2 & 1
\end{array}\right)
\left(\begin{array}{cc}
p_2 & q_2\\
r_2 & s_2
\end{array}\right)
A_2^{-1}A_1
\left(\begin{array}{c}
\xi_{m_1}\\
\xi_{l_1}
\end{array}\right).
\end{gathered}
\end{equation}
Combining \eqref{22012211} with \eqref{22012210}, we have
\begin{equation}\label{22031209}
\begin{gathered}
\left(\begin{array}{cc}
p_1 & q_1\\
r_1 & s_1
\end{array}\right)
\left(\begin{array}{c}
\xi_{m_1}\\
\xi_{l_1}
\end{array}\right)
=\left(\begin{array}{cc}
1 & 0\\
n_2 & 1
\end{array}\right)
\left(\begin{array}{cc}
p_2 & q_2\\
r_2 & s_2
\end{array}\right)
\left(\begin{array}{c}
\xi_{m_2}\\
\xi_{l_2}
\end{array}\right)\\
\Longrightarrow  p_1\xi_{m_1}+q_1\xi_{l_1}=p_2\xi_{m_2}+q_2\xi_{l_2}=-2\pi \sqrt{-1},
\end{gathered}
\end{equation}
concluding \eqref{22050401} is a Dehn filling point on $\log \mathcal{X}$ associated to $\mathcal{M}_{p_1/q_1, p_2/q_2}$. 

Finally, by the equality of the second rows in the first equation of \eqref{22031209}, we get
\begin{equation*}
e^{r_1\xi_{m_1}+s_1\xi_{l_1}}=e^{n_2(p_2\xi_{m_2} +q_2\xi_{l_2})+(r_2\xi_{m_2}+s_2\xi_{l_2})}=e^{n_2 (-2\pi \sqrt{-1})}e^{r_2\xi_{m_2}+s_2\xi_{l_2}}.
\end{equation*}
As $t_1=e^{r_1\xi_{m_1}+s_1\xi_{l_1}}$ and $t_2=e^{r_2\xi_{m_2}+s_2\xi_{l_2}}$, it follows that $t_1, t_2$ are multiplicatively dependent. But this contradicts the initial assumption, completing the proof.   
\end{proof}

\subsection{Codimension $2$ or $3$} \label{22042905}
Let $\mathcal{M}_{p_1/q_1, p_2/q_2}$ and $\mathcal{M}_{p'_1/q'_1, p'_2/q'_2}$ be Dehn fillings of $\mathcal{M}$ satisfying 
\begin{equation*}
\text{pvol}_{\mathbb{C}}\;\mathcal{M}_{p_{1}/q_{1},p_{2}/q_{2}}=\text{pvol}_{\mathbb{C}}\;\mathcal{M}_{p'_{1}/q'_{1},p'_{2}/q'_{2}}
\end{equation*}
with $|p_h|+|q_h|$ and $|p'_h|+|q'_h|$ sufficiently large $(h=1,2)$. By Theorem \ref{21091101}, there exists a proper subset of $\{t_1, t_2, t'_1, t'_2\}$ whose elements are multiplicatively dependent. If $t_1, t_2$ (resp. $t'_1, t'_2$) are multiplicatively dependent (resp. dependent), the desired result follows from Theorem \ref{20071505}. Hence, without loss of generality, we consider the following two:  
\begin{enumerate}
\item $t'_1, t'_2$ (resp. $t_1, t_2$) are multiplicatively independent (resp. independent);
\item $t'_1, t'_2$ (resp. $t_1, t_2$) are multiplicatively dependent (resp. independent). 
\end{enumerate}
Then Theorem \ref{21082901} is extended as below:

\begin{theorem}\label{20071401}
Suppose there exists a family  
\begin{equation*}
\{\mathcal{M}_{p_{1i}/q_{1i},p_{2i}/q_{2i}},\mathcal{M}_{p'_{1i}/q'_{1i},p'_{2i}/q'_{2i}} \}_{i\in \mathcal{I}}
\end{equation*}
of infinitely many pairs of Dehn fillings of $\mathcal{M}$ such that
\begin{equation}\label{22041501}
\text{pvol}_{\mathbb{C}}\;\mathcal{M}_{p_{1i}/q_{1i},p_{2i}/q_{2i}}=\text{pvol}_{\mathbb{C}}\;\mathcal{M}_{p'_{1i}/q'_{1i},p'_{2i}/q'_{2i}}
\end{equation}
with $|p_{hi}|+|q_{hi}|$ and $|p_{hi}|+|q_{hi}|$ $(h=1,2)$ sufficiently large for each $i\in \mathcal{I}$.
\begin{enumerate}
\item If $t'_{1i}, t'_{2i}$ (resp. $t_{1i}, t_{2i}$) are multiplicatively independent (resp. independent), then there exists a finite set $\mathcal{H}^{(2)}$ of algebraic subgroups of codimension $2$ such that a Dehn filling point $P_i$ associated to \eqref{22041501} is contained in some element of $\mathcal{H}^{(2)}$. In particular, each element of $\mathcal{H}^{(2)}$ is defined by equations of the following forms:
\begin{equation}\label{22041605}
M_1^{a_1}L_1^{b_1}M_2^{c_1}L_2^{d_1}(M'_1)^{a'_1}(L'_1)^{b'_1}=M_1^{a_2}L_1^{b_2}M_2^{c_2}L_2^{d_2}(M'_2)^{c'_2}(L'_2)^{d'_2}=1.
\end{equation}

\item If $t'_{1i}, t'_{2i}$ (resp. $t_{1i}, t_{2i}$) are multiplicatively dependent (resp. independent), then there exists a finite set $\mathcal{H}^{(3)}$ of algebraic subgroups of codimension $3$ such that a Dehn filling point $P_i$ associated to \eqref{22041501} is contained in some element of $\mathcal{H}^{(3)}$. In particular, each element of $\mathcal{H}^{(3)}$ is defined by equations of the following forms:
\begin{equation}\label{22041609}
\begin{gathered}
(M'_1)^{a'_j}(L'_1)^{b'_j}(M'_2)^{c'_j}(L'_2)^{d'_j}=M_1^{a_3}L_1^{b_3}M_2^{c_3}L_2^{d_3}(M'_1)^{a'_3}(L'_1)^{b'_3}=1, \quad (j=1,2).
\end{gathered}
\end{equation}
\end{enumerate}
\end{theorem}
First, if $t'_{1i}, t'_{2i}$ (resp. $t_{1i}, t_{2i}$) are multiplicatively independent (resp. independent), as there exists a proper subset of $\{t_{1i}, t_{2i}, t'_{1i}, t'_{2i}\}$ by Theorem \ref{21091101}, we assume, without loss of generality, $t_{1i}, t_{2i}, t'_{2i}$ are multiplicatively dependent for each $i\in \mathcal{I}$.\footnote{Note that, if $t_{1i}, t_{2i}, t'_{2i}$ are multiplicatively dependent, then $t_{1i}, t_{2i}, t'_{1i}$ are also multiplicatively dependent by the assumption $t_{1i}t_{2i}=t'_{1i}t'_{2i}$.} By Theorem \ref{21082901}, it is further assumed every $P_i$ contained in an algebraic subgroup defined by
\begin{equation}\label{22032507}
M_1^{a_{1}}L_1^{b_{1}}M_2^{c_{1}}L_2^{d_{1}}(M'_2)^{c'_{1}}(L'_2)^{d'_{1}}=1.
\end{equation}

Second, if $t'_{1i}, t'_{2i}$ (resp,  $t_{1i}, t_{2i}$) are multiplicatively dependent   (resp. independent), by Corollary \ref{20080403}, $P_i$ is contained in an algebraic subgroup defined by equations of the following forms:
\begin{equation*}
(M'_1)^{a'_j}(L'_1)^{b'_j}(M'_2)^{c'_j}(L'_2)^{d'_j}=1\quad (j=1,2).
\end{equation*}

Now Theorem \ref{20071401} is reduced to 
\begin{theorem}\label{22041601}
Adapting the same notation and assumptions as in Theorem \ref{20071401}, suppose all $P_i$ are contained in an algebraic subgroup $H$ defined by
\begin{equation}\label{22071401}
M_1^{a_{1}}L_1^{b_{1}}M_2^{c_{1}}L_2^{d_{1}}(M'_1)^{a'_{1}}(L'_1)^{b'_{1}}(M'_2)^{c'_{1}}(L'_2)^{d'_{1}}=1
\end{equation}
where, without loss of generality, $(c'_1, d'_1)\neq (0,0)$. Let $\overline{\text{Pr}\big((\mathcal{X}\times \mathcal{X})\cap H\big)}$ be the algebraic closure of the image of $(\mathcal{X}\times \mathcal{X})\cap_{(1, \dots, 1)} H$ under the following projection:
\begin{equation*}
\text{Pr}:\;(M_1, L_1, M_2, L_2, M'_1, L'_1, M'_2, L'_2)\longrightarrow (M_1, L_1, M_2, L_2, M'_1, L'_1).
\end{equation*}
Then there exists a finite set of algebraic subgroups such that each $\text{Pr}(P_i)$ is contained in some element of the set.
\end{theorem}

The idea of the proof of Theorem \ref{22041601} is similar to that in the proof of either Theorem \ref{21082901} or \ref{21082301}. That is, we analyze the structure of anomalous subvarieties of $\overline{\text{Pr}\big((\mathcal{X}\times \mathcal{X})\cap H\big)}$ of every possible dimension and narrow down the problem to several accessible cases of Zilber-Pink type questions.    

\begin{proof}[Proof of Theorem \ref{22041601}]
Note that 
\begin{equation}\label{22070601}
\text{Pr}(P_i)\in \overline{\text{Pr}\big((\mathcal{X}\times \mathcal{X})\cap H\big)}\cap (M_1^{p_{1i}}L_1^{q_{1i}}=M_2^{p_{2i}}L_2^{q_{2i}}=(M'_1)^{p'_{1i}}(L'_1)^{q'_{1i}}=1)
\end{equation}
and, as $t_{1i}, t_{2i}, t'_{1i}$ are multiplicatively dependent, $\text{Pr}(P_i)$ is additionally contained one more algebraic subgroup. Since the height of each $\text{Pr}(P_i)$ is uniformly bounded by Theorem \ref{20080404}, thanks to Theorem \ref{19090804}, there exists a set $\{H_i\}_{i\in \mathcal{I}}$ of algebraic subgroups such that 
\begin{itemize}
\item $\text{Pr}(P_i)\in \overline{\text{Pr}\big((\mathcal{X}\times \mathcal{X})\cap H\big)}\cap H_i$;
\item $\overline{\text{Pr}\big((\mathcal{X}\times \mathcal{X})\cap H\big)}\cap_{\text{Pr}(P_i)} H_i$ is an anomalous subvariety of $\overline{\text{Pr}\big((\mathcal{X}\times \mathcal{X})\cap H\big)}$. 
\end{itemize}
If $\bigcup_{i\in \mathcal{I}}\overline{\text{Pr}\big((\mathcal{X}\times \mathcal{X})\cap H\big)}\cap_{\text{Pr}(P_i)} H_i$ is contained in a finite number of algebraic subgroups, then we are done. Otherwise, by Theorem \ref{struc}, there exists an algebraic subgroup $K$ such that, for each $i\in \mathcal{I}$, 
\begin{equation*}
\overline{\text{Pr}\big((\mathcal{X}\times \mathcal{X})\cap H\big)}\cap_{\text{Pr}(P_i)} H_i\subset \overline{\text{Pr}\big((\mathcal{X}\times \mathcal{X})\cap H\big)}\cap k_iK
\end{equation*}
for some $k_i\in \mathbb{G}^6$. We also assume $K$ is an algebraic subgroup of the smallest dimension satisfying the above property. To simplify notation, denote $\overline{\text{Pr}\big((\mathcal{X}\times \mathcal{X})\cap H\big)}\cap_{\text{Pr}(P_i)} H_i$ by $\mathcal{Y}_i$ and a component of $\overline{\text{Pr}\big((\mathcal{X}\times \mathcal{X})\cap H\big)}\cap k_iK$ containing $\mathcal{Y}_i$ by $\mathcal{Z}_i$. 

\begin{enumerate}
\item First suppose $\dim \mathcal{Z}_i=2$. 
\begin{enumerate}
\item If $\mathcal{X}^{oa}\neq \emptyset$, by Lemma \ref{20072401}, $a_1=b_1=c_1=d_1=0$ in \eqref{22071401} and each $\mathcal{Z}_i$ is contained in a translation of $M'_1=L'_1=1$, that is, 
\begin{equation*}
\mathcal{Z}_i\subset (M'_1=(t'_{1i})^{-q'_{1i}}, L'_1=(t'_{1i})^{p'_{1i}}).\end{equation*}
\begin{enumerate}
\item If $\mathcal{Z}_i=\mathcal{Y}_i$, it implies 
\begin{equation*}
\big(H_i\cap (M'_1=(t'_{1i})^{-q'_{1i}}, L'_1=(t'_{1i})^{p'_{1i}})\big)=(M'_1=(t'_{1i})^{-q'_{1i}}, L'_1=(t'_{1i})^{p'_{1i}}), 
\end{equation*} 
which leads $t'_{1i}$ is a root of unity. But this is a contradiction. 

\item Now suppose $\dim \mathcal{Y}_i=1$ (and so $\dim H_i=3$). If $\dim \big(H_i\cap (M'_1=(t'_{1i})^{-q'_{1i}}, L'_1=(t'_{1i})^{p'_{1i}})\big)=3$, then it again contradicts the fact that $t'_{1i}$ is not a root of unity. Thus we assume 
\begin{equation*}
K_i:=H_i\cap (M'_1=(t'_{1i})^{-q'_{1i}}, L'_1=(t'_{1i})^{p'_{1i}})
\end{equation*}
is an algebraic coset of dimension at most $2$ in $\mathbb{G}^6$. By projecting each $\mathcal{Z}_i$ (resp. $K_i$) onto $\mathbb{G}^4(:=(M_1, L_1, M_2, L_2))$, we find that the image of each $\mathcal{Z}_i$ (resp. $K_i$) in $\mathbb{G}^4$ is $\mathcal{X}$ (resp. an algebraic coset of dim $2$). Therefore the projected image of each $\mathcal{Z}_i\cap K_i$ is regarded as an anomalous subvariety of $\mathcal{X}$. As $\mathcal{X}$ has only finitely many anomalous subvarieties by the assumption $\mathcal{X}^{oa}\neq \emptyset$, $\bigcup_{i\in \mathcal{I}}\mathcal{Z}_i\cap K_i$ is contained in the union of finitely many algebraic subgroups. But this contradicts our hypothesis that there is no such a finite set.   
\end{enumerate}

\item Now suppose $\mathcal{X}^{oa}=\emptyset$ and two cusps of $\mathcal{M}$ are not SGI. 
\begin{enumerate}
\item If $t'_{1i}, t'_{2i}$ are multiplicatively independent, then \eqref{22071401} is of the form given in \eqref{22032507} (i.e. $a'_1=b'_1=0$ in \eqref{22071401}). Using the preferred coordinates in Definition \ref{22071301}, if $\overline{\text{Pr}\big((\mathcal{X}\times \mathcal{X})\cap H\big)}$ is viewed as embedded in $\mathbb{G}^6(:=(\tilde{M_1}, \tilde{L_1}, \tilde{M_2}, \tilde{L_2}, M'_1, L'_1))$, then $\mathcal{Z}_i$ is contained in a translation of either $\tilde{M_1}=\tilde{L_1}=1$ or $\tilde{M_2}=\tilde{L_2}=1$ by Lemma \ref{21101501}. Without loss of generality, let us consider the first case and suppose
\begin{equation*}
\mathcal{Z}_i\subset (\tilde{M_1}=\zeta_{1i}, \tilde{L_1}=\zeta_{2i})
\end{equation*}
for some $\zeta_{1i}, \zeta_{2i}\in \mathbb{G}$. 

If $\mathcal{Z}_i=\mathcal{Y}_i$ or $\dim \big(H_i\cap  (\tilde{M_1}=\zeta_{1i}, \tilde{L_1}=\zeta_{2i})\big)>\dim H_i-2$, it means $H_i\subset \tilde{M_1}^{a}\tilde{L_1}^{b}=1$ for some $a, b\in \mathbb{Z}$, contradicting the fact that $t_{1i}, t_{2i}$ are multiplicatively independent. Thus $\dim \mathcal{Y}_i=1$ (with $\dim H_i =3$) and $\dim \big(H_i\cap  (\tilde{M_1}=\zeta_{1i}, \tilde{L_1}=\zeta_{2i})\big)=1$. However, in this case, one can verify that the intersection between $\overline{\text{Pr}\big((\mathcal{X}\times \mathcal{X})\cap H\big)}$ and $H_i\cap (\tilde{M_1}=\zeta_{1i}, \tilde{L_1}=\zeta_{2i})$ is simply $0$-dim by the same reasoning in Lemma \ref{20080305}, contradicting the fact that it contains a $1$-dim anomalous subvariety of $\overline{\text{Pr}\big((\mathcal{X}\times \mathcal{X})\cap H\big)}$. 

\item If $t'_{1i}, t'_{2i}$ are multiplicatively dependent, then $a_1=b_1=c_1=d_1=0$ in \eqref{22071401} and $\overline{\text{Pr}\big((\mathcal{X}\times \mathcal{X})\cap H\big)}$ is the product of three algebraic curves. Similar to the previous case, if $\overline{\text{Pr}\big((\mathcal{X}\times \mathcal{X})\cap H\big)}$ is embedded in $\mathbb{G}^6(:=(\tilde{M_1}, \tilde{L_1}, \tilde{M_2}, \tilde{L_2}, M'_1, L'_1))$, each $\mathcal{Z}_i$ is contained in a translation of one of the following by Lemma \ref{22060401}: $\tilde{M_1}=\tilde{L_1}=1, \tilde{M_2}=\tilde{L_2}=1$ or $M'_1=L'_1=1$. 

Since the first two cases were already discussed above, we assume 
$\mathcal{Z}_i$ and $\mathcal{Y}_i$ are neither contained in a translation of $\tilde{M_1}=\tilde{L_1}=1$ nor $\tilde{M_2}=\tilde{L_2}=1$, and 
\begin{equation*}
\mathcal{Z}_i\subset \big(M'_1=(t'_{1i})^{-q'_{1i}}, L'_1=(t'_{1i})^{p'_{1i}}\big). 
\end{equation*}
If $\mathcal{Z}_i=\mathcal{Y}_i$ or $\dim \big(H_i \cap (M'_1=(t'_{1i})^{-q'_{1i}}, L'_1=(t'_{1i})^{p'_{1i}})\big)>\dim H_i-2$, then $H_i\subset (M'_1)^{a'}(L'_1)^{b'}=1$, implying $(p'_{1i}, q'_{1i})$ is independent of $i\in \mathcal{I}$. But this contradicts the condition that $(p'_{1i}, q'_{1i})$ is sufficiently large. On the other hand, by the same reasoning as before, the intersection between $\overline{\text{Pr}\big((\mathcal{X}\times \mathcal{X})\cap H\big)}$ and $H_i \cap (M'_1=(t'_{1i})^{-q'_{1i}}, L'_1=(t'_{1i})^{p'_{1i}})$ becomes $0$-dim, contradicting the fact that it contains $\mathcal{Y}_i$.  
\end{enumerate} 
\end{enumerate}

\item Suppose $\dim \mathcal{Z}_i=1$ (i.e. $\mathcal{Z}_i=\mathcal{Y}_i$) for each $i\in \mathcal{I}$. In this case, by Theorem \ref{struc2}, one further has 
\begin{equation}\label{22022801}
k_iK\subset H_i\quad (\text{for}\;\;i\in \mathcal{I}).
\end{equation}
Note that $\dim K\geq 2$. (Otherwise, it contradicts the fact that $\dim \big(\overline{\text{Pr}\big((\mathcal{X}\times \mathcal{X})\cap H\big)}\cap k_iK\big)=1$ by Lemma \ref{20080305}.) Let $\mathcal{U}_K$ be defined as in Corollary \ref{22022603}. If $\dim K=\dim H_i=3$, by Corollary \ref{24100901}, we get a finite number of algebraic subgroup whose union contains $\bigcup_{i\in \mathcal{I}}\mathcal{Z}_i$. This, however, contradicts to our initial assumption. On the other hand, if $\dim K=2$ (so $\codim K=4$) and $\dim H_i=3$, as $\dim \mathscr{Z}_{K}$ is either $2$ or $3$, $\dim \overline{\mathcal{U}_K}$ is either $1$ or $2$, leading to the same contradiction by the same corollary. 
\end{enumerate}
\end{proof}

\subsection{Codimension $4$}\label{22042907}
We now complete the proof of Theorem \ref{22022601}. To be more precise, the following updated version of Theorem \ref{22022601} is established. 

\begin{theorem}\label{20070501}
Suppose there exist infinitely many pairs
\begin{equation}\label{22051101}
\{\mathcal{M}_{p_{1i}/q_{1i},p_{2i}/q_{2i}},\mathcal{M}_{p'_{1i}/q'_{1i},p'_{2i}/q'_{2i}} \}_{i\in \mathcal{I}}
\end{equation}
of Dehn fillings of $\mathcal{M}$ satisfying 
\begin{equation*}
\text{pvol}_{\mathbb{C}}\;\mathcal{M}_{p_{1i}/q_{1i},p_{2i}/q_{2i}}=\text{pvol}_{\mathbb{C}}\;\mathcal{M}_{p'_{1i}/q'_{1i},p'_{2i}/q'_{2i}}\quad (i\in \mathcal{I})
\end{equation*}
with $|p_{hi}|+|q_{hi}|$ and $|p'_{hi}|+|q'_{hi}|$ ($1\leq h\leq 2, i\in\mathcal{I}$) sufficiently large. Then there exists a finite set $\mathcal{H}^{(4)}$ of $4$-dim algebraic subgroups such that every $P_i$ is contained in some element of $\mathcal{H}^{(4)}$. More precisely, 
\begin{enumerate}
\item if $t'_{1i}, t'_{2i}$ (resp. $t_{1i}, t_{2i}$) are multiplicatively independent (resp. independent), then each element of $\mathcal{H}^{(4)}$ is defined by equations of the following forms
\begin{equation}\label{20071501}
\begin{gathered}
M_1^{a_j}L_1^{b_j}M_2^{c_j}L_2^{d_j}(M'_1)^{a'_j}(L'_1)^{b'_j}=M_1^{e_j}L_1^{f_j}M_2^{g_j}L_2^{h_j}(M'_2)^{g'_j}(L'_2)^{h'_j}=1\quad (j=1,2);
\end{gathered}
\end{equation} 
\item if $t'_{1i}, t'_{2i}$ (resp. $t_{1i}, t_{2i}$) are multiplicatively dependent (resp. independent), then each element of $\mathcal{H}^{(4)}$ is defined by equations of the following forms
\begin{equation}\label{22041611}
\begin{gathered}
M_1^{a_j}L_1^{b_j}M_2^{c_j}L_2^{d_j}(M'_1)^{a'_j}(L'_1)^{b'_j}=(M'_1)^{e'_j}(L'_1)^{f'_j}(M'_2)^{g'_j}(L'_2)^{h'_j}=1\quad (j=1,2).
\end{gathered}
\end{equation} 
\end{enumerate}
Moreover, for each $H^{(4)}\in \mathcal{H}^{(4)}$, the component of $(\mathcal{X}\times \mathcal{X})\cap H^{(4)}$ containing $(1, \dots, 1)$ is an anomalous subvariety of $\mathcal{X}\times \mathcal{X}$ of dimension $2$
\end{theorem}

We first prove the following theorem, to which Theorem \ref{20070501} is derived.  

\begin{theorem}\label{22041607}
Adapting the same assumptions and notation as in Theorem \ref{20070501}, we further suppose all the Dehn filling points $\{P_i\}_{i\in \mathcal{I}}$ associated to \eqref{22051101} are contained in an algebraic subgroup $H$ defined by either 
\begin{equation}\label{22041608}
\begin{gathered}
M_1^{a_1}L_1^{b_1}M_2^{c_1}L_2^{d_1}(M'_1)^{a'_1}=M_1^{a_2}L_1^{b_2}M_2^{c_2}L_2^{d_2}(M'_2)^{c'_2}=1,\quad (a'_1\neq 0,\quad c'_2\neq 0)
\end{gathered}
\end{equation}
or 
\begin{equation}\label{22081401}
\begin{gathered}
M_1^{a_1}L_1^{b_1}M_2^{c_1}L_2^{d_1}(M'_1)^{a'_1}=(M'_1)^{a'_2}(L'_1)^{b'_2}(M'_2)^{c'_2}(L'_2)^{d'_2}=1\quad (a'_1\neq 0,\quad (c'_2, d'_2)\neq (0, 0)). 
\end{gathered}
\end{equation} 
Let $\overline{\text{Pr}\big((\mathcal{X}\times \mathcal{X})\cap H\big)}$ be the algebraic closure of the image of $(\mathcal{X}\times \mathcal{X})\cap_{(1,\dots, 1)} H$ under
\begin{equation*}
\text{Pr}:\;\mathbb{G}^8(:=(M_1, L_1, M_2, L_2, M'_1, L'_1, M'_2, L'_2))\longrightarrow \mathbb{G}^5(:=(M_1, L_1,M_2, L_2, L'_1)).
\end{equation*}
Then there exists a finite set of algebraic subgroups such that each $\text{Pr}(P_i)$ is further contained in some element of the set.
\end{theorem}

\begin{proof}
First, note that $M'_1, L'_1, M'_2, L'_2$ are considered as functions depending on $M_1, M_2$ by \eqref{22041608} or \eqref{22081401}, and thus $\overline{\text{Pr}\big((\mathcal{X}\times \mathcal{X})\cap H\big)}$ is isomorphic to $\mathcal{X}$ and 
\begin{equation*}
\text{Pr}(P_i)\in \overline{\text{Pr}\big((\mathcal{X}\times \mathcal{X})\cap H\big)}\cap (M_1^{p_{1i}}L_1^{q_{1i}}=M_2^{p_{2i}}L_2^{q_{2i}}=1). 
\end{equation*}
Since $t_{1i}, t_{2i}, t'_{1i}$ are multiplicatively dependent each other, $\text{Pr}(P_i)$ is, additionally, contained in one more algebraic subgroup. Hence, by Corollary \ref{20080403}, there exists an algebraic subgroups $H_i$ ($i\in \mathcal{I}$) of dim $3$ such that 
\begin{itemize}
\item $\text{Pr}(P_i)\in \overline{\text{Pr}\big((\mathcal{X}\times \mathcal{X})\cap H\big)}\cap H_i$;
\item $\overline{\text{Pr}\big((\mathcal{X}\times \mathcal{X})\cap H\big)}\cap_{\text{Pr}(P_i)} H_i$ is an anomalous subvariety of $\overline{\text{Pr}\big((\mathcal{X}\times \mathcal{X})\cap H\big)}$.
\end{itemize}

If $\bigcup_{i\in \mathcal{I}}\overline{\text{Pr}\big((\mathcal{X}\times \mathcal{X})\cap H\big)}\cap_{\text{Pr}(P_i)} H_i$ is contained in the union of a finite set of algebraic subgroups, then we are done, and so suppose there is no such a finite set of algebraic subgroups. By Theorem \ref{struc2}, there is an algebraic subgroup $K$ where, for each $i\in \mathcal{I}$,  
\begin{itemize}
\item $\overline{\text{Pr}\big((\mathcal{X}\times \mathcal{X})\cap H\big)}\cap_{\text{Pr}(P_i)} H_i$ is a component of $\overline{\text{Pr}\big((\mathcal{X}\times \mathcal{X})\cap H\big)}\cap k_iK$ for some $k_i\in \mathbb{G}^5$;
\item  $k_iK\subset H_i$. 
\end{itemize}
Let $\mathcal{U}_K$ is defined as in Corollary \ref{22022603}. Then it is an algebraic curve in $\mathbb{G}^{\codim K}$, which, as $k_iK\subset H_i$ ($i\in \mathcal{I}$), intersects with infinitely many algebraic subgroups of codimension $2$. By Corollary \ref{24100901}, $\overline{\mathcal{U}_K}$ is contained in some algebraic subgroup, and this further implies every $\overline{\text{Pr}\big((\mathcal{X}\times \mathcal{X})\cap H\big)}\cap_{\text{Pr}(P_i)} H_i$ is contained in the same algebraic subgroup. But it contradicts our initial assumption that there is no such an algebraic subgroup. 
\end{proof}

Using Theorem \ref{22041607}, we now prove Theorem \ref{20070501}.
\begin{proof}[Proof of Theorem \ref{20070501}]
\begin{enumerate}
\item By Theorem \ref{20071401}, if $t_{1i}, t_{2i}$ and $t'_{1i}, t'_{2i}$ are multiplicatively independent, then there exists a finite set $\mathcal{H}$ of  algebraic subgroups such that each $P_i$ is contained in some $H\in \mathcal{H}$ and each $H$ is defined by equations of the forms given in \eqref{22041605}. Without loss of generality, we assume $\mathcal{H}=\{H\}$ and, changing basis if necessary, $H$ is defined by \eqref{22041608}. Projecting $(\mathcal{X}\times \mathcal{X})\cap H$ onto $\mathbb{G}^5(:=(M_1, L_1, M_2, L_2, L'_1))$, by Theorem \ref{22041607}, we get a finite set of algebraic subgroups whose union contains $\bigcup_{i\in \mathcal{I}}P_i$.   

Similarly, if we project $(\mathcal{X}\times \mathcal{X})\cap H$ onto $\mathbb{G}^5(:=(M_1, L_1, M_2, L_2, L'_2))$ instead of $\mathbb{G}^5(:=(M_1, L_1, M_2, L_2, L'_1))$, there exist finitely many algebraic subgroups, each of which is defined by an equation of the following form 
\begin{equation*}
M_1^{a}L_1^{b}M_2^{c}L_2^{d}(L'_2)^{d'}=1\quad (d'\neq 0)
\end{equation*} 
and whose union contains $\bigcup_{i\in \mathcal{I}}P_i$.

\item For $t'_{1i}, t'_{2i}$ (resp. $t_{1i}, t_{2i}$) being multiplicatively dependent (resp. independent), we analogously prrove the desired result by Theorems \ref{20071401} and \ref{22041607}.  
\end{enumerate}

Finally we confirm the last statement. Let $H^{(4)}$ be an algebraic subgroup containing all $P_i$ and defined by equations of the forms given in either \eqref{20071501} or \eqref{22041611}. On the contrary, if $\dim (\mathcal{X}\times \mathcal{X})\cap _{(1, \dots, 1)}H^{(4)}=1$, projecting $(\mathcal{X}\times \mathcal{X})\cap_{(1, \dots, 1)} H^{(4)}$ under
\begin{equation*}
\text{Pr}:\;(M_1, L_1, M_2, L_2, M'_1, L'_1, M'_2, L'_2)\longrightarrow (M_1, L_1,M_2, L_2),
\end{equation*}
we find an algebraic curve $\mathcal{C}:=\overline{\text{Pr}\big((\mathcal{X}\times \mathcal{X})\cap _{(1, \dots, 1)}H^{(4)}\big)}$ in $\mathbb{G}^4$. As 
\begin{equation*}
\text{Pr}(P_i)\in \mathcal{C}\cap (M_1^{p_{1i}}L_1^{q_{1i}}=M_2^{p_{2i}}L_2^{q_{2i}}=1)
\end{equation*}
for every $i\in \mathcal{I}$, by Theorem \ref{20080407}, there exists an algebraic subgroup $K$ containing $\mathcal{C}$ and further $\bigcup_{i\i}\text{Pr}(P_i)$. Note that $t_{1i}, t_{2i}$ are multiplicatively independent for each $i\in\mathcal{I}$ and hence the defining equation of $K$ is either of the form $M_1^aL_1^b=1$ or $M_2^cL_2^d=1$. Without loss of generality, if it is defined by the first one, then $-q_{1i}a+p_{1i}b=0$ for $i\in \mathcal{I}$. As $(p_{1i}, q_{1i})$ is a co-prime pair, this implies $(p_{1i}, q_{1i})$ is uniquely determined by $(a, b)$. However, this contradicts the fact that $|p_{1i}|+|q_{1i}|$ is sufficiently large. 
\end{proof}

\subsection{Refinement of Theorem \ref{20070501}}

In this subsection, we further refine Theorem \ref{20070501}, by showing that the second case of Theorem \ref{20070501} never arises. That is, we prove 
\begin{theorem}\label{22041616}
Let $\mathcal{M}$ and $\mathcal{X}$ be as usual. Let $\mathcal{M}_{p_{1}/q_{1},p_{2}/q_{2}}$ and $\mathcal{M}_{p'_{1}/q'_{1},p'_{2}/q'_{2}}$ be two Dehn fillings of $\mathcal{M}$ satisfying 
\begin{equation}\label{22050403}
\text{pvol}_{\mathbb{C}}\;\mathcal{M}_{p_{1}/q_{1},p_{2}/q_{2}}=\text{pvol}_{\mathbb{C}}\;\mathcal{M}_{p'_{1}/q'_{1},p'_{2}/q'_{2}}
\end{equation}
with $|p_h|+|q_h|$ and $|p'_h|+|q'_h|$ ($h=1,2$) sufficiently large. If $t_{1}, t_{2}$ (resp. $t'_{1}, t'_{2}$) are multiplicatively independent, then $t'_{1}, t'_{2}$ (resp. $t_{1}, t_{2}$) are multiplicatively independent.  
\end{theorem}

The strategy for proving of Theorem \ref{22041616} is similar to that for proving Theorem \ref{21091101}. Therefore, we first pursue an effective version of Theorem \ref{20070501} (2) in the lemma below. 

If $t'_1, t'_2$ (resp. $t_1, t_2$) are multiplicatively dependent (resp. independent), by Theorem \ref{20070501} (2), there exists an algebraic subgroup, defined by equations of the forms given in \eqref{22041611}, containing a Dehn filling point $P$ associated to \eqref{22050403}. Splitting \eqref{22041611} into 
\begin{equation}\label{22050801}
(M'_1)^{e'_j}(L'_1)^{f'_j}(M'_2)^{g'_j}(L'_2)^{h'_j}=1\quad (j=1,2) 
\end{equation}
and 
\begin{equation}\label{22050802}
M_1^{a_j}L_1^{b_j}M_2^{c_j}L_2^{d_j}(M'_1)^{a'_j}(L'_1)^{b'_j}=1\quad (j=1,2),
\end{equation}
let $H'$ and $H''$ be algebraic subgroups defined by \eqref{22050801} and \eqref{22050802} respectively. Regarding $H'$ as an algebraic subgroup in $\mathbb{G}^4(:=(M'_1, L'_1, M'_2, L'_2))$, if $\mathcal{X}\cap_{(1, \dots,1)} H'$ is viewed as a $1$-dim anomalous subvariety of $\mathcal{X}(\subset \mathbb{G}^4$)\footnote{Here, $\mathcal{X}$ is the second copy of $\mathcal{X}\times \mathcal{X}$.} and denoted by $\mathcal{C}$, it is natural to represent $(\mathcal{X}\times \mathcal{X})\cap_{(1, \dots,1)} H'$ as $\mathcal{X}\times \mathcal{C}$. Note that $(\mathcal{X}\times C)\cap_{(1, \dots,1)} H''$ is then a $2$-dim anomalous subvariety of $\mathcal{X}\times C$. 

Having this setting, the following is seen as an effective approach to Theorem \ref{20070501} (2).

\begin{lemma}\label{22041320}
Let $\mathcal{M}$ and $\mathcal{X}$ be as usual. Let $\mathcal{N}$ be a $1$-cusped hyperbolic $3$-manifold and $\mathcal{C}$ be its holonomy variety. Let $\mathcal{M}_{p_1/q_1, p_2/q_2}$ (resp. $\mathcal{N}_{p'_1/q'_1}$) be a Dehn filling of $\mathcal{M}$ (resp. $\mathcal{N}$) with the core holonomies $t_1, t_2$  (resp. $t'_1$). Suppose
\begin{enumerate}
\item $t_1^lt_2^l=t'_1\epsilon$ for some $l\in \mathbb{Q}$ and root of unity $\epsilon$ where $t_1, t_2$ (resp. $t'_1$) are the core geodesics of $\mathcal{M}_{p_1/q_1, p_2/q_2}$ (resp. $\mathcal{N}_{p'_1/q'_1}$);  
\item the elements in any proper subset of $\{t_1, t_2, t'_1\}$ are multiplicatively independent; 
\item a Dehn filling point associated to the pair $\mathcal{M}_{p_1/q_1, p_2/q_2}$ and $\mathcal{N}_{p'_1/q'_1}$ on $\mathcal{X}\times \mathcal{C}\big(\subset \mathbb{G}^4\times\mathbb{G}^2 (:=(M_1, L_1, M_2, L_2, M'_1, L'_1))\big)$ is contained in an algebraic subgroup $H''$ defined by
\begin{equation}\label{22041303}
M_1^{a_j}L_1^{b_j}M_2^{c_j}L_2^{d_j}=(M'_1)^{a'_j}(L'_1)^{b'_j},\quad (j=1,2). 
\end{equation}
\end{enumerate}
If $r'_1, s'_1, r_{j}, s_{j}$ ($j=1,2$) are integers such that 
\begin{equation}\label{22041313}
\det \left(\begin{array}{cc}
p'_{1} & q'_{1}\\
r'_{1} & s'_{1}
\end{array}\right)
=\det \left(\begin{array}{cc}
p_{1} & q_{1}\\
r_{1} & s_{1}
\end{array}\right)
=\det \left(\begin{array}{cc}
p_{2} & q_{2}\\
r_{2} & s_{2}
\end{array}\right)
=1, 
\end{equation}
then there exist $n_1, n_2\in \mathbb{Q}$ such that 
\begin{equation*}
\left(\begin{array}{cc}
p'_{1} & q'_{1}\\
r'_{1} & s'_{1}
\end{array}\right)
(A'_1)^{-1}A_1
=\left(\begin{array}{cc}
\frac{\det A_1}{(\det A'_1) l} & 0\\
n_{1} & l
\end{array}\right)
\left(\begin{array}{cc}
p_{1} & q_{1}\\
r_{1} & s_{1}
\end{array}\right)
\end{equation*}
and 
\begin{equation*} 
\left(\begin{array}{cc}
p'_{1} & q'_{1}\\
r'_{1} & s'_{1}
\end{array}\right)
(A'_1)^{-1}A_2
=\left(\begin{array}{cc}
\frac{\det A_2}{(\det A'_1) l} & 0\\
n_{2} & l
\end{array}\right)
\left(\begin{array}{cc}
p_{2} & q_{2}\\
r_{2} & s_{2}
\end{array}\right)
\end{equation*}
where $A'_1= \left(\begin{array}{cc}
a'_1 & b'_1\\
a'_2 & b'_2
\end{array}\right)$, 
$A_1= \left(\begin{array}{cc}
a_1 & b_1\\
a_2 & b_2
\end{array}\right)$ 
and $A_2= \left(\begin{array}{cc}
c_1 & d_1\\
c_2 & d_2
\end{array}\right)$. 
\end{lemma}

The proof of the lemma is attained by following the arguments verbatim used in the proof of Lemma \ref{22050501}. 

\begin{proof}
Moving to an analytic setting, $\log\big((\mathcal{X}\times \mathcal{C})\cap_{(1, \dots, 1)} H''\big)$ is defined by
\begin{equation}\label{22041304}
\begin{gathered}
A'_1\left(\begin{array}{c}
u'_1\\
v'_1
\end{array}
\right)
=A_1\left(\begin{array}{c}
u_1\\
v_1
\end{array}
\right)
+A_2\left(\begin{array}{c}
u_2\\
v_2
\end{array}
\right).
\end{gathered}
\end{equation}
Let 
\begin{equation}\label{220706011}
(\xi_{m_{1}}, \xi_{l_{1}}, \xi_{m_{2}}, \xi_{l_{2}}, \xi'_{m_{1}}, \xi'_{l_{1}})
\end{equation} 
be a Dehn filling point on $\log \big((\mathcal{X}\times \mathcal{C})\cap_{(1, \dots, 1)} H''\big)$ associated to the pair $(\mathcal{M}_{p_1/q_1, p_2/q_2}, \mathcal{N}_{p'_1/q'_1})$. That is, \eqref{220706011} is a point that satisfies
\begin{equation}\label{22070605}
p_{1}\xi_{m_{1}}+q_{1}\xi_{l_{1}}=p_{2}\xi_{m_{2}}+q_{2}\xi_{l_{2}}=p'_{1}\xi'_{m_{1}}+q'_{1}\xi'_{l_{1}}=-2\pi \sqrt{-1}
\end{equation}
and
\begin{equation}\label{22041307}
\begin{gathered}
A'_1\left(\begin{array}{c}
\xi'_{m_{1}}\\
\xi'_{l_{1}}
\end{array}
\right)
=A_1
\left(\begin{array}{c}
\xi_{m_{1}}\\
\xi_{l_{1}}
\end{array}
\right)
+A_2
\left(\begin{array}{c}
\xi_{m_{2}}\\
\xi_{l_{2}}
\end{array}
\right)
\Longrightarrow 
\left(\begin{array}{c}
\xi'_{m_{1}}\\
\xi'_{l_{1}}
\end{array}
\right)
=(A'_1)^{-1}A_1
\left(\begin{array}{c}
\xi_{m_{1}}\\
\xi_{l_{1}}
\end{array}
\right)
+(A'_1)^{-1}A_2
\left(\begin{array}{c}
\xi_{m_{2}}\\
\xi_{l_{2}}
\end{array}
\right).
\end{gathered}
\end{equation}
By changing basis of the cusp of $\mathcal{N}$ if necessary, we simply set $\left(\begin{array}{cc}
p'_1 & q'_1\\
r'_1 & s'_1
\end{array}\right)
=\left(\begin{array}{cc}
1 & 0\\
0 & 1
\end{array}\right)$, and so 
\begin{equation}\label{22071501}
\begin{gathered}
-2\pi \sqrt{-1}=
\left(\begin{array}{cc}
1 & 0
\end{array}
\right)
\bigg((A'_1)^{-1}A_1
\left(\begin{array}{c}
\xi_{m_{1}}\\
\xi_{l_{1}}
\end{array}
\right)
+(A'_1)^{-1}A_2
\left(\begin{array}{c}
\xi_{m_{2}}\\
\xi_{l_{2}}
\end{array}
\right)\bigg)
\end{gathered}
\end{equation}
by \eqref{22070605}-\eqref{22041307}. Recall the exponentials of 
\begin{equation}\label{22041306}
\begin{gathered}
\left(\begin{array}{cc}
1 & 0
\end{array}
\right)(A'_1)^{-1}A_1
\left(\begin{array}{c}
\xi_{m_{1}}\\
\xi_{l_{1}}
\end{array}
\right)\quad \text{and}\quad 
\left(\begin{array}{cc}
1 & 0
\end{array}
\right)(A'_1)^{-1}A_2
\left(\begin{array}{c}
\xi_{m_{2}}\\
\xi_{l_{2}}
\end{array}
\right)
\end{gathered}
\end{equation}
are powers of $t_1$ and $t_2$ respectively. Since $t_{1}, t_{2}$ are multiplicatively independent each other, the exponential of each in \eqref{22041306} is a root of unity by \eqref{22071501}, that is,  
\begin{equation}\label{22041312}
\begin{gathered}
\left(\begin{array}{cc}
1 & 0
\end{array}\right)(A'_1)^{-1}A_1=k_1 \left(\begin{array}{cc}
p_1 & q_1\\
\end{array}\right)\quad \text{and}\quad 
\left(\begin{array}{cc}
1 & 0
\end{array}
\right)(A'_1)^{-1}A_2=k_2 \left(\begin{array}{cc}
p_2 & q_2\\
\end{array}\right)
\end{gathered}
\end{equation}
for some $k_1, k_2\in \mathbb{Q}$ by \eqref{22070605}. 

Next, multiplying $\left(\begin{array}{cc}
0 & 1
\end{array}
\right)$ to \eqref{22041307}, it follows that 
\begin{equation}\label{22041309}
\begin{gathered}
\xi'_{l_1}=
\left(\begin{array}{cc}
0 & 1
\end{array}
\right)\bigg(
(A'_1)^{-1}A_1
\left(\begin{array}{c}
\xi_{m_1}\\
\xi_{l_1}
\end{array}
\right)
+(A'_1)^{-1}A_2
\left(\begin{array}{c}
\xi_{m_2}\\
\xi_{l_2}
\end{array}
\right)\bigg).
\end{gathered}
\end{equation}
Note that $t'_1=e^{\xi'_{l_1}}$, and, similar to \eqref{22041306}, the exponentials of 
\begin{equation*}
\begin{gathered}
\left(\begin{array}{cc}
0 & 1
\end{array}
\right)
(A'_1)^{-1}A_1
\left(\begin{array}{c}
\xi_{m_1}\\
\xi_{l_1}
\end{array}
\right)\quad \text{and}\quad
\left(\begin{array}{cc}
0 & 1
\end{array}
\right)
(A'_1)^{-1}A_2
\left(\begin{array}{c}
\xi_{m_2}\\
\xi_{l_2}
\end{array}
\right)
\end{gathered}
\end{equation*}
are powers of $t_1$ and $t_2$ respectively. Since $\epsilon t'_1=t_1^lt_2^l$ and the elements of any proper subset of $\{t_1, t_2, t'_1\}$ are multiplicatively independent, \eqref{22041309} is, in fact, equivalent to $\epsilon t'_1=t_1^lt_2^l$. That is, 
\begin{equation}\label{22041311}
\begin{gathered}
\left(\begin{array}{cc}
0 & 1
\end{array}
\right)
(A'_1)^{-1}A_1
=l\left(\begin{array}{cc}
r_{1} & s_{1}
\end{array}
\right)+n_1\left(\begin{array}{cc}
p_{1} & q_{1}
\end{array}
\right) \quad \text{and}\quad 
\left(\begin{array}{cc}
0 & 1
\end{array}
\right)
(A'_1)^{-1}A_2
=l\left(\begin{array}{cc}
r_{2} & s_{2}
\end{array}
\right)+n_2\left(\begin{array}{cc}
p_{2} & q_{2}
\end{array}
\right)
\end{gathered}
\end{equation}
for some $n_1, n_2\in \mathbb{Q}$. Finally, we obtain
\begin{equation*}
\begin{gathered}
(A'_1)^{-1}A_1
=\left(\begin{array}{cc}
k_1 & 0\\
n_{1} & l
\end{array}\right)
\left(\begin{array}{cc}
p_{1} & q_{1}\\
r_{1} & s_{1}
\end{array}\right)\quad\text{and}\quad  
(A'_1)^{-1}A_2
=\left(\begin{array}{cc}
k_2 & 0\\
n_{2} & l
\end{array}\right)
\left(\begin{array}{cc}
p_{2} & q_{2}\\
r_{2} & s_{2}
\end{array}\right)
\end{gathered}
\end{equation*}
from \eqref{22041312} and \eqref{22041311} respectively where $k_1
=\frac{\det A_1
}{(\det A'_1)l} \quad\text{and}\quad  
k_2
=\frac{\det A_2}{(\det A'_1)l}$ by \eqref{22041313}. 
\end{proof}

\begin{remark}\label{22071503}
\normalfont  In the above lemma, if $(\mathcal{X}\times \mathcal{C})\cap_{(1, \dots, 1)} H''$ is an anomalous subvariety of $\mathcal{X}\times \mathcal{C}$, similar to the one described in Remark \ref{22061607}, we obtain $\mathcal{X}^{oa}=\emptyset$ (meaning $\mathcal{X}\times \mathcal{C}$ is the product of three algebraic curves). Indeed, for any $(M'_1, L'_1)=(\zeta'_{m_1},\zeta'_{l_1})$ lying over $\mathcal{C}$, the intersection between $\mathcal{X}$ and the following coset\footnote{Note that \eqref{22071513} is $H''\cap (M'_1=\zeta'_{m_1},L'_1=\zeta'_{l_1})$.}
\begin{equation}\label{22071513}
M_1^{a_j}L_1^{b_j}M_2^{c_j}L_2^{d_j}=(\zeta'_{m_1})^{a'_j}(\zeta'_{l_1})^{b'_j},\quad (j=1,2)
\end{equation}
is an anomalous subvariety of $\mathcal{X}$, implying $\mathcal{X}$ has infinitely many anomalous subvarieties.   
\end{remark}

Using the above lemma, we now prove Theorem \ref{22041616}. Similar to the case of Theorem \ref{21091101}, we establish this by deriving a contradiction. 

\begin{proof}[Proof of Theorem \ref{22041616}]
Let $H'$, $H''$ and $\mathcal{C}$ be the same as given in the discussion before Lemma \ref{22041320}. 

Since a Dehn filling point associated to $\mathcal{M}_{p'_1/q'_1, p'_2/q'_2}$ is contained in $H'$, by Lemma \ref{22050501}, there exists a root of unity $\epsilon$ satisfying 
\begin{equation*}
(t'_1)^{\det E'}=\epsilon(t'_2)^{\det G'}
\end{equation*}
where $E'=\left(\begin{array}{cc}
e'_1 & f'_1\\
e'_2 & f'_2
\end{array}\right)$ and 
$G'=\left(\begin{array}{cc}
g'_1 & h'_1\\
g'_2 & h'_2
\end{array}\right)$ (see \eqref{22050801}), which further implies 
\begin{equation}\label{22050805}
t_1t_2=t'_1t'_2\Longleftrightarrow t_1t_2=\epsilon' (t'_1)^{1+\frac{\det E'}{\det G'}}
\end{equation}
with some root of unity $\epsilon'$. 

To simplify the problem, we project both $\mathcal{X}\times \mathcal{C}$ and $H''$ under
\begin{equation*}
\text{Pr}\;:\;\mathbb{G}^8(:=(M_1,L_1, \dots, M'_2, L'_2)) \longrightarrow \mathbb{G}^6(:=(M_1,L_1,\dots , M'_1, L'_1)), 
\end{equation*}
and, by abuse of notation, still denote the projected images of the two by $\mathcal{X}\times \mathcal{C}$ and $H''$ respectively. 

Applying Lemma \ref{22041320} to $\mathcal{X}\times \mathcal{C}$ and $H''$ with \eqref{22050805}, it follows that
\begin{equation}\label{22050605}
\begin{gathered}
\left(\begin{array}{cc}
p'_{1} & q'_{1}\\
r'_{1} & s'_{1}
\end{array}\right)
=\left(\begin{array}{cc}
\frac{\det A_1}{(\det A'_1)l} & 0\\
n_{1} & l
\end{array}\right)
\left(\begin{array}{cc}
p_{1} & q_{1}\\
r_{1} & s_{1}
\end{array}\right)A_1^{-1}A'_1
=\left(\begin{array}{cc}
\frac{\det A_2}{(\det A'_1)l} & 0\\
n_{2} & l
\end{array}\right)
\left(\begin{array}{cc}
p_{2} & q_{2}\\
r_{2} & s_{2}
\end{array}\right)A_2^{-1}A'_1
\end{gathered}
\end{equation}
where $n_1, n_2\in \mathbb{Q}$ and $l=\frac{\det E'}{\det G'+\det E'}$.
 
Now we claim $t_1, t_2$ are multiplicatively dependent. Analogous to that of Theorem \ref{21091101}, we verify this by showing that a Dehn filling point  associated with $\mathcal{M}_{p_1/q_1, p_2/q_2}$ is contained in an anomalous subvariety of $\mathcal{X}$. 

Moving to an analytic setting, $\log\big((\mathcal{X}\times \mathcal{C})\cap_{(1, \dots, 1)} H''\big)\big(\subset \mathbb{C}^6\big(:=(u_1, v_1, \dots, u'_1, v'_1)\big)\big)$ is defined by
\begin{equation}\label{22031703}
\begin{gathered}
\left(\begin{array}{cc}
a'_1 & b'_1\\
a'_2 & b'_2
\end{array}\right)
\left(\begin{array}{c}
u'_1\\
v'_1
\end{array}\right)
+
\left(\begin{array}{cc}
a_1 & b_1\\
a_2 & b_2
\end{array}\right)
\left(\begin{array}{c}
u_1\\
v_1
\end{array}\right)
+\left(\begin{array}{cc}
c_1 & d_1\\
c_2 & d_2
\end{array}\right)
\left(\begin{array}{c}
u_2\\
v_2
\end{array}\right)
=\left(\begin{array}{c}
0\\
0
\end{array}\right),
\end{gathered}
\end{equation}
and, letting $u'_1=0$ in \eqref{22031703}, \eqref{22031703} is reduced to 
\begin{equation}\label{22050901}
\begin{gathered}
\left(\begin{array}{cc}
a_1 & b_1\\
a_2 & b_2
\end{array}\right)
\left(\begin{array}{c}
u_1\\
v_1
\end{array}\right)
+\left(\begin{array}{cc}
c_1 & d_1\\
c_2 & d_2
\end{array}\right)
\left(\begin{array}{c}
u_2\\
v_2
\end{array}\right)
=\left(\begin{array}{c}
0\\
0
\end{array}\right). 
\end{gathered}
\end{equation}
By further projecting it onto $\mathbb{C}^4(:=(u_1, v_1, u_2, v_2))$, \eqref{22050901} is viewed as a $1$-dim anomalous analytic subset of $\log \mathcal{X}$ and thus    
\begin{equation}\label{22031704}
\begin{gathered}
\left(\begin{array}{cc}
a_1 & b_1\\
a_2 & b_2
\end{array}\right)
\left(\begin{array}{c}
u_1\\
v_1
\end{array}\right)
=\left(\begin{array}{cc}
c_1 & d_1\\
c_2 & d_2
\end{array}\right)
\left(\begin{array}{c}
u_2\\
v_2
\end{array}\right)
\end{gathered}
\end{equation}
is also a $1$-dim anomalous analytic subset of $\log \mathcal{X}$ by Theorem \ref{potential}. Let 
\begin{equation}\label{22050807}
(u_1, v_1, u_2, v_2)=(\xi_{m_1}, \xi_{l_1}, \xi_{m_2}, \xi_{l_2})
\end{equation}
be an intersection point between $\eqref{22031704}$ and $p_1u_1+q_1v_1=-2\pi \sqrt{-1}$. Namely,  
\begin{equation}\label{22041318}
A_1\left(\begin{array}{c}
\xi_{m_1}\\
\xi_{l_1}
\end{array}\right)
=A_2
\left(\begin{array}{c}
\xi_{m_2}\\
\xi_{l_2}
\end{array}\right)\quad \text{and}\quad p_1\xi_{m_1}+q_1\xi_{l_1}=-2\pi \sqrt{-1}
\end{equation}
where $A_1=\left(\begin{array}{cc}
a_1 & b_1\\
a_2 & b_2
\end{array}\right)
$ and $A_2=\left(\begin{array}{cc}
c_1 & d_1\\
c_2 & d_2
\end{array}\right)$. By \eqref{22050605}, 
\begin{equation*}
\begin{gathered}
\left(\begin{array}{cc}
p_{1} & q_{1}\\
r_{1} & s_{1}
\end{array}\right)A_1^{-1}A_2
=\left(\begin{array}{cc}
\frac{\det A_1}{(\det A'_1)l} & 0\\
n_{1} & l
\end{array}\right)^{-1}
\left(\begin{array}{cc}
\frac{\det A_2}{(\det A'_1)l} & 0\\
n_{2} & l
\end{array}\right)
\left(\begin{array}{cc}
p_{2} & q_{2}\\
r_{2} & s_{2}
\end{array}\right)
\end{gathered}
\end{equation*}
and, combining it with \eqref{22041318}, it follows that 
\begin{equation}\label{22050813}
\begin{aligned}
\left(\begin{array}{cc}
p_1 & q_1\\
r_1 & s_1
\end{array}\right)
\left(\begin{array}{c}
\xi_{m_1}\\
\xi_{l_1}
\end{array}\right)
=&\left(\begin{array}{cc}
p_1 & q_1\\
r_1 & s_1
\end{array}\right)A_1^{-1}A_2
\left(\begin{array}{c}
\xi_{m_2}\\
\xi_{l_2}
\end{array}\right)\\
=&\left(\begin{array}{cc}
\frac{\det A_1}{(\det A'_1)l} & 0\\
n_{1} & l
\end{array}\right)^{-1}
\left(\begin{array}{cc}
\frac{\det A_2}{(\det A'_1)l} & 0\\
n_{2} & l
\end{array}\right)
\left(\begin{array}{cc}
p_{2} & q_{2}\\
r_{2} & s_{2}
\end{array}\right)
\left(\begin{array}{c}
\xi_{m_2}\\
\xi_{l_2}
\end{array}\right)\\
=&\left(\begin{array}{cc}
\frac{\det A_2}{\det A_1} & 0\\
\frac{1}{l\det A_1}(-n_1\det A_2+ n_2\det A_1)& 1
\end{array}\right)
\left(\begin{array}{cc}
p_{2} & q_{2}\\
r_{2} & s_{2}
\end{array}\right)
\left(\begin{array}{c}
\xi_{m_2}\\
\xi_{l_2}
\end{array}\right).
\end{aligned}
\end{equation}
Since $\mathcal{X}^{oa}=\emptyset$ by Remark \ref{22071503}, $\det A_1=\det A_2$ by Proposition \ref{21073101} and so \eqref{22050813} is reduced farther to 
 \begin{equation}\label{22071510}
\begin{gathered}
\left(\begin{array}{cc}
p_1 & q_1\\
r_1 & s_1
\end{array}\right)
\left(\begin{array}{c}
\xi_{m_1}\\
\xi_{l_1}
\end{array}\right)
=\left(\begin{array}{cc}
1 & 0\\
\frac{n_2-n_1}{l} & 1
\end{array}\right)
\left(\begin{array}{cc}
p_{2} & q_{2}\\
r_{2} & s_{2}
\end{array}\right)
\left(\begin{array}{c}
\xi_{m_2}\\
\xi_{l_2}
\end{array}\right).
\end{gathered}
\end{equation}

In conclusion, we have 
\begin{equation*}
p_1\xi_{m_1}+q_1\xi_{l_1}=p_2\xi_{m_2}+q_2\xi_{l_2}=-2\pi \sqrt{-1}, 
\end{equation*}
meaning \eqref{22050807} is, actually, a Dehn filling point associated with $\mathcal{M}_{p_1/q_1, p_2/q_2}$ lying over \eqref{22031704}. However, as
\begin{equation*}
\begin{gathered}
e^{r_1\xi_{m_1}+s_1\xi_{l_1}}=e^{\frac{n_2-n_1}{l}(p_2\xi_{m_2} +q_2\xi_{l_2})+(r_2\xi_{m_2}+s_2\xi_{l_2})}=e^{\frac{n_2-n_1}{l} (-2\pi \sqrt{-1})}e^{r_2\xi_{m_2}+s_2\xi_{l_2}}\Longrightarrow t_1=e^{\frac{n_2-n_1}{l} (-2\pi \sqrt{-1})}t_2
\end{gathered}
\end{equation*}
by \eqref{22071510}, it contradicts the fact that $t_1, t_2$ are multiplicatively independent. This completes the proof of Theorem \ref{22041616}. 
\end{proof}

\subsection{Final summary}

Combining Theorems \ref{20071505}-\ref{22012104} with Theorems \ref{20070501} and \ref{22041616}, what we have obtained so far in Sections \ref{ZPCI}-\ref{ZPCII} are now summarized in the following theorem:  

\begin{theorem}\label{22050203}
Let $\mathcal{M}$ be a $2$-cusped hyperbolic $3$-manifold and $\mathcal{X}$ be its holonomy variety. Then there exists a finite set $\mathcal{H}$ of algebraic subgroups such that, for any two Dehn fillings $\mathcal{M}_{p_{1}/q_{1},p_{2}/q_{2}}$ and $\mathcal{M}_{p'_{1}/q'_{1},p'_{2}/q'_{2}}$ of $\mathcal{M}$ satisfying 
\begin{equation}\label{22041613}
\text{pvol}_{\mathbb{C}}\;\mathcal{M}_{p_{1}/q_{1},p_{2}/q_{2}}=\text{pvol}_{\mathbb{C}}\;\mathcal{M}_{p'_{1}/q'_{1},p'_{2}/q'_{2}}
\end{equation}
with $|p_h|+|q_h|$ and $|p'_h|+|q'_h|$ ($h=1,2$) sufficiently large, a Dehn filling point $P$ associated to \eqref{22041613} is contained in some $H\in \mathcal{H}$. More precisely, we have the following dichotomy.  
\begin{enumerate}
\item If $t_1, t_2$ (resp. $t'_1, t'_2$) are multiplicatively dependent (resp. dependent), then $H(\in \mathcal{H})$ containing $P$ is defined by equations of the forms   
\begin{equation}\label{220708011}
\begin{gathered}
M_1^{a_j}L_1^{b_j}M_2^{c_j}L_2^{d_j}=(M'_1)^{a'_j}(L'_1)^{b'_j}(M'_2)^{c'_j}(L'_2)^{d'_j}=M_1^{e_j}L_1^{f_j}(M'_1)^{e'_j}(L'_1)^{f'_j}=1, \quad (j=1,2). 
\end{gathered}
\end{equation} 
Further, each $(\mathcal{X}\times \mathcal{X})\cap_{(1, \dots, 1)} H$ is a $1$-dim anomalous subvariety of $\mathcal{X}\times \mathcal{X}$. 

\item If $t_{1}, t_{2}$ (resp. $t'_{1}, t'_{2}$) are multiplicatively independent (resp. independent), then $H(\in \mathcal{H})$ containing $P$ is defined by equations of the following forms
\begin{equation}\label{22051405}
\begin{gathered}
M_1^{a_j}L_1^{b_j}M_2^{c_j}L_2^{d_j}(M'_1)^{a'_j}(L'_1)^{b'_j}=M_1^{e_j}L_1^{f_j}M_2^{g_j}L_2^{h_j}(M'_2)^{g'_j}(L'_2)^{h'_j}=1\quad (j=1,2).
\end{gathered}
\end{equation} 
Further, each $(\mathcal{X}\times \mathcal{X})\cap_{(1, \dots, 1)} H$ is a $2$-dim anomalous subvariety of $\mathcal{X}\times \mathcal{X}$. 
\end{enumerate}
\end{theorem}
Finally, we state a corollary of the above theorem. Before doing so, the following definition shall first be provided. 
\begin{definition}\label{22070813}
\normalfont We say two vectors $\bold{v}$ and $\bold{w}$ over $\mathbb{Q}$ are \textit{parallel} to each other if either $\bold{v}=c\bold{w}$ or 
$\bold{w}=c\bold{v}$ holds for some $c\in \mathbb{Q}$, and denote by $\bold{v}\newparallel \bold{w}$.  
\end{definition}

\begin{corollary}\label{20121907}
Let $\mathcal{M}$ and $\mathcal{X}$ be the same as in the above theorem. Let $\tau_1$ and $\tau_2$ be the cusp shapes of $\mathcal{M}$. 
\begin{enumerate}
\item For $H$ defined by \eqref{220708011}, if $\dim \big((\mathcal{X}\times \mathcal{X})\cap_{(1, \dots, 1)} H\big)=1$, then 
\begin{equation*}
\begin{gathered}
\left(\begin{array}{cc}
a_1 & b_1 \\
a_2 & b_2 
\end{array}
\right)
\left(\begin{array}{c}
1 \\
\tau_1 
\end{array}
\right)
\newparallel
\left(\begin{array}{cc}
c_1 & d_1 \\
c_2 & d_2 
\end{array}
\right)
\left(\begin{array}{c}
1 \\
\tau_2 
\end{array}
\right), \quad 
\left(\begin{array}{cc}
a'_1 & b'_1 \\
a'_2 & b'_2 
\end{array}
\right)
\left(\begin{array}{c}
1 \\
\tau_1 
\end{array}
\right)
\newparallel
\left(\begin{array}{cc}
c'_1 & d'_1 \\
c'_2 & d'_2 
\end{array}
\right)
\left(\begin{array}{c}
1 \\
\tau_2 
\end{array}
\right) \\
\text{and}\quad \left(\begin{array}{cc}
e_1 & f_1 \\
e_2 & f_2 
\end{array}
\right)
\left(\begin{array}{c}
1 \\
\tau_1 
\end{array}
\right)
\newparallel
\left(\begin{array}{cc}
e'_1 & f'_1 \\
e'_2 & f'_2 
\end{array}
\right)
\left(\begin{array}{c}
1 \\
\tau_1
\end{array}
\right).
\end{gathered}
\end{equation*}

\item For $H$ defined by \eqref{22051405}, if $\dim \big((\mathcal{X}\times \mathcal{X})\cap_{(1, \dots, 1)} H\big)=2$, then 
\begin{equation}\label{22070810}
\begin{gathered}
\left(\begin{array}{cc}
a_1 & b_1 \\
a_2 & b_2 
\end{array}
\right)
\left(\begin{array}{c}
1 \\
\tau_1 
\end{array}
\right)
\newparallel
\left(\begin{array}{cc}
c_1 & d_1 \\
c_2 & d_2 
\end{array}
\right)
\left(\begin{array}{c}
1 \\
\tau_2 
\end{array}
\right)
\newparallel
\left(\begin{array}{cc}
a'_1 & b'_1 \\
a'_2 & b'_2 
\end{array}
\right)
\left(\begin{array}{c}
1 \\
\tau_1 
\end{array}
\right)\\
\text{and}\\
\left(\begin{array}{cc}
e_1 & f_1 \\
e_2 & f_2 
\end{array}
\right)
\left(\begin{array}{c}
1 \\
\tau_1 
\end{array}
\right)
\newparallel
\left(\begin{array}{cc}
g_1 & h_1 \\
g_2 & h_2 
\end{array}
\right)
\left(\begin{array}{c}
1 \\
\tau_2 
\end{array}
\right)
\newparallel
\left(\begin{array}{cc}
g'_1 & h'_1 \\
g'_2 & h'_2 
\end{array}
\right)
\left(\begin{array}{c}
1 \\
\tau_2
\end{array}
\right).
\end{gathered}
\end{equation}
\end{enumerate}
\end{corollary}
\begin{proof}
We only consider the second case (as the proof of the first case is similar). 

For $H$ given as \eqref{22051405}, moving to the analytic holonomy variety, $\log \big((\mathcal{X}\times \mathcal{X})\cap_{(1, \dots, 1)} H\big)$ is defined by
\begin{equation}\label{22070803}
\begin{gathered}
\left(\begin{array}{cccc}
a'_1 & b'_1 & 0 & 0\\
a'_2 & b'_2 & 0 & 0\\
0 & 0 & g'_1 & h'_1\\
0 & 0 & g'_2 & h'_2 
\end{array}\right)
\left(\begin{array}{c}
u'_1\\
v'_1\\
u'_2\\
v'_2
\end{array}\right)
+\left(\begin{array}{cccc}
a_1 & b_1 & c_1 & d_1\\
a_2 & b_2 & c_2 & d_2\\
e_1 & f_1 & g_1 & h_1\\
e_2 & f_2 & g_2 & h_2 
\end{array}\right)
\left(\begin{array}{c}
u_1\\
v_1\\
u_2\\
v_2
\end{array}\right)
=\left(\begin{array}{c}
0\\
0\\
0\\
0
\end{array}\right),
\end{gathered}
\end{equation}
whose Jacobian at $(0, \dots, 0)$ is 
\begin{equation}\label{22070805}
\left(\begin{array}{cccc}
a'_1 & b'_1 & 0 & 0\\
a'_2 & b'_2 & 0 & 0\\
0 & 0 & g'_1 & h'_1\\
0 & 0 & g'_2 & h'_2 
\end{array}\right)
\left(\begin{array}{c}
1\\
\tau_1\\
1\\
\tau_2
\end{array}\right)
+\left(\begin{array}{cccc}
a_1 & b_1 & c_1 & d_1\\
a_2 & b_2 & c_2 & d_2\\
e_1 & f_1 & g_1 & h_1\\
e_2 & f_2 & g_2 & h_2 
\end{array}\right)
\left(\begin{array}{c}
1\\
\tau_1\\
1\\
\tau_2
\end{array}\right).
\end{equation}
Since \eqref{22070803} is an analytic variety of dimension $2$, the rank of \eqref{22070805} is $2$, and thereby \eqref{22070810} follows. 
\end{proof}

\newpage
\section{Quantification I}\label{quantaI}
In Sections \ref{quantaI}-\ref{20121903}, we make Theorem \ref{22050203} more effective. First, in this section, we determine the explicit relationship between two Dehn filling coefficients in \eqref{22041613} and the exponents of the defining equations of $H$ (given in Theorem \ref{22050203}). The basic strategy of the proof is similar to the one exhibited in either Lemma \ref{21082802} or \ref{22041320}.

By Theorem \ref{22050203}, if two Dehn fillings $\mathcal{M}_{p_{1}/q_{1},p_{2}/q_{2}}$ and $\mathcal{M}_{p'_{1}/q'_{1},p'_{2}/q'_{2}}$ of $\mathcal{M}$ with $|p_i|+|q_i|$ and $|p'_i|+|q'_i|$ sufficiently large ($i=1,2$) satisfy \eqref{22041613}, a Dehn filling point $P$ associated to them is contained in an algebraic subgroup $H$ defined by equations of the forms given in \eqref{20071501}. Moving to the analytic holonomy variety, $\log\big((\mathcal{X}\times \mathcal{X})\cap H\big)$ is a complex manifold defined by 
\begin{equation}\label{20091903}
\begin{gathered}
\left(\begin{array}{cccc}
a'_1 & b'_1 & 0 & 0\\
a'_2 & b'_2 & 0 & 0\\
0 & 0 & g'_1 & h'_1\\
0 & 0 & g'_2 & h'_2 
\end{array}\right)
\left(\begin{array}{c}
u'_1\\
v'_1\\
u'_2\\
v'_2
\end{array}\right)
=\left(\begin{array}{cccc}
a_1 & b_1 & c_1 & d_1\\
a_2 & b_2 & c_2 & d_2\\
e_1 & f_1 & g_1 & h_1\\
e_2 & f_2 & g_2 & h_2 
\end{array}\right)
\left(\begin{array}{c}
u_1\\
v_1\\
u_2\\
v_2
\end{array}\right).
\end{gathered}
\end{equation}
To simplify the problem, let us assume 
\begin{equation}\label{22022007}
\left(\begin{array}{cc}
a'_1 & b'_1 \\
a'_2 & b'_2 
\end{array}\right)=
\left(\begin{array}{cc}
g'_1 & h'_1 \\
g'_2 & h'_2 
\end{array}\right)=
\left(\begin{array}{cc}
1 & 0 \\
0 & 1 
\end{array}\right)
\end{equation} 
and 
\begin{equation}\label{22022005}
\left(\begin{array}{cccc}
a_1 & b_1 & c_1 & d_1\\
a_2 & b_2 & c_2 & d_2\\
e_1 & f_1 & g_1 & h_1\\
e_2 & f_2 & g_2 & h_2 
\end{array}\right)\in \text{Mat}_{4\times 4}(\mathbb{Q}).
\end{equation}
Moreover, denoting the matrix in \eqref{22022005} by $M$ and letting 
\begin{equation}\label{22022010}
A_1:=\left(\begin{array}{cc}
a_1 & b_1 \\
a_2 & b_2 
\end{array}
\right), \quad 
A_2:=\left(\begin{array}{cc}
c_1 & d_1\\
c_2 & d_2
\end{array}
\right), \quad 
A_3:=\left(\begin{array}{cc}
e_1 & f_1 \\
e_2 & f_2 
\end{array}
\right), \quad
A_4:=\left(\begin{array}{cc}
g_1 & h_1\\
g_2 & h_2
\end{array}
\right), 
\end{equation}
we represent $M$ as 
\begin{equation}\label{21072901}
\left(\begin{array}{cc}
A_1 & A_2\\
A_3 & A_4
\end{array}
\right)\Bigg(=\left(\begin{array}{cccc}
a_1 & b_1 & c_1 & d_1\\
a_2 & b_2 & c_2 & d_2\\
e_1 & f_1 & g_1 & h_1\\
e_2 & f_2 & g_2 & h_2 
\end{array}
\right)\Bigg).
\end{equation}
Then the first quantified version of Theorem \ref{22050203} is given as follows:
\begin{theorem}\label{22031403}
Let $\mathcal{M}$ and $\mathcal{X}$ be as in Theorem \ref{22050203}. Let $\mathcal{M}_{p_{1}/q_{1}, p_{2}/q_{2}}$ and $\mathcal{M}_{p'_{1}/q'_{1}, p'_{2}/q'_{2}}$ be two Dehn fillings of $\mathcal{M}$ with sufficiently large $|p_i|+|q_i|, |p'_i|+|q'_i|$ ($i=1,2$) satisfying
\begin{equation}\label{22031409}
\text{pvol}_{\mathbb{C}}\;\mathcal{M}_{p_{1}/q_{1}, p_{2}/q_{2}}=\text{pvol}_{\mathbb{C}}\;\mathcal{M}_{p'_{1}/q'_{1}, p'_{2}/q'_{2}}. 
\end{equation}
Suppose the core holonomies $t_1, t_2$ (resp. $t'_1, t'_2$) of $\mathcal{M}_{p_{1}/q_{1}, p_{2}/q_{2}}$ (resp. $\mathcal{M}_{p'_{1}/q'_{1}, p'_{2}/q'_{2}}$) are multiplicatively independent and a Dehn filling point $P$ associated with \eqref{22031409} is contained in an analytic variety defined by \eqref{20091903} with \eqref{22022007}-\eqref{22022005}. Then the following statements are true. 
\begin{enumerate}
\item For $r_i, s_i, r'_i, s'_i(\in \mathbb{Z})$ satisfying 
\begin{equation}\label{22080802}
\det \left( \begin{array}{cc}
p_i & q_i\\
r_i & s_i
\end{array}\right)=
\det \left( \begin{array}{cc}
p'_i & q'_i\\
r'_i & s'_i
\end{array}\right)
=1\quad (i=1,2), 
\end{equation}
there exist $k_j, l_j, n_j\in \mathbb{Q}$ ($1\leq j\leq 4$) such that 
\begin{equation}\label{22031407}
\begin{gathered}
\left( \begin{array}{cc}
p'_1 & q'_1\\
r'_1 & s'_1
\end{array}\right)
A_1=\left( \begin{array}{cc}
k_1 & 0\\
n_1 & l_1
\end{array}\right)
\left( \begin{array}{cc}
p_1 & q_1\\
r_1 & s_1
\end{array}\right), \\
\left( \begin{array}{cc}
p'_1 & q'_1\\
r'_1 & s'_1
\end{array}\right)
A_2
=\left( \begin{array}{cc}
k_2 & 0\\
n_2 & l_2
\end{array}\right)
\left( \begin{array}{cc}
p_2 & q_2\\
r_2 & s_2
\end{array}\right), \\
\left( \begin{array}{cc}
p'_2  & q'_2\\
r'_2  & s'_2\\
\end{array}\right)
A_3
=\left( \begin{array}{cc}
k_3 & 0\\
n_3 & l_3
\end{array}\right)
\left( \begin{array}{cc}
p_1 & q_1\\
r_1 & s_1
\end{array}\right), \\ 
\left( \begin{array}{cc}
p'_2  & q'_2\\
r'_2  & s'_2
\end{array}\right)
A_4
=\left( \begin{array}{cc}
k_4 & 0\\
n_4 & l_4
\end{array}\right)
\left( \begin{array}{cc}
p_2 & q_2\\
r_2 & s_2
\end{array}\right), \\
\end{gathered}
\end{equation}
\begin{equation}\label{22030403}
k_1+k_2=k_3+k_4=1,
\end{equation}
and 
\begin{equation}\label{20101101}
\begin{gathered}
l_1+l_3=l_2+l_4=1, \quad k_jl_j=\det A_j\;(1\leq j\leq 4).
\end{gathered}
\end{equation} 
\item If $\det A_1=0$ or $\det A_4=0$, then 
\begin{equation*}
A_1=A_4=\left(\begin{array}{cc}
0 & 0\\
0 & 0
\end{array}\right) \quad\text{and}\quad \det A_3=\det A_2=1.
\end{equation*} 
\item If $\det A_2=0$ or $\det A_3=0$, then 
\begin{equation*}
A_2=A_3=\left(\begin{array}{cc}
0 & 0\\
0 & 0
\end{array}\right)\quad\text{and}\quad \det A_1=\det A_4=1.
\end{equation*} 
\end{enumerate}
\end{theorem}
We prove the theorem by splitting it into several lemmas. The first argument of Theorem \ref{22031403} will be verified in Lemmas \ref{20092003}-\ref{20101102}, and the second and third ones in Lemma \ref{22050205}. 

In the following, we show there exist $k_j$ ($1\leq j\leq 4$) satisfying \eqref{22031407}-\eqref{22030403}. 
\begin{lemma}\label{20092003}
Having the same notation and assumptions as in Theorem \ref{22031403},  
\begin{enumerate}
\item there exist $k_j\in \mathbb{Q}$ ($1\leq j\leq 4$) such that 
\begin{equation}\label{20092001}
\begin{gathered}
\left( \begin{array}{cc}
p'_1 & q'_1
\end{array}\right)
A_1
=k_1\left( \begin{array}{cc}
p_1 & q_1
\end{array}\right), \quad 
\left( \begin{array}{cc}
p'_1 & q'_1
\end{array}\right)
A_2
=k_2\left( \begin{array}{cc}
p_2 & q_2
\end{array}\right)\\
\left( \begin{array}{cc}
p'_2  & q'_2
\end{array}\right)
A_3
=k_3\left( \begin{array}{cc}
p_1 & q_1
\end{array}\right), \quad 
\left( \begin{array}{cc}
p'_2  & q'_2
\end{array}\right)
A_4
=k_4\left( \begin{array}{cc}
p_2 & q_2
\end{array}\right)
\end{gathered}
\end{equation}
and 
\begin{equation}\label{220304031}
k_1+k_2=k_3+k_4=1;
\end{equation}
\item for each $j$ ($1\leq j\leq 4$), if $\det A_j=0$, then $A_j=\left( \begin{array}{cc}
0 & 0\\
0 & 0
\end{array}\right)$ and $k_j=0$. 
\end{enumerate}
\end{lemma}

\begin{proof}
\begin{enumerate}
\item Let $P$ be given by
\begin{equation}\label{21011407}
(u_1, v_1, \dots, u'_2, v'_2)=(\xi_{m_{1}}, \xi_{l_{1}}, \dots , \xi'_{m_{2}}, \xi'_{l_{2}})\quad (\in \mathbb{C}^8). 
\end{equation}
That is, the coordinates in \eqref{21011407} satisfy
\begin{equation}\label{21011307}
p_i\xi_{m_{i}}+q_i\xi_{l_{i}}=p'_i\xi'_{m_{i}}+q'_i\xi'_{l_{i}}=-2\pi \sqrt{-1},\quad (i=1,2)
\end{equation}
and 
\begin{equation}\label{22030407}
\begin{gathered}
\left(\begin{array}{c}
\xi'_{m_1}\\
\xi'_{l_1}\\
\xi'_{m_2}\\
\xi'_{l_2}
\end{array}\right)
=\left(\begin{array}{cccc}
a_1 & b_1 & c_1 & d_1\\
a_2 & b_2 & c_2 & d_2\\
e_1 & f_1 & g_1 & h_1\\
e_2 & f_2 & g_2 & h_2 
\end{array}\right)
\left(\begin{array}{c}
\xi_{m_1}\\
\xi_{l_1}\\
\xi_{m_2}\\
\xi_{l_2}
\end{array}\right).  
\end{gathered}
\end{equation}
Combining \eqref{21011307} with \eqref{22030407}, we have
\begin{equation*}
\begin{gathered}
\left(\begin{array}{cccc}
p'_1 & q'_1 & 0 & 0\\
0 & 0 & p'_2 & q'_2
\end{array}\right)
\left(\begin{array}{c}
\xi'_{m_1}\\
\xi'_{l_1}\\
\xi'_{m_2}\\
\xi'_{l_2}
\end{array}\right)
=\left(\begin{array}{c}
-2\pi\sqrt{-1}\\
-2\pi\sqrt{-1}
\end{array}\right)
=\left(\begin{array}{cccc}
p'_1 & q'_1 & 0 & 0\\
0 & 0 & p'_2 & q'_2
\end{array}\right)
\left(\begin{array}{cccc}
a_1 & b_1 & c_1 & d_1\\
a_2 & b_2 & c_2 & d_2\\
e_1 & f_1 & g_1 & h_1\\
e_2 & f_2 & g_2 & h_2 
\end{array}\right)
\left(\begin{array}{c}
\xi_{m_1}\\
\xi_{l_1}\\
\xi_{m_2}\\
\xi_{l_2}
\end{array}\right), 
\end{gathered}
\end{equation*}
and thus
\begin{equation*}
\begin{gathered}
-2\pi\sqrt{-1}=p'_1(a_1\xi_{m_1}+b_1\xi_{l_1}+c_1\xi_{m_2}+d_1\xi_{l_2})+q'_1(a_2\xi_{m_1}+b_2\xi_{l_1}+c_2\xi_{m_2}+d_2\xi_{l_2}),\\
-2\pi\sqrt{-1}=p'_2(e_1\xi_{m_1}+f_1\xi_{l_1}+g_1\xi_{m_2}+h_1\xi_{l_2})+q'_2(e_2\xi_{m_1}+f_2\xi_{l_1}+g_2\xi_{m_2}+h_2\xi_{l_2}), 
\end{gathered}
\end{equation*}
which are rewritten as  
\begin{equation}\label{22030404}
\begin{gathered}
-2\pi\sqrt{-1}=(p'_1a_1+q'_1a_2)\xi_{m_1}+(p'_1b_1+q'_1b_2)\xi_{l_1}+(p'_1c_1+q'_1c_2)\xi_{m_2}+(p'_1d_1+q'_1d_2)\xi_{l_2},\\
-2\pi\sqrt{-1}=(p'_2e_1+q'_2e_2)\xi_{m_1}+(p'_2f_1+q'_2f_2)\xi_{l_1}+(p'_2g_1+q'_2g_2)\xi_{m_2}+(p'_2h_1+q'_2h_2)\xi_{l_2}.
\end{gathered}
\end{equation}
Recall (from Remark \ref{22081303}) that 
\begin{equation}\label{22070911}
e^{\xi_{m_{i}}}=t_i^{-q_i}\quad \text{ and  }\quad e^{\xi_{l_{i}}}=t_i^{p_i}\quad (i=1,2), 
\end{equation}
and so \eqref{22030404} implies
\begin{equation*}
\begin{gathered}
1=t_1^{-(p'_1a_1+q'_1a_2)q_1+(p'_1b_1+q'_1b_2)p_1}t_2^{-(p'_1c_1+q'_1c_2)q_1+(p'_1d_1+q'_1d_2)p_2},\\
1=t_1^{-(p'_2e_1+q'_2e_2)q_1+(p'_2f_1+q'_2f_2)p_1}t_2^{-(p'_2g_1+q'_2g_2)q_2+(p'_2h_1+q'_2h_2)p_2}.
\end{gathered}
\end{equation*}
As $t_1, t_2$ are multiplicatively independent, it follows that
\begin{equation*}
t_1^{-(p'_1a_1+q'_1a_2)q_1+(p'_1b_1+q'_1b_2)p_1},\quad t_2^{-(p'_1c_1+q'_1c_2)q_1+(p'_1d_1+q'_1d_2)p_2}
\end{equation*}
and
\begin{equation*}
t_1^{-(p'_2e_1+q'_2e_2)q_1+(p'_2f_1+q'_2f_2)p_1}, \quad t_2^{-(p'_2g_1+q'_2g_2)q_2+(p'_2h_1+q'_2h_2)p_2}
\end{equation*}
are all roots of unity. Since $t_i$ is not a root of unity and $(p_i, q_i)$ is a co-prime pair, we conclude 
\begin{equation*}
\begin{gathered}
\left(\begin{array}{cc}
p'_1a_1+q'_1a_2 & p'_1b_1+q'_1b_2
\end{array}\right)
=k_1\left(\begin{array}{cc}
p_1 & q_1
\end{array}\right), \quad 
\left(\begin{array}{cc}
p'_1c_1+q'_1c_2 & p'_1d_1+q'_1d_2
\end{array}\right)
=k_2\left(\begin{array}{cc}
p_2 & q_2
\end{array}\right),\\
\left(\begin{array}{cc}
p'_2e_1+q'_2e_2 & p'_2f_1+q'_2f_2
\end{array}\right)
=k_3\left(\begin{array}{cc}
p_1 & q_1
\end{array}\right), \quad 
\left(\begin{array}{cc}
p'_2g_1+q'_2g_2 & p'_2h_1+q'_2h_2
\end{array}\right)
=k_4\left(\begin{array}{cc}
p_2 & q_2
\end{array}\right)
\end{gathered}
\end{equation*}
for some $k_j\in \mathbb{Q}$ ($1\leq j\leq 4$). Finally \eqref{220304031} follows from \eqref{21011307} and \eqref{22030404}. 

\item We claim if $\det A_j=0$, then $A_j$ is the zero matrix for each $j$ ($1\leq j\leq 4$). Without loss of generality, it will be only proven for $j=1$, as the rest can be derived similarly. Suppose $\det A_1=0,  A_1\neq 0$ and 
\begin{equation*}
A_1=\left(\begin{array}{cc}
a_1 & b_1 \\
ma_1 & mb_1
\end{array}\right) 
\end{equation*}
for some $(a_1,b_1)\neq (0,0)$ and $m\in \mathbb{Q}$. By \eqref{20092001}, 
\begin{equation*}
\begin{gathered}
\left( \begin{array}{cc}
p'_1 & q'_1
\end{array}\right)
\left( \begin{array}{cc}
a_1 & b_1  \\
ma_1 & mb_1  
\end{array}\right)
=(p'_1+mq'_1)
\left( \begin{array}{cc}
a_1 & b_1
\end{array}\right)
=k_1\left( \begin{array}{cc}
p_1 & q_1
\end{array}\right)
\end{gathered}
\end{equation*}
for some $k_1\in \mathbb{Q}$ and so
\begin{equation*}
\begin{gathered}
\left( \begin{array}{cc}
a_1 & b_1
\end{array}\right)
\newparallel
\left( \begin{array}{cc}
p_1 & q_1
\end{array}\right). 
\end{gathered}
\end{equation*}
As $(p_1, q_1)$ is a co-prime pair, it is uniquely determined by $(a_1, b_1)$. However, as $(a_1, b_1)$ depends only on $\mathcal{X}$, it contradicts the fact that $|p_1|+|q_1|$ is sufficiently large. Thus $A_1= 0$ and $k_1=0$ follows by \eqref{20092001}. 
\end{enumerate}
\end{proof}

The following lemma completes the proof of the first claim in Theorem \ref{22031403}. 
\begin{lemma}\label{20101102}
Adapting the same notation and assumptions as in Theorem \ref{22031403}, let $k_j$ ($1\leq j\leq 4$) be the same as in Lemma \ref{20092003}. Then there further exist $l_j, n_j$ ($1\leq j\leq 4$) satisfying \eqref{22031407} and \eqref{20101101}.
\end{lemma}
\begin{proof}
Suppose 
\begin{equation}\label{20100907}
t'_1=\epsilon_1 t_1^{l_1}t_2^{l_2}, \quad t'_2=\epsilon_2 t_2^{l_3}t_2^{l_4}
\end{equation} 
for some $l_j\in \mathbb{Q}$ ($1\leq j\leq 4$) and torsion $\epsilon_i$ ($1\leq i\leq 2$). As $t_1, t_2$ are multiplicatively independent and $t'_1t'_2=t_1t_2$, we have 
\begin{equation}\label{22030501}
l_1+l_3=l_2+l_4=1. 
\end{equation}
Since 
\begin{equation*}
\log t_i\equiv r_i\xi_{m_{i}}+s_i\xi_{l_{i}}\mod 2\pi \sqrt{-1} \quad\text{and}\quad \log t'_i\equiv  r'_i\xi'_{m_{i}}+s'_i\xi'_{l_{i}}\mod 2\pi \sqrt{-1} 
\end{equation*}
where $r_i, s_i, r'_i, s'_i$ ($i=1,2$) are as given in \eqref{22080802}, \eqref{20100907} implies
\begin{equation}\label{201009091}
\begin{gathered}
\left(\begin{array}{c}
r'_1\xi'_{m_1}+s'_1\xi'_{l_1}\\
r'_2\xi'_{m_2}+s'_2\xi'_{l_2}
\end{array}\right)
=\left(\begin{array}{cc}
l_1 & l_2 \\
l_3 & l_4  
\end{array}\right)
\left(\begin{array}{cccc}
r_1 & s_1 & 0 & 0\\
0 & 0 & r_2 & s_2 
\end{array}\right)
\left(\begin{array}{c}
\xi_{m_1}\\
\xi_{l_1}\\
\xi_{m_2}\\
\xi_{l_2}
\end{array}\right)
+2\pi\sqrt{-1}\left(\begin{array}{c}
n_{12}\\
n_{34}
\end{array}\right)
\end{gathered}
\end{equation}
for some $n_{12}, n_{34}\in \mathbb{Q}$. By \eqref{22030407}, we get
\begin{equation}\label{22072001}
\begin{gathered}
\left(\begin{array}{cccc}
r'_1 & s'_1 & 0 & 0\\
0 & 0 & r'_2 & s'_2 
\end{array}\right)
\left(\begin{array}{c}
\xi'_{m_1}\\
\xi'_{l_1}\\
\xi'_{m_2}\\
\xi'_{l_2}
\end{array}\right)
=\left(\begin{array}{cccc}
r'_1 & s'_1 & 0 & 0\\
0 & 0 & r'_2 & s'_2 
\end{array}\right)
\left(\begin{array}{cccc}
a_1 & b_1 & c_1 & d_1\\
a_2 & b_2 & c_2 & d_2\\
e_1 & f_1 & g_1 & h_1\\
e_2 & f_2 & g_2 & h_2 
\end{array}\right)
\left(\begin{array}{c}
\xi_{m_1}\\
\xi_{l_1}\\
\xi_{m_2}\\
\xi_{l_2}
\end{array}\right). 
\end{gathered}
\end{equation}
Combining \eqref{22072001} with \eqref{201009091}, it follows that 
\begin{equation*}
\begin{gathered}
\small
\left(\begin{array}{cc}
l_1 & l_2 \\
l_3 & l_4  
\end{array}\right)
\left(\begin{array}{cccc}
r_1 & s_1 & 0 & 0\\
0 & 0 & r_2 & s_2 
\end{array}\right)
\left(\begin{array}{c}
\xi_{m_1}\\
\xi_{l_1}\\
\xi_{m_2}\\
\xi_{l_2}
\end{array}\right)
+2\pi\sqrt{-1}\left(\begin{array}{c}
n_{12}\\
n_{34}
\end{array}\right)
=\left(\begin{array}{cccc}
r'_1 & s'_1 & 0 & 0\\
0 & 0 & r'_2 & s'_2 
\end{array}\right)
\left(\begin{array}{cccc}
a_1 & b_1 & c_1 & d_1\\
a_2 & b_2 & c_2 & d_2\\
e_1 & f_1 & g_1 & h_1\\
e_2 & f_2 & g_2 & h_2 
\end{array}\right)
\left(\begin{array}{c}
\xi_{m_1}\\
\xi_{l_1}\\
\xi_{m_2}\\
\xi_{l_2}
\end{array}\right),
\end{gathered}
\end{equation*}
which further implies
\begin{equation}\label{22030401}
\begin{gathered}
e^{2\pi\sqrt{-1}n_{12}}=e^{(r'_1a_1+s'_1a_2)\xi_{m_1}+(r'_1b_1+s'_1b_2)\xi_{l_1}-l_1(r_1\xi_{m_1}+s_1\xi_{l_1})}e^{(r'_1c_1+s'_1c_2)\xi_{m_2}+(r'_1d_1+s'_1d_2)\xi_{l_2}-l_2(r_2\xi_{m_2}+s_2\xi_{l_2})},\\
e^{2\pi\sqrt{-1}n_{34}}=e^{(r'_2e_1+s'_2e_2)\xi_{m_1}+(r'_2f_1+s'_2f_2)\xi_{l_1}-l_3(r_1\xi_{m_1}+s_1\xi_{l_1})}e^{(r'_2g_1+s'_2g_2)\xi_{m_2}+(r'_2h_1+s'_2h_2)\xi_{l_2}-l_4(r_2\xi_{m_2}+s_2\xi_{l_2})}.
\end{gathered}
\end{equation}
By \eqref{22070911},   
\begin{equation}\label{22030408}
e^{(r'_1a_1+s'_1a_2)\xi_{m_1}+(r'_1b_1+s'_1b_2)\xi_{l_1}-l_1(r_1\xi_{m_1}+s_1\xi_{l_1})} \quad\text{and}\quad e^{(r'_1c_1+s'_1c_2)\xi_{m_2}+(r'_1d_1+s'_1d_2)\xi_{l_2}-l_2(r_2\xi_{m_2}+s_2\xi_{l_2})} 
\end{equation}
are powers of $t_1$ and $t_2$ respectively, and, as $t_1, t_2$ are multiplicatively independent each other, both elements in \eqref{22030408} are indeed roots of unity by the first equation in \eqref{22030401}. 
That is,  
\begin{equation}\label{21012105}
\begin{gathered}
2\pi\sqrt{-1}n_{1}=(r'_1a_1+s'_1a_2)\xi_{m_1}+(r'_1b_1+s'_1b_2)\xi_{l_1}-l_1(r_1\xi_{m_1}+s_1\xi_{l_1}),\\
2\pi\sqrt{-1}n_{2}=(r'_1c_1+s'_1c_2)\xi_{m_2}+(r'_1d_1+s'_1d_2)\xi_{l_2}-l_2(r_2\xi_{m_2}+s_2\xi_{l_2}) 
\end{gathered}
\end{equation}
for some $n_{1}, n_{2}\in \mathbb{Q}$ satisfying $n_{1}+n_{2}=n_{12}$, and so 
\begin{equation*}
\begin{gathered}
n_{1}(p_1\xi_{m_1}+q_1\xi_{l_1})=(r'_1a_1+s'_1a_2)\xi_{m_1}+(r'_1b_1+s'_1b_2)\xi_{l_1}-l_1(r_1\xi_{m_1}+s_1\xi_{l_1}),\\
n_{2}(p_2\xi_{m_2}+q_2\xi_{l_2})=(r'_1c_1+s'_1c_2)\xi_{m_2}+(r'_1d_1+s'_1d_2)\xi_{l_2}-l_2(r_2\xi_{m_2}+s_2\xi_{l_2}). 
\end{gathered}
\end{equation*}

Similarly, we get
\begin{equation}\label{21012106}
\begin{gathered}
2\pi\sqrt{-1}n_{3}=n_{3}(p_1\xi_{m_1}+q_1\xi_{l_1})=(r'_2e_1+s'_2e_2)\xi_{m_1}+(r'_2f_1+s'_2f_2)\xi_{l_1}-l_3(r_1\xi_{m_1}+s_1\xi_{l_1}),\\
2\pi\sqrt{-1}n_{4}=n_{4}(p_2\xi_{m_2}+q_2\xi_{l_2})=(r'_2g_1+s'_2g_2)\xi_{m_2}+(r'_2h_1+s'_2h_2)\xi_{l_2}-l_4(r_2\xi_{m_2}+s_2\xi_{l_2})
\end{gathered}
\end{equation}
for some $n_{3}, n_{4}\in \mathbb{Q}$ satisfying $n_{3}+n_{4}=n_{34}$ from the second equation in \eqref{22030401}. 

Remark that if $a\xi_{m_{i}}+b\xi_{l_{i}}=0$ ($i=1,2$) where $a,b\in \mathbb{Q}$, then $a=b=0$ (since $p_i\xi_{m_{i}}+q_i\xi_{l_{i}}=-2\pi \sqrt{-1}$ and $e^{\xi_{m_{i}}}, e^{\xi_{l_{i}}}$ are not roots of unity). Thus \eqref{21012105} and \eqref{21012106} imply
\begin{equation}\label{21012107}
\begin{gathered}
n_{1}p_1+l_1r_1=r'_1a_1+s'_1a_2, \quad 
n_{1}q_1+l_1s_1=r'_1b_1+s'_1b_2,\\
n_{2}p_2+l_2r_2=r'_1c_1+s'_1c_2,\quad 
n_{2}q_2+l_2s_2=r'_1d_1+s'_1d_2
\end{gathered}
\end{equation}
and 
\begin{equation}\label{21012108}
\begin{gathered}
n_{3}p_1+l_3r_1=r'_2e_1+s'_2e_2, \quad 
n_{3}q_1+l_3s_1=r'_2f_1+s'_2f_2,\\
n_{4}p_2+l_4r_2=r'_2g_1+s'_2g_2,\quad 
n_{4}q_2+l_4s_2=r'_2h_1+s'_2h_2
\end{gathered}
\end{equation}
respectively. Finally, \eqref{22031407} is attained by \eqref{20092001}, \eqref{21012107} and \eqref{21012108}. 

Since 
\begin{equation*}
\det\left(\begin{array}{cc}
p_{i} & q_{i} \\
r_{i} & s_{i}
\end{array}\right)=\det\left(\begin{array}{cc}
p'_{i} & q'_{i} \\
r'_{i} & s'_{i}
\end{array}\right)=1 \quad (i=1,2),
\end{equation*}
$k_jl_j=\det A_j$ follows ($1\leq j\leq 4$) and, together with \eqref{22030501}, \eqref{20101101} is obtained.  
\end{proof}

Lastly, the second and third statements of Theorem \ref{22031403} are established in the following:   
\begin{lemma} \label{22050205}
Adapting the same notation and assumptions as in Theorem \ref{22031403}, 
\begin{enumerate}
\item if $\det A_1=0$ or $\det A_4=0$, then 
\begin{equation*}
A_1=A_4=\left(\begin{array}{cc}
0 & 0\\
0 & 0
\end{array}\right), \quad\det A_2=\det A_3=1;
\end{equation*} 
\item if $\det A_2=0$ or $\det A_3=0$, then 
\begin{equation*}
A_2=A_3=\left(\begin{array}{cc}
0 & 0\\
0 & 0
\end{array}\right), \quad \det A_1=\det A_4=1.
\end{equation*} 
\end{enumerate}
\end{lemma}

Note that if two cusps of $\mathcal{M}$ are SGI each other, by Theorem \ref{22012104}, either $A_1=A_4=\left(\begin{array}{cc}
0 & 0\\
0 & 0\\
\end{array}\right)$ or $A_2=A_3=\left(\begin{array}{cc}
0 & 0\\
0 & 0\\
\end{array}\right)$ is always true. That is, the lemma holds in this case, and thus we assume two cusps of $\mathcal{M}$ are not SGI each other in the following proof. 
\begin{proof}
We only prove the first case, since the proof of the second case is analogous. 

Recall from Lemma \ref{20092003} that if $\det A_1=0$, then $A_1=\left(\begin{array}{cc}
0 & 0\\
0 & 0\\
\end{array}\right)$. We now claim 
\begin{claim}\label{21072507}
If $A_1=\left(\begin{array}{cc}
0 & 0\\
0 & 0\\
\end{array}\right)$, then $\det A_4=0$ (and so $A_4=\left(\begin{array}{cc}
0 & 0\\
0 & 0\\
\end{array}\right)$ by Lemma \ref{20092003}). 
\end{claim}
\begin{proof}[Proof of Claim \ref{21072507}]
On the contrary, suppose $\det A_4\neq 0$. Let 
\begin{equation*}
\Phi(u_1, u_2)=\tau_1u_1^2+\tau_2 u_2^2+2\sum_{\substack{\alpha+\beta=k,\\ \alpha,\beta:even}} m_{\alpha,\beta} u_1^{\alpha}u_2^{\beta}+\cdots
\end{equation*}
be the Neumann-Zagier potential function of $\mathcal{M}$, thus $v_1$ and $v_2$ are given as
\begin{equation}\label{22072005}
\begin{gathered}
\tau_1u_1+\sum_{\substack{\alpha+\beta=k,\\ \alpha,\beta:even}} \alpha m_{\alpha,\beta}  u_1^{\alpha-1}u_2^{\beta}+\cdots\quad\text{and}\quad \tau_2u_2+\sum_{\substack{\alpha+\beta=k,\\ \alpha,\beta:even}}\beta m_{\alpha,\beta} u_1^{\alpha}u_2^{\beta-1}+\cdots
\end{gathered}
\end{equation}
respectively. Since 
\begin{equation}\label{22072003}
\left(\begin{array}{c}
u'_1\\
v'_1\\
u'_2\\
v'_2
\end{array}\right)
=\left(\begin{array}{cccc}
0 & 0 & c_1 & d_1\\
0 & 0 & c_2 & d_2\\
e_1 & f_1 & g_1 & h_1\\
e_2 & f_2 & g_2 & h_2 
\end{array}\right)
\left(\begin{array}{c}
u_1\\
v_1\\
u_2\\
v_2
\end{array}\right)
\end{equation}
defines a $2$-dim anomalous analytic subset of $\log (\mathcal{X}\times \mathcal{X})$, plugging 
\begin{equation*}
\begin{gathered}
u'_1=c_1u_2+d_1v_2 \quad \text{and}\quad u'_2=e_1u_1+f_1v_1+g_1u_2+h_1v_2
\end{gathered}
\end{equation*}
into 
\begin{equation*}
v'_1=\tau_1u'_1+\sum_{\substack{\alpha+\beta=k,\\ \alpha,\beta:even}} m_{\alpha,\beta} \alpha(u'_1)^{\alpha-1}(u'_2)^{\beta}+\cdots,
\end{equation*}
we get 
\begin{equation*}
\begin{gathered}
v'_1=\tau_1(c_1u_2+d_1v_2)+\sum_{\substack{\alpha+\beta=k,\\ \alpha,\beta:even}} \alpha m_{\alpha,\beta}(c_1u_2+d_1v_2)^{\alpha-1}(e_1u_1+f_1v_1+g_1u_2+h_1v_2)^{\beta}+\cdots 
\end{gathered}
\end{equation*}
and, as $v'_1=c_2u_2+d_2v_2$ in \eqref{22072003}, the following equality  
\begin{equation}\label{22072006}
\begin{gathered}
\tau_1(c_1u_2+d_1v_2)+\sum_{\substack{\alpha+\beta=k,\\ \alpha,\beta:even}} \alpha m_{\alpha,\beta}(c_1u_2+d_1v_2)^{\alpha-1}(e_1u_1+f_1v_1+g_1u_2+h_1v_2)^{\beta}+\cdots
=c_2u_2+d_2v_2.
\end{gathered}
\end{equation}
Thanks to \eqref{22072005}, \eqref{22072006} is further expanded as 
\small
\begin{equation}\label{21072510}
\begin{aligned}
&\tau_1\Big(c_1u_2+d_1\big(\tau_2u_2+\sum_{\substack{\alpha+\beta=k,\\ \alpha,\beta:even}} \beta m_{\alpha,\beta} u_1^{\alpha}u_2^{\beta-1}+\cdots \big)\Big)+\sum_{\substack{\alpha+\beta=k,\\ \alpha,\beta:even}} \bigg(\alpha m_{\alpha,\beta}\Big(c_1u_2+d_1\big(\tau_2u_2+\sum_{\substack{\alpha+\beta=k,\\ \alpha,\beta:even}} \beta m_{\alpha,\beta}  u_1^{\alpha}u_2^{\beta-1}+\cdots\big) \Big)^{\alpha-1}\\ 
&\times \Big(e_1u_1+f_1\big(\tau_1u_1+\sum_{\substack{\alpha+\beta=k,\\ \alpha,\beta:even}} \alpha m_{\alpha,\beta}  u_1^{\alpha-1}u_2^\beta+\cdots\big)
+g_1u_2+h_1\big(\tau_2u_2+\sum_{\substack{\alpha+\beta=k,\\ \alpha,\beta:even}} \beta m_{\alpha,\beta} u_1^{\alpha}u_2^{\beta-1}+\cdots\big) \Big)^{\beta}\bigg)+\cdots\\
=&c_2u_2+d_2\big(\tau_2u_2+\sum_{\substack{\alpha+\beta=k,\\ \alpha,\beta:even}} \beta m_{\alpha,\beta} u_1^{\alpha}u_2^{\beta-1}+\cdots\big), 
\end{aligned}
\end{equation}
\normalsize
and, extracting the terms of homogeneous degree $k-1$ in \eqref{21072510}, we obtain
\begin{equation*}
\begin{gathered}
\tau_1d_1\sum_{\substack{\alpha+\beta=k,\\ \alpha,\beta:even}} \beta m_{\alpha,\beta} u_1^{\alpha}u_2^{\beta-1}+\sum_{\substack{\alpha+\beta=k,\\ \alpha,\beta:even}} \alpha m_{\alpha,\beta}(c_1u_2+d_1\tau_2u_2)^{\alpha-1}(e_1u_1+f_1\tau_1u_1+g_1u_2+h_1\tau_2u_2)^{\beta}\\
=d_2\sum_{\substack{\alpha+\beta=k,\\ \alpha,\beta:even}} \beta m_{\alpha,\beta} u_1^{\alpha}u_2^{\beta-1},
\end{gathered}
\end{equation*}
which is equal to 
\begin{equation}\label{21072509}
\begin{gathered}
\sum_{\substack{\alpha+\beta=k,\\ \alpha,\beta:even}} \alpha m_{\alpha,\beta}(c_1u_2+d_1\tau_2u_2)^{\alpha-1}\big((e_1+f_1\tau_1)u_1+(g_1+h_1\tau_2)u_2\big)^{\beta}=(d_2-d_1\tau_1)\sum_{\substack{\alpha+\beta=k,\\ \alpha,\beta:even}} \beta m_{\alpha,\beta} u_1^{\alpha}u_2^{\beta-1}.
\end{gathered}
\end{equation}
\normalfont
If $g_1+h_1\tau_2\neq 0$, the left side in \eqref{21072509} contains a non-trivial term of the form $u_1^{\alpha-1}u_2^{\beta}$ with $\alpha,\beta$ even, however every term in the right side of \eqref{21072509} is of the form $u_1^{\alpha}u_2^{\beta-1}$ with $\alpha,\beta$ even. Thus we get a contradiction and conclude $g_1=h_1=0$. This completes the proof of Claim \ref{21072507}. 
\end{proof}

If $A_1=A_4=\left(\begin{array}{cc}
0 & 0\\
0 & 0\\
\end{array}\right)$, then $k_1=k_4=l_1=l_4=0$ by \eqref{22031407} and so 
$k_2=k_3=l_2=l_3=\det A_2=\det A_3=1$ by \eqref{22030403}-\eqref{20101101}. 

Similarly if $\det A_4=0$, one can show $A_1=A_4=\left(\begin{array}{cc}
0 & 0\\
0 & 0
\end{array}\right)$ and $\det A_2=\det A_3=1$. This completes the proof of Lemma \ref{22050205}. 
\end{proof}

Combining Lemmas \ref{20092003}-\ref{20101102} with Lemma \ref{22050205}, Theorem \ref{22031403} is established by the same methods above. 

Recall that Theorem \ref{22031403} is attained by letting 
$\left(\begin{array}{cc}
a'_1 & b'_1\\
a'_2 & b'_2
\end{array}\right)$ and 
$\left(\begin{array}{cc}
g'_1 & h'_1\\
g'_2 & h'_2
\end{array}\right)$ be the identity matrix in \eqref{22022007}. The general statement without this assumption is given as below. We use it later to prove Theorem \ref{20103003} in Section \ref{22041806}. 
\begin{corollary}\label{22010903}
Let $\mathcal{M}$ and $\mathcal{X}$ be as usual. Let $\mathcal{M}_{p_{1}/q_{1},p_{2}/q_{2}}$ and $\mathcal{M}_{p'_{1}/q'_{1},p'_{2}/q'_{2}}$ be two Dehn fillings of $\mathcal{M}$ satisfying 
\begin{equation}\label{22030507}
\text{pvol}_{\mathbb{C}}\;\mathcal{M}_{p_{1}/q_{1},p_{2}/q_{2}}=\text{pvol}_{\mathbb{C}}\;\mathcal{M}_{p'_{1}/q'_{1},p'_{2}/q'_{2}}
\end{equation}
with $|p_{i}|+|q_{i}|$ and $|p'_{i}|+|q'_{i}|$ ($i=1,2$) sufficiently large. Further, suppose $t_{1}, t_{2}$ (resp. $t'_1, t'_2$) are multiplicatively independent (resp. independent), and a Dehn filling point $P$ associated to \eqref{22030507} is contained in an analytic variety defined by \eqref{20091903}. Let 
\begin{equation*}
A'_1:=\left(\begin{array}{cc}
a'_1 & b'_1\\
a'_2 & b'_2
\end{array}\right), \quad A'_4:=\left(\begin{array}{cc}
g'_1 & h'_1\\
g'_2 & h'_2
\end{array}\right), 
\end{equation*}
and $A_j$ ($1\leq j\leq 4$) be the same as in \eqref{22022010}. Then there are $k_j\in\mathbb{Q}$ $(1\leq j\leq 4)$ satisfying 
\begin{equation}\label{22051901}
\begin{gathered}
\left(\begin{array}{cc}
p'_{1} & q'_{1} 
\end{array}\right)(A'_1)^{-1}A_1=
k_{1}\left(\begin{array}{cc}
p_{1} & q_{1} 
\end{array}\right), \quad 
\left(\begin{array}{cc}
p'_{1} & q'_{1} 
\end{array}\right)(A'_1)^{-1}A_2=
k_{2} \left(\begin{array}{cc}
p_{2} & q_{2} 
\end{array}\right),\\ 
\left(\begin{array}{cc}
p'_{2} & q'_{2} 
\end{array}\right)(A'_4)^{-1}A_3=
k_{3} \left(\begin{array}{cc}
p_{1} & q_{1} 
\end{array}\right), \quad 
\left(\begin{array}{cc}
p'_{2} & q'_{2} 
\end{array}\right)(A'_4)^{-1}A_4
=k_{4} \left(\begin{array}{cc}
p_{2} & q_{2} 
\end{array}\right),
\end{gathered}
\end{equation}
and, for $\det A_j\neq 0$ $(1\leq j\leq 4)$, it follows that 
\begin{equation}\label{22080406}
A_1^{-1}A_2 =\dfrac{k_{2}k_{3}}{k_{1}k_{4}}A_3^{-1}A_4.
\end{equation}
\end{corollary}
\begin{proof}
The first claim follows immediately from Theorem \ref{22031403}, and hence we only prove \eqref{22080406}. 

If $\det A_j\neq 0$ ($1\leq j\leq 4$), then 
\begin{equation*}
k_{1} \left(\begin{array}{cc}
p_{1} & q_{1} 
\end{array}\right)A_1^{-1}=k_{2} 
\left(\begin{array}{cc}
p_{2} & q_{2} 
\end{array}\right)A_2^{-1},\quad
k_{3}\left(\begin{array}{cc}
p_{1} & q_{1} 
\end{array}\right)A_3^{-1}=
k_{4} 
\left(\begin{array}{cc}
p_{2} & q_{2} 
\end{array}\right)A_4^{-1}
\end{equation*}
by \eqref{22051901}, and so
\begin{equation*}
\left(\begin{array}{cc}
p_{2} & q_{2} 
\end{array}\right)(A_1^{-1}A_2)^{-1}A_3^{-1}A_4=
\frac{k_{1}k_{4}}{k_{2}k_{3}} 
\left(\begin{array}{cc}
p_{2} & q_{2} 
\end{array}\right).
\end{equation*}
That is, $\dfrac{k_{1}k_{4}}{k_{2}k_{3}}$ is an eigenvalue of $\big((A_1^{-1}A_2)^{-1}A_3^{-1}A_4\big)^{T}$ and $ 
\left(\begin{array}{cc}
p_{2} & q_{2} 
\end{array}\right)$ is an eigenvector associated to it. Since $p_2$ and $q_2$ are co-prime to each other with $|p_2|+|q_2|$ sufficiently large, we conclude the dimension of the eigenvector space associated to $\dfrac{k_{1}k_{4}}{k_{2}k_{3}}$ is $2$, implying
\begin{equation*}
(A_1^{-1}A_2)^{-1}A_3^{-1}A_4=\dfrac{k_{1}k_{4}}{k_{2}k_{3}}I.
\end{equation*}
\end{proof}

\newpage
\section{Quantification II}\label{20121903}
This section is a continuation of the previous section and we further investigate another effective version of Theorem \ref{22050203}. 

Adapting the same notation given in Section \ref{quantaI}, let $A_j$ ($1\leq j\leq 4$) and $M$ be the same as \eqref{22022010} and \eqref{21072901} respectively. Let $\mathcal{Y}$ be an anomalous analytic subset of $\log(\mathcal{X}\times \mathcal{X})$ of $\dim 2$ defined by 
\begin{equation}\label{22022101}
\left(\begin{array}{c}
u'_1\\
v'_1\\
u'_2\\
v'_2
\end{array}\right)
=\left(\begin{array}{cccc}
a_1 & b_1 & c_1 & d_1\\
a_2 & b_2 & c_2 & d_2\\
e_1 & f_1 & g_1 & h_1\\
e_2 & f_2 & g_2 & h_2 
\end{array}\right)\left(\begin{array}{c}
u_1\\
v_1\\
u_2\\
v_2
\end{array}\right)
\Bigg(=M\left(\begin{array}{c}
u_1\\
v_1\\
u_2\\
v_2
\end{array}\right)\Bigg).
\end{equation}
Then the second quantification of Theorem \ref{22050203} is stated as below: 
\begin{theorem}\label{20091911}
Let $\mathcal{M}, \mathcal{X}$ be the same as in Theorem \ref{22050203}. Let $\mathcal{M}_{p_{1}/q_{1},p_{2}/q_{2}}$ and $\mathcal{M}_{p'_{1}/q'_{1},p'_{2}/q'_{2}}$ be two Dehn fillings of $\mathcal{M}$ having the same pseudo complex volume with $|p_k|+|q_k|$ and $|p'_k|+|q'_k|$ sufficiently large ($k=1,2$). Moreover, we suppose the core holonomies $t_1, t_2$ (resp. $t'_1, t'_2$) of $\mathcal{M}_{p_{1}/q_{1},p_{2}/q_{2}}$ (resp. $\mathcal{M}_{p'_{1}/q'_{1},p'_{2}/q'_{2}}$) are multiplicatively independent and a Dehn filling point associated to the pair is contained in $\mathcal{Y}$ as defined above. If $\det A_j\neq 0$ ($1\leq j\leq 4$),\footnote{By Theorem \ref{22012104}, if two cusps of $\mathcal{M}$ are SGI each other, then we have either $A_1=A_4=\left(\begin{array}{cc}
0 & 0\\
0 & 0
\end{array}\right)$ or $A_2=A_3=\left(\begin{array}{cc}
0 & 0\\
0 & 0
\end{array}\right)$. Thus the condition (i.e. $\det A_j\neq 0$ ($1\leq j\leq 4$)) in the theorem implies that two cusps of $\mathcal{M}$ is not SGI each other.} then 
\begin{equation}\label{22052309}
\det A_1=\det A_4, \quad \det A_2=\det A_3,\quad \det A_1+\det A_3=1, \quad (\det A_1) A_1^{-1}A_2=-(\det A_3)A_3^{-1}A_4.
\end{equation}
\end{theorem}
This result, together with Theorem \ref{22031403}, will be used to show Theorem \ref{22041805} in Section \ref{PCV}. 

The proof of Theorem \ref{20091911} is based on symmetric properties of $\log\mathcal{X}$ described in Theorem \ref{potential}. To explain this more in detail, first let us consider $u'_1$ and $u'_2$ as functions of $u_1$ and $u_2$ by \eqref{22022101}. Then we get
\begin{equation}\label{22052501}
\begin{gathered}
\dfrac{\partial v'_1}{\partial u_1}=\dfrac{\partial v'_1}{\partial u'_1} \dfrac{\partial u'_1}{\partial u_1} + \dfrac{\partial v'_1}{\partial u'_2}\dfrac{\partial u'_2}{\partial u_1},\quad  
\dfrac{\partial v'_1}{\partial u_2} =\dfrac{\partial v'_1}{\partial u'_1} \dfrac{\partial u'_1}{\partial u_2}+\dfrac{\partial v'_1}{\partial u'_2}
\dfrac{\partial u'_2}{\partial u_2} \\
\dfrac{\partial v'_2}{\partial u_1}= \dfrac{\partial v'_2}{\partial u'_1}\dfrac{\partial u'_1}{\partial u_1}+\dfrac{\partial v'_2}{\partial u'_2}\dfrac{\partial u'_2}{\partial u_1},\quad  
\dfrac{\partial v'_2}{\partial u_2}=
\dfrac{\partial v'_2}{\partial u'_1}\dfrac{\partial u'_1}{\partial u_2} + 
\dfrac{\partial v'_2}{\partial u'_2}\dfrac{\partial u'_2}{\partial u_2} 
\end{gathered}
\end{equation}
by the chain rule,
\begin{equation}\label{22052301}
\begin{gathered}
\dfrac{\partial v'_1}{\partial u_1}=a_2+b_2\dfrac{\partial v_1}{\partial u_1}+d_2\dfrac{\partial v_2}{\partial u_1}, \quad 
\dfrac{\partial v'_1}{\partial u_2}=b_2\dfrac{\partial v_1}{\partial u_2}+c_2+d_2\dfrac{\partial v_2}{\partial u_2},\\
\dfrac{\partial v'_2}{\partial u_1}=e_2+f_2\dfrac{\partial v_1}{\partial u_1}+h_2\dfrac{\partial v_2}{\partial u_1},\quad 
\dfrac{\partial v'_2}{\partial u_2}=f_2\dfrac{\partial v_1}{\partial u_2}+g_2+h_2\dfrac{\partial v_2}{\partial u_2}
\end{gathered}
\end{equation}
by \eqref{22022101}, and  
\begin{equation}\label{21072201}
\begin{gathered}
\left(\begin{array}{c}
\dfrac{\partial v'_1}{\partial u_1} \\
\dfrac{\partial v'_1}{\partial u_2} 
\end{array}\right)
=
\left(\begin{array}{cc}
\dfrac{\partial u'_1}{\partial u_1} & \dfrac{\partial u'_2}{\partial u_1} \\
\dfrac{\partial u'_1}{\partial u_2} & \dfrac{\partial u'_2}{\partial u_2} 
\end{array}\right)\left(\begin{array}{c}
\dfrac{\partial v'_1}{\partial u'_1} \\
\dfrac{\partial v'_1}{\partial u'_2}
\end{array}\right)=\left(\begin{array}{c}
a_2+b_2\dfrac{\partial v_1}{\partial u_1}+d_2\dfrac{\partial v_2}{\partial u_1} \\
b_2\dfrac{\partial v_1}{\partial u_2}+c_2+d_2\dfrac{\partial v_2}{\partial u_2}
\end{array}\right),\\
\left(\begin{array}{c}
\dfrac{\partial v'_2}{\partial u_1} \\
\dfrac{\partial v'_2}{\partial u_2} 
\end{array}\right)
=
\left(\begin{array}{cc}
\dfrac{\partial u'_1}{\partial u_1} & \dfrac{\partial u'_2}{\partial u_1} \\
\dfrac{\partial u'_1}{\partial u_2} & \dfrac{\partial u'_2}{\partial u_2} 
\end{array}\right)\left(\begin{array}{c}
\dfrac{\partial v'_2}{\partial u'_1} \\
\dfrac{\partial v'_2}{\partial u'_2}
\end{array}\right)=\left(\begin{array}{c}
e_2+f_2\dfrac{\partial v_1}{\partial u_1}+h_2\dfrac{\partial v_2}{\partial u_1} \\
f_2\dfrac{\partial v_1}{\partial u_2}+g_2+h_2\dfrac{\partial v_2}{\partial u_2}
\end{array}\right)
\end{gathered}
\end{equation}
by \eqref{22052501}-\eqref{22052301}. Denoting
\begin{equation}\label{21072202}
\left(\begin{array}{cc}
\dfrac{\partial u'_1}{\partial u_1} & \dfrac{\partial u'_2}{\partial u_1} \\
\dfrac{\partial u'_1}{\partial u_2} & \dfrac{\partial u'_2}{\partial u_2} 
\end{array}\right)
\end{equation} 
by $J$, the proof of Theorem \ref{20091911} splits into two parts, depending on whether $\det J=0$ or not. 

Before proceeding further, we give the following, which is seen as a generalization of Definition \ref{22070813}.  
\begin{definition}
\normalfont Let $\mathbb{C}[[u_1, u_2]]$ the set of power series of two variables $u_1, u_2$.  For $f_i, g_i\in\mathbb{C}[[u_1, u_2]]$ ($i=1,2$), we say two vectors $\left(\begin{array}{c}
f_1 \\
f_2 
\end{array}\right)$ and 
$\left(\begin{array}{c}
g_1  \\
g_2 
\end{array}\right)$ are \textit{parallel} to each other if $f_1g_2=f_2g_1$, and represent it by $\left(\begin{array}{c}
f_1 \\
f_2 
\end{array}\right)
\newparallel 
\left(\begin{array}{c}
g_1  \\
g_2 
\end{array}\right)$.
\end{definition}

\begin{enumerate}
\item First, if $\det J=0$, we get
\begin{equation}\label{22052310}
\left(\begin{array}{c}
\dfrac{\partial u'_1}{\partial u_1} \\
\dfrac{\partial u'_1}{\partial u_2} 
\end{array}\right)\newparallel
\left(\begin{array}{c}
\dfrac{\partial u'_2}{\partial u_1} \\
\dfrac{\partial u'_2}{\partial u_2} 
\end{array}\right)
\end{equation}
by the definition of $J$, and so, by \eqref{21072201}, 
\begin{equation}\label{21072206}
\begin{gathered}
\left(\begin{array}{c}
\dfrac{\partial u'_1}{\partial u_1} \\
\dfrac{\partial u'_1}{\partial u_2} 
\end{array}\right)\newparallel
\left(\begin{array}{c}
\dfrac{\partial u'_2}{\partial u_1} \\
\dfrac{\partial u'_2}{\partial u_2} 
\end{array}\right)\newparallel
\left(\begin{array}{c}
\dfrac{\partial v'_1}{\partial u_1} \\
\dfrac{\partial v'_1}{\partial u_2} 
\end{array}\right)\newparallel
\left(\begin{array}{c}
\dfrac{\partial v'_2}{\partial u_1} \\
\dfrac{\partial v'_2}{\partial u_2} 
\end{array}\right)\\
\Longrightarrow 
\small
\left(\begin{array}{c}
a_1+b_1\dfrac{\partial v_1}{\partial u_1}+d_1\dfrac{\partial v_2}{\partial u_1}\\
b_1\dfrac{\partial v_1}{\partial u_2}+c_1+d_1\dfrac{\partial v_2}{\partial u_2}
\end{array}\right)
\newparallel
\left(\begin{array}{c}
e_1+f_1\dfrac{\partial v_1}{\partial u_1}+h_1\dfrac{\partial v_2}{\partial u_1}\\
f_1\dfrac{\partial v_1}{\partial u_2}+g_1+h_1\dfrac{\partial v_2}{\partial u_2}
\end{array}\right)
\newparallel
\left(\begin{array}{c}
a_2+b_2\dfrac{\partial v_1}{\partial u_1}+d_2\dfrac{\partial v_2}{\partial u_1}\\
b_2\dfrac{\partial v_1}{\partial u_2}+c_2+d_2\dfrac{\partial v_2}{\partial u_2}
\end{array}\right)
\newparallel
\left(\begin{array}{c}
e_2+f_2\dfrac{\partial v_1}{\partial u_1}+h_2\dfrac{\partial v_2}{\partial u_1}\\
f_2\dfrac{\partial v_1}{\partial u_2}+g_2+h_2\dfrac{\partial v_2}{\partial u_2}
\end{array}\right).
\end{gathered}
\end{equation}

\item If $\det J\neq 0$, then
\begin{equation}
\begin{gathered}
\left(\begin{array}{c}
\dfrac{\partial v'_1}{\partial u'_1} \\
\dfrac{\partial v'_1}{\partial u'_2}
\end{array}\right)=J^{-1}
\left(\begin{array}{c}
a_2+b_2\dfrac{\partial v_1}{\partial u_1}+d_2\dfrac{\partial v_2}{\partial u_1} \\
b_2\dfrac{\partial v_1}{\partial u_2}+c_2+d_2\dfrac{\partial v_2}{\partial u_2}
\end{array}\right),\quad 
\left(\begin{array}{c}
\dfrac{\partial v'_2}{\partial u'_1} \\
\dfrac{\partial v'_2}{\partial u'_2}
\end{array}\right)=J^{-1}
\left(\begin{array}{c}
e_2+f_2\dfrac{\partial v_1}{\partial u_1}+h_2\dfrac{\partial v_2}{\partial u_1} \\
f_2\dfrac{\partial v_1}{\partial u_2}+g_2+h_2\dfrac{\partial v_2}{\partial u_2}
\end{array}\right)
\end{gathered}
\end{equation}
by \eqref{21072201}. As 
\begin{equation*}
\dfrac{\partial v'_1}{\partial u'_2}
=\dfrac{\partial v'_2}{\partial u'_1}
\end{equation*}
by Theorem \ref{potential}, it follows that 
\begin{equation}\label{22052505}
\begin{gathered}
\left(\begin{array}{cc}
-\dfrac{\partial u'_1}{\partial u_2} & \dfrac{\partial u'_1}{\partial u_1} 
\end{array}\right)
\left(\begin{array}{c}
a_2+b_2\dfrac{\partial v_1}{\partial u_1}+d_2\dfrac{\partial v_2}{\partial u_1} \\
b_2\dfrac{\partial v_1}{\partial u_2}+c_2+d_2\dfrac{\partial v_2}{\partial u_2}
\end{array}\right)
=\left(\begin{array}{cc}
\dfrac{\partial u'_2}{\partial u_2} & -\dfrac{\partial u'_2}{\partial u_1} 
\end{array}\right)
\left(\begin{array}{c}
e_2+f_2\dfrac{\partial v_1}{\partial u_1}+h_2\dfrac{\partial v_2}{\partial u_1} \\
f_2\dfrac{\partial v_1}{\partial u_2}+g_2+h_2\dfrac{\partial v_2}{\partial u_2}
\end{array}\right).
\end{gathered}
\end{equation}
\end{enumerate}

Later in Theorem \ref{21072701} (resp. Theorem \ref{21071405} and Proposition \ref{21072601}), it will be shown that \eqref{22052309} is deduced from \eqref{21072206} (resp.  \eqref{22052505}) for $\det J=0$ (resp. $\det J\neq 0$). 

The following four lemmas shall be required for the proofs of the aforementioned theorems and proposition. 

\begin{lemma}\label{21072210}
Having the same assumptions as in Theorem \ref{20091911}, if $\det A_j\neq 0$ ($1\leq j\leq 4$), then $\det A_j>0$.
\end{lemma}
\begin{proof}
By Corollary \ref{20121907}, $\dfrac{1}{\tau_1}=\dfrac{a_1+b_1\tau_1}{a_2+b_2\tau_1}$ and so $b_1\tau_1^2+(a_1-b_2)\tau_1-a_2=0.$ Since $\tau_1$ is a non-real complex number, $(a_1-b_2)^2+4a_2b_1=(a_1+b_2)^2-4\det A_1<0$, implying $\det A_1>0$. 

Let $\tau_2=m+n\sqrt{-1}$ where $n>0$. Then 
\begin{equation*}
\begin{aligned}
\tau_1&=\dfrac{c_2+d_2(m+n\sqrt{-1})}{c_1+d_1(m+n\sqrt{-1})}=\dfrac{\big(c_2+d_2(m+n\sqrt{-1})\big)\big(c_1+d_1(m-n\sqrt{-1})\big)}{(c_1+d_1m)^2+d_1^2n^2}\\
&=\dfrac{(c_2+d_2m)(c_1+d_1m)+d_1d_2n^2+\big((c_1+d_1m)d_2n-(c_2+d_2m)d_1n\big)\sqrt{-1}}{(c_1+d_1m)^2+d_1^2n^2}\\
&=\dfrac{(c_2+d_2m)(c_1+d_1m)+d_1d_2n^2+(c_1d_2-c_2d_1)n\sqrt{-1}}{(c_1+d_1m)^2+d_1^2n^2}.
\end{aligned}
\end{equation*}
Since $\text{Im}\;\tau_1>0$, we get $\dfrac{(c_1d_2-c_2d_1)n}{(c_1+d_1m)^2+d_1^2n^2}>0$, that is, $c_1d_2-c_2d_1=\det A_2>0$.

Similarly one can show $\det A_3>0$ and $\det A_4>0$.  
\end{proof}

\begin{lemma}\label{22062503}
Having the same assumptions and notation given in Theorem \ref{20091911}, we have either $\det M=1$ or $-1$. 
\end{lemma}
\begin{proof}
We first claim $\det M\neq 0$. Indeed, if $\det M=0$, by the definition of $M$,  
\begin{equation*}
au'_1+bv'_1+cu'_2+dv'_2=0
\end{equation*}
over $\mathcal{Y}$ for some non-trivial $(a, b, c, d)\in \mathbb{Z}^4$. As $t'_1, t'_2$ are multiplicatively independent, it implies either 
\begin{equation*}
(a,b)\neq (0,0),  (c,d)= (0,0)\quad \text{ or }\quad (a,b)=(0,0),  (c,d)\neq (0,0).
\end{equation*}
Without loss of generality, if we assume the first case, then $(a,b)\newparallel (p'_1,q'_1)$ and so $(p'_1, q'_1)$ is uniquely determined by $(a,b)$. But this contradicts the fact that $|p'_1|+|q'_1|$ is sufficiently large. 

Next we claim $\det M=\pm 1$. Note that if
\begin{equation*}
\left(\begin{array}{c}
u'_1\\
v'_1\\
u'_2\\
v'_2
\end{array}\right)
=M\left(\begin{array}{c}
u_1\\
v_1\\
u_2\\
v_2
\end{array}\right)
\end{equation*}
is an anomalous analytic subset of $\log (\mathcal{X}\times \mathcal{X})$ with $\det M\neq 0$, then  
\begin{equation*}
\left(\begin{array}{c}
u'_1\\
v'_1\\
u'_2\\
v'_2
\end{array}\right)
=M^n\left(\begin{array}{c}
u_1\\
v_1\\
u_2\\
v_2
\end{array}\right)
\end{equation*}
is also an anomalous analytic subset of $\log (\mathcal{X}\times \mathcal{X})$ for any $n\in \mathbb{Z}$. As $\log (\mathcal{X}\times \mathcal{X})$ has only finitely many anomalous subsets containing the origin by Theorem \ref{struc}, it follows that $M^n=I$ for some $n\in \mathbb{Z}$ and thus $\det M=\pm 1$. 
\end{proof}

The next two lemmas are well-known in linear algebra.
\begin{lemma}\label{21072204}
Let $A_j$ ($1\leq j\leq 4$) and $M$ be as given in \eqref{22022010}-\eqref{21072901}. If $A_1$ is invertible, then 
\begin{equation*}
\det M=\det A_1 \det (A_4-A_3A_1^{-1}A_2).
\end{equation*}
\end{lemma}
\begin{lemma}\label{21072603}
Let $A_j$ ($1\leq j\leq 4$) and $M$ be the same as in \eqref{22022010}-\eqref{21072901}. If $M$ and $A_1$ are invertible, then 
\begin{equation*}
M^{-1}=\left(\begin{array}{cc}
A_1^{-1}+A_1^{-1}A_2(M/A_1)^{-1}A_3A_1^{-1} & -A_1^{-1}A_2(M/A_1)^{-1}\\
-(M/A_1)^{-1}A_3A_1^{-1} & (M/A_1)^{-1}
\end{array}
\right)
\end{equation*}
where $M/A_1:=A_4-A_3A_1^{-1}A_2$.  
\end{lemma}

\subsection{$\det J=0$} 
We first suppose $\det J=0$ and claim  
\begin{theorem}\label{21072701}
Under the same assumptions as in Theorem \ref{20091911}, if $\det J= 0$, then Theorem \ref{20091911} holds. 
\end{theorem}
\begin{proof}
By the equality between the first and third fractions in \eqref{21072206}, 
\begin{equation*}
\begin{gathered}
\Big(a_1+b_1\dfrac{\partial v_1}{\partial u_1}+d_1\dfrac{\partial v_2}{\partial u_1}\Big)\Big(b_2\dfrac{\partial v_1}{\partial u_2}+c_2+d_2\dfrac{\partial v_2}{\partial u_2}\Big)=\Big(b_1\dfrac{\partial v_1}{\partial u_2}+c_1+d_1\dfrac{\partial v_2}{\partial u_2}\Big)\Big(a_2+b_2\dfrac{\partial v_1}{\partial u_1}+d_2\dfrac{\partial v_2}{\partial u_1}\Big)\\
\Longrightarrow 
(a_1c_2-c_1a_2)+(b_1c_2-c_1b_2)\dfrac{\partial v_1}{\partial u_1}+(a_1d_2-d_1a_2)\dfrac{\partial v_2}{\partial u_2}+(a_1b_2-a_2b_1-c_1d_2+c_2d_1)\dfrac{\partial v_1}{\partial u_2}\\+(b_1d_2-d_1b_2)\bigg(\dfrac{\partial v_1}{\partial u_1}\dfrac{\partial v_2}{\partial u_2}-\Big(\dfrac{\partial v_1}{\partial u_2}\Big)^2\bigg)=0.
\end{gathered}
\end{equation*}
As $\dfrac{\partial v_1}{\partial u_2}$ is the only one containing a term of the form $u_1^{\alpha}u_2^{\beta}$ with $\alpha,\beta$ odd among 
\begin{equation*}
\dfrac{\partial v_1}{\partial u_1},\quad \dfrac{\partial v_2}{\partial u_2},\quad \dfrac{\partial v_1}{\partial u_2},\quad\text{and}\quad   \dfrac{\partial v_1}{\partial u_1}\dfrac{\partial v_2}{\partial u_2}-\Big(\dfrac{\partial v_1}{\partial u_2}\Big)^2, 
\end{equation*} 
we have 
\begin{equation*}
a_1b_2-a_2b_1-c_1d_2+c_2d_1=0, 
\end{equation*}
implying
\begin{equation}\label{22010901}
\det A_1=\det A_2. 
\end{equation}
Similarly, 
\begin{equation}\label{22010902}
\det A_3=\det A_4 
\end{equation}
is attained from the equality between the second and fourth ones in \eqref{21072206}. 

As $A_1^{-1}A_2\newparallel A_3^{-1}A_4$ (by Corollary \ref{22010903}), either 
\begin{equation*}
A_1^{-1}A_2=A_3^{-1}A_4\quad  \text{or}\quad A_1^{-1}A_2=-A_3^{-1}A_4
\end{equation*}
holds by \eqref{22010901}-\eqref{22010902}. However, if $A_1^{-1}A_2=A_3^{-1}A_4$, then $\det M=0$ by Lemma \ref{21072204}, contradicting the conclusion of Lemma \ref{22062503}.  

For $A_1^{-1}A_2=-A_3^{-1}A_4$, by Theorem \ref{22031403} and Corollary \ref{22010903}, there exist $k_j\in \mathbb{Q}$ ($1\leq j\leq 4$) such that 
\begin{equation*}
\begin{gathered}
\dfrac{k_1k_4}{k_2k_3}=-1\quad\text{and}\quad k_1+k_2=k_3+k_4= \dfrac{\det A_1}{k_1}+\dfrac{\det A_3}{k_3}=\dfrac{\det A_2}{k_2}+\dfrac{\det A_4}{k_4}=1,
\end{gathered}
\end{equation*}
implying
\begin{equation}\label{21072209}
\begin{gathered}
k_1+k_2=k_3-\dfrac{k_2k_3}{k_1}= \dfrac{\det A_1}{k_1}+\dfrac{\det A_3}{k_3}=\dfrac{\det A_2}{k_2}-\dfrac{k_1\det A_4}{k_2k_3}=1\\
\Longrightarrow 
\dfrac{1}{k_3}=\dfrac{1}{\det A_3}\Big(1-\dfrac{\det A_1}{k_1}\Big)=-\dfrac{k_2}{k_1\det A_4}\Big(1-\dfrac{\det A_2}{k_2}\Big).
\end{gathered}
\end{equation}
Combining \eqref{21072209} with \eqref{22010901}-\eqref{22010902}, 
\begin{equation*}
\begin{gathered}
\dfrac{1}{\det A_3}\Big(1-\dfrac{\det A_1}{k_1}\Big)=-\dfrac{k_2}{k_1\det A_3}\Big(1-\dfrac{\det A_1}{k_2}\Big)
\end{gathered}
\end{equation*}
and so
\begin{equation}\label{22062501}
\begin{gathered}
1-\dfrac{\det A_1}{k_1}=-\dfrac{k_2}{k_1}\Big(1-\dfrac{\det A_1}{k_2}\Big)\Longrightarrow 1-\dfrac{\det A_1}{k_1}=\dfrac{k_1-1}{k_1}+\dfrac{\det A_1}{k_1}\Longrightarrow \det A_1=\det A_2=\dfrac{1}{2}.
\end{gathered}
\end{equation}
Since  
\begin{equation*}
\det M=\det A_1\det (-2A_3A_1^{-1}A_2)=4\det (A_2A_3)=\pm 1
\end{equation*} 
by Lemmas \ref{22062503}-\ref{21072204}, we get  
\begin{equation*}
4\det (A_2A_3)=1\Longrightarrow \det A_3=\det A_4=\dfrac{1}{2}
\end{equation*}
by Lemma \ref{21072210}, \eqref{22010902} and \eqref{22062501}. In conclusion,
\begin{equation*}
A_1^{-1}A_2=-A_3^{-1}A_4,\quad \det A_j=\dfrac{1}{2}\quad (1\leq j\leq 4), 
\end{equation*}
which clearly satisfy \eqref{22052309}. This completes the proof of Theorem \ref{21072701}. 
\end{proof}

\subsection{$\det J\neq 0$} 
Now we assume $\det J\neq 0$. Recall from the earlier discussion that, for $\det J\neq 0$, the equality \eqref{22052505} holds. The following theorem is derived from expanding this equality.
\begin{theorem}\label{21071405}
Under the same assumptions as in Theorem \ref{20091911}, if $J$ is invertible, then
\begin{equation}\label{21071403}
\det A_1+\det A_3=\det A_2+\det A_4.
\end{equation}
Further,
\begin{enumerate}
\item if $(\det A_1) A_1^{-1}A_2=-(\det A_3)A_3^{-1}A_4$, then 
\begin{equation*}
\det A_1=\det A_4, \;\; \det A_2=\det A_3, \;\;   \det A_1+\det A_3=1;
\end{equation*}
\item if $(\det A_1) A_1^{-1}A_2\neq -(\det A_3)A_3^{-1}A_4$, then 
\begin{equation}\label{21071404}
a+b\dfrac{\partial v_1}{\partial u_1}+c\dfrac{\partial v_2}{\partial u_2}+d\Big(\dfrac{\partial v_1}{\partial u_2}\dfrac{\partial v_2}{\partial u_1}-\dfrac{\partial v_1}{\partial u_1}\dfrac{\partial v_2}{\partial u_2}\Big)=0
\end{equation}
where $\left(\begin{array}{cc}
-b & d\\
a & c
\end{array}\right)=A_1^{-1}A_2$.
\end{enumerate}
\end{theorem}
 
Clearly the conclusion of Theorem \ref{21071405} (1) implies \eqref{22052309}. After the proof of Proposition \ref{21072601}, it will be shown that Theorem \ref{21071405} (2) never arises, thus completing the proof of Theorem \ref{20091911} for $\det J\neq 0$.  

\begin{proof}[Proof of Theorem \ref{21071405}]
Since $\det J\neq 0$, by \eqref{22052505}, we get 
\begin{equation}\label{22052509}
\begin{gathered}
-\dfrac{\partial u'_1}{\partial u_2}\Big(a_2+b_2\dfrac{\partial v_1}{\partial u_1}+d_2\dfrac{\partial v_2}{\partial u_1}\Big)+\dfrac{\partial u'_1}{\partial u_1}\Big(
b_2\dfrac{\partial v_1}{\partial u_2}+c_2+d_2\dfrac{\partial v_2}{\partial u_2}
\Big)\\
=\dfrac{\partial u'_2}{\partial u_2} \Big(
e_2+f_2\dfrac{\partial v_1}{\partial u_1}+h_2\dfrac{\partial v_2}{\partial u_1}\Big)-\dfrac{\partial u'_2}{\partial u_1} \Big(
f_2\dfrac{\partial v_1}{\partial u_2}+g_2+h_2\dfrac{\partial v_2}{\partial u_2}\Big).
\end{gathered}
\end{equation}
As
\begin{equation*}
u'_1=a_1u_1+b_1v_1+c_1u_2+d_1v_2, \quad u'_2=e_1u_1+f_1v_1+g_1u_2+h_1v_2
\end{equation*}
by \eqref{22022101}, \eqref{22052509} implies  
\begin{equation*}
\begin{gathered}
-\Big(b_1\dfrac{\partial v_1}{\partial u_2}+c_1+d_1\dfrac{\partial v_2}{\partial u_2}\Big)\Big(a_2+b_2\dfrac{\partial v_1}{\partial u_1}+d_2\dfrac{\partial v_2}{\partial u_1}\Big)+\Big(a_1+b_1\dfrac{\partial v_1}{\partial u_1}+d_1\dfrac{\partial v_2}{\partial u_1}\Big)\Big(
b_2\dfrac{\partial v_1}{\partial u_2}+c_2+d_2\dfrac{\partial v_2}{\partial u_2}
\Big)\\
=\Big(f_1\dfrac{\partial v_1}{\partial u_2}+g_1+h_1\dfrac{\partial v_2}{\partial u_2}\Big)\Big(
e_2+f_2\dfrac{\partial v_1}{\partial u_1}+h_2\dfrac{\partial v_2}{\partial u_1}\Big)-\Big(e_1+f_1\dfrac{\partial v_1}{\partial u_1}+h_1\dfrac{\partial v_2}{\partial u_1}\Big)\Big(
f_2\dfrac{\partial v_1}{\partial u_2}+g_2+h_2\dfrac{\partial v_2}{\partial u_2}\Big),
\end{gathered}
\end{equation*}
which is further expanded as  
\begin{equation}\label{21070901}
\begin{aligned}
&(a_1c_2-a_2c_1+e_1g_2-e_2g_1)+(b_1c_2-b_2c_1+f_1g_2-f_2g_1)\dfrac{\partial v_1}{\partial u_1}
+(a_1b_2-a_2b_1+e_1f_2-e_2f_1)\dfrac{\partial v_1}{\partial u_2}\\
+&(c_2d_1-c_1d_2+g_2h_1-g_1h_2 )\dfrac{\partial v_2}{\partial u_1}+(a_1d_2-a_2d_1+e_1h_2-e_2h_1)\dfrac{\partial v_2}{\partial u_2}\\
+&(b_2d_1-b_1d_2+f_2h_1-f_1h_2)\dfrac{\partial v_1}{\partial u_2}\dfrac{\partial v_2}{\partial u_1}+(b_1d_2-b_2d_1+f_1h_2-f_2h_1)\dfrac{\partial v_1}{\partial u_1}\dfrac{\partial v_2}{\partial u_2}=0.
\end{aligned}
\end{equation} 
Since $\dfrac{\partial v_1}{\partial u_2}\Big(=\dfrac{\partial v_2}{\partial u_1}\Big)$ is the only one containing a term of the form $u_1^{\alpha}u_2^{\beta}$ with $\alpha,\beta$ odd among 
\begin{equation*}
\dfrac{\partial v_1}{\partial u_1},\quad \dfrac{\partial v_1}{\partial u_2},\quad \dfrac{\partial v_2}{\partial u_2},\quad  \dfrac{\partial v_1}{\partial u_2}\dfrac{\partial v_2}{\partial u_1}\quad\text{and}\quad   \dfrac{\partial v_1}{\partial u_1}\dfrac{\partial v_2}{\partial u_2}, 
\end{equation*} 
we get 
\begin{equation}\label{22030601}
(a_1b_2-a_2b_1+e_1f_2-e_2f_1)+(c_2d_1-c_1d_2+g_2h_1-g_1h_2 )=0
\end{equation}
from \eqref{21070901}. Now \eqref{21071403} follows from \eqref{22030601}, and \eqref{21070901} is reduced to 
\begin{equation}\label{21072703}
\begin{aligned}
&(a_1c_2-a_2c_1+e_1g_2-e_2g_1)+(b_1c_2-b_2c_1+f_1g_2-f_2g_1)\dfrac{\partial v_1}{\partial u_1}+(a_1d_2-a_2d_1+e_1h_2-e_2h_1)\dfrac{\partial v_2}{\partial u_2}\\
+&(b_2d_1-b_1d_2+f_2h_1-f_1h_2)\Big(\dfrac{\partial v_1}{\partial u_2}\dfrac{\partial v_2}{\partial u_1}-\dfrac{\partial v_1}{\partial u_1}\dfrac{\partial v_2}{\partial u_2}\Big)=0.
\end{aligned}
\end{equation}
Note that 
\begin{equation}\label{21072704}
(\det A_1)A_1^{-1}A_2=\left(\begin{array}{cc}
b_2c_1-b_1c_2 & b_2d_1-b_1d_2\\
a_1c_2-a_2c_1 & a_1d_2-a_2d_1
\end{array}\right), \quad 
(\det A_3)A_3^{-1}A_4=\left(\begin{array}{cc}
f_2g_1-f_1g_2 & f_2h_1-f_1h_2\\
e_1g_2-e_2g_1 & e_1h_2-e_2h_1
\end{array}\right)
\end{equation}
from a direct computation. 

\begin{enumerate}
\item If $(\det A_1)A_1^{-1}A_2=-(\det A_3)A_3^{-1}A_4$, then, by \eqref{21071403}, 
\begin{equation*}
\det A_1+\det A_3=\det A_2+\dfrac{\det A_1}{\det A_3}\det A_2=\det A_2\Big(\dfrac{\det A_1+\det A_3}{\det A_3}\Big), 
\end{equation*}
implying either $\det A_1+\det A_3=0$ or $\det A_2=\det A_3$. For $\det A_1+\det A_3=0$, as $\det A_1\neq 0$ and $\det A_3\neq 0$, it contradicts the conclusion in Lemma \ref{21072210}. If $\det A_2=\det A_3$, then 
\begin{equation*}
\det M=\det A_1\det (A_4-A_3A_1^{-1}A_2)=\det A_1\Big(1+\dfrac{\det A_1}{\det A_3}\Big)^2\dfrac{\det A_3\det A_2}{\det A_1}
=(\det A_1+\det A_3)^2=1,
\end{equation*} 
implying $\det A_1+\det A_3=1$. By \eqref{21071403}, $\det A_1=\det A_4$ follows. 
\item If $(\det A_1)A_1^{-1}A_2\neq -(\det A_3)A_3^{-1}A_4$, since $A_1^{-1}A_2\newparallel A_3^{-1}A_4$, the conclusion follows from \eqref{21072703}-\eqref{21072704}.
\end{enumerate}
\end{proof}

Now we state the following, which completes the proof of Theorem \ref{20091911}. 

\begin{proposition}\label{21072601}
Theorem \ref{21071405} (2) never arises. 
\end{proposition}

Before proving the proposition, two lemmas shall be proved first.

The first lemma concerns about the rigidity of the relation appeared in \eqref{21071404}.
\begin{lemma}\label{21071407}
Suppose 
\begin{equation*}
a+b\dfrac{\partial v_1}{\partial u_1}+c\dfrac{\partial v_2}{\partial u_2}+d\Big(\dfrac{\partial v_1}{\partial u_1}\dfrac{\partial v_2}{\partial u_2}-\dfrac{\partial v_1}{\partial u_2}\dfrac{\partial v_2}{\partial u_1}\Big)=0
\end{equation*}
for some non-trivial $(a, b, c, d)\in \mathbb{Q}^4$. Then $(a, b, c, d)$ is uniquely determined up to a constant multiple. That is, if 
\begin{equation}\label{22071605}
\begin{gathered}
a+b\dfrac{\partial v_1}{\partial u_1}+c\dfrac{\partial v_2}{\partial u_2}+d\Big(\dfrac{\partial v_1}{\partial u_1}\dfrac{\partial v_2}{\partial u_2}-\dfrac{\partial v_1}{\partial u_2}\dfrac{\partial v_2}{\partial u_1}\Big)=a'+b'\dfrac{\partial v_1}{\partial u_1}+c'\dfrac{\partial v_2}{\partial u_2}+d'\Big(\dfrac{\partial v_1}{\partial u_1}\dfrac{\partial v_2}{\partial u_2}-\dfrac{\partial v_1}{\partial u_2}\dfrac{\partial v_2}{\partial u_1}\Big)=0,
\end{gathered}
\end{equation}
then 
\begin{equation*}
(a,b,c,d)=m (a',b',c',d')
\end{equation*}
for some $m\in \mathbb{Q}$. 
\end{lemma}
\begin{proof}
On the contrary, suppose there is no $m\in \mathbb{Q}$ such that $(a,b,c,d)\neq m(a',b',c',d')$. Without loss of generality, we assume $d'=0$ and $d=c'=1$, and normalize \eqref{22071605} as  
\begin{equation}\label{22071607}
\begin{gathered}
\dfrac{\partial v_2}{\partial u_2}=-a'-b'\dfrac{\partial v_1}{\partial u_1},\quad
a+b\dfrac{\partial v_1}{\partial u_1}+c\dfrac{\partial v_2}{\partial u_2}+\Big(\dfrac{\partial v_1}{\partial u_1}\dfrac{\partial v_2}{\partial u_2}-\dfrac{\partial v_1}{\partial u_2}\dfrac{\partial v_2}{\partial u_1}\Big)=0. 
\end{gathered}
\end{equation} 
Combining the two equations in \eqref{22071607}, it follows that
\begin{equation}\label{21071402}
\begin{aligned}
&a+b\dfrac{\partial v_1}{\partial u_1}-c\Big(a'+b'\dfrac{\partial v_1}{\partial u_1}\Big)-\dfrac{\partial v_1}{\partial u_1}\Big(a'+b'\dfrac{\partial v_1}{\partial u_1}\Big)-\dfrac{\partial v_1}{\partial u_2}\dfrac{\partial v_2}{\partial u_1}\\
=&a-ca'+(b-cb'-a')\dfrac{\partial v_1}{\partial u_1}-b'\Big(\dfrac{\partial v_1}{\partial u_1}\Big)^2-\Big(\dfrac{\partial v_1}{\partial u_2}\Big)^2=0.
\end{aligned}
\end{equation}
Let 
\begin{equation*}
\dfrac{\partial v_1}{\partial u_1}=\tau_1+m_{\alpha} u_1^{\alpha}+\cdots\quad (\text{where }m_{\alpha}\neq 0). 
\end{equation*} 
Since $\dfrac{\partial v_1}{\partial u_2}$ does not contain a term of the form $u_1^{\alpha}$, computing the coefficient of $u_1^{\alpha}$ in \eqref{21071402}, we have
\begin{equation*}
(b-cb'-a')m_{\alpha}-b'(2\tau_1m_{\alpha})=0\Longrightarrow (b-cb'-a')-b'(2\tau_1)=0, 
\end{equation*} 
which implies $b-cb'-a'=b'=0$ (as $b,c,b',a'\in \mathbb{Q}$ and $\tau_1\notin\mathbb{R}$). As a result, \eqref{21071402} is reduced to $a-ca'-\Big(\dfrac{\partial v_1}{\partial u_2}\Big)^2=0$. But this is a contradiction, since $\dfrac{\partial v_1}{\partial u_2}$ is a non-constant function.  
\end{proof}

The next lemma is induced by applying various results obtained previously to \begin{equation*}
\begin{gathered}
\left(
  \begin{array}{c}
  u_1\\
  v_1\\
  u_2\\
  v_2
  \end{array}
\right)=M^{-1}\left(
  \begin{array}{c}
  u'_1\\
  v'_1\\
  u'_2\\
  v'_2
  \end{array}
\right), 
\end{gathered}
\end{equation*}
which is an equivalent form of \eqref{22022101}. 
\begin{lemma}\label{20091912}
Suppose $\det M, \det J\neq 0$ and Theorem \ref{21071405} (2) holds. Then 
\begin{equation}\label{21070903}
\det A_1=\det A_4, \quad \det A_2=\det A_3, \quad A_1^{-1}A_2\newparallel A_1A_3^{-1}
\end{equation}
and so $M$ is either of the form 
\begin{equation}
\begin{gathered}
\left(\begin{array}{cc}
A_1 & A_1E \\
(\det E) E^{-1}A_1 & E^{-1}A_1E
\end{array}\right)
\quad \text{or}\quad
\left(\begin{array}{cc}
A_1 & A_1E \\
-(\det E) E^{-1}A_1 & -E^{-1}A_1E
\end{array}\right)
\end{gathered}
\end{equation}
for some $(2\times 2)$-matrix $E$. 
\end{lemma}
\begin{proof}
By Lemma \ref{21072603},  
\begin{equation}\label{22101201}
\begin{gathered}
M^{-1}=\left(\begin{array}{cc}
  \dfrac{s}{s-1}A_1^{-1} & -\dfrac{1}{s-1}A_3^{-1} \\
-\dfrac{1}{s-1}A_2^{-1} & \dfrac{s}{s-1}A_4^{-1}
  \end{array}
\right)
\end{gathered}
\end{equation}
where $s(\in \mathbb{Q})$ is a constant satisfying $A_4=sA_3A_1^{-1}A_2$. To simplify notation, denote \eqref{22101201} by 
\begin{equation*}
\left(\begin{array}{cc}
A'_1 & A'_2 \\
A'_3 & A'_4
  \end{array}
\right)
\end{equation*}
and let
\begin{equation*}
J':=\left(\begin{array}{cc}
\dfrac{\partial u_1}{\partial u'_1} & \dfrac{\partial u_2}{\partial u'_1} \\
\dfrac{\partial u_1}{\partial u'_2} & \dfrac{\partial u_2}{\partial u'_2} 
\end{array}\right). 
\end{equation*}

We first claim that $J'$ is invertible. If not, by Theorem \ref{21072701}, we apply \eqref{22052309} to $M^{-1}$ and obtain
\begin{equation}\label{22101202}
\begin{gathered}
\det A'_1=\det A'_4, \quad \det A'_2=\det A'_3, \quad (\det A'_1)(A'_1)^{-1}A'_2=-(\det A'_3)(A'_3)^{-1}A'_4.
\end{gathered}
\end{equation}
The first two equalities in \eqref{22101202} are equivalent to 
\begin{equation*}
\det A_1=\det A_4,\quad  \det A_2=\det A_3
\end{equation*}
respectively, and the last one in \eqref{22101202} implies
\begin{equation*}
\begin{gathered}
\dfrac{s^2}{(s-1)^2\det A_1}  \dfrac{s-1}{s}A_1\Big(-\dfrac{1}{s-1}A_3^{-1}\Big)=-\dfrac{1}{(s-1)^2\det A_2}(1-s)A_2 \dfrac{s}{s-1}A_4^{-1}\\
\Longrightarrow
(\det A_2) A_1A_3^{-1}=-(\det A_1) A_2A_4^{-1}\Longrightarrow (\det A_1)A_1^{-1}A_2= -(\det A_3)A_3^{-1}A_4.
\end{gathered}
\end{equation*}
However this contradicts the assumption $(\det A_1)A_1^{-1}A_2\neq -(\det A_3)A_3^{-1}A_4$.

As $\det J'\neq 0$, applying \eqref{21071403} in Theorem \ref{21071405} to $M^{-1}$, 
\begin{equation*}
\begin{aligned}
\det A'_1+\det A'_3&=\det A'_2+\det A'_4\\
\Longrightarrow 
\det\Big(\dfrac{s}{s-1}A_1^{-1}\Big)+\det\Big(-\dfrac{1}{s-1}A_2^{-1}\Big)&=\det\Big(-\dfrac{1}{s-1}A_3^{-1}\Big)+\det\Big(\dfrac{s}{s-1}A_4^{-1}\Big)\\
\Longrightarrow \Big(\dfrac{s}{s-1}\Big)^2\det A_1^{-1}+\Big(\dfrac{1}{s-1}\Big)^2\det A_2^{-1}&=\Big(\dfrac{1}{s-1}\Big)^2\det A_3^{-1}+\Big(\dfrac{s}{s-1}\Big)^2\det A_4^{-1}\\
\Longrightarrow s^2\det A_1^{-1}+\det A_2^{-1}&=\det A_3^{-1}+s^2\det A_4^{-1}\\
\Longrightarrow \dfrac{\det A_1\det A_4}{\det A_2\det A_3}\dfrac{1}{\det A_1}+\dfrac{1}{\det A_2}&=\dfrac{1}{\det A_3}+\dfrac{\det A_1\det A_4}{\det A_2\det A_3}\dfrac{1}{\det A_4}\\
\Longrightarrow \dfrac{\det A_4}{\det A_2\det A_3}+\dfrac{1}{\det A_2}&=\dfrac{1}{\det A_3}+\dfrac{\det A_1}{\det A_2\det A_3}\\
\Longrightarrow \det A_4+\det A_3&=\det A_2+\det A_1.
\end{aligned}
\end{equation*}
Combining with \eqref{21071403}, it follows that 
\begin{equation*}
\det A_1=\det A_4, \quad \det A_2=\det A_3.  
\end{equation*}
If $(\det A'_1)(A'_1)^{-1}A'_2=-(\det A'_3)(A'_3)^{-1}A'_4$, as noted earlier, it contradicts our assumption $(\det A_1)A_1^{-1}A_2\neq -(\det A_3)A_3^{-1}A_4$. If $(\det A'_1)(A'_1)^{-1}A'_2\neq -(\det A'_3)(A'_3)^{-1}A'_4$, applying Theorem \ref{21071405} (2) to $M^{-1}$, we attain
\begin{equation*}
a'+b'\dfrac{\partial v'_1}{\partial u'_1}+c'\dfrac{\partial v'_2}{\partial u'_2}+d'\Big(\dfrac{\partial v'_1}{\partial u'_2}\dfrac{\partial v'_2}{\partial u'_1}-\dfrac{\partial v'_1}{\partial u'_1}\dfrac{\partial v'_2}{\partial u'_2}\Big)=0
\end{equation*}
where 
\begin{equation*}
\left(\begin{array}{cc}
-b' & d'\\
a' & c'
\end{array}\right)=A_1A_3^{-1}.
\end{equation*}
Since \eqref{21071404} holds (by the assumption), $A_1^{-1}A_2\newparallel A_1A_3^{-1}$ follows by Lemma \ref{21071407}. Letting $E:=A_1^{-1}A_2$, as 
\begin{equation*}
A_1^{-1}A_2 \newparallel A_3^{-1}A_4 \newparallel A_1A_3^{-1}, \quad \det A_1=\det A_4, \quad \det A_2=\det A_3,
\end{equation*} 
we conclude $M$ is one of the following forms: 
\begin{equation*}
\begin{aligned}
\left(\begin{array}{cc}
A_1 & A_1E \\
(\det E) E^{-1}A_1 & E^{-1}A_1E
\end{array}\right), \quad 
\left(\begin{array}{cc}
A_1 & A_1E \\
-(\det E) E^{-1}A_1 & -E^{-1}A_1E
\end{array}\right),\\
\left(\begin{array}{cc}
A_1 & A_1E \\
(\det E) E^{-1}A_1 & -E^{-1}A_1E
\end{array}\right)\quad\text{or}\quad 
\left(\begin{array}{cc}
A_1 & A_1E \\
-(\det E) E^{-1}A_1 & E^{-1}A_1E
\end{array}\right).
\end{aligned}
\end{equation*}
But the last two contradicts the assumption $(\det A_1) A_1^{-1}A_2\neq -(\det A_3)A_3^{-1}A_4$. This completes the proof. 
\end{proof}

Amalgamating Lemmas \ref{21071407} and \ref{20091912}, we finally establish Proposition \ref{21072601}. 
\begin{proof}[Proof of Proposition \ref{21072601}]
First, let us assume $M$ is given as
\begin{equation*}
\left(\begin{array}{cc}
A_1 & A_1E \\
(\det E) E^{-1}A_1 & E^{-1}A_1E
\end{array}\right)
\end{equation*}
for some matrix $E$. By Theorem \ref{22031403}, there are $k_j\in \mathbb{Q}\backslash \{0\}$ ($1\leq j\leq 4$) such that 
\begin{equation}\label{21072105}
\begin{gathered}
k_1+k_2=k_3+k_4=\dfrac{\det A_1}{k_1}+\dfrac{\det A_1\det E}{k_3}=
 \dfrac{\det A_1}{k_4}+ \dfrac{\det A_1\det E}{k_2}=1,
\end{gathered}
\end{equation}
and, by Corollary \ref{22010903}, they further satisfy 
\begin{equation}\label{22052813}
\begin{gathered}
\det E=\dfrac{k_2k_3}{k_1k_4}. 
\end{gathered}
\end{equation}
By \eqref{21072105}-\eqref{22052813}, 
\begin{equation}\label{22052901}
\begin{aligned}
&k_1+k_2=k_3+k_4=1,\quad \dfrac{1}{k_1}+\dfrac{k_2}{k_1k_4}=\dfrac{1}{k_4}+\dfrac{k_3}{k_1k_4}\Big(=\dfrac{1}{\det A_1}\Big)\\
\Longrightarrow &k_1+k_2=k_3+k_4=1, \quad k_4+k_2=k_1+k_3\\
\Longrightarrow &k_1=k_4,\quad k_2=k_3, \quad k_1+k_2=1. 
\end{aligned}
\end{equation}
Since 
\begin{equation*}
\det M=\det A_1\det (A_4-A_3A_1^{-1}A_2)=\pm 1
\end{equation*}
by Lemmas \ref{22062503}-\ref{21072204}, we have  
\begin{equation*}
\begin{gathered}
\det A_1\det (A_4-A_3A_1^{-1}A_2)=\det A_1\det \big(E^{-1}A_1E-(\det E)E^{-1}A_1E\big)=(\det A_1)^2(1-\det E)^2=\pm 1\\
\Longrightarrow \det A_1-\det A_1\det E=\pm 1.
\end{gathered}
\end{equation*}
\begin{enumerate}
\item If $\det A_1-\det A_1\det E=1$, then $1-\det E=\dfrac{1}{\det A_1}$ and,  since
\begin{equation*}
\det E=\dfrac{k_2^2}{k_1^2},  \quad \dfrac{1}{\det A_1}=\dfrac{1}{k_1}+\dfrac{\det E}{k_2}
\end{equation*}
by \eqref{21072105}-\eqref{22052901}, we get
\begin{equation*}
1-\dfrac{k_2^2}{k_1^2}=\dfrac{1}{k_1}+\dfrac{k_2}{k_1^2}\Longrightarrow k_1^2-k_2^2=k_1+k_2. 
\end{equation*}
As $k_1+k_2=1$, $k_1=1$ and $k_2=0$, contradicting the fact that $k_2\neq 0$. 
\item If $\det A_1-\det A_1\det E=-1$, then  
\begin{equation*}
\det E-1=\dfrac{1}{\det A_1}\Longrightarrow \dfrac{k_2^2}{k_1^2}-1=\dfrac{1}{k_1}+\dfrac{k_2}{k_1^2}\Longrightarrow k_2^2-k_1^2=k_1+k_2 
\end{equation*}
by \eqref{21072105}-\eqref{22052901}, concluding $k_1=0$ and $k_2=1$. But this again contradicts the fact $k_1\neq 0$. 
\end{enumerate}

Similarly, for $M$ given as 
\begin{equation*}
\left(\begin{array}{cc}
A_1 & A_1E \\
-(\det E) E^{-1}A_1 & -E^{-1}A_1E
\end{array}\right)
\end{equation*}
with some matrix $E$, one gets the same result. 
\end{proof}

Theorem \ref{21072701}, along with Theorem \ref{21071405} and Proposition \ref{21072601}, now completes the proof of Theorem \ref{20091911}. 

Finally, before proceeding to the next section, we record the following corollary, which is a combination of Theorems \ref{22031403} and \ref{20091911}. The corollary will be used in the proofs of Theorems \ref{22052203} and \ref{20103003} in Sections \ref{PCV} and \ref{22041806} respectively. 
\begin{corollary}\label{22051413}
Let $\mathcal{M}$ be a $2$-cusped hyperbolic $3$-manifold and $\mathcal{X}$ be its holonomy variety. Let $\mathcal{M}_{p_{1}/q_{1},p_{2}/q_{2}}$ and $\mathcal{M}_{p'_{1}/q'_{1},p'_{2}/q'_{2}}$ be two Dehn fillings of $\mathcal{M}$ having the same pseudo complex volume. Suppose the core holonomies $t_{1}, t_{2}$ (resp. $t'_{1},t'_{2}$ ) of $\mathcal{M}_{p_{1}/q_{1},p_{2}/q_{2}}$ (resp. $\mathcal{M}_{p'_{1}/q'_{1},p'_{2}/q'_{2}}$) are multiplicatively independent. If a Dehn filling point associated to the pair is contained in an analytic set defined by \eqref{22022101}, adapting the notation introduced in \eqref{22022010}, one of the following holds:
\begin{enumerate}
\item $\det A_3=\det A_2=1$ and $A_1=A_4=\left(\begin{array}{cc}
0 & 0\\
0 & 0
\end{array}\right)$;
\item $\det A_1=\det A_4=1$ and $A_2=A_3=\left(\begin{array}{cc}
0 & 0\\
0 & 0
\end{array}\right)$;
\item $\det A_j\neq 0\;(1\leq j\leq 4),\quad \det A_1=\det A_4, \quad \det A_2=\det A_3,\quad \det A_1+\det A_3=1$, and $(\det A_1) A_1^{-1}A_2=-(\det A_3)A_3^{-1}A_4$. 
\end{enumerate}
\end{corollary}

\newpage
\section{Pseudo complex volume $\Longrightarrow$ Complex volume}\label{PCV}
Throughout the section, $\mathcal{M}$ is a $2$-cusped hyperbolic $3$-manifold and $\mathcal{X}$ is its holonomy variety as usual. 

The aim of this section is proving Theorem \ref{22041805}. That is, if two Dehn fillings $\mathcal{M}_{p_{1}/q_{1},p_{2}/q_{2}}$ and $\mathcal{M}_{p'_{1}/q'_{1},p'_{2}/q'_{2}}$ of $\mathcal{M}$ with sufficiently large coefficients have the same pseudo complex volume, then we show their complex volumes are the same as well. The proof is divided into two parts, depending on whether the core holonomies of each Dehn filling are multiplicatively dependent or not.

\subsection{The core holonomies are multiplicatively dependent}

Let us first assume the core holonomies of the given two Dehn fillings are multiplicatively dependent. Before proving Theorem \ref{22041805}, we first claim the following, which is seen as a quantified version of Theorem \ref{20071505} and whose proof is based on Lemma \ref{22050501}.

\begin{lemma}\label{21121906}
Let $\mathcal{M}_{p_{1}/q_{1},p_{2}/q_{2}}$ and $\mathcal{M}_{p'_{1}/q'_{1},p'_{2}/q'_{2}}$ be two Dehn fillings of $\mathcal{M}$ satisfying the following:
\begin{itemize}
\item $\text{pvol}_{\mathbb{C}}\;\mathcal{M}_{p_{1}/q_{1},p_{2}/q_{2}}=\text{pvol}_{\mathbb{C}}\;\mathcal{M}_{p'_{1}/q'_{1},p'_{2}/q'_{2}}$;
\item the core holonomies $t_{1}, t_{2}$ (resp. $t'_1, t'_2$) of $\mathcal{M}_{p_{1}/q_{1},p_{2}/q_{2}}$ (resp. $\mathcal{M}_{p'_{1}/q'_{1},p'_{2}/q'_{2}}$) are multiplicatively dependent.
\end{itemize}
Let $H$ be an algebraic subgroup containing a Dehn filling point associated to the pair obtained in Theorem \ref{20071505}. If $H$ is defined by
\begin{equation}\label{21121804}
M_1^{a_j}L_1^{b_j}M_2^{c_j}L_2^{d_j}=(M'_1)^{a'_j}(L'_1)^{b'_j}(M'_2)^{c'_j}(L'_2)^{d'_j}=M_1^{e_j}L_1^{f_j}(M'_1)^{e'_j}(L'_1)^{f'_j}=1, \quad (j=1,2), 
\end{equation}
then
\begin{equation*}
\begin{gathered}
\left(\begin{array}{cc}
p_1  & q_1 
\end{array}\right)A_1^{-1}
=-\left(\begin{array}{cc}
p_2  & q_2 
\end{array}\right)A_2^{-1}, \;
\left(\begin{array}{cc}
p'_1  & q'_1 
\end{array}\right)(A'_1)^{-1}
=-\left(\begin{array}{cc}
p'_2  & q'_2 
\end{array}\right)(A'_2)^{-1}, \\ 
\left(\begin{array}{cc}
p_1  & q_1 
\end{array}\right)A_3^{-1}
=-\left(\begin{array}{cc}
p'_1  & q'_1 
\end{array}\right)(A'_3)^{-1}
\end{gathered}
\end{equation*}
and 
\begin{equation}\label{22051402}
1+\dfrac{\det A_1
}{\det A_2} 
=\dfrac{\det A_3}
{\det A'_3}\Big(1+\dfrac{\det A'_1}{\det A'_2}\Big)
\end{equation}
where
\begin{equation}\label{22080407}
A_1:=\left(\begin{array}{cc}
a_1 & b_1 \\
a_2 & b_2 
\end{array}\right), \quad 
A_2:=\left(\begin{array}{cc}
c_1 & d_1 \\
c_2 & d_2 
\end{array}\right), \quad 
A_3:=\left(\begin{array}{cc}
e_1 & f_1 \\
e_2 & f_2 
\end{array}\right)
\end{equation}
and 
\begin{equation}\label{22080408}
A'_1:=\left(\begin{array}{cc}
a'_1 & b'_1 \\
a'_2 & b'_2 
\end{array}\right), \quad 
A'_2:=\left(\begin{array}{cc}
c'_1 & d'_1 \\
c'_2 & d'_2 
\end{array}\right), \quad 
A'_3:=\left(\begin{array}{cc}
e'_1 & f'_1 \\
e'_2 & f'_2 
\end{array}\right). 
\end{equation}
\end{lemma}
\begin{proof}
Splitting \eqref{21121804} into 
\begin{equation}\label{22072101}
M_1^{a_j}L_1^{b_j}M_2^{c_j}L_2^{d_j}=(M'_1)^{a'_j}(L'_1)^{b'_j}(M'_2)^{c'_j}(L'_2)^{d'_j}=1, \quad (j=1,2)
\end{equation}
and 
\begin{equation}\label{22072102}
M_1^{e_j}L_1^{f_j}(M'_1)^{e'_j}(L'_1)^{f'_j}=1, \quad (j=1,2), 
\end{equation}
denote \eqref{22072101} and \eqref{22072102} by $H_1$ and $H_2$ respectively. As seen in Theorem \ref{20071505}, $(\mathcal{X}\times \mathcal{X})\cap H_1$ is a $2$-dim anomalous subvariety of $\mathcal{X}\times \mathcal{X}$, and, applying Lemma \ref{22050501}, we obtain
\begin{equation*}
\left(\begin{array}{cc}
p_1  & q_1 
\end{array}\right)A_1^{-1}
=-\left(\begin{array}{cc}
p_2  & q_2 
\end{array}\right)A_2^{-1}, \;
\left(\begin{array}{cc}
p'_1  & q'_1 
\end{array}\right)(A'_1)^{-1}
=-\left(\begin{array}{cc}
p'_2  & q'_2 
\end{array}\right)(A'_2)^{-1}
\end{equation*}
and
\begin{equation}\label{22051401}
t_1^{\det A_1}=\epsilon_1t_2^{\det A_2}, \quad (t'_1)^{\det A'_1}=\epsilon_2(t'_2)^{\det A'_2}
\end{equation}
for some torsion $\epsilon_k$ ($k=1,2$). 

Projecting $(\mathcal{X}\times \mathcal{X})\cap H_1$ onto $\mathbb{G}^4(:=M_1, L_1, M'_1, L'_1))$, if we represent the projected image by $\mathcal{C}\times\mathcal{C}'$ and consider $H_2$ as an algebraic subgroup in $\mathbb{G}^4$, $(\mathcal{C}\times \mathcal{C}')\cap H_2$ is an anomalous subvariety of $\mathcal{C}\times \mathcal{C}'$ by Theorem \ref{20071505}. Again, thanks to Lemma \ref{22050501}, 
\begin{equation*}
\left(\begin{array}{cc}
p_1  & q_1 
\end{array}\right)A_3^{-1}
=-\left(\begin{array}{cc}
p'_1  & q'_1 
\end{array}\right)(A'_3)^{-1}
\end{equation*}
and 
\begin{equation}\label{22072105}
 t_1^{\det A_3}=\epsilon_3(t'_1)^{\det A'_3}
\end{equation}
for some torsion $\epsilon_3$. Since 
\begin{equation*}
t_1t_2=t'_1t'_2,
\end{equation*}
combining it with \eqref{22051401}-\eqref{22072105}, \eqref{22051402} follows. 
\end{proof}

Using the lemma, we now prove 

\begin{theorem}
Let $\mathcal{M}_{p_{1}/q_{1},p_{2}/q_{2}}$ and $\mathcal{M}_{p'_{1}/q'_{1},p'_{2}/q'_{2}}$ be two Dehn fillings of $\mathcal{M}$ with sufficiently large $|p_k|+|q_k|$ and $|p'_k|+|q'_k|$ ($k=1,2$) satisfying 
\begin{itemize}
\item $\text{pvol}_{\mathbb{C}}\;\mathcal{M}_{p_{1}/q_{1},p_{2}/q_{2}}=\text{pvol}_{\mathbb{C}}\;\mathcal{M}_{p'_{1}/q'_{1},p'_{2}/q'_{2}}$;
\item the core holonomies $t_{1}, t_{2}$ (resp. $t'_1, t'_2$) of $\mathcal{M}_{p_{1}/q_{1},p_{2}/q_{2}}$ (resp. $\mathcal{M}_{p'_{1}/q'_{1},p'_{2}/q'_{2}}$) are multiplicatively dependent.
\end{itemize}
Then
\begin{equation}\label{22051501}
\text{vol}_{\mathbb{C}}\;\mathcal{M}_{p_{1}/q_{1}, p_{2}/q_{2}}=
\text{vol}_{\mathbb{C}}\;\mathcal{M}_{p'_{1}/q'_{1}, p'_{2}/q'_{2}}.
\end{equation}
\end{theorem}

\begin{proof}
By Theorem \ref{20071505}, a Dehn filling point associated to the pair is contained in an algebraic subgroup $H$ defined by \eqref{21121804}. Moving to an analytic setting, $\log \big((\mathcal{X}\times \mathcal{X})\cap H\big)$ is defined by
\begin{equation}\label{21121805}
\begin{gathered}
A_1\left(\begin{array}{c}
u_1 \\
v_1  
\end{array}\right)
+A_2\left(\begin{array}{c}
u_2 \\
v_2  
\end{array}\right)
=A'_1\left(\begin{array}{c}
u'_1 \\
v'_1  
\end{array}\right)
+A'_2\left(\begin{array}{c}
u'_2 \\
v'_2  
\end{array}\right)
=A_3\left(\begin{array}{c}
u_1 \\
v_1  
\end{array}\right)
+A'_3\left(\begin{array}{c}
u'_1 \\
v'_1  
\end{array}\right)
=\left(\begin{array}{c}
0\\
0  
\end{array}\right)
\end{gathered}
\end{equation}
where $A_j, A'_j$ ($1\leq j\leq 3$) are the same as in \eqref{22080407}-\eqref{22080408}. 

Since
\begin{equation*}
\begin{gathered}
\text{vol}_{\mathbb{C}}\;\mathcal{M}_{p_1/q_1, p_2/q_2}=\text{pvol}_{\mathbb{C}}\;\mathcal{M}_{p_1/q_1, p_2/q_2}+\sum_{k=1}^2\int u_kdv_k-v_kdu_k
\end{gathered}
\end{equation*}
by \cite{nz}, verifying \eqref{22051501} is equivalent to establishing 
\begin{equation}\label{21121903}
\sum_{k=1}^2\int u_kdv_k-v_kdu_k=\sum_{k=1}^2\int u'_kdv'_k-v'_kdu'_k 
\end{equation}
under the transformation given in \eqref{21121805}. 

Rewriting \eqref{21121805} as 
\begin{equation*}
\begin{gathered}
\left(\begin{array}{c}
u_2 \\
v_2  
\end{array}\right)=-A_2^{-1}A_1\left(\begin{array}{c}
u_1 \\
v_1  
\end{array}\right), \quad 
\left(\begin{array}{c}
u'_2 \\
v'_2 
\end{array}\right)=
-(A'_2)^{-1}A'_1\left(\begin{array}{c}
u'_1 \\
v'_1  
\end{array}\right), \quad 
\left(\begin{array}{c}
u'_1 \\
v'_1  
\end{array}\right)
=-(A'_3)^{-1}A_3
\left(\begin{array}{c}
u_1 \\
v_1  
\end{array}\right),   
\end{gathered}
\end{equation*}
it follows that\footnote{Note that if $\left(\begin{array}{c}
u' \\
v'  
\end{array}\right)
=A\left(\begin{array}{c}
u \\
v  
\end{array}\right)$, then $\left(\begin{array}{c}
du' \\
dv'  
\end{array}\right)
=A\left(\begin{array}{c}
du \\
dv  
\end{array}\right)$ and $u'dv'-v'du'=\det A(udv-vdu)$.} 
\begin{equation}\label{22051504}
\begin{gathered}
\int (u_1dv_1-v_1du_1)+\int (u_2dv_2-v_2du_2)=
\Big(1+\dfrac{\det A_1}{\det A_2}\Big)\int (u_1dv_1-v_1du_1). 
\end{gathered}
\end{equation}
and
\begin{equation}\label{22051503}
\begin{gathered}
\allowdisplaybreaks
\int (u'_1dv'_1-v'_1du'_1)+\int (u'_2dv'_2-v'_2du'_2)
=\dfrac{\det A_3}{\det A'_3}\int (u_1dv_1-v_1du_1)+\dfrac{\det A'_1}{\det A'_2}\int (u'_1dv'_1-v'_1du'_1)\\
=\dfrac{\det A_3}{\det A'_3}\int (u_1dv_1-v_1du_1)+\dfrac{\det A'_1}{\det A'_2}\dfrac{\det A_3}{\det A'_3}\int (u_1dv_1-v_1du_1)
=\Big(\dfrac{\det A_3}{\det A'_3}+\dfrac{\det A'_1}{\det A'_2}\dfrac{\det A_3}{\det A'_3}\Big)\int (u_1dv_1-v_1du_1)
\end{gathered}
\end{equation}
By \eqref{22051402}, \eqref{22051504} is equal to \eqref{22051503}, implying \eqref{21121903}. 
\end{proof}

\subsection{The core holonomies are multiplicatively independent}

Second, we assume the core holonomies of two Dehn fillings are multiplicatively independent. In this case, the result will be obtained using Corollary \ref{22051413}.
\begin{theorem}\label{22052203}
Let $\mathcal{M}_{p_{1}/q_{1},p_{2}/q_{2}}$ and $\mathcal{M}_{p'_{1}/q'_{1},p'_{2}/q'_{2}}$ be two Dehn fillings of $\mathcal{M}$ with $|p_k|+|q_k|$ and $|p'_k|+|q'_k|$ ($k=1,2$) sufficiently large satisfying 
\begin{itemize}
\item $\text{pvol}_{\mathbb{C}}\;\mathcal{M}_{p_{1}/q_{1},p_{2}/q_{2}}=\text{pvol}_{\mathbb{C}}\;\mathcal{M}_{p'_{1}/q'_{1},p'_{2}/q'_{2}}$;
\item the core holonomies $t_{1}, t_{2}$ (resp. $t'_{1}, t'_{2}$) of $\mathcal{M}_{p_{1}/q_{1},p_{2}/q_{2}}$ (resp. $\mathcal{M}_{p'_{1}/q'_{1},p'_{2}/q'_{2}}$) are multiplicatively independent. 
\end{itemize}
Then
\begin{equation}\label{22051407}
\text{vol}_{\mathbb{C}}\;\mathcal{M}_{p_{1}/q_{1}, p_{2}/q_{2}}=
\text{vol}_{\mathbb{C}}\;\mathcal{M}_{p'_{1}/q'_{1}, p'_{2}/q'_{2}}.
\end{equation}
\end{theorem}
\begin{proof}
Adapting the same notation given in Section \ref{20121903}, let us assume a Dehn filling point associated to the pair $(\mathcal{M}_{p_{1}/q_{1},p_{2}/q_{2}},\mathcal{M}_{p'_{1}/q'_{1},p'_{2}/q'_{2}})$ is contained in an anomalous analytic subset defined by \eqref{22022101}. By Corollary \ref{22051413}, we have
\begin{equation*}
\det A_1+\det A_3=\det A_2+\det A_4=1
\end{equation*}
and 
\begin{equation*}
(\det A_1) A_1^{-1}A_2=-(\det A_3)A_3^{-1}A_4, 
\end{equation*}
which are explicitly represented as
\begin{equation}\label{22011104}
\begin{gathered}
a_1b_2-a_2b_1+e_1f_2-e_2f_1=c_1d_2-c_2d_1+g_1h_2-g_2h_1=1
\end{gathered}
\end{equation}
and 
\begin{equation}\label{22011105}
\begin{gathered}
\left(\begin{array}{cc}
b_2c_1-b_1c_2 & b_2d_1-b_1d_2\\
a_1c_2-a_2c_1 & a_1d_2-a_2d_1
\end{array}\right)=-\left(\begin{array}{cc}
f_2g_1-f_1g_2 & f_2h_1-f_1h_2\\
e_1g_2-e_2g_1 & e_1h_2-e_2h_1
\end{array}\right)
\end{gathered}
\end{equation}
respectively. 

Similar to the previous theorem, to prove \eqref{22051407}, it is enough to show  
\begin{equation*}
\sum_{k=1}^2\int u_kdv_k-v_kdu_k=\sum_{k=1}^2\int u'_kdv'_k-v'_kdu'_k. 
\end{equation*}
By \eqref{22022101},  
\begin{equation*}
\begin{gathered}
\left(\begin{array}{c}
du'_1\\
dv'_1\\
du'_2\\
dv'_2
\end{array}
\right)
=\left(\begin{array}{cccc}
a_1 & b_1 & c_1 & d_1\\
a_2 & b_2 & c_2 & d_2\\
e_1 & f_1 & g_1 & h_1\\
e_2 & f_2 & g_2 & h_2 
\end{array}
\right)
\left(\begin{array}{c}
du_1\\
dv_1\\
du_2\\
dv_2
\end{array}
\right)
\end{gathered}
\end{equation*}
and thus
\begin{equation*}
\begin{gathered}
\allowdisplaybreaks
\int u'_1dv'_1-v'_1du'_1+\int u'_2dv'_2-v'_2du'_2\\
=\int (a_1u_1+b_1v_1+c_1u_2+d_1v_2)(a_2du_1+b_2dv_1+c_2du_2+d_2dv_2)\\
-\int (a_2u_1+b_2v_1+c_2u_2+d_2v_2)(a_1du_1+b_1dv_1+c_1du_2+d_1dv_2)\\
+\int (e_1u_1+f_1v_1+g_1u_2+h_1v_2)(e_2du_1+f_2dv_1+g_2du_2+h_2dv_2)\\
-\int (e_2u_1+b_2v_1+c_2u_2+d_2v_2)(e_1du_1+f_1dv_1+e_1du_2+f_1dv_2)\\
\end{gathered}
\end{equation*}
\begin{equation}\label{21072211}
\begin{gathered}
=(a_1b_2-a_2b_1+e_1f_2-e_2f_1)\int (u_1dv_1-v_1du_1)
+(c_1d_2-c_2d_1+g_1h_2-g_2h_1)\int (u_2dv_2-v_2du_2)\\
+ (a_1c_2-a_2c_1+e_1g_2-e_2g_1)\int (u_1du_2-u_2du_1)
+(a_1d_2-a_2d_1+e_1h_2-e_2h_1)\int (u_1dv_2-v_2du_1)\\
+(b_1c_2-b_2c_1+f_1g_2-f_2g_1)\int (v_1du_2-u_2dv_1)
+(b_1d_2-b_2d_1+f_1h_2-f_2h_1)\int (v_1dv_2-v_2dv_1).
\end{gathered}
\end{equation}
By \eqref{22011104}-\eqref{22011105}, \eqref{21072211} is further reduced to 
\begin{equation*}
\int (u_1dv_1-v_1du_1)+\int (u_2dv_2-v_2du_2),  
\end{equation*}
which completes the proof. 
\end{proof}

\newpage
\section{Proofs of Theorems \ref{2201}-\ref{20103003}}\label{22041806}
In this section, we prove Theorems \ref{2201}-\ref{20103003}. The theorems are attained by integrating several results acquired in previous sections.  

By Theorem \ref{22050203}, there exists a finite set $\mathcal{H}$ of algebraic subgroups such that, for any two Den fillings $\mathcal{M}_{p_1/q_1, p_2/q_2}$ and $\mathcal{M}_{p'_1/q'_1, p'_2/q'_2}$ of $\mathcal{M}$ having the same pseudo complex volume, a Dehn filling point associated to the pair is contained in some $H\in \mathcal{H}$, defined by equations either of the forms 
\begin{equation}\label{22070801}
\begin{gathered}
M_1^{a_j}L_1^{b_j}M_2^{c_j}L_2^{d_j}=(M'_1)^{a'_j}(L'_1)^{b'_j}(M'_2)^{c'_j}(L'_2)^{d'_j}=M_1^{e_j}L_1^{f_j}(M'_1)^{e'_j}(L'_1)^{f'_j}=1, \quad (j=1,2)
\end{gathered}
\end{equation} 
or
\begin{equation}\label{22051601}
\begin{gathered}
M_1^{a_j}L_1^{b_j}M_2^{c_j}L_2^{d_j}(M'_1)^{a'_j}(L'_1)^{b'_j}=M_1^{e_j}L_1^{f_j}M_2^{g_j}L_2^{h_j}(M'_2)^{g'_j}(L'_2)^{h'_j}=1,\quad (j=1,2).
\end{gathered}
\end{equation}
In particular, if the core holonomies $t_1, t_2$ (resp. $t'_1, t'_2$) of $\mathcal{M}_{p_1/q_1, p_2/q_2}$ (resp. $\mathcal{M}_{p'_1/q'_1, p'_2/q'_2}$) are multiplicatively dependent, then $H$ is defined by \eqref{22070801}. Otherwise, it belongs to the second class. 
 
To simplify notation, let
\begin{equation*}
A_1:=\left(\begin{array}{cc}
a_1 & b_1\\
a_2 & b_2
\end{array}\right), \quad 
A_2:=\left(\begin{array}{cc}
c_1 & d_1\\
c_2 & d_2
\end{array}\right), \quad 
 A_3:=\left(\begin{array}{cc}
e_1 & f_1\\
e_2 & f_2
\end{array}\right), \quad  
A_4:=\left(\begin{array}{cc}
g_1 & h_1\\
g_2 & h_2
\end{array}\right)
\end{equation*}
and 
\begin{equation*}
A'_1:=\left(\begin{array}{cc}
a'_1 & b'_1\\
a'_2 & b'_2
\end{array}\right),\quad  
A'_2:=\left(\begin{array}{cc}
c'_1 & d'_1\\
c'_2 & d'_2
\end{array}\right), \quad  
A'_3:=\left(\begin{array}{cc}
e'_1 & f'_1\\
e'_2 & f'_2
\end{array}\right), \quad
A'_4:=\left(\begin{array}{cc}
g'_1 & h'_1\\
g'_2 & h'_2
\end{array}\right).
\end{equation*}

First, if $H$ is defined by \eqref{22070801}, then 
\begin{equation}\label{22071803}
A_1
\left( \begin{array}{c}
1  \\
\tau_1  
\end{array}\right)
\newparallel
A_2
\left( \begin{array}{c}
1  \\
\tau_2
\end{array}\right),\quad 
A'_1
\left( \begin{array}{c}
1  \\
\tau_1  
\end{array}\right)
\newparallel
A'_2\left( \begin{array}{c}
1  \\
\tau_2  
\end{array}\right),\quad 
A_3
\left( \begin{array}{c}
1  \\
\tau_1  
\end{array}\right)\newparallel
A'_3
\left( \begin{array}{c}
1  \\
\tau_1  
\end{array}\right)
\end{equation}
by Corollary \ref{20121907}, and  
\begin{equation}\label{22071801}
A_1\left( \begin{array}{c}
-q_1  \\
p_1  
\end{array}\right)\newparallel
A_2\left( \begin{array}{c}
-q_2  \\
p_2  
\end{array}\right), \quad 
A'_1\left( \begin{array}{c}
-q'_1  \\
p'_1  
\end{array}\right)
\newparallel
A'_2\left( \begin{array}{c}
-q'_2  \\
p'_2  
\end{array}\right),\quad 
A_3\left( \begin{array}{c}
-q_1  \\
p_1  
\end{array}\right)
\newparallel
A'_3\left( \begin{array}{c}
-q'_1  \\
p'_1  
\end{array}\right)
\end{equation}
by Lemma \ref{21121906}.

Second, if $H$ is defined by \eqref{22051601}, then
\begin{equation}\label{20122307}
A'_1
\left( \begin{array}{c}
1  \\
\tau_1  
\end{array}\right)
\newparallel
A_1
\left( \begin{array}{c}
1  \\
\tau_1  
\end{array}\right)
\newparallel
A_2
\left( \begin{array}{c}
1  \\
\tau_2
\end{array}\right),\quad 
A'_4\left( \begin{array}{c}
1  \\
\tau_2  
\end{array}\right) 
\newparallel
A_3
\left( \begin{array}{c}
1  \\
\tau_1  
\end{array}\right)
\newparallel
A_4\left( \begin{array}{c}
1  \\
\tau_2  
\end{array}\right)
\end{equation}
by Corollary \ref{20121907}, and  
\begin{equation}\label{22051805}
A'_1\left( \begin{array}{c}
-q'_1  \\
p'_1  
\end{array}\right)\newparallel
A_1\left( \begin{array}{c}
-q_1  \\
p_1  
\end{array}\right)\newparallel
A_2\left( \begin{array}{c}
-q_2  \\
p_2  
\end{array}\right), \quad 
A'_4\left( \begin{array}{c}
-q'_2  \\
p'_2  
\end{array}\right)
\newparallel
A_3\left( \begin{array}{c}
-q_1  \\
p_1  
\end{array}\right)
\newparallel
A_4\left( \begin{array}{c}
-q_2  \\
p_2  
\end{array}\right)
\end{equation}
by Corollary \ref{22010903}. 

Note Theorem \ref{2201} immediately follows from \eqref{22071801} and \eqref{22051805}. Hence we only cover Theorem \ref{20103003}, which is restated below.
\\

\noindent\textbf{Theorem \ref{20103003}}
\textit{Let $\mathcal{M}$ be a $2$-cusped hyperbolic $3$-manifold with two cusp shapes $\tau_1, \tau_2$ not belonging to the same quadratic field. Suppose 
\begin{equation}\label{22011701}
\text{pvol}_{\mathbb{C}}\;\mathcal{M}_{p_1/q_1, p_2/q_2}=\text{pvol}_{\mathbb{C}}\;\mathcal{M}_{p'_1/q'_1, p'_2/q'_2}
\end{equation}
with $|p_k|+|q_k|$ and $|p'_k|+|q'_k|$ sufficiently large ($k=1,2$). Then the following statements hold. 
\begin{enumerate} 
\item Suppose neither $\tau_1$ nor $\tau_2$ is contained in $\mathbb{Q}(\sqrt{-1})$ or $\mathbb{Q}(\sqrt{-3})$. 
\begin{enumerate}
\item If $\tau_1$ and $\tau_2$ are relatively independent, then $(p_1/q_1, p_2/q_2)=(p'_1/q'_1, p'_2/q'_2)$.
\item If $\tau_1$ and $\tau_2$ are relatively dependent, there exists unique $\rho\in S$ such that $(p_1/q_1, p_2/q_2)$ is either $(p'_1/q'_1, p'_2/q'_2)$ or $\big(\rho^{-1}(p'_2/q'_2), \rho(p'_1/q'_1)\big)$.
\end{enumerate}
\item If $\tau_1\in \mathbb{Q}(\sqrt{-1})$ and $\tau_2\notin\mathbb{Q}(\sqrt{-3})$ (resp. $\tau_1\in \mathbb{Q}(\sqrt{-3})$ and $\tau_2\notin \mathbb{Q}(\sqrt{-1})$), then there exists unique $\sigma\in S$ of order $2$ (resp. $3$) such that $(p_1/q_1, p_2/q_2)=(\sigma^i(p'_1/q'_1), p'_2/q'_2)$ where $0\leq i\leq 1$ (resp. $0\leq i\leq 2$). 
\item If $\tau_1\in \mathbb{Q}(\sqrt{-1})$ and $\tau_2\in \mathbb{Q}(\sqrt{-3})$, then there exist unique $\sigma$ and $\phi\in S$ of order $2$ and $3$ respectively such that $(p_1/q_1, p_2/q_2)=(\sigma^i(p'_1/q'_1), \phi^j(p'_2/q'_2))$ for some $0\leq i\leq 1, 0\leq j\leq 2$.
\end{enumerate}
}
\vspace{0.2cm}

To show Theorem \ref{20103003}, we shall need the following lemma, whose proof is given in \cite{jeon5}. 

\begin{lemma}\label{17052301}
Let $A_1$ and $A_2$ be $(2\times 2)$-matrices with rational entries and $\det A_1\neq 0$. Let $\tau_1$ and $\tau_2$ be non-real algebraic numbers. 
\begin{enumerate}
\item If $\tau_1$ is non-quadratic and $
A_1\left( \begin{array}{c}
1  \\
\tau_1
\end{array} \right)\newparallel A_2\left( \begin{array}{c}
1  \\
\tau_1
\end{array} \right)$, then $A_1\newparallel A_2$. 
\item Suppose $\tau_1$ and $\tau_2$ are relatively independent. If $
A_1\left( \begin{array}{c}
1  \\
\tau_1
\end{array} \right)\newparallel A_2\left( \begin{array}{c}
1  \\
\tau_2
\end{array} \right)$, then $A_2=\left( \begin{array}{cc}
0 & 0\\
0 & 0
\end{array} \right)$.
\end{enumerate}
\end{lemma}

Now we prove Theorem \ref{20103003}. 

\begin{proof}[Proof of Theorem \ref{20103003}]
\begin{enumerate}
\item First, let us assume $\tau_i\notin \big(\mathbb{Q}(\sqrt{-1})\cup \mathbb{Q}(\sqrt{-3})\big)$ for $i=1,2$. 
\begin{enumerate}
\item If $\tau_1$ and $\tau_2$ are relatively independent, a Dehn filling point associated to \eqref{22011701} is contained in an algebraic subgroup of the second kind (i.e. \eqref{22051601}). Further,  by \eqref{20122307} and Lemma \ref{17052301},   
\begin{equation*}
A_2=A_3=\left(\begin{array}{cc}
0 & 0\\
0 & 0
\end{array}\right),
\end{equation*}
and thus \eqref{22051601} is reduced to
\begin{equation}\label{22051602}
\begin{gathered}
M_1^{a_j}L_1^{b_j}(M'_1)^{a'_j}(L'_1)^{b'_j}=M_2^{g_j}L_2^{h_j}(M'_2)^{g'_j}(L'_2)^{h'_j}=1,\quad (j=1,2).
\end{gathered}
\end{equation}
Let 
\begin{equation*}
\mathcal{X}_1\times\mathcal{X}_1:=(\mathcal{X}\times \mathcal{X})\cap (M_1=L_1=M'_1=L'_1=1). 
\end{equation*}
Then 
\begin{equation*}
\begin{gathered}
(\mathcal{X}_1\times \mathcal{X}_1)\cap \big(M_2^{g_j}L_2^{h_j}(M'_2)^{g'_j}(L'_2)^{h'_j}=1,\; (j=1,2)\big)
\end{gathered}
\end{equation*}
is an anomalous subvariety of $\mathcal{X}_1\times \mathcal{X}_1$. Since $\tau_1, \tau_2\notin\big(\mathbb{Q}(\sqrt{-1})\cup\mathbb{Q}(\sqrt{-3})\big)$, the only anomalous subvarieties of $\mathcal{X}_1\times \mathcal{X}_1$ are the trivial ones by Theorem \ref{20102501} and so $A_1=\pm A'_1$. Similarly, one also gets $A_4=\pm A'_4$. 

In conclusion, \eqref{22051602} is equivalent to 
\begin{equation*}
M_1=(M'_1)^{\pm 1 }, L_1=(L'_1)^{\pm 1}, \quad M_2=(M'_2)^{\pm 1}, L_2=(L'_2)^{\pm 1}, 
\end{equation*}
implying $(p_1/q_1, p_2/q_2)=(p'_1/q'_1, p'_2/q'_2)$ by \eqref{22051805}. 

\item If $\tau_1$ and $\tau_2$ are relatively dependent, both belong to the same (non-quadratic) field. 
\begin{enumerate}
\item We first assume the core holonomies $t_1, t_2$ (resp. $t'_1, t'_2$) of $\mathcal{M}_{p_1/q_1, p_2/q_2}$ (resp. $\mathcal{M}_{p'_1/q'_1, p'_2/q'_2}$) are multiplicatively dependent and thus a Dehn filling point associated to \eqref{22011701} is contained in an algebraic subgroup defined by \eqref{22070801}. 

By \eqref{22071803} and Lemma \ref{17052301}, $A_3\newparallel A'_3$ and so 
\begin{equation}\label{22071805}
p_1/q_1=p'_1/q'_1 
\end{equation}
by \eqref{22071801}. Similarly,
\begin{equation*}
\left( \begin{array}{c}
1  \\
\tau_2  
\end{array}\right)
\newparallel
A_2^{-1}A_1
\left( \begin{array}{c}
1  \\
\tau_1
\end{array}\right)
\newparallel
(A'_2)^{-1}A'_1
\left( \begin{array}{c}
1  \\
\tau_1
\end{array}\right)
\end{equation*}
by \eqref{22071803}, and thus $A_2^{-1}A_1\newparallel (A'_2)^{-1}A'_1$ by Lemma \ref{17052301}. As
\begin{equation*}
A_2^{-1}A_1\left( \begin{array}{c}
-q_1  \\
p_1  
\end{array}\right)\newparallel
\left( \begin{array}{c}
-q_2  \\
p_2  
\end{array}\right), \quad 
(A'_2)^{-1}A'_1\left( \begin{array}{c}
-q'_1  \\
p'_1  
\end{array}\right)
\newparallel
\left( \begin{array}{c}
-q'_2  \\
p'_2  
\end{array}\right)
\end{equation*} 
by \eqref{22071801}, 
\begin{equation*}
p_2/q_2=p'_2/q'_2
\end{equation*}
follows by \eqref{22071805}. 

\item Next we assume the core holonomies $t_1, t_2$ (resp. $t'_1, t'_2$) are multiplicatively independent (resp. independent) and so a Dehn filling point associated to \eqref{22011701} is contained in an algebraic subgroup defined by \eqref{22051601}. 
\begin{enumerate}
\item If $\det A_1\neq 0$ and $\det A_4\neq 0$, as $\tau_1$ and $\tau_2$ are non-quadratic, 
\begin{equation*}
A_1\newparallel A'_1, \quad A_4\newparallel A'_4
\end{equation*}
by \eqref{20122307} and Lemma \ref{17052301}, thus
\begin{equation*}
(p_{1}/q_{1}, p_{2}/q_{2})=(p'_{1}/q'_{1}, p'_{2}/q'_{2})
\end{equation*} 
by \eqref{22051805}.

\item If $\det A_1=0$ or $\det A_4=0$, then $A_1=A_4=
\left(\begin{array}{cc}
0 & 0\\
0 & 0
\end{array}\right)$ and $\det A_2, \det A_3\neq 0$ by Corollary \ref{22051413}. By \eqref{20122307}, 
\begin{equation*}
\left( \begin{array}{c}
1  \\
\tau_2  
\end{array}\right)
\newparallel
A_2^{-1}A'_1
\left( \begin{array}{c}
1  \\
\tau_1  
\end{array}\right)\newparallel
(A'_4)^{-1}A_3
\left( \begin{array}{c}
1  \\
\tau_1  
\end{array}\right)
\end{equation*}
and, as $\tau_1$ is non-quadratic, it follows that 
\begin{equation*}
A_2^{-1}A'_1\newparallel (A'_4)^{-1}A_3
\end{equation*}
by Lemma \ref{17052301}. Since 
\begin{equation}\label{20122313}
\left( \begin{array}{c}
-q_{2}  \\
p_{2}  
\end{array}\right)
\newparallel
A_2^{-1}A'_1
\left( \begin{array}{c}
-q'_{1}  \\
p'_{1}  
\end{array}\right), \quad 
(A'_4)^{-1}A_3
\left( \begin{array}{c}
-q_{1}  \\
p_{1}  
\end{array}\right)
\newparallel
\left( \begin{array}{c}
-q'_{2}  \\
p'_{2}  
\end{array}\right)
\end{equation} 
by \eqref{22051805}, if $\big((A_2^{-1}A'_1)^T\big)^{-1}$ is denoted by $\rho$, then 
\begin{equation*}
p_{2}/q_{2}=\rho(p'_{1}/q'_{1}),\quad p'_{2}/q'_{2}=\rho(p_{1}/q_{1})
\end{equation*}
by \eqref{20122313}.  
\end{enumerate} 
\end{enumerate}
\end{enumerate}

\item We only consider the case $\tau_1\in \mathbb{Q}(\sqrt{-1})$ and $\tau_2\notin\big(\mathbb{Q}(\sqrt{-1})\cup\mathbb{Q}(\sqrt{-3})\big)$. As $\tau_1$ and $\tau_2$ are relatively independent, $A_2=A_3=\left( \begin{array}{cc}
0 & 0  \\
0 & 0 
\end{array}\right)$ by \eqref{20122307} and Lemma \ref{17052301}. So \eqref{22051601} is reduced to 
\begin{equation*}
M_1^{a_j}L_1^{b_j}(M'_1)^{a'_j}(L'_1)^{b'_j}=M_2^{g_j}L_2^{h_j}(M'_2)^{g'_j}(L'_2)^{h'_j}=1,\quad (j=1,2).
\end{equation*}
For $k=1,2$, let 
\begin{equation*}
\mathcal{X}_k\times\mathcal{X}_k:=(\mathcal{X}\times \mathcal{X})\cap (M_k=L_k=M'_k=L'_k=1). 
\end{equation*}
\begin{enumerate}
\item Since $\tau_2\notin \big(\mathbb{Q}(\sqrt{-1})\cup\mathbb{Q}(\sqrt{-3})\big)$, by Theorem \ref{20102501}, 
\begin{equation*}
\begin{gathered}
(\mathcal{X}_1\times \mathcal{X}_1)\cap \big(M_2^{g_j}L_2^{h_j}(M'_2)^{g'_j}(L'_2)^{h'_j}=1,\; (j=1,2)\big)
\end{gathered}
\end{equation*}
is a trivial anomalous subvariety of $\mathcal{X}_1\times \mathcal{X}_1$. That is, $A_1=\pm A'_1$, implying $p_1/q_1=p'_1/q'_1$. 

\item Since $\tau_1\in \mathbb{Q}(\sqrt{-1})$, if
\begin{equation*}
\begin{gathered}
(\mathcal{X}_2\times \mathcal{X}_2)\cap \big(M_1^{a_j}L_1^{b_j}(M'_1)^{a'_j}(L'_1)^{b'_j}=1,\; (j=1,2)\big)
\end{gathered}
\end{equation*}
is a non-trivial anomalous subvariety of $\mathcal{X}_2\times \mathcal{X}_2$, $(A_1^{-1}A'_1)^4=I$ by Theorem \ref{20102501}. As $(A_1^{-1}A'_1)^2=-I$ (again by Theorem \ref{20102501}), the desired result follows.
\end{enumerate} 
\item The proof for the last case is similar to the previous cases. 
\end{enumerate}
\end{proof}

\vspace{5 mm}
\noindent Department of Mathematics, POSTECH\\
77 Cheong-Am Ro, Pohang, South Korea\\
\\
\text{Email Address}: bogwang.jeon@postech.ac.kr


\begin{thebibliography}{99}
\bibitem{Ahbijit} A. ~Champanerkar, \textit{A-polynomial and Bloch invariants of hyperbolic $3$-manifolds}, preprint.
\bibitem{BMZ0} E. ~Bombieri, D. ~Masser, U. ~Zannier, \textit{Intersecting a plane with algebraic subgroups of multiplicative groups}, Ann. Scuola Norm. Sup. Pisa Cl. Sci. \textbf{7} (2008), 51-80. 
\bibitem{za} E. ~Bombieri, D. ~Masser, U. ~Zannier, \textit{Anomalous subvarieties-structure theorems and applications}, IMRN \textbf{19} (2007), 1-33.
\bibitem{BMZ1}  E. ~Bombieri, D. ~Masser, U. ~Zannier, \textit{On unlikely intersection of complex varieties with tori}, Acta Arith. \textbf{133} 4 (2008), 309-323.

\bibitem{bg} E. ~Bombieri, W. ~Gubler, \textit{Heights in Diophantine Geometry}, Cambridge University Press (2006).

\bibitem{calegari} D. ~Calegari, \textit{A note on strong geometric isolation in 3-orbifolds}, Bull. Aust. Math. Soc. \textbf{53} 2 (1996), 271–280.
\bibitem{ccgls} D. ~Cooper, M. ~Culler, H. ~Gillett, D.~D. ~Long, P.~B. ~Shalen, \textit{Plane curves associated to character varieties of knot complements}, Invent. Math. \textbf{118} (1994), 47-84.
\bibitem{hab} P. ~Habegger, \textit{On the bounded height conjecture}, IMRN \textbf{5} (2009), 860-886.
\bibitem{jeon-1} B. ~Jeon, \textit{Classification of hyperbolic Dehn fillings II: Quadratic case}, preprint.  
\bibitem{jeon1} B. ~Jeon, \textit{Hyperbolic $3$-manifolds of bounded volume and trace field degree}, ProQuest LLC, Ann Arbor, MI, 2013, Thesis (Ph.D.)–University of Illinois at Urbana-Champaign.  
\bibitem{jeon2} B. ~Jeon, \textit{Hyperbolic $3$-manifolds of bounded volume and trace field degree II}, arxiv.org/abs/1409.2069. 
\bibitem{jeon3} B. ~Jeon, \textit{On the number of Dehn fillings of a given volume}, Trans. Amer. Math. Soc. \textbf{374} (2021), 3947-3969.  
\bibitem{jeon5} B. ~Jeon, \textit{On anomalous subvarieties of holonomy varieties of hyperbolic $3$-manifolds}, arxiv.org/abs/2005.02481.
\bibitem{jeon0} B. ~Jeon, \textit{Rigidity in hyperbolic Dehn fillings}, arxiv.org/abs/1910.11159.  
\bibitem{jeon-2} B. ~Jeon, S. ~Oh, \textit{Classification of hyperbolic Dehn fillings III: Examples}, in preparation.  
\bibitem{mau} G. ~Maurin, \textit{Courbes alg\'{e}briques et \'{e}quations multiplicatives}, Math. Ann. \textbf{341} (2008), 789–824.
\bibitem{mumford} D. ~Mumford, \textit{Algebraic Geometry I: Complex Projective Varieties}, Springer (1995). 
\bibitem{rigidity} W. ~Neumann, A.~Reid, \textit{Rigidity of cusps in deformations of hyperbolic 3-orbifolds}, Math. Ann. \textbf{295} (1993), 223-237.
\bibitem{nz} W. ~Neumann, D. ~Zagier, \textit{Volumes of hyperbolic three-manifolds}, Topology \textbf{24} 3 (1985), 307-332.
\bibitem{suslin} A. A.~Suslin, \textit{Algebraic K-theory of fields}, Proceedings of the ICM, Berkeley (1986), vol. 1, 222-244.   
\bibitem{thu} W.~Thurston, \textit{The Geometry and Topology of 3-manifolds}, Princeton University Mimeographed Notes (1979).
\bibitem{Y} T.~Yoshida, \textit{The $\eta$-invariant of hyperbolic $3$-manifolds}, Invent. Math. \textbf{81} (1985), 473-514.
\bibitem{zan} U.~Zannier, \textit{Some problems of unlikely intersections in Arithmetic and Geometry}, Princeton University Press (2012). 
\bibitem{Zhang} S. Zhang, \textit{Positive line bundles on arithmetic varieties},  J. Amer. Math. Soc. \textbf{8} 1(1995), 187-221. 
\end{thebibliography}
\end{document}